\documentclass[10pt]{book}
\usepackage{amsmath,amssymb,amsfonts,amsthm} 
\usepackage{graphics}                 
\usepackage{color}                    
\usepackage{hyperref}                 
\usepackage{makeidx}
\usepackage{amsmath,amssymb,amsthm,amsmath}
\usepackage{mathtools}
\usepackage{marginnote}
\usepackage[final]{showkeys}
\usepackage{setspace}
\usepackage{a4wide}

\renewcommand{\leq}{\leqslant}
\renewcommand{\geq}{\geqslant}

\DeclareMathOperator{\diag}{diag}
\renewcommand{\hat}{\widehat}

\newtheorem{theorem}{Theorem}[section]

\newtheorem{proposition}[theorem]{Proposition}
\newtheorem{corollary}[theorem]{Corollary}

\newtheorem{lemma}[theorem]{Lemma}

\newtheorem{remark}[theorem]{Remark}

\newtheorem{definition}[theorem]{Definition}
\theoremstyle{example}
\newtheorem{example}{Example}[section]
\makeindex
\def\div{\hbox{div}}
\def\dim{\hbox{dim}}
\def\grad{\hbox{grad}}
\def\max{\hbox{max}}
\def\min{\hbox{min}}
\def\diag{\hbox{diag}}
\def\col{\hbox{col}}
\def\det{\hbox{det}}
\def\exp{\hbox{exp}}
\def\diag{\hbox{diag}}
\begin{document}
\title{PARTIAL DIFFERENTIAL EQUATIONS\\
AN INTRODUCTION}
\pagestyle{empty}
\author{A.D.R. Choudary, Saima Parveen\\Constantin Varsan\\\\\\\\First Edition\\\\\\\\\\\\\\\\\\\\Abdus Salam School of Mathematical Sciences, Lahore, Pakistan.}

\date{ }
\maketitle
\pagestyle{plain}
\pagenumbering{roman}

\noindent \textbf{\textit{A.D.Raza Choudary}}\\
Abdus Salam School of Mathematical Sciences,\\
GC University, Lahore, Pakistan.\\
choudary@cwu.edu\\

\noindent \textbf{\textit{Saima Parveen}}\\
Abdus Salam School of Mathematical Sciences,\\
GC University, Lahore, Pakistan.\\
saimashaa@gmail.com\\

\noindent \textbf{\textit{Constatin Varsan}}\\
Mathematical Institute of Romanian Academy\\
Bucharest, Romania.\\
constantin.varsan@imar.ro\\

\noindent {\large PARTIAL DIFFERENTIAL EQUATIONS: AN INTRODUCTION}\\
\noindent All rights reserved. No part of this publication may be reproduced, stored in any electronic or mechanical form, including photocopy, recording or otherwise, without the prior permission of the publisher.\\

\noindent First Edition 2010.\\

\noindent \textbf{ISBN    978-969-9236-07-6}\\
\noindent \copyright \,respective authors, 2010

\newpage
\vspace{5cm}
\begin{center}
{\huge\emph{\bigskip To our families}}
\end{center}

\newpage
\tableofcontents

\newpage

\chapter*{Introduction}
\addcontentsline{toc}{chapter}{Introduction}
\pagenumbering{arabic}
\pagestyle{myheadings}
\markboth{Introduction}{Introduction}
 This book is encompassing those mathematical
methods used in describing and solving second order partial
differential equation\index{Partial differential equations}  (PDE) of elliptic, hyperbolic and parabolic type.

\noindent Our priority is to make this difficult subject accessible to those
interested in applying mathematics using differential equations as
models. It is accomplished by adding some fundamental results\index{Fundamental!results} from
ordinary differential equations \index{Ordinary differential equations}(ODE) regarding flows and their
differentiability properties which are useful in constructing
solution of Hamilton Jacobi equations.\\
The analysis of first order Cauchy-Kowalevska system\index{Cauchy!Kowalevska theorem} is limited to their
application in constructing analytical solution for hyperbolic and
elliptic equations\index{Elliptic equation} which are frequently used in
Physics and Mechanics.\\
The exposition is subjected to a gradually presentations and the
classical methods called as Fourier\index{Fourier }, Riemann, Fredholm integral
 equations and the corresponding Green functions\index{Green!function}  are
analyzed by solving significant examples.\\
The analysis is not restricted to the linear equations. Non
linear parabolic equation\index{Parabolic equation} or elliptic equations are included
enlarging the meaning of the weak solution.\\
This university text includes a Lie-geometric analysis of gradient
systems\index{Gradient systems} of vector fields and their algebraic representation with a
direct implication in treating both first order overdetermined
 systems \index{Overdetermined system}with singularities and solutions for first order PDE.\\
We are aware that a scientific presentation of PDE must contain
an additional text introducing several results from multidimensional
analysis and it is accomplished including Gauss-Ostrogadsky\index{Gauss-Ostrogadsky}
formula, variational methods \index{Variational method}of deriving PDE and recovering
harmonic functions\index{Harmonic functions } from its boundary values
 (see appendices I,II,III of chapter III).\\
Each chapter of this book includes exercises and problems which can
be solved provided the given hints are used.\\
Partially, some subjects of this university-text have been lectured at Abdus Salam School of Mathematical Sciences (ASSMS), Lahore and it is our belief that this
presentation deserve a wider distribution among universities having graduate program in
mathematics.\\
The authors express their gratitude to Abdus Salam School of Mathematical Sciences (ASSMS) of GC University Lahore for their encouragements and
assistance in writing this book.\\
This book also includes two subjects that encompass good connection of PDE with differential geometry and stochastic analysis\index{Stochastic!analysis}.
 The first subject is represented by gradient systems\index{Gradient systems} of vector fields and their influence in solving first order overdetermined
systems\index{Overdetermined system}. The second subject introduces approximations of
SDE by  ODE which can be meaningful in deriving significant results of stochastic partial differential equations\index{Partial differential equations}($SPDE$)
represented here by stochastic rule of derivation.Our belief is that such subjects must be presented in any introductory monography treating $PDE$.
\pagestyle{headings}

\chapter{Ordinary Differential Equations(ODE)}\index{Ordinary differential equations}
\section{Linear System of Differential Equations}
A linear system of differential equations is described by the
following scalar differential equations
\begin{eqnarray}\label{eq:1.1}
  \frac{dy^i}{dx} &=& \mathop{\sum}\limits_{j=1}^{n}a_{ij}(x)y^{j}+b_{i}(x),x\in
I\subseteq \mathbb{R},i=1,2,...,n
\end{eqnarray}
where the scalar functions $a_{ij}$ and $b_i$ are  continuous on the
interval $I\subseteq \mathbb{R} \,(a_{ij} , b_i \in
\mathbb{C}(I;\mathbb{R}))$. Denote $z=col(y^1,\ldots,y^n)\in
\mathbb{R}^{n}$ as the column vector of the unknown functions and
rewrite \eqref{eq:1.1} as a linear vector system
\begin{eqnarray}\label{eq:1.2}
\frac{dz}{dx} =A(x)z + b(x), x\in I\subseteq \mathbb{R} ,
z\in \mathbb{R}^n
\end{eqnarray}
where $A=(a_{ij}),\,(i,j)\in\{1,\ldots,n\}$ stands for
the corresponding $n\times n$ continuous  matrix valued function and
$b=\col(b_1,\ldots,b_n)\in \mathbb{R}^n$ is a continuous function on
$I$ valued in $\mathbb{R}^n$. A solution of the system \eqref{eq:1.2} means
a continuous function $z(x):I\rightarrow \mathbb{R}^n$ which is
continuously differentiable $(z\in \mathcal{C}^{1}(I;\mathbb{R}^n))$
and satisfies \eqref{eq:1.2} for any $x\in I$, i.e
$\frac{dz}{dx}(x)=A(x)z(x)+ b(x)$, for all $x\in I$. The
computation of a solution as a combination of elementary functions
and integration is not possible without assuming some particular
structures regarding the involved matrix $A$. As far as some
qualitative results are concerned (existence and uniqueness of a
solution \index{Existence and uniqueness of a
solution }) we do not need to add new assumptions on the $A$ and a
Cauchy problem solution receives a positive answer. For a fixed pair
$(x_0,z_0)\in I\times \mathbb{R}^n$ we define $\{z(x):x\in I\}$ as
the solution of the system \eqref{eq:1.2} satisfying $z(x_0)=z_0$ (Cauchy
problem solution). To prove uniqueness of a Cauchy problem solution
we recall a standard lemma.
\begin{lemma}\label{le:1.1}(Gronwall) Let $\varphi(x),\alpha(x):[a,b]\rightarrow
[0,\infty)$ and a constant $M \geq 0$ be given such that the
following integral inequality is valid
$$\varphi(x)\leq M+\int_{a}^{x}\alpha(t)\varphi(t)dt,x\in[a,b]$$
where $\alpha$ and $\varphi$ are continuous scalar functions. Then$$
\varphi (x) \leq M
\exp\int_{a} ^ {b} \alpha (t)dt,\forall\, x\in [a,b]$$
\end{lemma}
\begin{proof}
 Denote $\psi(x)=M+\int_{a}^{b}\alpha(x)\varphi(t)dt$
and a straight computation lead us to $ \varphi(x)\leq\psi(x),\alpha\in[a,b]$ and
 \begin{eqnarray}\label{eq:1.3}\left\{
  \begin{array}{ll}
    \frac{d\psi(x)}{dx}= & \alpha(x)\varphi(x)\leq\alpha(x)\psi(x),\,x\in[a,b]\ \\
    \psi(a)= &M
  \end{array}
\right.
\end{eqnarray}
The differential inequality in \eqref{eq:1.3} can be written as a scalar equation
\begin{eqnarray}\label{eq:1.4}
\left\{
   \begin{array}{ll}
     \frac{d\psi(x)}{dx} =& \alpha(x)\psi(x)+\delta(x),\, x\in[a,b] \\
    \psi(a)= &M
   \end{array}
 \right.
 \end{eqnarray}
 where $\delta(x)\leq 0,\forall\,x\in\, [a,b]$. Using the integral
representation of of the scalar equation solution we get
$$\psi(x)=[\exp\int_{a}^{x}\alpha(t)dt]\{M+\int_{a}^{x}(\exp-\int_{a}^{t}\alpha(s)ds)\delta(t)dt\}$$
and as consequence(see $\delta(x)\leq 0,x\in[a,b]$)we obtain
$$\varphi(x)\leq\psi(x)
 \leq M [\exp\, \int_{a} ^ {b} \alpha (t)dt], \forall\, x\in [a,b]
$$
\end{proof}
\begin{theorem}\label{th:1.1}
(Existence and uniqueness of Cauchy problem solution)\\
Let $ A(x):I\rightarrow M_{n\times n}(n\times n \hbox{ matrices})$ and
$b(x):I\rightarrow \mathbb{R}^n $ be continuous functions
and $x_{0}\in I \subseteq \mathbb{R}^n$ be fixed. Then there exists
a unique solution $z(x):I\rightarrow \mathbb{R}^n$ of the system
satisfying the initial condition $z(x_{0})=z_{0}$ (Cauchy
condition).
\end{theorem}
\begin{proof}
We notice that a solution $\{z(x):x \in I\}$ of \eqref{eq:1.2}  satisfying $z(x_{0})=z_{0}$ fulfils the following integral equation
\begin{eqnarray}\label{eq:1.5}
z(x)=z_{0}+\int_{x_{0}}^{x}[A(t)z(t)+b(t)]dt ,\forall x \in
I\end{eqnarray}
and conversely any continuous function $z(x):I\rightarrow
\mathbb{R}^n$  which satisfies \eqref{eq:1.5} is continuously differentiable \\
and
$$\frac{dz}{dx}=A(x)z+b(x) ,\forall \,x \in I ,\  z(x_{0})=z_{0}$$
We suggest to look for a continuous function satisfying integral
equation \eqref{eq:1.5} and it is accomplished using Picard's iteration method
which involves the following sequence of continuous functions
$\{z_{k}(x):x\in I\}_{k\geq{o}}$
\begin{eqnarray}\label{eq:1.6}
z_{0}(x)=z_{0},z_{k+1}(x)=z_{0}+\int_{x_{0}}^{x}[A(t)z_{k}(t)+b(t)]dt
\end{eqnarray}
For each compact interval $J\subseteq I,x_{0}\in J$, the sequence
$\{z_{k}(x):x\in J\}_{k\geq{o}}\subseteq
\mathcal{C}(J;\mathbb{R}^n)$ (Banach space) is a Cauchy sequence in
a Banach space, where
 \begin{eqnarray}\label{eq:1.7}
 z_{k+1}(x)=z_{0}+u_{1}+u_{2}...+u_{k+1},
u_{j+1}(x)=z_{j+1}(x)-z_{j}(x),j\geq{0}
\end{eqnarray}
In this respect, $\{z_{k}(x):x\in J \} $ rewritten as in \eqref{eq:1.7}
coincides with a partial sum of the following series
\begin{eqnarray}\label{eq:1.8}
\sum(x)=z_{0} +u_{1}+u_{2}...+u_{k }+...,x \in J
\end{eqnarray}
and the convergence of $\{ \sum(x):x\in J\}\subseteq
\mathcal{C}(J;\mathbb{R}^{n})$ can be obtained using a convergent
numerical series as an upper bound for it. The corresponding
convergent numerical series is of exponential\index{Exponential!} type and it is
contained in the following estimate. Using standard induction
argument we prove
\begin{eqnarray}\label{eq:1.9}
\mid u_{j+1}(x)\mid\leq(1+\mid z_{0}\mid)
\frac{k^{j+1}|x-x_{0}|^{j+1}}{(j+1)!},\forall j\geq{0},x\in J
\end{eqnarray}
where $$ k=\max(\max_{t\in J}\mid A(t)\mid,\max_{t\in J}\mid b(t)\mid)$$
For $j=0$ we see easily that $\mid u_{1}\mid \leq (1+\mid
z_{0}\mid)k\mid(x-x_{0})\mid,x\in J$, and using
$$u_{j+2}(x)=z_{j+2}(x)-z_{j+1}(x)=\int_{x_{0}}^{x}A(t)u_{j+1}(t)dt$$
we get
\begin{eqnarray}\label{eq:1.10}
\mid u_{j+2}(x) \mid\leq k\int_{x_{0}}^{\mid
(x-x_{0})\mid}\mid u_{j+1}(t) \mid dt\leq (1+\mid
z_{0}\mid)\frac{k^{j+2}|(x-x_{0})|^{j+2}}{(j+2)!},x\in J
\end{eqnarray}
provided $\{u_{j+1}(x):x\in J\}$ satisfies \eqref{eq:1.9}. As a consequence, the inequality \eqref{eq:1.9} show us that
\begin{eqnarray}\label{eq:1.11}
\mid u_{j+1}(x) \mid \leq(1+\mid z_{0}\mid) \frac{k^{j+1}\mid
J\mid^{j+1}}{(j+1)},\forall j\geq{0},x\in J
\end{eqnarray}
where $\mid J\mid=\max_{x\in J}\mid x-x_{0}\mid$and the corresponding
convergent numerical series is given by $(1+\mid
z_{0}\mid)\exp{k\mid J\mid}$, where the constant $k\geq{0}$ is
defined in \eqref{eq:1.9}. Using\eqref{eq:1.11} into \eqref{eq:1.8} we obtain that
sequence$\{z_{k}(x):x\in I\}_{k\geq{o}}$ defined in \eqref{eq:1.6} is uniformly
convergent to a continuous function $\{z_{j}(x):x\in J\}$ and  by
passing $k\rightarrow \infty $ in \eqref{eq:1.6} we get
\begin{eqnarray}\label{eq:1.12}
z_{j}(x)=z_{0}+\int_{x_{0}}^{x}[A(t)z_{j}(t)+b(t)]dt ,\forall x
\in J\end{eqnarray}
It shows that $\{z_{j}(x):x\in J\}$ is continuously
differentiable, $z_{j}(x_{0})=z_{0}$ and satisfies \eqref{eq:1.2} for any $x\in
J\subseteq I$. Define $\{\hat{z}(x):x\in I\}$ as a "inductive limit"
of the sequence $\{z_{m}(x)=z_{j_{m}}(x) :x\in
J_{m}\}_{m\geq{1}}$\\
where
$$I=\bigcup_{m=1}^{\infty}J_{m},
J_{m+1}\supseteq J_{m},x_{0}\in J_{m}$$
Then ${\hat{z}(x)=z_{m}(x),x\in J_{m},m\geq{1}}$ is continuously
differentiable function satisfying
integral equation \eqref{eq:1.5}.\\
\textbf{Uniqueness}\\
It will be proved by contradiction and assuming another solution
$\{z_{1}(x):x\in I\}$ of (2)exits such that $z_{1 }(x_{0})=z_{0}$
and $z_{1 }(x^*)\neq \tilde{z}(x^*)$ for some $x^*\in I $ then
$u(x)=\widehat{z}(x)-z_{1 }(x),x\in I$, verifies the following linear
system
\begin{eqnarray}\label{eq:1.13}
u(x)=\int_{x_0}^{x}A(t)u(t)dt, x\in I, \hbox{where}
\{\hat{z}(x),x\in I\}\, \hbox{is defined above}\,.
\end{eqnarray}
Let $J\subseteq I$ be a
compact interval such that $x_{0},x^*\in I$ and assuming that
$x_{0}\leq x^*$ we get
\begin{eqnarray}\label{eq:1.14}
\varphi(x)=\mid u(x)\mid\leq
\int_{x_{0}}^{x_{0}+\mid x-x_{0}\mid}\alpha(t)\varphi(t)dt,\,\,\forall
x\in [x_{0},x^*]
\end{eqnarray}
where\,$ \alpha(t)=\mid A(t)\mid\geq{0}$ and $\varphi(t)=\mid
u(t)\mid$ are continuous scalar functions. Using Lemma \ref{le:1.1}  with$ M=0 $
we obtain $\varphi(x)=0$ for any $x\in [x_{0},x^*]$ contradicting
$\varphi(x^*)>{0}$. The proof for $x_{0}>x^*$ is similar. The proof is
complete.
\end{proof}
\section{ Fundamental Matrix\index{Fundamental!matrix} \index{Matrix!Fundamental} of Solution}
A linear homogenous system is described
by
\begin{eqnarray}\label{eq:1.17}
\frac{dz}{dx}=A(x)z,z\in \mathbb{R}^n
\end{eqnarray}
where the $(n\times n)$ matrix $A(x),x\in I$ is a continuous mapping
valued in $M_{n\times n}$. Let $x_{0}\in I$ be fixed and denote
$\{z_{i}(x):x\in I\}$ the unique solution of \eqref{eq:1.17} satisfying the
Cauchy condition $z_{i}(x_{0})=e_{i}\in \mathbb{R}^n,i\in \{1,2,...,n\}$\\
where $\{e_{1},...,e_{n}\}\subseteq \mathbb{R}^n$ is the canonical
basis. Denote
\begin{eqnarray}\label{eq:1.18}
C(x;x_{0})=\parallel
z_{1}(x)....z_{n}(x)\parallel,x\in I
\end{eqnarray}
where $z_{i}(x)\in
\mathbb{R}^n$ is a column vector. The matrix
$C(x;x_{0}),x\in I$ defined in \eqref{eq:1.18}, is called the
fundamental matrix\index{Fundamental!matrix}\index{Matrix!fundamental}  of solutions associated with linear system
\eqref{eq:1.17}. Let $S\subseteq C(I;\mathbb{R}^n) $ be the real linear
space consisting of all solutions verifying \eqref{eq:1.17}. The following
properties of the fundamental matrix $\{C(x;x_{0}),x\in
I\}$ are obtained by straight computation.
\subsection {Properties}
\begin{lemma}\label{le:1.2}
Let $x_{0}\in I$ be fixed and consider the fundamental matrix
$\{C(x;x_{0}),x\in I\}$  defined in\eqref{eq:1.18}
\begin{equation}\label{eq:1.3a}
Let \{z(x):x\in I\}\hbox{ be an arbitrary solution of \eqref{eq:1.17}}
\end{equation}
Then
$$z(x)=C(x;x_{0})z_{0},x\in I, \hbox{where}\,z(x_{0})=z_{0}$$
\begin{eqnarray}\label{eq:1.19}
\{C(x;x_{0}),x\in I\} \,\hbox{is continuously differentiable}
\end{eqnarray}
\begin{eqnarray}\label{eq:1.20}
\left\{
  \begin{array}{ll}
    \frac{dC(x;x_{0})}{dx}=&A(x)C(x;x_{0}),x\in I\ \\
    C(x_{0};x_{0})= &I_{n} \,\hbox{(unit matrix) }\index{Matrix!unit}
  \end{array}
\right.\end{eqnarray}
\begin{eqnarray}\label{eq:1.21}
\det\, C(x;x_{0})\neq 0,\forall x\in I \,\hbox{and} \,D(x;x_{0})=[C(x;x_{0})]^{-1} \,\hbox{satisfies}
\end{eqnarray}
\begin{eqnarray}\label{eq:1.22}
\left\{
  \begin{array}{ll}
    \frac{dD(x;x_{0})}{dx}=&-D(x;x_{0})A(x),x\in I\ \\
    D (x_{0};x_{0})= &I_{n} \,\hbox{(unit matrix)}
  \end{array}
\right.
\end{eqnarray}
\begin{eqnarray}\label{eq:1.23}
\dim \,S=n
\end{eqnarray}
\end{lemma}
 \begin{proof} For $z(.)\in S$, let
$z_{0}=z(x_{0})\in \mathbb{R}^n$ be column vector and consider  $
\hat{z}(x)=C(x;x_{0})z_{0},x\in I$. Notice that each column of $C
(x;x_{0})$ is a solution of \eqref{eq:1.17}. As far as S is a linear space
then $ \{\hat{z}(x):x\in I\}$ is a solution of \eqref{eq:1.17} fulfilling the
Cauchy condition
$$\hat{z}(x_{0})=C (x_{0};x_{0})z_{0}$$
Therefore
$$\{z(x):x\in I\} \hbox{and} \{\hat{z}(x):x\in I\}$$ are solutions for the
linear system \eqref{eq:1.17} satisfying the same Cauchy condition
$z(x_0)=z_{0}=\hat{z}(x_{0})$ and using the uniqueness of the Cauchy
problem solution (see Theorem
\ref{th:1.1}) we get
$$z(x)=\hat{z}(x)=C(x;x_{0})z_{0} \hbox{for any} x\in I$$
The conclusion \eqref{eq:1.3a} is proved. To get \eqref{eq:1.20} we notice that each
component $ \{z_{i}(x):x\in I\}$ of $\{C(x;x_{0}),x\in
I\}$, satisfies \eqref{eq:1.17} and it shows directly that
\begin{eqnarray}\label{eq:1.25}
\label{eq:1.24}
\frac{dC(x;x_{0})}{dx}=\parallel
\frac{dz_{1}}{dx}...\frac{dz_{n}}{dx}
\parallel=A(x)C(x;x_{0}),x\in I
\end{eqnarray}
and
\begin{eqnarray}\label{eq:1.26}
C(x_0;x_0)=\parallel
e_{1}...e_{n}\parallel=I_{n}
\end{eqnarray}
The property \eqref{eq:1.21} can be proved by contradiction. Assuming that
$C(x^*;x_{0})z_{0}=0$ for some $x^*\in I$ and $z_{0}\in
\mathbb{R}^n,z_{0}\neq 0$, we get that the solution
$z^*(x)=C(x;x_{0})z_{0},x\in I$ of \eqref{eq:1.17} satisfies $z^*(x^*)=0$ which
implies (see uniqueness) $z^*(x)=0,\forall x\in I$ contradicting
$z^*(x_{0})=z_{0}\neq 0$.Therefore $det C(x,x_{0})\neq 0$ for any
$x\in I$ and denote $D(x;x_{0})=[C(x;x_{0})]^{-1}$. On the other
hand, consider $\{D_{1}(x;x_{0}):x\in I\}$ as the unique solution of
the linear matrix system
\begin{eqnarray}\label{eq:1.27}
\left\{
  \begin{array}{ll}
    \frac{dD_{1}(x;x_{0})}{dx}&=-D_{1}(x;x_{0})A(x),x\in I\ \\
    D_{1} (x_{0};x_{0})&= I_{n}\hbox{(unit matrix)}
  \end{array}\right.
\end{eqnarray}
and by straight derivation we get
\begin{eqnarray}\label{eq:1.28}
D_{1} (x;x_{0}C(x;x_{0})=I_{n},\forall x\in I
\end{eqnarray}
The equation \eqref{eq:1.27} shows that $D(x;x_{0})=[C(x;x_{0})]^{-1},x\in I$.
The last property \eqref{eq:1.23} is a direct consequence of \eqref{eq:1.3a} and \eqref{eq:1.21}. Using
\eqref{eq:1.3a}, we see easily that $\dim S\leq n$ and from \eqref{eq:1.21} we obtain that
$\{z_{1}(x),...z_{n}(x):x\in I\}$ are n linearly independent solutions \index{Independent!solutions}of \eqref{eq:1.17}. The proof is complete.
\end{proof}
Any basis of S will be called a fundamental system of solutions\index{Fundamental!system of solutions}
satisfying \eqref{eq:1.17} and we shall conclude this section recalling
  Liouville'theorem.
 \begin{theorem}\label{th:1.2}
 Let the $(n\times n)$ continuous matrix \index{Matrix!continuous} $A(x)=[a_{ij}(x)]_{i,j}$ be given and consider n solutions $\{\hat{z}_{1}(x),...,\hat{z}_{n}(x):x\in I\}$
 satisfying the linear system \eqref{eq:1.17}. Then $\{W(x):x\in
  I\}$ is the solution of the following linear scalar equation
  \begin{eqnarray}\label{eq:1.29}
  \frac{dW(x)}{dx}=(TrA(x))W(x),x\in I\,,\, W(x)=\det\parallel \hat{z}_1(x),...,\hat{z}_n(x)\parallel
  \end{eqnarray}
  where
  $$TrA(x)=\mathop{\sum}\limits_{i=1}^{n}a_{ii}(x) \,\,\,\hbox{and}\,\, \, W(x)=W(x_{0})\exp\int_{x_{0}}^{x}[Tr A(t)]dt,x\in I$$ for some fixed $x_{0}\in
  I$.
  \end{theorem}
  \begin{proof}
  By definition
  $\frac{d\tilde{z}_{i}(x)}{dx}=A(x)\tilde{z}_{i}(x),x\in I,i\in
  \{1,...,n\}$ and the matrix $Z(x)=\parallel
  \tilde{z}_{1}(x)...\tilde{z}_{n}(x)\parallel$ satisfies
 \begin{eqnarray}\label{eq:1.30}
 \frac{dZ(x)}{dx}=A(x)Z(x),x\in I
 \end{eqnarray}
 Rewrite $Z(x)$ using row
 vectors
  \begin{eqnarray}\label{eq:1.31}
  Z(x)=\left(
               \begin{array}{c}
                 \varphi_{1}(x) \\
                . \\
                 . \\
                 .\\
                 \varphi_{n}(x) \\
               \end{array}
             \right)
  ,\alpha\in I
  \end{eqnarray}
  and from \eqref{eq:1.29} we get easily
    \begin{eqnarray}\label{eq:1.32}
    \frac{d\varphi_{i}(x)}{dx}=\mathop{\sum}\limits_{j=1}^{n}a_{ij}(x)\varphi_{j}(x),i\in
  \{1,...,n\},x\in I
  \end{eqnarray}
  where $(a_{i1}(x),...a_{in}(x))$ stands for
  the row $"i"$ of the involved matrix\\ $A(x)=(a_{ij}(x))_{i,j\in
  \{1,...,n\}}.$. On the other hand, the standard rule of derivation
  for a $\det\,Z(x)=W(x)$ gives us
 \begin{eqnarray}\label{eq:1.33}
 \frac {dW(x)}{dx}=\mathop{\sum}\limits_{i=1}^{n}W_{i}(x),\,
 \hbox{ where}\, \,W_{i}(x)= \det\left(
                        \begin{array}{c}
                          \varphi_{1}(x) \\
                          . \\
                          . \\
                          . \\
                         \varphi_{i-1}(x) \\
                          \frac{d\varphi_{i}(x)} {dx}\\
                          \varphi_{i+1}(x) \\
                          . \\
                          . \\
                          . \\
                          \varphi_{n}(x) \\
                        \end{array}
                      \right)
  \end{eqnarray}
  and using \eqref{eq:1.31} we obtain
  \begin{equation}\label{eq:1.13a}
  W_{i}(x)=a_{ii}(x)W(x)\,,\,\,i\in\{1,...,n\}
  \end{equation}
  Combining
  \eqref{eq:1.32} and \eqref{eq:1.13a} we get the scalar equation
  \begin{equation}\label{eq:1.14a}
  \frac{dW(x)}{dx}=(\mathop{\sum}\limits_{i=1}^{n}a_{ii}(x))W(x),x\in I
  \end{equation}
  which is the conclusion of theorem.The proof is complete.
  \end{proof}
  We shall conclude by recalling the constant variation formula used
  for integral representation of a solution satisfying a linear
  system.
  \subsection{ Constant Variation Formula}
  \begin{theorem}\label{th:1.3}(Constant variation formula)
  We are given continuous mappings
  $$A(x):I\rightarrow M_{n\times n}\,\,\hbox{and}\,\, b(x):I\rightarrow \mathbb{R}^n$$
  Let $\{\mathbb{C}(x;x_{0}),x\in I\}$ be the
  fundamental matrix \index{Mundamental!matrix} of solutions associated with the linear
  homogenous system
  \begin{equation}\label{eqw.15}
  \frac{dz}{dx}=A(x)z, \,x\in I, \,z\in \mathbb{R}^n
  \end{equation}
 Let $\{z(x),x\in I\}$be the unique solution of the linear system
 with a Cauchy condition $z(x_{0})=z_{0}$
\begin{equation}\label{eqw.16}
\left\{
  \begin{array}{ll}
    \frac{dz}{dx}=&A(x)z+b(x),x\in I\ \\
   z(x_{0})= &z_{0}
  \end{array}
\right.
\end{equation}
Then
\begin{equation}\label{eq:1.34}
z(x)=\mathcal{C}(x;x_{0})[z_{0}+\int_{x_{0}}^x\mathcal{C}^{-1}(t;x_{0})b(t)dt],x\in
I
\end{equation}
\end{theorem}
\begin{proof} By definition $\{z(x),x\in I\}$
fulfilling \eqref{eq:1.34} is a solution of the linear system \eqref{eqw.16} provided a
straight derivation calculus is used. In addition using $z(x)=z_{0}$
and uniqueness of the Cauchy problem
solution (see Theorem \ref{th:1.1}) we get the conclusion.
\end{proof}
\begin{remark}
 The constant variation formula expressed in \eqref{eq:1.34} suggest that the
 general solution of the linear system \eqref{eqw.16} can be written as a sum
 of the general solution $\mathbb{C}(x;x_{0})z_{0},x\in I$ fulfilling linear
 homogeneous system\index{Homogeneous!system} \eqref{eqw.15} and a particular solution \eqref{eqw.16} (see $
 z_{0}=0$) given by
\begin{equation}\label{eq:1.35}
\mathbb{C}(x;x_{0})\int_{x_{0}}^x\mathcal{C}^{-1}(t;x_{0})b(t)dt,x\in
 I \end{equation}
 \end{remark}
\begin{remark}
 There is no real obstruction for defining solution of a linear
 system of integral equation
\begin{equation}\label{eq:1.36}
z(x)=z_{0}+\int_{a}^x[A(t)z(t)+b(t)]dt,x\in [a,b]\subseteq \mathbb{R},z\in
 \mathbb{R}^n
\end{equation}
  where the matrix $A(x)\in M_{n\times n}$ and the vector $b(x)\in
 \mathbb{R}^n$ are piecewise continuous mappings of $x\in[a,b] $ such that
$A(x)$ and $b(x)$ are continuous functions on $[x_{i},x_{i+1})$ admitting bounded left limits
\begin{equation}
 A(x_{i+1}-0)=\mathop{lim}\limits_{x\rightarrow x_{i+1}}A(x),b(x_{i+1}-0)=\mathop{lim}\limits_{x\rightarrow
 x_{i+1}}b(x)\hbox{for each} i\in \{1,2,...,N-1\}
 \end{equation}
 Here $a=x_{0}<x_{1}<...<x_{N}=b$ is an increasing sequence such
 that $[a,b]=[x_{0},x_{N}]$. Starting with an arbitrary Cauchy condition $z(x_{0})=z_{0}\in \mathbb{R}^n$
 , we construct the corresponding solution $z(x),x\in [a,b]$, as a
 continuous mapping which is continuously differentiable on each
 open interval $(x_{i},x_{i+1}),i\in \{1,2,...,N-1\}$ such that the
 conclusions of the above given result are preserved. We have to take
 the case of these equations implying derivation
 $\frac{dz(x;x_{0},z_{0})}{dx}$ of the solution and to mention that
 they are valid on each open interval $(x_{i},x_{i+1}),i\in
 \{1,2,...,N-1\}$ where the matrix $A(x)$ and the vector b(x) are
 continuous functions. Even more, admitting that the components of the matrix $\{A(x):x\in
 [a,b]\}$ and vector $\{b(x):x\in
 [a,b]\}$are complex valued functions \index{Complex!valued functions}then the unique Cauchy problem
 solution of the linear system \eqref{eq:1.36} is defined as $\{z(x)\in \mathbb{C}^n:x\in [a,b],z(a)=z_{0}\in
 \mathbb{C}^n\}$ satisfying \eqref{eq:1.36} $\forall x\in [a,b]$.
 \end{remark}
 \section{Exercises and Some Problem Solutions}
 \subsection[Linear Constant Coefficients Equations]{Linear Constant Coefficients Equations(Fundamental System of
 Solutions\index{Fundamental!system of solutions} }
 \textbf{$(P_{1})$}. Compute the fundamental matrix of solutions $C(x,0)$  for a linear constant coefficients system
\begin{equation}\label{eq:1.37}
\frac{dz(x)}{dx}=Az,x\in \mathbb{R},z\in \mathbb{R}^n,A\in M_{n\times n}
\end{equation}
\textbf{Solution}
 Since A is a constant matrix \index{Matrix!constant}we are looking for the fundamental matrix\index{Fundamental!matrix}  of solutions
 $C(x)=C(x,0)$,\,where $x_{0}=0$ is fixed and as it is mentioned in
Lemma \eqref{le:1.2} (see equation \eqref{eq:1.19}). We need to solve the
following matrix
 equation
\begin{equation}\label{eq:1.38}
\left\{
  \begin{array}{ll}
    \frac{dC(x)}{dx}=&AC(x),x\in \mathbb{R}\ \\
   C(0)= &I_{n}
  \end{array}
\right.
\end{equation}
The computation of $\{C(x):x\in \mathbb{R}\}$ relies on the fact
that the unique matrix solution of \eqref{eq:1.37} is given by the following
matrix
exponential series
\begin{equation}\label{eq:1.39}
\exp Ax=I_{n}+\frac{x}{1!}A+...+\frac{x^k}{k!}A^k+...\end{equation}
where the uniform convergence of the matrix series on compact
intervals is a direct consequence of comparing it with a numerical
convergent series. By direct derivation we get
\begin{equation}\label{eq:1.40}
\left\{
  \begin{array}{ll}
    \frac{d(\exp Ax)}{dx}=&A+\frac{x}{1!}A^{2}+...+\frac{x^k}{k!}A^{k+1}+...\ =A(\exp\,Ax),\forall x\in \mathbb{R}\\
   (\exp Ax)_{x=0}= &I_{n}
  \end{array}
\right.\end{equation}
and it shows that $C(x)=\exp\,Ax,\forall x\in \mathbb{R}$. It
suggest the first method of computing C(x) by using partial sum of
the series \eqref{eq:1.39}\\
\textbf{$(P_2)$}. Find a fundamental system of
solutions(basis of S) for the
linear constant coefficients system
\begin{equation}\label{eq:1.41}
\frac{dz}{dx}=Az,x\in \mathbb{R},z\in \mathbb{R}^n,A\in M_{n\times n}\end{equation}
using the eigenvalues\index{Eigen!values} $\lambda\in \sigma(A)$(spectrum of A).\\
\textbf{Case I}\\ The characteristic polynomial $det(A-\lambda
I_{n})=P(\lambda)$ has n real distinct roots
$\{\lambda_{1},...\lambda_{n}\}\subseteq \mathbb{R},i.e\,
\sigma(A)=\{\lambda_{1},...\lambda_{n}\},\lambda_{i}\neq\lambda_{j},i\neq
j$. Let $v_{i}\neq0,v_{i}\in \mathbb{R}^n$ be an eigenvector\index{Eigen!vector}
corresponding to the
eigenvalue $\lambda_{i}\in \sigma(A)$, i.e
\begin{equation}\label{eq:1.42}
Av_{i}=\lambda_{i}v_{i},i\in \{1,...,n\}\end{equation}
By definition ,$\{v_{1},..,v_{n}\}\subseteq \mathbb{R}^n$is a basis
in $\mathbb{R}^n$
and define the vector functions
\begin{equation}\label{eq:1.43}
\hat{z}_{i}(x)=(\exp\lambda_{i}x)v_{i},x\in \mathbb{R},i\in
\{1,...,n\}\end{equation}
Each $\{\hat{z}_{i}(x):x\in \mathbb{R}\}$ satisfies
\eqref{eq:1.40} and $\{\hat{z}_{1}(x),...,\hat{z}_{n}(x)\}$ is a basis of the
linear space S consisting of all solutions satisfying
\eqref{eq:1.40}. Therefore,any solution$\{z(x):x\in I\}$ of the system \eqref{eq:1.40} can be
found as a linear combination of $\hat{z}_{1}(.),...,\hat{z}_{n}(.)$
and $\{\alpha_{1},...,\alpha_{n}\}$ from
$z(x)=\sum_{i=i}^n\alpha_{i}\tilde{z}_{i}(x),x\in \mathbb{R}$ will
be determined by imposing Cauchy condition $z(0)=z_{0}\subseteq \mathbb{R}^n$.\\
\textbf{Case II}\\
The characteristic polynomial $P(\lambda)=\det(A-\lambda I_{n})$ has
$n$ complex numbers\index{Complex!numbers} as roots
$\{\lambda_{1},...,\lambda_{n}\}=\sigma(A),\lambda_{i}\neq\lambda_{j},i\neq
j$. As in the real case we define $n$ complex valued solutions\index{Complex!valued solutions}
satisfying \eqref{eq:1.40}
\begin{equation}\label{eq:1.44}
\hat{z}_{j}(x)=(\exp\lambda_{j}x)v_{j},j\in \{1,...,n\}\end{equation}
where $v_{j}\subseteq \mathbb{R}^n$ is an eigenvector\index{Eigen!vector} corresponding to a
real eigenvalue $\lambda_{j}\in \sigma(A)$ and $v_{j}\in
\mathbb{C}^n$ is a complex eigenvector\index{complex!eigen vector} corresponding to the complex
eigenvalue\index{Eigen!value} $\lambda_{j}\in \sigma(A)$  such that
$v\bar{}_{j}=a_{j}-ib_{j},(v_{j}=a_{j}+ib_{j})$ is the eigenvector
corresponding to the eigen value $\overline{\lambda}_{j}=\alpha_{j}-i\beta_{j},(\lambda_{j}=\alpha_{j}+i\beta_{j},\beta_{j}\neq
0)$. From $\{\hat{z}_{1}(x),...,\hat{z}_{n}(x);x\in \mathbb{R}\}$
defined in \eqref{eq:1.44} we construct another n real solutions
$\{\hat{z}_{1}(x),...,\hat{z}_{n}(x);x\in \mathbb{R}\}$ as
follows $(n=m+2k)$. The first m real solutions
\begin{equation}\label{eq:1.45}
\hat{z}_{i}(x)=(\exp\lambda_{i}x)v_{i},i\in \{1,...,m\}\end{equation}
when $\{\lambda_{1},...,\lambda_{m}\}\subseteq \sigma(A)$ are the
real eigenvalues\index{Eigen!values} of $A$ and another $2k$ real solutions
\begin{equation}\label{eq:1.46}
\left\{
  \begin{array}{ll}
    \hat{z}_{m+j}(x)=&\hbox{Re}\, \hat{z}_{m+j}(x)=\frac{ \hat{z}_{m+j}(x)+ \hat{z}_{m+j+k}(x)}{2}\ \\
    \hat{z}_{m+k+j}(x)= &\hbox{Im}\, \hat{z}_{m+j}(x)=\frac{ \hat{z}_{m+j}(x)-
    \hat{z}_{m+k+j}(x)}{2i}
  \end{array}
\right.\end{equation}
for any $j\in \{1,...,k\}$. Here
$\{\lambda_{m+1},...,\lambda_{m+2k}\}\subseteq \sigma(A)$ are the
complex eigenvalues such that $\lambda_{m+j+k}=\lambda\bar{}_{m+j}$
for any $j\in \{1,...,k\}$. Since
$\{\hat{z}_{1}(x),...,\hat{z}_{n}(x);x\in \mathbb{R}\}$ are linearly
independent \index{Independent!linearly} over reals and the linear transformation used in \eqref{eq:1.44}
and \eqref{eq:1.45} is a nonsingular one, we get
$\{\hat{z}_{1}(x),...,\hat{z}_{n}(x);x\in \mathbb{R}\}$as a
basis of S.\\
\textbf{Case III: General Case}\\
$\sigma(A)=\{\lambda_{1},...\lambda_{d}\}\subseteq \mathbb{R}$,
where the eigenvalue\index{Eigen!value} $\lambda_{j}$ has a multiplicity $n_{j}$ and
$n=n_{1}+...+n_{d}$.In this case the canonical Jordan form of the
matrix $A$ is involved which allows to construct a basis of S using an
adequate transformation $T:\mathbb{R}^n\rightarrow \mathbb{R}^n$. It
relies on the factors decomposition of the characteristic polynomial
$\det(A-\lambda
I_{n})=P(\lambda)=(\lambda-\lambda_{1})^{n_{1}}...(\lambda-\lambda_{d})^{n_{d}}$
and using Caylay-Hamilton theorem we get
\begin{equation}\label{eq:1.47}
\left\{
  \begin{array}{ll}
  & P(A)=0 \hbox{(null  matrix)}\\
   &(A-\lambda_{1}I_{n})^{n_{1}}...(A-\lambda_{d}I_{n})^{n_{d}}= 0 \hbox{(null mapping} :\mathbb{C}^{n}\rightarrow\mathbb{C}^{n})
  \end{array}
\right.
\end{equation}
The equations \eqref{eq:1.46} are essential for finding a $n\times n$
nonsingular matrix\index{Matrix!nonsingular} $Q=\parallel v_{1}...v_{n}\parallel$ such
that the linear transformation $z=Qy$ leads us to a similar matrix\index{Matrix!similar}
$B=Q^{-1}AQ(B\sim A)$ for which the corresponding linear system
\begin{equation}\label{eq:1.48}
\frac{dy}{dx}=By,B=\diag(B_{1},...,B_{d}),\dim B_{j}=n_{j},j\in
\{1,...,d\}
\end{equation}
has a fundamental matrix of solution.
$$Y(x)=\exp Bx=\diag(\exp B_{1}x,...,\exp B_{d}x)$$
Here
$$B_{j}=\lambda_{j}I_{n_{j}}+C_{n_{j}}\,, \hbox{and} \,\left(
                                                 \begin{array}{cccccc}
                                                   0 & 0 & 0 & . & .& 0 \\
                                                   1 & 0 & 0 & . & . & 0 \\
                                                   0 & 1& 0& .& .    &  o\\
                                                   . & . & . & . & . & . \\
                                                   0 & 0 & 0 &. & 1 & 0 \\
                                                 \end{array}
                                               \right)=C_{n_{j}}$$\\
 has the nilpotent property $(C_{n_{j}})^{n_{j}}=0$ (null matrix)\index{Matrix!null}. The general case can be computed in a more attractive way when the
 linear constant coefficients system comes from n-th order scalar
 differential equation . In this respect consider a linear n-th
 order scalar differential equation
\begin{equation}\label{eq:1.49}
y^{n}(x)+a_{1}y^{(n-1)}(x)+...+a_{n}y(x)=0,x\in \mathbb{R}
\end{equation}
Denote $z(x)=(y(x),y^{(1)}(x),...,y^{(n-1)}(x))\in \mathbb{R}^n$  and
\eqref{eq:1.48} can be written as a linear system for the unknown z.
\begin{equation}\label{eq:1.50}
\frac{dz}{dx}=A_{0}z,z\in \mathbb{R}^n,x\in \mathbb{R}\end{equation}
 where $$A_{0}=\left(
                \begin{array}{c}
                  b_{1} \\
                  . \\
                  . \\
                  . \\
                  b_{n} \\
                \end{array}
              \right)
 $$
   $$b_{1}=(0,1,0,...,0),...b_{n-1}=(0,...,o,1) \hbox{and}
 \,b_{n}=(-a_{n},...,-a_{1})$$
The characteristic polynomial associated with $A_{0}$ is given by
\begin{equation}\label{eq:1.51}
P_{0}(\lambda)=det(A_{0}-\lambda
 I_{n})=\lambda^n+a_{1}\lambda^{n-1}+...+a_{n-1}\lambda+a_{n}
 \end{equation}
  where\,$a_{i}\in \mathbb{R},i\in\{1,...,n\}$ are given in \eqref{eq:1.49}. In this case, a fundamental system of solutions \index{Fundamental!system of solutions} (basis) $\{y_{1}(x),...,y_{n}(x):x\in
 \mathbb{R}\}$ for \eqref{eq:1.48} determines a basis $\{z_{1}(x),...,z_{n}(x):x\in \mathbb{R}\}$) for \eqref{eq:1.49}, where
 $$z_{i}(x)=column\{y_{i}(x),y_{i}^{(1)}(x)...,y_{i}^{(n-1)}(x)\},i\in
 \{1,...,n\}$$
 \noindent In addition, a basis $\{y_{1}(x),y_{2}(x)...,y_{n}(x)\}$
 for the scalar equation \eqref{eq:1.48} is given by
 $\textit{F}=\bigcup_{j=1}^{d}\textit F_{j}$ where
\begin{equation}\label{eq:1.52}
\textit{F}_{j}=\{(\exp\lambda_{j}x),x(\exp\lambda_{j}x),...,x^{n_{j}-1}(\exp\lambda_{j}x)\}
\end{equation}
 if the real $\lambda_{j}\in \sigma(A_{0})$ has the multiplicity degree
 $n_{j}$, and
\begin{equation}\label{eq:1.53}
\textit F_{j}=\textit F_{}j(cos)\bigcup
 \textit F_{j}(sin), \,\hbox{if the complex}\,\lambda_{j}\in
 \sigma(A_{0})\,,\,\lambda_{j}=\alpha_{j}+i\beta_{j}
 \end{equation}
 has the multiplicity degree $n_{j}$, where
 $$\textit
 F_{j}(cos)=\{(\exp\alpha_{j}x),x(\exp\alpha_{j}x),...,x^{n_{j}-1}(\exp\alpha_{j}x)\}(cos\beta_{j}x),x\in\mathbb{R}$$
$$\textit
F_{j}(sin)=\{(\exp\alpha_{j}x),x(\exp\alpha_{j}x),...,x^{n_{j}-1}(\exp\alpha_{j}x)\}(sin\beta_{j}x),x\in\mathbb{R}$$
\subsection{ Some Stability Problems and Their Solution} \textbf{($P_{1}$)}. $\hat{z}(x;z_{0})=[\exp Ax]z_{0},x\in [0,\infty)$, be the solution of the
\begin{equation}\label{eq:1.54}
\frac{dz}{dx}=Az,x\in [0,\infty),z(0)=z_{0}\in
\mathbb{R}^n
\end{equation}
We say that $\{\hat{z}(x,z_{0}):x\geq 0\}$is exponentially stable\index{Exponential !stable} if
\begin{equation}\label{eq:1.55}
\mid \hat{z}(x,z_{0})\mid \leq \mid z_{0}\mid (\exp-\gamma x),\forall
x\in [0,\infty),z_{0}\in \mathbb{R}^{n}
\end{equation}
where the constant $\gamma> 0$  does not depend on $z_{0}$. Assume that
\begin{equation}\label{eq:1.56}
\sigma (A+A^t)=\{\lambda_{1},...,\lambda_{d}\}\,\hbox{satisfies}\,
\lambda_{i}< 0 \,\hbox{for any} \,i\in\{1,...,d\}
\end{equation}
Then $\{\hat{z}(x,z_{0}):x\in [0.\infty)\}$is exponentially stable\index{Exponential !stable}\\
\textbf{Solution of $P_{1}$}\\
For each $z_{0}\in \mathbb{R}^n$, the corresponding Cauchy problem
solution  $\hat{z}(x,z_{0})=(\exp Ax)\\ z_{0},x\geq 0$, satisfies \eqref{eq:1.53} and
in addition,the scalar function $\{\varphi(x)=\mid
\hat{z}(x,z_{0})\mid^2,x\geq 0\}$ fulfils the following
differential inequality
\begin{equation}\label{eq:1.57}
\frac{d\varphi(x)}{dx}=\langle(A+A^t)\hat{z}(x,z_{0}),\hat{z}(x,z_{0})\rangle\leq
2w\varphi(x), \forall x\geq 0
\end{equation}
where
$$2w=\max\{\lambda_{1},...,\lambda_{d}\}< o$$provided the condition
\eqref{eq:1.55} is assumed. A simple explanation of this statement comes from
the diagonal representation of the symmetric matrix $(A+A^{t})$
when an orthogonal transformation $T:\mathbb{R}^{n}\rightarrow \mathbb{R}^{n},z=Ty (T=T^{-1})$is performed.We get
\begin{equation}\label{eq:1.58}
\langle (A+A^t)\hat{z}(x,z_{0}),\hat{z}(x,z_{0})\rangle=\langle
[T^{-1}(A+A^t)T]\hat{y}(x;z_{0}),\hat{y}(x;z_{0})\rangle
\end{equation}
where
$$T^{-1}(A+A^t)T=\diag(\nu_{1},...,\nu_{n})$$
and
$$\nu_{i}\in\sigma(A+A^t)\,\,\hbox{ for each}\,\,\{1,...,n\}$$
Using \eqref{eq:1.57}we see easily that
\begin{equation}\label{eq:1.59}
\langle
(A+A^t)\hat{z}(x,z_{0}),\hat{z}(x,z_{0})\rangle\leq[\max\{\lambda_{1},...\lambda_{d}\}]\mid\hat{z}(x,z_{0})\mid^{2}=2w\varphi(x)
\end{equation}
where $$2w=\max\{\lambda_{1},...,\lambda_{d}\}< o$$
and
$$\varphi(x)=\mid \hat{z}(x;,z_{0})\mid^{2}=\langle
T\hat{y}(x;z_{0}),T\hat{y}(x;z_{0})\rangle=\langle
\hat{y}(x;z_{0}),T^tT\hat{y}(x;z_{0})\rangle=\mid
\hat{y}(x;z_{0})\mid^2 $$
 are used. The inequality \eqref{eq:1.58} shows that
\eqref{eq:1.56} is valid and denoting $$0\geq
\beta(x)=\frac{d\varphi(x)}{dx}-2w\varphi(x),x\geq 0$$
we rewrite \eqref{eq:1.56} as a scalar differential equation
\begin{equation}\label{eq:1.60}
\left\{
       \begin{array}{ll}
         \frac{d\varphi(x)}{dx}= & 2w\varphi(x)+\beta(x),x\geq 0 \\
         \varphi(0)= & \mid z_{0}\mid\\
       \end{array}
     \right.
\end{equation}
where $\beta(x)\leq 0,x\in [0,\infty)$, is a continuous function.The
unique solution of \eqref{eq:1.59} can be represented by
\begin{equation}\label{eq:1.61}
\varphi(x)=[\exp 2wx][  \mid
z_{0}\mid^2+\int_{0}^x(\exp-2wt)\beta(t)dt]
\end{equation}
 and using $(\exp-2wt)\beta(x)\leq 0$ for any $t\geq 0$ we get
\begin{equation}\label{eq:1.62}
\left\{
         \begin{array}{ll}
           \varphi(x)\leq  & [\exp 2wx]\mid z_{0}\mid^2 \\
           \mid \hat{z}(x,z_{0})= & [\varphi(x)]^\frac{1}{2}\leq
[\exp wx]\mid z_{0}\mid\\
         \end{array}
       \right.
\end{equation}
for any $x\geq 0$. It shows that the exponential stability
expressed
in \eqref{eq:1.54} is valid when the condition \eqref{eq:1.55} is assumed.\\
\textbf{$(P_{2})$} (Lyapunov exponent associated with linear
system and piecewise continuous solutions)
Let $\{y(t,x)\in \mathbb{R}^n:t\geq 0\}$ be the unique solution of
the following linear system of differential equations
\begin{equation}\label{eq:1.63}
\left\{
         \begin{array}{ll}
           \frac{dy(t)}{dt}= & f(y(t),\hat{\sigma} (t)),t\geq 0,y(t)\in\mathbb{R}^n \\
           y(0)= & x\in\mathbb{R}^n\\
         \end{array}
       \right.
\end{equation}
where the vector field $f(y,\sigma):\mathbb{R}^n \times \sum
\rightarrow \mathbb{R}^n$ is a continuous mapping of $(y,\sigma)\in
\mathbb{R}^n\times \sum$ and linear with respect to $y\in
\mathbb{R}^n$
\begin{equation}\label{eq:1.64}
f(y,\sigma)=A(\sigma)y +a(\sigma),A(\sigma)\in M_{n\times
n},a(\sigma)\in\mathbb{R}^{n}\,.
\end{equation}
Here $\hat{\sigma}(t):[0,\infty)\rightarrow \sum (\hbox{bounded set})\subseteq
\mathbb{R}^d$ is an arbitrary piecewise
continuous function satisfying
\begin{equation}\label{eq:1.65}
\sigma\hat{}(t)=\sigma\hat{}(t\hat{}_{k}),t\in
[t\hat{}_{k},t\hat{}_{k+1}),k\geq0
\end{equation}
where $0=\hat{t_{0}}\leq\hat{t_{1}}...\leq\hat{t_{k}} $ is an
increasing sequence with $\mathop{lim}\limits_{k\rightarrow
\infty}\hat{t_{k}}=\infty$. The analysis will be done around a
piecewise constant trajectory $\hat{y}(t):[0,\infty)\rightarrow Y
\hbox{(bounded set)} \subseteq R^n$
such that
\begin{equation}\label{eq:1.66}
\hat{\lambda}(t)=(\hat{y}(t),\hat{\sigma}(t)):[o,\infty)\rightarrow
Y\times \sum =\Lambda \subseteq \mathbb{R}^n\times \mathbb{R}^d
\end{equation}
satisfies $\hat{\lambda}(t)=\hat{\lambda}(\hat{t}_{k}),t\in
[\hat{t}_{k},\hat{t}_{k+1}),k\geq 0$, where the increasing sequence
$\{\hat{t}_{k}\}_{k\geq0}$ is fixed in \eqref{eq:1.65}. Define a linear vector
field $g(z,\lambda)$ by
\begin{equation}\label{eq:1.67}
g(z,\lambda)=f(z+\nu;\sigma)=A(\sigma)z+f(\lambda),\lambda=(\nu,\sigma)\in
\Lambda,z\in \mathbb{R}^n\,\hbox{where}\,f(y;\sigma)
\end{equation}
is given in \eqref{eq:1.63}. Let the continuous mapping $\{z(t,x):t\geq 0\}$ be the unique
solution of the following differential equation
\begin{equation}\label{eq:1.68}
\left\{
       \begin{array}{ll}
         \frac{dz(t)}{dt}= & g(z(t);\hat{\lambda}(t)),t\geq 0 \\
         z(0) & =x
       \end{array}
     \right.
\end{equation}
where the piecewise continuous function $\{\hat{\lambda}(t),t\geq
0\}$ is given in \eqref{eq:1.65}. We may and do associate the following
piecewise continuous mapping $\{z(t,x):t\geq 0\}$ satisfying the
following linear system with jumps
\begin{equation}\label{eq:1.69}
\left\{
       \begin{array}{ll}
         \frac{d\hat{z}(t)}{dt}= & f(\hat{z}(t),\hat{\sigma}(t)),t\in
[\hat{t}_{k},\hat{t}_{k+1}),\hat{z}(t)\in\mathbb{R}^n\\
        \hat{z}(\hat{t}_{k})=
&z(\hat{t}_{k},x)+\hat{y}(t_{k}),k\geq 0
       \end{array}
     \right.
\end{equation}
where $\{z(t,x):t\geq 0\}$is the continuous mapping satisfying
\eqref{eq:1.67}. It is easily seen that the piecewise continuous mapping
$\{\hat{z}(t,x):t\geq 0\}$ can be decomposed as follows
\begin{equation}\label{eq:1.70}
\hat{z}(t,x)=z(t,x)+\hat{y}(t,x),t\geq 0,x\in \mathbb{R}^n
\end{equation}
and the asymptotic behavior\\
$\mathop{lim}\limits_{t\rightarrow\infty}\hat{z}(t,x)=\mathop{lim}\limits_{t\rightarrow\infty}\hat{y}(t),x\in\mathbb{R}^n$, is valid provided $\mathop{lim}\limits_{t\rightarrow\infty}\hat{y}(t)$ exists and
\begin{equation}\label{eq:1.71}
\mathop{lim}\limits_{t\rightarrow\infty}\mid z(t,x)\mid^2=0
\end{equation}
for $x\in\mathbb{R}^n$. We say that$\{z(t,x):t\geq 0\}$ satisfying
\eqref{eq:1.67} is asymptotically stable \index{Asymptotically stable}if \eqref{eq:1.70} is valid.
\begin{definition}
A constant $\gamma< 0$ is a Lypounov exponent
for $\{z(t,x):t\geq 0\}$ if $\{\hat{z}(t,x)=(\exp\gamma
t)z(t,x):t\geq 0\}$is asymptotically stable\index{Asymptotically stable}.
\end{definition}
\begin{remark}
Notice that if $\{\hat{z}(t,x):t\geq 0\}$ satisfies \eqref{eq:1.67} then
$\{z_{\gamma}(t,x)=(\exp\gamma t)z(t,x):t\geq 0\}$, satisfies the following
augmented linear system
\begin{equation}\label{eq:1.72}
\left\{
       \begin{array}{ll}
         \frac{dz_{\gamma}(t)}{dt}= & \gamma z_\gamma (t)+(\exp\gamma t)g(z(t),\lambda\hat{}(t)),t\geq 0\\
         z_\gamma (0)= &x
       \end{array}
     \right.
\end{equation}
In addition, the piecewise continuous mapping
\begin{equation}\label{eq:1.73}
\hat{z}_\gamma(t,x)=z_\gamma(t,x)+\hat{y }(t),t\geq 0
\end{equation}
satisfies the following system with jumps
\begin{equation}\label{eq:1.74}
\left\{
       \begin{array}{ll}
         \frac{d\hat{z}_{\gamma}(t)}{dt}= & \gamma [\hat{z}_\gamma (t)-y\hat{} (t)]+(\exp\gamma t)f(\hat{z}_\gamma(t),\lambda\hat{}(t)),t\geq 0\\
         \hat{z}_\gamma (t\hat{}_{k})= &z_\gamma
(\hat{t}_{k},x)+\hat{y}(\hat{t}_{k})\,.
       \end{array}
     \right.
\end{equation}
It shows that the Lyapunov exponent $\gamma< 0$ found for
the continuous mapping $\{z(t,x):t\geq 0\}$ satisfying \eqref{eq:1.67} gives
the answer for the following asymptotic behavior associated with
$\{\hat{z}(t,x)\}$
\begin{equation}\label{eq:1.75}
\mathop{lim}\limits_{t\rightarrow\infty}\mid \hat{z}_\gamma
(t,x)-\hat{y}(t)\mid^2=\mathop{lim}\limits_{t\rightarrow\infty}\mid
z_\gamma(t,x)\mid=0\,, \hbox{for each} \,x\in \mathbb{R}^{n}\end{equation}
and the analysis will be focussed on getting Lyapunov exponents for
the
continuous mapping $\{z(t,x):t\geq 0\}$
\end{remark}
A description of the Lyapunov exponent associated with the
continuous mapping $\{z(t,x):t\geq 0\}$  can be associated using the
corresponding integral equation satisfied by a scalar continuous
function
\begin{equation}\label{eq:1.76}
h_\gamma(t,x)=(\exp2\gamma t)h(z(t,x)),t\geq 0 ,\,\hbox{where}
\,h(z)=\mid z \mid ^2
\end{equation}
In this respect,applying standard rule of derivation we get
\begin{equation}\label{eq:1.77}
\left\{
      \begin{array}{ll}
         \frac{d\hat{z}_{\gamma}(t,x)}{dt}= &(\exp2\gamma t)[2\gamma h +L_{g} (h)](\hat{z}_{\gamma}(t,x), \hat{\lambda}(t),t\geq 0\\
      h_{\gamma} (0,x)= &\mid x\mid^{2}
      \end{array}
    \right.
\end{equation}
where the first order differential operator\index{First order differential operator} $L_{g}:P_{2}(z,\lambda)\rightarrow P_{2}(z,\lambda)$ is given
by
\begin{equation}\label{eq:1.78}
L_{g}(\varphi)(z,\lambda)=<\partial_{z}\varphi (z,\lambda)),g(z,\lambda)>
\end{equation}
Here $P_{2}(z,\lambda)$ consist of all polynomial scalar functions
of second degree with respect to the variables $z=(z_{1},...,z_{n})$
and with continuous coefficients as functions of $\lambda \in
\Lambda$. In particular, for $h\in P_{2}(z,\lambda)$, $h(z)=\mid
z\mid^2$, we obtain
\begin{eqnarray}\label{eq:1.79}
  L_{g}(\varphi)(z,\lambda) &=& <\partial_{z}h(z),g(z,\lambda)>=<[A(\sigma)+A^{*}(\sigma)z,z]>+2<f(\lambda),z>\nonumber \\
  &=& <B(\sigma)z,z>+2<f(\lambda),z>
\end{eqnarray}

where the matrix $B(\sigma)=A(\sigma)+A^{*}(\sigma)$is symmetric for
each $\sigma\in \sum \subseteq \mathbb{R}^{d}$ and for
$f(\lambda)=A(\sigma)v+a(\sigma),\lambda=(v,\sigma)\in
\Lambda=Y\times \sum$, given in \eqref{eq:1.63}. Rewrite \eqref{eq:1.76} as follows
(see\eqref{eq:1.68})
\begin{eqnarray}\label{eq:1.80}
\frac{dh_{\gamma} (t,x)}{dt}&=&-\mid \gamma \mid
h_{\gamma}(t,x)-(\exp2\gamma t)<[\mid \gamma \mid
I_{n}-B(\hat{\sigma}(t))]z(t,x)>\nonumber\\
&+&2(\exp2\gamma
t)<f(\lambda\hat{}(t)),z(t,x)>,t\geq 0\,.
\end{eqnarray}
Regarding the symmetric
matrix
$$Q_{\gamma}(\sigma)=\mid\gamma\mid
I_{n}-B(\sigma),\sigma \in \sum \subseteq \mathbb{R}^{d}$$
we notice that it can be defined as positively defined matrix
uniformly with respect to $\sigma \in \sum $
\begin{equation}\label{eq:1.81}
<Q_{\gamma}(\sigma)z,z>\geq c\mid z\mid^2,\forall z\in
\mathbb{R}^n,\sigma \in \sum , \hbox{for some}\,c\geq 0
\end{equation}
provided, $\mid
\gamma\mid\geq\parallel B\parallel$, where
$$\parallel
B\parallel=\sup_{\sigma\in \sum}\parallel B(\sigma)\parallel$$
In this respect, let $T(\sigma):\mathbb{R}^{n}\rightarrow
\mathbb{R}^n$ be an orthogonal matrix
$(T^{*}(\sigma)=T^{-1}(\sigma))$ such that
\begin{equation}\label{eq:1.82}
 \left\{
     \begin{array}{ll}
        T^{-1}(\sigma)B(\sigma)T(\sigma)= & \diag(\nu_{1}(\sigma),...,\nu_{n}(\sigma)) =\Gamma_{\gamma}(\sigma)\\
        B(\sigma)e_{j}(\sigma)=&\nu_{j}(\sigma)e_{j}(\sigma),j\in\{1,2,...,n\}
      \end{array}
    \right.
\end{equation}
    where $T(\sigma)=\parallel
    e_{1}(\sigma),...,e_{n}(\sigma)\parallel$.
On the other hand ,using the same matrix $T(\sigma)$ we get
\begin{equation}\label{eq:1.83}
T^{-1}(\sigma)Q_{\gamma}(\sigma)T(\sigma)=
\diag(\nu_{1}^{\gamma}(\sigma),...,\nu_{n}^{\gamma}(\sigma))
=\Gamma_{\gamma}(\sigma)
\end{equation}

where $\nu_{i}^{\gamma}(\gamma)=\mid\gamma \mid-\nu_{i}(\sigma)\geq c>0$
for any $\sigma\in \sum,i\in \{1,...,n\}$, provided $\mid \gamma
\mid>\parallel B\parallel$ and we get
\begin{equation}\label{eq:1.44a}
\mid \gamma \mid>\parallel B\parallel\geq
\sup_{\sigma\in\sum}\parallel B(\sigma)\parallel\geq \parallel
B(\sigma)e_{j}(\sigma)\parallel =\sup_{\sigma\in\sum}\mid
\nu_{j}(\sigma)\mid,\forall j\in\{1,...,n\}
\end{equation}
Using \eqref{eq:1.82}, it makes  sense to consider the square root of the
positively defined matrix $Q_{\gamma}(\sigma)$
\begin{equation}\label{eq:1.84}
\sqrt{[Q_{\gamma}(\sigma)]}=T(\sigma)[\Gamma_{\gamma}(\sigma)]^{\frac{1}{2}}.T^{-1}(\sigma)=P_{\gamma}(\sigma),\sigma
\in \sum \subseteq\mathbb{R}^{d}
\end{equation}
and rewrite $<Q_{\gamma}(\sigma)z,z>-2<f(\lambda),z>=\varphi
_{\gamma}(z,\lambda)$ as follows
\begin{equation}\label{eq:1.85}
\varphi(z,\lambda)=\mid
P_{\gamma}(\sigma)z-R_{\gamma}(\sigma)f(\lambda)\mid^{2}-\mid
R_{\gamma}(\sigma)f(\lambda)\mid^2
\end{equation}
where
$$R_{\gamma}(\sigma)=\sqrt{Q_{\gamma}^{-1}(\sigma)}=T(\sigma)[\Gamma_{\gamma}(\sigma)]^{-1/2}T^{-1}(\sigma)\,\,.$$
Using \eqref{eq:1.84} we get the following differential equation
\begin{eqnarray}\label{eq:1.47a}
\frac{dh_{\gamma}(t;x)}{dt}&=&-\mid \gamma \mid
h_{\gamma}(t;x)-(\exp2\gamma t)\mid
P_{\gamma}(\hat{\sigma}(t)z(t;x)-R_{\gamma}(\hat{\sigma},t)f(\hat{\lambda}(t)))\mid^{2}\nonumber\\
&+&(\exp2\gamma
t)\mid R_{\gamma}(\hat{\sigma}(t))f(\hat{\lambda})(t)\mid^2,t\geq
0
\end{eqnarray}
The integral representation of the solution $\{h_{\gamma}(t;x):t\geq 0\}$ fulfilling \eqref{eq:1.47a} leads us directly to
\begin{eqnarray}\label{eq:1.48a}
h_{\gamma}(t;x)&=&(\exp\gamma t)[\mid x \mid
^{2}+\int_{0}^{t}(\exp\gamma s)\mid
R_{\gamma}(\sigma\hat{})(s)f(\lambda\hat{}(s))\mid
^{2}ds]\\
&-&(\exp\gamma t)\int_{0}^{t}(\exp\gamma
s)\mid P_{\gamma}(\hat{\sigma}(t)z(s;x)
-R_{\gamma}(\hat{\sigma})(s)f(\hat{\lambda}(s)))\mid^2
ds\nonumber
\end{eqnarray}
for any $t\geq 0 and x\in \mathbb{R}^{n}$. As far as $\{R_{\gamma}(\hat{\sigma})(t)f(\hat{\lambda}(t)):t\geq
0\}$ is a continuous and bounded function on $[0,\infty)$ we
obtain that
\begin{equation}\label{eq:1.86}
\int_{0}^{t}(\exp\gamma s)\mid
R_{\gamma}(\sigma\hat{})(s)f(\lambda\hat{}(s))\mid ^{2}ds\leq
C_{\gamma},\forall t\geq 0
\end{equation}
and
\begin{equation}\label{eq:1.87}
h_{\gamma }(t,x)\leq (\exp\gamma t)[\mid x \mid
^{2}+C_{\gamma}],\forall t\geq 0
\end{equation}
where $C_{\gamma}$ is a constant. In conclusion, for each $\gamma < 0,\mid \gamma \mid >
\sup_{\sigma\in\sum}\mid A(\sigma)+A^{*}(\sigma)\mid$, we obtain
 $$\mathop{lim}\limits_{t\rightarrow \infty}\mid z_{\gamma}(t,x)\mid
^{2}=\mathop{lim}\limits_{t\rightarrow \infty}h_{\gamma}(t,x)=0$$ for each
$x\in\mathbb{R}$. The above given computations  can be stated as
\begin{theorem}\label{th:1.4}  Let the vector field
$f(y,\sigma):\mathbb{R}^{n}\times \sum\rightarrow \mathbb{R}^{n}$ be
given such that \eqref{eq:1.63} is satisfied. Then any $\gamma < 0$ satisfying
$\mid \gamma \mid>\sup_{\sigma\in\sum}\mid
A(\sigma)+A(\sigma^{*})\mid$ is a Lyapunov exponent for the
continuous mapping $\{z(t,x):t\geq 0\}$verifying \eqref{eq:1.67}, where
$\hat{\lambda}(t)=(\hat{y}(t),\hat{\sigma}(t)):[0,\infty)\rightarrow
\Lambda$ is fixed arbitrarily. In addition, let $\{\hat{z}(t,x):t\geq
0\}$ be the piecewise continuous solution fulfilling the
corresponding system with  jumps \eqref{eq:1.73}. Then $\ \mathop{lim}\limits_{t\rightarrow
\infty}\mid \hat{z}_{\gamma}(t,x)-\hat{y}(t)\mid=\mathop{lim}\limits_{t\rightarrow
\infty}\mid z_{\gamma}(t,x)\mid=0$ for each $x\in \mathbb{R}^{n}$
\end{theorem}
\begin{remark}
 The result stated in Theorem \ref{th:1.4} make use of
some bounds $\sup_{\sigma\in\Sigma}\mid
A(\sigma)+A(\sigma^{*})\mid=\parallel B\parallel$ associated with
the unknown matrix $A(\sigma),\sigma \in \Sigma\,\hbox{(bounded set)}\subseteq
\mathbb{R}^{d}$. In the particular case $\Sigma=\{\sigma_{1},...,\sigma_{m}\}$, we get a
finite set of matrices ${A(\sigma_{1}),...A(\sigma_{m})}$ for which
$\parallel B\parallel=\max \{\mid A_{1}+A_{1}^{*}\mid,...,\mid
A_{m}+A_{m}^{*}\mid\}$ where $A_{i}=A(\sigma_{i})$.\\
An estimate of the the Theorem \ref{th:1.4} stating that $lim _{t\rightarrow
\infty}(\exp\gamma t)\mid z(t,x)\mid=0$\\
provided $\mid \gamma\mid=-\gamma>\parallel B\parallel$ and  in
addition ${z(t,x),t\geq 0}$ can be measured as continuous solution
of the system \eqref{eq:1.67}. This information can be used for estimating the
bounds $\parallel B\parallel$and as far as we get $lim
_{t\rightarrow \infty}(\exp\gamma\hat{t})\mid z(t,x)\mid=0$\\
for some $\gamma\hat{}< 0$. We may predict that $\parallel
B\parallel<\mid \gamma \hat{}\mid$. On the other hand ,if we are able to measure only some projections\\
$\nu_{i}(t,x)=<b_{i},z(t,x)>,i\in \{1,2,..,m\}$ of the solution
${z(t,x):t\geq 0}$ satisfying \eqref{eq:1.67} then
$\nu(t,x)=(\nu_{1}(t,x),...\nu_{m}(t,x))$ fulfils the following
linear equation
\begin{equation}\label{eq:1.88}
\left\{
       \begin{array}{ll}\\
         \frac{d\nu(t,x)}{dt}= &D(\gamma\hat{}(t))\nu(t,x)+d(\lambda\hat{}(t)),t\geq 0\\
        \nu(0,x)= &\nu^{0}(x)=(b_{1},x>,...,b_{m},x>)
       \end{array}
     \right.
\end{equation}
where $D(\sigma)$ is an $(m\times m)$ continuous matrix\index{Matrix!continuous} and\\
$d(\sigma)=column(<b_{1},f(\lambda)>,...,<b_{m},f(\lambda)>) $\\
Here we have assumed that
$A^{*}(\sigma)[b_{1},...,b_{m}]=[b_{1},...,b_{m}]D^{*}(\sigma)$ and
\eqref{th:1.4} gets the corresponding version when \eqref{eq:1.87} replaces the
original system \eqref{eq:1.67}
\end{remark}
\section{Nonlinear Systems of Differential Equations}
Let $f(x,\lambda,y):I\times\Lambda \times G\rightarrow
\mathbb{R}^{n}$be a continuous function,where $I(interval)\subseteq
\mathbb{R}$ and $\Lambda \subseteq \mathbb{R}^{m}$ are some open
sets. Consider a system of differential equations(normal form)
\begin{equation}\label{eq:1.89}
\frac{dy}{dx}=f(x,\lambda,y)
\end{equation}
By a solution of \eqref{eq:1.88} we mean a continuous function
$y(x,\lambda):J\times \Sigma\rightarrow G$ which is
continuously derivable with respect to $x\in J$ such that
$$\frac{dy(x,\lambda)}{dx}=f(x,\lambda,y(x,\lambda)),\forall x\in
J,\lambda\in \Sigma$$
where $J\subseteq I $ is an interval and
$\Sigma \subseteq\Lambda$ is a compact set. The Cauchy problem for the nonlinear system \eqref{eq:1.88}
$C.P (f;x_{0},y_{0})$, has the meaning that we must determine a
solution of (1)$y(x,\lambda):J\times \Sigma\rightarrow G$ which
satisfies $y(x_{0},\lambda)=y_{0},\lambda\in \Sigma$,where $x_{0}\in
I$and $y_{0}\in G$ are fixed. To get a unique $C.P(f;x_{0},y_{0})$ solution we need to replace the
continuity property of f by a Lipschitz condition with respect to
$y\in G$
\begin{definition} We say that a continuous function $f(x,\lambda,y):I\times\Lambda\times G\rightarrow \mathbb{R}^{n}$ is locally Lipschitz continuous with respect to $y\in G$ if for each compact set $W=J\times\sum\times K\subseteq I\times\Lambda\times G$ there exists a constant $L(W)>0  $ such that$\mid f(x,\lambda,y'')-f(x,\lambda,y')\leq L\mid y''-y'\mid,\,\forall\,y',y''\in K,x\in J,\lambda\in\Lambda$
\end{definition}
\begin{definition}We say that a $C.P (f;x_{0},y_{0})$ solution $\{y(x;\lambda):x\in J,\lambda\in\sum\}$is unique if for any other solution $\{y_{1}(x;\lambda):x\in J_{1},\lambda\in\sum_{1}\}$ of \eqref{eq:1.89} which verifies $y_{1}(x_{0})=y_{0}$,we get $y(x,\lambda)=y_{1}(x,\lambda),\forall\,(x,\lambda)\in(J\times\sum)\bigcap(J_{1}\times\sum_{1})$.Denote $B(y_{0},b)\subseteq\mathbb{R}_{n}$ the ball centered at $y_{0}\in\mathbb{R}^{n}$\,whose radius is $b>0$
\end{definition}
\begin{remark}By a straight computation we get that if the right hand side of \eqref{eq:1.88} is continuously differentiable function with respect to $y\in G$, i.e
$$\frac{\partial f}{\partial y_{i}}(x,\lambda,y):I\times\Lambda\times G\rightarrow \mathbb{R}^{n},i\in\{1,...,n\}$$ are continuous functions and $G$ is a convex domain then $f$ is locally Lipschitz with respect to $y\in G $
\end{remark}
\subsection{Existence and Uniqueness\index{Existence and uniqueness of a
solution }
of\,C.P$(f,x_{0},y_{0})$}
 \begin{theorem}\label{th:1.5} (Cauchy Lipschitz\index{Cauchy!Lipschitz}) Let the
continuous function $f(x;\lambda,y):I\times \Lambda \times
G\rightarrow \mathbb{R}^{n}$ be locally Lipschitz continuous with
respect to $y\in G$, where $I\subseteq \mathbb{R}, \Lambda\subseteq
\mathbb{R}^{m},G\subseteq \mathbb{R}^{n}$ are open sets.For some
$x_{0}\in I,y_{0}\in G \,\,\hbox{and}\,\, \Sigma(compact)\subseteq \Lambda$
fixed, we take $a,b > 0$ such that
$I_{a}(x_{0})=[x_{0}-a,x_{0}+a]\subseteq I$and $B(y_{0},b)\subseteq
G$. Let $M=\max\{\mid f(x,\lambda,y)\mid :x\in I_{a}(x_{0}),\lambda\in
\Sigma,\,y\in B(y_{0},b)\}$. Then there exist $\alpha >
0, \alpha=\min(a,\frac{b}{M})$ and a unique C.P $(f;x_{0},y_{0})$
solution $y(x,\lambda):I_{\alpha}(x_{0})\times \Sigma\rightarrow
B(y_{0},b)$ of \eqref{eq:1.89}.
\end{theorem}
\begin{proof}
We associate the corresponding integral equation (as
in the linear case)
\begin{equation}\label{eq:1.90}
y(x,\lambda)=y_{0}+\int_{x_{0}}^{x}f(t,\lambda,y(t,\lambda))dt,x\in
I_{a}(x_{0}),\lambda\in \Sigma
\end{equation}
where $I_{a}(x_{0})\subseteq I$ and $\Sigma(compact)\subseteq
\Lambda$ are fixed. By a direct inspection, we see that the two systems \eqref{eq:1.88} and \eqref{eq:1.89} are
equivalent using their solutions and the existence of solution for
\eqref{eq:1.88} with $y(x_{0})=y_{0}$ will be obtained proving that \eqref{eq:1.89} has a
solution.In this respect, a sequence of continuous functions
$\{y_{k}(x,\lambda):(x,\lambda)\in I_{\alpha}(t_{0})\times
\Sigma\}_{k\geq 0}$ is constructed such that
\begin{equation}\label{eq:1.91}
y_{0}(x,\lambda)=y_{0},y_{k+1}(x,\lambda)=y_{0}+\int_{x_{0}}^{x}f(t,\lambda,y_{k}(t,\lambda))dt,k\geq
0
\end{equation}
Consider $\alpha =\min(a,\frac{b}{M})$ and
$I_{\alpha}(x_{0})=[x_{0}-\alpha,x_{0}+\alpha]$. We see easily that
the sequence $\{y_{k}(.)\}_{k\geq 0}$ constructed in \eqref{eq:1.90} is uniformly
bounded if the variable $x$ is restricted to $x\in
I_{\alpha}(x_{0})\subseteq I_{a}(x_{0})$. More precisely
\begin{equation}\label{eq:1.92}
y_{k}(x,\lambda)\in B(y_{0},b),\forall (x,\lambda)\in
I_{\alpha}(x_{0})\times\Sigma,k\geq 0
\end{equation}
It will be proved by induction and assuming that \eqref{eq:1.91} is satisfied
for $k\geq 0$ we compute
\begin{equation}\label{eq:1.93}
\mid y_{k+1}(x,\lambda)-y_{0}\mid\leq \int_{x_{0}}^{x_{0}+\mid
x-x_{0}\mid}f(t,\lambda,y_{k}(t,\lambda))dt\leq M\alpha \leq b
\end{equation}
for
any
$\{(x,\lambda)\in I_{\alpha}(x_{0})\times \Sigma\}$.
Next step is to notice that $\{y_{k}(x,\lambda):(x,\lambda)\in
I_{\alpha}(x_{0})\times \Sigma\}_{k\geq 0}$ is a Cauchy sequence\index{Cauchy!sequence} in a
Banach space $C(I_{\alpha}(x_{0}))\times \sigma _{j} \mathbb{R}^{n}$
and it is implied by the following estimates
\begin{equation}\label{eq:1.94}
\mid y_{k+1}(x,\lambda)-y_{k}(x,\lambda)\mid \leq bL^{k}\frac
{\mid x-x_{0}\mid^{k}}{k!},\forall(x,\lambda)\in
I_{\alpha}(x_{0})\times \Sigma,k\geq 0
\end{equation}
where $L=L(W)> 0$ is a
Lipschitz constant associated with  $f$
and $W=I_{\alpha}(x_{0})\times \Sigma\times B(y_{0},b)$, a compact
set of $I\times \Lambda\times G$. A verification of \eqref{eq:1.93} uses the
standard induction argument and for $k=0$ they are proved in
\eqref{eq:1.92}. Assuming \eqref{eq:1.93} for $k$ we compute $\mid
y_{k+2}(x,\lambda)-y_{k+1}(x,\lambda)\mid\leq
\int_{x_{0}}^{x_{0}+\mid x-x_{0}\mid}\mid
f(t,\lambda,y_{k+1}(t,\lambda))- f(t,\lambda,y_{k}(t,\lambda))\mid
dt\leq L\int_{x_{0}}^{x_{0}+\mid x-x_{0}\mid} \mid
y_{k+1}(t,\lambda)-y_{k}(t,\lambda)\mid dt\leq L^{k+1}\frac{\mid
x-x_{0}\mid^{k+1}}{k+1!}$ and \eqref{eq:1.93} is verified.Rewrite \eqref{eq:1.93}\\
$y_{k+1}(x,\lambda)=y_{0}+(y_{1}(x,\lambda)-y_{0})+...,+y_{k+1}(x,\lambda)-y_{k}(x,\lambda)=\\
=y_{0}+\Sigma_{j=0}^{k+1}u_{j}(x,\lambda)$, where
$u_{j}=y_{j+1}-y_{j}$\\and consider the following series of continuous functions
\begin{equation}\label{eq:1.95}
S(x,\lambda)=y_{0}+\Sigma_{j=0}^{\infty}u_{j}(x,\lambda)
\end{equation}
The series \eqref{eq:1.94} is convergent in the Banach space
$C(I_{\alpha}(x_{0})\times \Sigma;\mathbb{R}^{n})$ if it is bounded
by a numerical convergent series and notice that each
$\{u_{j}(x,\lambda):(x,\lambda)\in I_{\alpha}(x_{0})\times \Sigma\}$
of \eqref{eq:1.94} satisfies (see\eqref{eq:1.93})
\begin{equation}\label{eq:1.96}
\mid u_{j}(x,\lambda)\mid\leq b\frac{L^{j}\mid
x-x_{0}\mid^{j}}{j!}\leq
b\frac{L.\alpha^{j}}{j!},\,\,j\geq 0
\end{equation}
In conclusion, the series $S(x,\lambda)$ given in \ref{eq:1.95} is bounded by
the following series
$$u=\mid y_{0}\mid
+b(1+\frac{L\alpha}{1}+...+\frac{L\alpha^{k}}{k!}+...)=\mid
y_{0}\mid+b\,\exp\,L\alpha$$
and the sequence of continuous functions
$$\{y_{k}(x,\lambda):(x,\lambda)\in I_{\alpha}(x_{0})\times
\Sigma\}_{k\geq} 0$$
constructed in \eqref{eq:1.90} is uniformly convergent to a
continuous function
\begin{equation}\label{eq:1.97}
y(x,\lambda)=\mathop{lim}\limits_{k\rightarrow \infty}y_{k}(x,\lambda)\,
\hbox{uniformly on} \,(x,\lambda)\in I_{\alpha}(x_{0})\times \Sigma
\end{equation}
It allows to pass $k\rightarrow \infty$ into integral equation \eqref{eq:1.90} and
we get
\begin{equation}\label{eq:1.98}
y(x,\lambda)=y_{0}+\int_{x_{0}}^{x}f(t,\lambda,y(t,\lambda))dt,\forall
(x,\lambda)\in I_{\alpha}(x_{0})\times \Sigma
\end{equation}
which proves the existence of the\,$C.P(f;x_{0},y_{0})$ solution.\\
\textbf{Uniqueness}. Let
$y_{1}(x,\lambda):J_{1}\times\Sigma_{1}\rightarrow G$ another
solution of \eqref{eq:1.88} satisfying $y_{1}(x_{0})=y_{0}$, where $J_{1}$(compact
set)$\subseteq I$ and $\Sigma_{1}(compact)\subseteq \Lambda$. Define
a compact set $K_{1}\subseteq G$ such that it contains all values of
the continuous function
$$\{y_{1}(x,\lambda):(x,\lambda)\in
J_{1}\times \Sigma_{1}\}\subseteq K_{1}$$
Denote
$$\widetilde{K}=B(y_{0},b)\cap K_{1}
,\widetilde{J}=I_{\alpha}(x_{0})\cap
J_{1},\hat{\Sigma}=\Sigma\cap \Sigma_{1}$$
and let
$$\widetilde{L}=L(\widetilde{J}\times \widetilde{\Sigma}\times
\widetilde{K})> 0$$
 be the corresponding Lipschitz constant
associated with $f$ and compact set
$\widetilde{W}=\widetilde{J}\times \widetilde{\Sigma}\times
\widetilde{K}$.We have
$\widetilde{J}=[x_{0}-\delta_{1},x_{0}+\delta_{2}]$ for some $\delta_{i}\geq 0$ and
\begin{equation}\label{eq:1.11b}
\mid y(x,\lambda)-y_{1}(x,\lambda)\mid\leq \widetilde{L}\int
_{x_{0}}^{x_{0}+\mid x-x_{0}\mid}\mid
y(t,\lambda)-y_{1}(t,\lambda)\mid dt,\forall (x,\lambda)\in
\widetilde{J}\times \widetilde{\Sigma}
\end{equation}
For $\lambda\in \widetilde{\Sigma}$ fixed,denote $\varphi(x)=\mid
y(x,\lambda)-y_{1}(x,\lambda)\mid$ and inequality \eqref{eq:1.11b} becomes\\
\begin{equation}\label{eq:1.12b}
\varphi \leq \widetilde{L}\int_{x_{0}}^{t}\varphi(s)ds,t\in
[x_{0},x_{0}+\delta_{2}]
\end{equation}
which shows that $\varphi(t)=0,\forall t\in
[x_{0},x_{0}+\delta_{2}]$(see Gronwall Lemma in \eqref{le:1.1})\\
Similarly we get $\varphi(t)=0 ,t\in [x_{0}-\delta_{1},x_{0}]$\\
and\ $\varphi(t)=0 \,\,\forall \,t\in \widetilde{J}$, lead us to the
conclusion $y(x,\lambda)=y_{1}(x,\lambda),\forall x\in
\widetilde{J}$ and for an arbitrary fixed $\lambda\in
\widetilde{\Sigma}$. The proof is
complete.
\end{proof}
\begin{remark} The local Lipschitz continuity of the function $f$
is essential for getting uniqueness of a C.P solution. Assuming that
$f$ is only a continuous function of $y\in G$,we can construct
examples supporting the idea that a C.P solution is not unique. In
this respect, consider the scalar function $f(y)=2\sqrt{\mid
y\mid},y\in \mathbb{R}$ and the equation
$\frac{dy(t)}{dt}=2\sqrt{\mid y(t)\mid}$ with $y(0)=0$. There are two
C.P solution.$y_{1}(t)=0,t\geq 0$ and $y_{2}(t)=t^2,t\geq 0$ where a continuous but not a Lipschitz continuous function was used.
\end{remark}
\textbf{Comment}. There is a general fixed point theorem which can be used for proving the existence
and uniqueness of a
Cauchy problem solution\index{Existence and uniqueness of a
Cauchy problem solution } .In this respect we shall recall so called
fixed point theorem associated with contractive mappings. Let $T:X\rightarrow X$be a continuous mapping satisfying
$$\rho(Tx,Ty)\leq \alpha \rho(x,y) \,\hbox{ for any} \, x,y\in X$$
where $0<\alpha<1$ is a constant and $(X,\rho)$ is a complete metric
space. A fixed point for the mapping $T$ satisfies $T\hat{x}=\hat{x}$
and it can be obtained as a limit point of the following Cauchy
sequence $\{x_{n}\}_{n\geq 0}$ defined by $x_{n+1}=Tx_{n},n\geq
0$. By definition we get
$\rho(x_{k+1},x_{k})=\rho(Tx_{k},Tx_{k-1})\leq \alpha\rho
(x_{k},x_{k-1})\leq \alpha^{k}\rho(x_{1},x_{0})$ for any $k\geq 1$
and $\rho(Tx_{k+m},Tx_{k})\leq
\Sigma_{j=1}^{m}\rho(x_{k+j},x_{k+j-1})\leq(\Sigma_{j=1}^{m}\alpha^{k+j-1})\rho(x_{1},x_{0})$, where
$\Sigma_{k=0}^{\infty}\alpha^{k}=\frac{1}{1-\alpha}$ and
$\{x_{n}\}_{n\geq 1}$ is Cauchy sequence.
\subsection{Differentiability of Solutions with Respect
to \\Parameters} In Theorem \ref{th:1.5}we have obtained the continuity
property of the $C.P(f,x_{0},y_{0})$ solution with respect to
parameters $\lambda\in\Lambda$ satisfying a differential
system. Assume that the continuous function
$f(x,\lambda,y):I\times\Lambda\times G\rightarrow \mathbb{R}^{n}$ is
continuously differentiable with respect $y\in G$ and
$\lambda\in\Lambda$ i.e there exist continuous partial derivatives
\begin{equation}\label{eq:1.13b}
\frac{\partial{f(x,\lambda,y)}}{\partial{y_{}}},\frac{\partial{f(x,\lambda,y)}}{\partial{\lambda_{j}}}:I\times\Lambda\times
G\rightarrow\mathbb{R}^{n},i\in\{1,...,n\},j\in\{1,...,m\}
\end{equation}
where
$\Lambda\subseteq\mathbb{R}^{m},G\subseteq\mathbb{R}^{n}$ are open
sets.
\begin{remark}\label{re:1.1}The assumption \eqref{eq:1.13b} leads us directly to the local
Lipschitz property of $f$ with respect to $y\in G$. In addition, let
$y(x,\lambda):I_{\alpha}(x_{0})\times \Sigma\rightarrow
B(y_{0},b)$ be the $C.P(f;x_{0},y_{0})$ solution and define\\
$F(x,z(x))=f(x,\lambda,y(x,\lambda)),z(x)=(\lambda,y(x,\lambda))$
for each $x\in I_{\alpha}(x_{0})$. Using \eqref{eq:1.13b} we get that $F(x,z(x))$
satisfies the following differentiability property
\begin{eqnarray}\label{eq:1.14b}
F(x,z''(x)-F(x,z'(x)))&=&\frac{\partial
f(x,\lambda',y(x,\lambda'))}{\partial
y}[y(x,\lambda'')-y(x,\lambda')]\nonumber\\
&+&\sum_{j=1}^{m}\frac{\partial
f(x,\lambda',y(x,\lambda'))(\lambda''_{j}-\lambda'_{j})}{\partial
\lambda_{j}}+\theta (x,\lambda',\lambda'')(\mid
y(x,\lambda'')\nonumber\\&-&y(x,\lambda')\mid+\mid \lambda''-\lambda' \mid)
\end{eqnarray}
where$$\mathop{lim}\limits_{\lambda''\rightarrow
\lambda'}\theta(x,\lambda',\lambda'')=0$$
uniformly with respect
to $x\in I_{\alpha}(x_{0})$. To get \eqref{eq:1.14b} we rewrite
$F(x,z''(x))-F(x,z'(x))$ as follows
$$F(x,z''(x))-F(x,z'(x))=\int_{0}^{1}[\frac{dh}{d\theta}(x,\theta)]d\theta$$
where
$$h(x,\theta)=F(x,z'(x)+\theta(z''(x)-z'(x)),\theta\in[0,1],x\in
I_{\alpha}(x_{0})$$
The computation of the derivatives  allows to
see easily that \eqref{eq:1.14b} is valid. In addition using the assumption
\eqref{eq:1.13b}we obtain the Lipschitz continuity of the solution
$\{y(t,\lambda):\lambda\in\sum\}$ and
\begin{equation}\label{eq:1.15b}
\mid y(x,\lambda'')-y(x,\lambda') \mid\leq C \mid
\lambda''-\lambda'\mid ,\forall x\in
I_{\alpha}(x_{0}),\lambda',\lambda''\in
B(\lambda_{0},\beta)=\Sigma
\end{equation}
where $C>0$ is a constant. The
property \eqref{eq:1.15b} is obtained applying lemma Gronwall for the integral
inequality associated with the equation
$$y(x,\lambda'')-y(x,\lambda')=\int
_{0}^{x}[F(t,z''(t))-F(t,z'(t))]dt,x\in I_{\alpha}(x_{0})$$
where $F(x,z(x))$ fulfils \eqref{eq:1.14b}.
\end{remark}
\begin{theorem}\label{th:1.7}(differentiability of a
solution)\\
Let $f(x,\lambda,y):I\times\Lambda\times G \rightarrow
\mathbb{R}^{n}$ be given such that \eqref{eq:1.13b}is satisfied. Let$x_{0}\in
I,y_{0}\in G, \lambda_0\in \Lambda$ be fixed and define
$a,b,\beta>0$ such that $I_a(x_{0})\subseteq I, B(y_{0},b)\subseteq
G,\sum= B(\lambda_0,\beta)\subseteq\Lambda$. Then there
exist $\alpha>0$ and $y(x,\lambda):
I_{\alpha}(x_{0})\times\sum\rightarrow B(y_0,b)$  as a unique
$C.P(f,x_{0},y_{0})$ solution for \eqref{eq:1.88} such that for each
$\hat{\lambda}\in int\Sigma$ there exist $\frac{\partial
y(x,\hat{\lambda})}{\partial \lambda_j}=\hat{y}_{j}(x),x\in
I_{\alpha}(x_{0}),j\in
{1,...,m}$,satisfying the following linear system
\begin{equation}\label{eq:1.99}
\frac{dz}{dx}=\frac{\partial
f(x,\hat{\lambda},y(x,\hat{\lambda}))}{\partial y}z+\frac{\partial
f(x,\hat{\lambda},y(x,\hat{\lambda}))}{\partial
\lambda_j},\,\,z(x_{0})=0
\end{equation}
$x\in
I_{\alpha}(x_{0})=[x_{0}-\alpha,x_{0}+\alpha]$ for each $j\in
\{1,...,m\}$.
\end{theorem}
\begin{proof}
By hypothesis, the conditions
of Theorem \ref{th:1.5} are fulfilled and let
$y(x,\lambda):I_{\alpha}(x_{0})\times \Sigma \rightarrow B(y_{0},b)$
be the unique solution of \eqref{eq:1.88} satisfying
$y(x_{0},\lambda)=y_{0},\lambda \in \Sigma$.\\
Consider $\hat{\lambda}\in int\Sigma$ and notice that
differentiability of the solution with respect to parameters at
$\lambda=\hat{\lambda}$ is equivalent to showing
\begin{equation}\label{eq:1.100}
\mathop{lim}\limits_{\tau\rightarrow 0}E_{\tau}^{j}(x)=0,\forall x\in
I_{\alpha}(x_{0}),j\in {1,...,m}
\end{equation}
where$E_{\tau}^{j}(x)$ is defined
by (${e_{1},...,e_{m}}$ is canonical basis for $\mathbb{R}^{m}$)
\begin{equation}\label{eq:1.101}
E_{\tau}^{j}(x)=\frac {1}{\tau}[y(x,\hat{\lambda}+\tau
e_{j})-y(x,\hat{\lambda})-\tau \hat{y_{j}}(x)], \tau\neq 0
\end{equation}
and
$\hat{y_{j}}(x),x\in
I_{\alpha}(x_{0})$, is the unique solution of \eqref{eq:1.99}. Using remark \ref{re:1.1} for $\lambda''=\hat{\lambda}+\tau
e_{g},\lambda'=\hat{\lambda}$ we get
\begin{eqnarray}\label{eq:1.102}
E_{\tau}^{j}(x)&=&\int_{x_{0}}^{x} \frac {\partial f}{\partial
y}(t,\hat{\lambda},y(t,\hat{\lambda}))E_{\tau}^{j}(x)dt+
\nonumber\\&+&\int_{x_{0}}^{x}(\theta
(t,\hat{\lambda}+\tau e_{j}\hat{\lambda}){\frac{1}{\tau}\mid
y(t,\hat{\lambda}+\tau e_{j}-y(t,\hat{\lambda})\mid
+1\}dt}
\end{eqnarray}
Denote $M_{1}=C+1$, where$\frac{1}{\tau}\mid y(t,\hat{\lambda}+\tau
e_{j})-y(t,\hat{\lambda})) \mid\leq C$ and let $L> 0$ be such
that $\mid\frac {\partial
f(t,\hat{\lambda},y(t,\hat{\lambda}))}{\partial y}\mid\leq L
\,\forall t\in
I_{\alpha}(x_{0})$. Then the following integral inequality is valid
\begin{equation}\label{eq:1.103}
\mid E_{\tau}^{j}(x)\mid\leq L\int_{x_{0}}^{x_{0}+\mid
x-x_{0}\mid} E_{\tau}^{j}(x)\mid
dt+M_{1}\int_{x_{0}}^{x_{0}+\alpha}\mid
\theta(t,\hat{\lambda}+\tau e_{j},\hat{\lambda})\mid dt,x\in
I_{\alpha}(x_{0})
\end{equation}
Applying Lemma Gronwall, from \eqref{eq:1.103} we obtain
\begin{equation}\label{eq:1.21b}
\mid E_{r}^{j}(x)\mid\leq M_{1}(\int_{x_{0}}^{x_{0}+\alpha}\mid
\theta(t,\hat{\lambda}+\tau e_{j},\hat{\lambda})\mid dt)\exp L\alpha
\end{equation}
Using $\mathop{lim}\limits_{\tau\rightarrow0}\theta(t,\hat{\lambda}+\tau
e_{j},\hat{\lambda})=0$
uniformly of $t\in I_{\alpha}(x_{0})$ (see\eqref{eq:1.14b}) and passing $\tau\rightarrow 0$ into \eqref{eq:1.21b} we obtain
\begin{equation}\label{eq:1.22b}
\mathop{lim}\limits_{r\rightarrow 0}E_{\tau}^{j}(x)=0,\forall x\in
I_{\alpha}(x_{0}),j\in {1,...,m}
\end{equation}
The proof is complete.
\end{proof}
\subsection{The Local Flow\index{Flow!local}(Differentiability
Properties)} Consider a continuous function $g(x,z):I\times
G\rightarrow\mathbb{R}^{n}$, where $I\subseteq \mathbb{R}$ and
$G\subseteq\mathbb{R}^{n}$ are open sets. Define a new nonlinear
system of differential equations
\begin{equation}\label{eq:1.104}
\frac{dz}{dx}=g(x,z),z(x_{0})=\lambda\in B(z_{0},\rho)\subseteq
G\,,\,\hbox{where} \,x_{0}\in I,z_{0}\in G
\end{equation}
are fixed and $\lambda\in
B(z_{0},\rho)$ is a variable Cauchy condition. Assume
\begin{equation}\label{eq:1.24b}
\hbox{there exist continuous partial derivatives} \,
\frac{\partial g(x,z)}{\partial z_{i}}:I\times G\rightarrow
\mathbb{R}^{n}\,i\in {1,...,n}
\end{equation}
The unique solution of \eqref{eq:1.104}, $z(x,\lambda):I_{\alpha}(x_{0})\times
B(z_{0},\rho)\rightarrow G$ satisfying $z(x_{0},\lambda)=\lambda\in
B(z_{0},\rho)\subseteq \mathbb{R}^{n}$ will be called the local flow\index{Flow!local}
associated with the vector field $g$ satisfying \eqref{eq:1.24b}.
\begin{theorem}\label{th:1.6}(differentiability of local flow)
 Consider that the
vector field $g\in C(I\times G;\mathbb{R}^{n})$ satisfies the
assumption \eqref{eq:1.24b}. Then there exist $\alpha>0$ and a continuously
differentiable local flow $z(x,\lambda):I_{\alpha}(x_{0})\times
B(z_{0},\rho)\rightarrow G$ of
$g$ fulfilling the following properties
\begin{equation}\label{eq:1.25b}
\hbox{for each} \, \hat{\lambda}\in \hbox{int}\, B(z_{0},\rho)\,\, \,\hbox{there exists a
nonsingular} \, (n\times n) matrix
\end{equation}
$\hat{Z}(x)=\frac{\partial
z(x,\hat{\lambda})}{\partial \lambda},x\in
I_{\alpha}(x_{0})$, satisfying
$$\left\{
  \begin{array}{ll}
    \frac{d\hat{Z}(x)}{dx}= & \frac{\partial g(x,z(x,\lambda\hat{}))}{\partial z}\hat{Z}(x),\forall x\in I_{\alpha}(x_{0}) \\
   \hat{Z}(x_{0})= & I_{n}
  \end{array}
\right.$$
\end{theorem}
\begin{proof}
For $z_{0}\in G$ fixed define
$\Lambda=\{z\in G:\mid z-z_{0}\mid<\rho+\epsilon\}=\hbox{int} B(z_{0},\rho
+ \epsilon)$, where $\epsilon> 0$ is sufficiently small such that
$\Lambda\subseteq G$. Associate a new vector field depending on
parameter $\lambda\in \Lambda$
\begin{equation}\label{eq:1.26b}
f(x,\lambda,y)=g(x,y+\lambda),x\in I,\lambda\in \Lambda,y\in
\hat{G}=G-B(z_{0},\rho+\epsilon)
\end{equation}
Let $b>0$ be such that
$B(0,b)\subseteq \hat{G}$ and
$B(z_{0},\rho+\epsilon+b)\subseteq G$. Notice that according to \eqref{eq:1.24b} we get a smooth vector field $f$ with
respect to $(\lambda,y)\in \Lambda\times \hat{G}$ and the following
system of differential equation with parameters
$\lambda=(\lambda_{1},...,\lambda_{n})\subseteq \Lambda$
\begin{equation}\label{eq:1.105}
\frac{dy}{dx}=f(x,\lambda,y),y(x_{0})=0,y\in
\hat{G}\subseteq \mathbb{R}^{n}
\end{equation}
satisfies the differentiability conditions of the Theorem \ref{th:1.7}. Let $y(x,\lambda):I_{\alpha}(x_{0})\times
B(z_{0},\rho)\rightarrow B(0,b)\subseteq \hat{G}$ be the unique
solution of \eqref{eq:1.105} which is continuously differentiable on
$(x,\lambda)\in I_{\alpha}(x_{0})\times int B(z_{0},\rho)$ and
$y_{j}\hat{}(x)=\frac{\partial y(x,\hat{\lambda})}{\partial y} ,x\in
I_{\alpha}(x_{0}),\hat{\lambda}\in int
B(z_{0},\rho)$,fulfils
\begin{equation}\label{eq:1.106}
\left\{
      \begin{array}{ll}
        \frac{du}{dx}= & \frac{\partial f(x,\hat{\lambda},y(x,\hat{\lambda}))}{\partial y}u+ \frac{\partial f(x,\hat{\lambda},y(x,\hat{\lambda}))}{\partial \widetilde{\lambda_{j}}},u\in \mathbb{R}^{n},x\in I_{\alpha}(x_{0})\\
        u(x_{0})= & 0
      \end{array}
    \right.
    \end{equation}
for each $j\in {1,..,n}$. Then $z(x,\lambda)=y(x,\lambda)+\lambda,x\in
I_{\alpha}(x_{0}),\lambda\in B(z_{0},\rho)$, is the unique
continuously differentiable solution of the nonlinear system \eqref{eq:1.104} with $z(x_{0},\lambda)=\lambda$ and
\begin{equation}\label{eq:1.29b}
\hat{Z}(x)=\frac{\partial z(x,\hat{\lambda})}{\partial
\lambda}=\frac{\partial y(x,\hat{\lambda})}{\partial
\lambda}+I_{n},x\in I_{\alpha}(x_{0}),\hat{\lambda}\in int
B(z_{0},\rho)
\end{equation}
where $\hat{Y}(x)=\frac{\partial
y(x,\hat{\lambda})}{\partial \lambda}$ verifies the following
system
\begin{equation}\label{eq:1.107}
\left\{
      \begin{array}{ll}
        \frac{d\hat{Y}(x)}{dx}= & \frac{\partial f(x,\hat{\lambda},y(x,\hat{\lambda}))}{\partial y}\hat{Y}(x)+ \frac{\partial f(x,\hat{\lambda},y(x,\hat{\lambda}))}{\partial \lambda},x\in I_{\alpha}(x_{0})\\
        \hat{Y}(x_{0})= &\hbox{ null matrix}
      \end{array}
    \right.
    \end{equation}
Using \eqref{eq:1.29b} and \eqref{eq:1.107} we see easily that $\hat{Z}(x_{0})=I_{n}$
and
\begin{eqnarray}\label{eq:1.31b}
\frac{d\hat{Z}(x)}{dx}&=&\frac{d[\hat{Y}(x)]}{dx}=\frac{\partial
g(x,z(x,\hat{\lambda}))}{\partial z}[\hat{Z}(x)-I_{n}]
+\frac {\partial g(x,z(x,\hat{\lambda}))}{\partial z}\nonumber\\&=&\frac
{\partial g(x,z(x,\hat{\lambda}))}{\partial
z}\hat{Z}(x),x\in I_{\alpha}(x_{0})
\end{eqnarray}
The matrix satisfying \eqref{eq:1.31b} is a nonsingular one(see Liouville
theorem) and the proof is complete.
\end{proof}
\begin{remark}\label{re:1.2}
Consider the nonsingular matrix $\hat{Z}(x),x\in I_{\alpha}(x_{0})$
given in the above theorem and define
$\hat{H}(x)=[\hat{Z}(x)]^{-1},x\in I_{\alpha}(x_{0})$. Then
$\hat{H}(x),x\in I_{\alpha}(x_{0})$
satisfies the following linear matrix system
\begin{equation}\left\{
  \begin{array}{ll}
    \frac {d\hat{H}(x)}{dx} =& -\hat{H}(x)\frac{\partial g(x,z(x,\hat{\lambda}))}{\partial z},x\in I_{\alpha}(x_{0}) \\
   H(x_{0})=& I_{n}
  \end{array}
\right.
\end{equation}
It can be proved by computing the derivative
$$\frac{d[\hat{H}(x)\hat{Z}(x)]}{dx}=[\frac{d\hat{H}(x)}{dx}]\hat{Z}(x)+\hat{H}(x)[\frac{dZ}{dx}]=0,x\in
I_{\alpha}(x_{0})$$
\end{remark}
\textbf{Exercise(differentiability with respect to $ x_{0}\in
I)$}\\
Let $g(x,z):I\times G\rightarrow \mathbb{R}^{n}$ be continuously
differentiable mapping with respect to $z\in G$.Let $x_{0}\in
I,z_{0}\in G$ and $\beta>0$ be fixed such that $
I_{\beta}(x_{0})=[x_{0}-\beta,x_{0}+\beta]\subseteq I$. Then there
exist $\alpha>0$ and a continuously differentiable mapping
$z(x;s):I_{\alpha}(s)\times I_{\beta}(x_{0})\rightarrow G$
satisfying
\begin{equation}\label{eq:1.108}
\left\{
       \begin{array}{ll}
         \frac{dz(x;s)}{dx}= & g(x,z(x;s)),x\in I_{\alpha}(s)=[s-\alpha,s+\alpha] \\
        z(s;s)= & z_{0}, \,\,\hbox{for each}\, s\in I_{\beta}(x_{0})
       \end{array}
     \right.
\end{equation}
and $\hat{z}(x)=\frac{\partial z(x,\hat{s})}{\partial s},x\in
I_{\alpha}(\hat{s})\,(\hbox{for} \,\hat{s}\in \hbox{int}
I_{\beta}(x_{0}))$ fulfills the following linear system
\begin{equation}\label{eq:1.109}
\frac{d\hat{z}(x)}{dx}=\frac {\partial
g(x,z(x,\hat{s}))}{\partial s}\hat{z}(x),x\in
I_{\alpha}(\hat{s}),\hat{z}(\hat{s})=-g(\hat{s},z_{0})
\end{equation}
\textbf{Hint}.
A system with parameters is associated as in
Theorem \ref{th:1.6} and it lead us to a solution
$$z(x,s):I_{\alpha}(s)\times I_{\beta}(x_{0})\rightarrow B(z_{0},b)\subseteq G$$
of the following integral equation
$$z(x,s)=z_{0}+\int_{s}^{x}g(t,z(t,s))dt,x\in I_{\alpha}(s),s\in
I_{\beta}(x_{0})$$
Take $\hat{s}\in int I_{\beta}(x_{0})$ and using a similar
computation given in Theorem \ref{th:1.7} we get $\mathop{lim}\limits_{\tau\rightarrow
0}E_{\tau}(x)=0,x\in I_{\alpha}(\hat{s})$, where
$$E_{\tau}(x)=\frac{1}{\tau}[z(x,\hat{s}+\tau)-z(x,\hat{s})-\tau\hat{z}(x)],\tau\neq
0$$
\subsection{Applications(Using Differentiability of a Flow)}\index{Flow!}
 (a) The local flow $z(x,\lambda):I_{\alpha}(x_{0})\times
B(z_{0},\rho)\rightarrow \mathbb{R}^{n}$ defined in Theorem \ref{th:1.7}
preserve the volume of any bounded domain $D\subseteq B(z_{0},\rho)$
for $\hat{\lambda}\in D$ provided $Tr\frac{\partial
g(x,z(x,\hat{\lambda}))}{\partial z}=\Sigma_{i=1}^{n}\frac{\partial
g_{i}(x,z(x,\hat{\lambda}))}{\partial z_{i}}=0, x\in
I_{\alpha}(x_{0}),\hat{\lambda}\in D$. In this respect, denote $D(x)=\{y\in
\mathbb{R}^{n},y=z(x,\hat{\lambda}),\hat{\lambda}\in D\},x\in
I_{\alpha}(x_{0})$ and notice that $D(x_{0})=D$. Using multiple
integrals we compute $vol D(x)=\int...\int dy_{1}...dy_{n}$ which reduces to
$$D(x)=\int...\int\mid det \frac{\partial
z(x,\hat{\lambda})}{\partial \lambda}\mid
d\lambda_{1}...d\lambda_{n}$$
Using Liouville theorem  and
$Tr\frac{\partial g(x,z(x,\hat{\lambda}))}{\partial z}=0,\forall
x\in I_{\alpha}(x_{0}),\hat{\lambda}\in D$, we obtain $det \frac{\partial z(x,\hat{\lambda})}{\partial
\lambda}=1$ for any $x\in I_{\alpha}(x_{0}),\hat{\lambda}\in
D$, where
$$\frac{\partial z(x,\hat{\lambda})}{\partial \lambda},x\in
I_{\alpha}(x_{0})$$
satisfies the linear system \eqref{eq:1.25b}.
As a result,$det ]\frac{\partial z(x,\hat{\lambda})}{\partial
\lambda}]=1,\forall x\in I_{\alpha}(x_{0})\,\hbox{and}\, \hat{\lambda}\in
D$, which proves that $vol D(x)=vol D,x\in I_{\alpha}(x_{0})$\\
(b) The linear system  \eqref{eq:1.25b} given in Theorem \ref{th:1.6} is called the
linearized system associated with \eqref{eq:1.104}.\\
If $g(x,z)=g(z)$ such that $g(z)$ satisfies a linear growth
condition
\begin{equation}\label{eq:1.110}
\mid g(z)\mid\leq C(1+\mid z \mid),\forall z\in\mathbb{R}^{n}
\end{equation}
where $C>0$ is a constant ,then the unique solution $z(x,\lambda),\lambda\in B(z_{0},\rho)$ verifying
\begin{equation}\label{eq:1.111}
\frac{dz}{dx}(x,\lambda)=g(z(x,\lambda)),z(0)=\lambda
\end{equation}
can be extended to the entire half line $x\in [0,\infty)$.\\
In addition, if $g(z_{0})=0$($z_{0}$ is a stationary point)then the
asymptotic behaviour of $z(x,\lambda)$ for $x\rightarrow \infty$ and
$\lambda \in B(z_{0},\rho)$ can be obtained analyzing the
corresponding linear constant coefficients system
\begin{equation}\label{eq:1.112}
\frac{dz}{dx} =\frac{\partial g(z_{0})}{\partial
z}z,z(0)=\lambda
\end{equation}
It will be assuming that
\begin{equation}\label{eq:1.113}
A=\frac{\partial g(z_{0})}{\partial z}\, \hbox{is a Hurwitz matrix i.e}
\end{equation}
any$\lambda\in \sigma(A)(P(\lambda)=det(A-\lambda I_n)=0)$
satisfies $Re\lambda<0$, which implies $\mid \exp Ax\mid\leq M
[\exp-wx]$ for some $M>0,w>0$.\\
We say that the nonlinear system \eqref{eq:1.111} is locally asymptotically
stable\index{Asymptotically stable} around the stationary solution ${z_{0}}$(or ${z_{0}}$ is
locally asymptotically stable\index{Asymptotically stable})if there exist $\rho>0$ such that
\begin{equation}\label{eq:1.39b}
\mathop{lim}\limits_{x\rightarrow \infty}z(x,\lambda)=0,\forall \lambda \in
B(z_{0},\rho)
\end{equation}
There is a classical result connecting \eqref{eq:1.113} and
\eqref{eq:1.39b}.
\begin{theorem}(Poincare-Lyapunov) Assume that the $n\times n$
matrix $A$ is Hurwitz such that
$$\mid \exp Ax\mid\leq M( \exp-wx),\forall x\in [0,\infty)$$
for some constant $M>0,w>0$. Let $f(y):\Omega\subseteq \mathbb{R}^{n}\times \mathbb{R}^{n}$ be a
local Lipschitz continuous function such that $\Omega$ is an open
set and $\mid f(y)\mid\leq L\mid y\mid, y\in \Omega (0\in
\Omega)$, where $L>0$ is a constant. If $LM-w<0$,then $y=0$ is asymptotically stable\index{Asymptotically stable} for the perturbed system
\begin{equation}\label{eq:1.40b}
\frac{dy}{dx}=Ay+f(y),y\in \Omega,y(0)=\lambda\in
B(0,\rho)\subseteq \Omega
\end{equation}
Here the unique solution $\{y(x,\lambda):x\geq 0\}\, of (40)$
fulfils $\mathop{lim}\limits_{x\rightarrow \infty}y(x,\lambda)=0$ for each
$\lambda\in B(0,\rho)$ provided $\rho=\frac{\delta}{2M}$ and
$B(0,\rho)\subseteq \Omega$.\end{theorem}
\begin{proof}
Let$\lambda\in \Omega$ be fixed and consider the unique solution
$\{y(x,\lambda):x\in [0,T]\}$ satisfying \eqref{eq:1.40b}. Notice that if
$\lambda$ is sufficiently small then ${y(x,\lambda):x\in [0,T]}$,can
be extended to the entire half line $x\in [0,\infty)$. In this
respect, using the constant variation formula for $x\in[0,T]$ and
$b(x)=f(y(x,\lambda))$ we get
$y(x,\lambda)=(\exp Ax)\lambda+\int_{0}^{x}(\exp(x-t)A)f(y(t,\lambda))dt$,for
any $x\in[0,T]$. Using the above given representation we see easily
the following estimation
$$\mid y(x,\lambda)\mid\leq \mid \exp Ax\mid
\mid\lambda\mid+\int_{0}^{x}\mid \exp(x-t)A\mid\mid
f(y(t,\lambda))\mid dt\leq$$\\
$$\leq\mid\lambda\mid M(\exp-wx)+\int_{0}^{x}LM(\exp-w(x-t))\mid
y(t,\lambda)\mid dt \hbox{for any} x\in[0,T]$$
Multiplying by $(\exp wx)>0$, we obtain
$$(\exp wx)\mid y(x,\lambda\mid)\leq\mid\lambda\mid
M+\int_{0}^{x}LM(\exp wt)\mid y(t,\lambda\mid)dt,x\in[0,T]$$
and applying Gronwall Lemma for $\varphi(x)=(\exp wx)\mid
y(x,\lambda)\mid$ we get $\varphi(x)\leq\mid\lambda\mid M(\exp LMx)$
for any $x\in[0,T]$. Take $\delta>0$ and $\rho=\frac{\delta}{2M}$ such
that $B(0,\delta)\subseteq\Omega$. It implies $\mid
y(x,\lambda)\mid\leq\frac{\delta}{2}$ for any$x\in[0,T]$ if $\mid\lambda\mid\leq\rho$. In addition, $\mid y(T,\lambda)\mid\leq\frac{\delta}{2}\exp(-\alpha
T)$, where $\alpha=w-LM>0$ and choosing $T>0$ sufficiently large such
that $M\exp-\alpha T\leq 1$ we may extend the solution
${y(x,\lambda),x\in[0,T]}$ for $x\in[0,2T]$ such that
$$\mid y(2T,\lambda)\mid=\mid y(2T,y(T,\lambda))\mid\leq M\mid
y(T,\lambda)\mid \exp-\alpha T\leq (\frac{\delta}{2})$$
and
$$\mid y(x,y(T,\lambda))\mid\leq M\mid y(T,\lambda)\mid
\exp(LM-w)x\leq\frac{\delta}{2}$$
for any $x\in[T,2T]$. It shows that $y(x,\lambda)$ can be extended to the entire half line
$x\in[0,\infty)$ if $\mid\lambda\mid\leq\rho=\frac{\delta}{2M}$ and
$\delta>0$ satisfies $B(0,\delta)\subseteq\Omega$\\
In addition the inequality established for $x\in[0,T]$ is preserved
for the extended solution and $\mid
y(x,\lambda)\mid\leq\mid\lambda\mid M \exp(LM-w)x\leq
\frac{\delta}{2}\exp(LM-w)x$, for any $x\in[0,\infty)$, if
$\mid\lambda\mid\leq\rho=\frac{\delta}{2}$. Passing to the limit
$x\rightarrow\infty$, from the last inequality we get
$\mathop{lim}\limits_{x\rightarrow\infty}\mid y(x,\lambda)\mid=0$ uniformly with
respect to $\lambda\mid\leq\rho$ and the proof is complete.
\end{proof}
The following is a direct consequence of the above theorem.
\begin{remark}\label{re:1.3}
Let $A$ be $n\times n$ Hurwitz matrix and
$f(y):\Omega(open)\subseteq\mathbb{R}^{n}\rightarrow \mathbb{R}^{n}$
is a local Lipschitz continuous function where $0\in\Omega$. If $\mid
f(y)\mid\leq\alpha(\mid y\mid),\forall y\in\Omega$,where
$\alpha(\mid y\mid):[0,\infty)\rightarrow[0,\infty)$ satisfies
$\mathop{lim}\limits_{\gamma\rightarrow 0}\frac{\alpha(\gamma)}{\gamma}=0$.
Then the unique solution of the system
$$\frac{dy}{dx}=Ay+f(y),y(0)=\lambda\in B(0,\rho),x\geq 0$$ is a
asymptotically stable\index{Asymptotically stable}($\mathop{lim}\limits_{x\rightarrow\infty}\mid
y(x,\lambda\mid=0)$) if $\rho>0$ is sufficiently
small.
\end{remark}
\begin{proof}
Let $M\geq 1$ and $w>0$ such that $\mid
\exp Ax\mid\leq M(\exp-wx)$. Consider $L>0$ such that $LM-w<0$ and
$\eta>0$ with the property $\alpha(r)\leq Lr$ for any
$r\in[0,\eta]$. Define $\widetilde{\Omega}={y\in\Omega;\mid y\mid<\eta}$. Notice that
$g(y)=Ay+f(y),y\in\widetilde{\Omega}$, satisfies the assumption of
Poincare-Lyapunov theorem which allows to get the conclusion.
\end{proof}
We are in position to mention those sufficient conditions which
implies that the stationary solution $z_{0}\in\mathbb{R}^{n}$ of the
system
\begin{equation}\label{eq:1.114}
\frac{dz}{dx}=g(z),g(z_{0})=0
\end{equation}
is asymptotically stable\index{Asymptotically stable}.Let
$g(z):D(open)\subseteq\mathbb{R}^{n}\rightarrow\mathbb{R}^{n}$ be
given such that $g(z_{0})=0$ and \\
(i) $g(z):D\rightarrow\mathbb{R}^{n}$ is continuously
differentiable.\\
(ii) The matrix $\frac{\partial g(z_{0})}{\partial z}= A$ is
Hurwitz.\\
Under the hypothesis (i) and (ii) we rewrite
$$g(y)=Ay+f(y), \hbox{where} y=z-z_{0} \hbox{and} f(y)=g(z_{0}+y)-g(z_{0})$$
Consider the linear system
$$\frac{dy}{dx}=Ay+f(y),y\in\{y\in\mathbb{R}^{n}:\mid
y\mid<\mu\}=\Omega$$
where $\mu>0$ is taken such that $z_{0}+\Omega\subseteq D$. Using (i) we get that $f$ is locally Lipschitz continuous on $\Omega$ and the assumptions of the above given remark are satisfied. We see
easily that
$$f(y)=\int_{0}^{1}[\frac{\partial g(z_{0}+\theta y)}{\partial
z}-\frac{\partial g(z_{0})}{\partial z}]yd\theta,y\in\Omega$$
and
$$\mid f(y)\mid\leq\alpha(\mid y\mid)=(\int_{0}^{1}h(\theta,\mid
y\mid)d\theta)\mid y\mid,y\in\Omega$$
Here $h(\theta,\mid y\mid)=\max_{\mid w\mid\leq\mid
y\mid}\mid\frac{\partial g(z_{0}+\theta w)}{\partial
z}-\frac{g(z_{0})}{\partial z}\mid$ satisfies \\
$\mathop{lim}\limits_{\mid y\mid\rightarrow 0}h(\theta,\mid y\mid)=0$ uniformly on $\theta\in[0,1]$ and $\mathop{lim}\limits_{\gamma\rightarrow 0}\frac{\alpha(\gamma)}{\gamma}=0$
therefore, the assumptions of the above given remark are satisfied
when considering the nonlinear system \eqref{eq:1.114} and the conclusion will
be stated as
\begin{proposition}Let $g(z):D\subseteq\mathbb{R}^{n}\rightarrow
\mathbb{R}^{n}$ be a continuous function and $z_{0}\in D$ fixed such
that $g(z_{0})=0$. Assume that the condition (i) and (ii) are
fulfilled. Then the stationary solution $z=z_{0}$ of the system \eqref{eq:1.114}
is asymptotically stable\index{Asymptotically stable}, i.e\\
$\mathop{lim}\limits_{x\rightarrow\infty}\mid z(x,\lambda)-z_{0}\mid=0$ for any
$\mid\lambda-z_{0}\mid\leq\rho$, if$\rho>0$ is sufficiently
small,where ${z(x,\lambda):x\geq 0}$ is the unique solution of \eqref{eq:1.114}
with $z(0,\lambda)=\lambda$\\
 \textbf{Problem}. Prove Poincare-Lyapunov
theorem replacing Hurwitz property of the matrix $A$ with $\sigma
(A+A^{*})=\{\lambda_{1},...\lambda_{d}\}$ where
$\lambda_{i}<0,i\in{1,...,d}$.
\end{proposition}
\textbf{Comment}(on global existence of a solution)
\begin{theorem}(global existence)Let
$g(x,z):[0,\infty)\times\mathbb{R}^{n}\rightarrow\mathbb{R}^{n}$ be
a continuous function which is locally Lipschitz continuous with
respect to $z\in\mathbb{R}^{n}$. Assume that for each $T>0$ there
exists $C_{T}>0$ such that\\
$$\mid g(x,z)\mid\leq C_{T}(1+\mid z\mid),\forall
x\in[0,T],\,z\in\mathbb{R}^{n}$$
Then the unique $C.P(g;0,z_{0})$ solution $z(x,z_{0})$ satisfying\\
$\frac{dz(x;z_{0})}{dx}=g(x,z(x,z_{0}))$, $z(0,z_{0})=z_{0}$, is
defined for any $x\in[0,\infty)$.
\end{theorem}
\begin{proof}
\textbf{(sketch)}It is used the associated integral equation\index{Integral equation}
$$z(x,z_{0})=z_{0}+\int_{0}^{x}g(t,z(t,z_{0}))dt,x\in[0,K T]$$
and the corresponding Cauchy sequence\index{Cauchy!sequence} allows to get a unique solution\\
${z_{k}(x,z_{0}):x\in[0,K T]}$ for each $K\geq 1$. The global
solution will be defined as an inductive
limit $\widetilde{Z}(z,z_{0}):[0,\infty)\rightarrow\mathbb{R}^{n},\widetilde{Z}(z,z_{0})=Z_{k}(x,z_{0})$, if
$x\in[0,K T]$, for each $k\geq 0$.
\end{proof}
\section[Gradient Systems of Vector Fields and Solutions]{Gradient Systems\index{Gradient systems} of Vector Fields  and their
Solutions; Frobenius Theorem\index{Frobenius!Theorem}}
\begin{definition}Let
$X_{j}(p;y)\in\mathbb{R}^{n}$ for $p=(t_{1},...,t_{m})\in
D_{m}=\prod_{1}^{m}(-a_{i},a_{i}),y\in V\subseteq \mathbb{R}^{n}$ be
continuously differentiable $(X_{j}\in \mathcal{C}^{1}(D_{m}\times
V;\mathbb{R}^{n}))$ for $j\in\{1,...,m\}$.\\
We say that $\{X_{1}(p;y),...,X_{m}(p;y)\}$ defines a gradient system of vector
fields(or fulfils the Frobenius integrability
condition)\index{Frobenius! integrability condition} if\\
$\frac{\partial X_{j}}{\partial\,t_{i}}(p;y)-\frac{\partial\,X_{i}}{\partial\,t_{j}}(p;y)=[X_{i}(p;.),X_{j}(p;.)](y)\forall
i,j\in\{1,...,m\}$, where
$[Z_{1},Z_{2}](y)=\frac{\partial{Z_{1}(y)}}{\partial{y}}Z_{2}(y)-\frac{\partial{Z_{2}(y)}}{\partial{y}}Z_{1}(y)$(Lie
bracket).
\end{definition}
\subsection[Gradient System Associated to Finite Set of
Vec-Fields]{The Gradient System\index{Gradient systems} Associated with a Finite Set of
Vector Fields}
\begin{theorem}\label{th:1.8}  Let $Y_{j}\in
\mathcal{C}^{2}(\mathbb{R}^{n},\mathbb{R}^{n}),x_{0}\in\mathbb{R}^{n},j\in\{1,...,m\}$
be fixed .Then there exist
$D_{m}=\prod_{1}^{m}(-a_{i},a_{i}),V(x_{0})\subseteq \mathbb{R}^{n}$
and $X_{j}(p;y)\in\mathbb{R}^{n}\\X_{j}\in \mathcal{C}^{1}(D_{j-1}\times
V;\mathbb{R}^{n}),j\in\{1,...,m\},X_{1}(p_{1};y)=Y_{1}(y),$such
that
\begin{equation}\label{eq:1.1c}
\frac{\partial y}{\partial t_{1}}=Y_{1}(y),\frac{\partial
y}{\partial t_{2}}=X_{2}(t_{1};y),...,\frac{\partial y}{\partial
t_{m}}=X_{m}(t_{1},...,t_{m-1};y)
\end{equation}
is a gradient system $(\frac{\partial {X_{j}}(p_{j};y)}{\partial
t_{i}}=[X_{i}(p_{i};.),X_{j}(p_{j},.)](y),i<j)$ and
\begin{equation}\label{eq:1.115}
G(p;x)=G_{1}(t)\circ...\circ
G_{m}(t_{m}(x)),p=(t_{1},...,t_{m})\in D_{m},x\in V(x_{0})
\end{equation}
is the solution for \eqref{eq:1.1c} satisfying Cauchy condition  $y(0)=x\in
V(x_{0})$. Here
 $G_{j}(t;x)$ is the local flow generated by\index{Flow!local}
the vector field $Y_{j},j\in\{1,...,m\}$.
\end{theorem}
\begin{proof}
We shall use the standard induction argument and for
$m=1$ we notice that the equations \eqref{eq:1.1c} and \eqref{eq:1.115} express the existence
and uniqueness  of a local flow associated with a nonlinear system
of differential equation (see Theorem \ref{th:1.6}). Assume that for
$(m-1)$ given vector fields $Y_{j}\in
\mathcal{C}^{2}(\mathbb{R}^{n},\mathbb{R}^{n})$ the conclusions \eqref{eq:1.1c} and \eqref{eq:1.115}
are satisfied,i.e there exist
\begin{equation}\label{eq:1.3c}
\hat{X_{2}}(\hat{y})=Y_{2},\hat{X_{3}}(t_{2},y),...,\hat{X_{m}}(t_{2},...,t_{m-1};y)
\end{equation}
continuously differentiable with respect to\\
$\hat{p}=(t_{2},...,t_{m})\in\prod_{2}^{m}(-a_{i},a_{i})$ and
$y\in\hat{V}(x_{0})\subseteq\mathbb{R}^{n}$ such
that
\begin{equation}\label{eq:1.4c}
\frac{\partial y}{\partial t_{2}}=Y_{2}(y),\frac{\partial
y}{\partial t_{3}}=\hat{X_{3}}(t_{2};y),...,\frac{\partial
y}{\partial
t_{m}}=\hat{X_{m}}(t_{2},...,t_{m-1};y)
\end{equation}
satisfying Cauchy condition \index{Cauchy!condition}$y(0)=x\in\hat{V}(x_{0})$. Recalling that
\eqref{eq:1.4c} is a gradient system\index{Gradient systems} we notice
\begin{equation}\label{eq:1.5c}\frac{\partial
\hat{X_{j}}(\hat{p};y)}{\partial
t_{i}}=[\hat{X_{i}}(\hat{p_{i}},.),\hat{X_{j}}(\hat{p_{j};.})](y),2\leq
i<j=3,...,m
\end{equation}
where $\hat{p_{i}}=(t_{2},...,t_{i-1})$
Let vector fields $Y_{j}\in
\mathcal{C}^{2}(\mathbb{R}^{n},\mathbb{R}^{n}),j\in\{1,...,m\}$ and denote
\begin{equation}\label{eq:1.6c}
y=G(p;x)=G_{1}(t_{1})\circ
\hat{G}(\hat{p};x),p=(t_{1},\hat{p})\in\prod_{1}^{m}(-a_{i},a_{i}),x\in
V(x_{0})\subseteq\hat{V}(x_{0})
\end{equation}
where $G_{1}(t)(x)$ is the local
flow \index{Flow!local}generated by $Y_{1}$ and
$\hat{G}(\hat{p},x)$ is defined in \eqref{eq:1.3c}. By definition $\frac{\partial G(p;y)}{\partial
t_{1}}=Y_{1}(G(p;x))$ and to prove that $y=G(p;x)$ defined in \eqref{eq:1.6c}
fulfils a gradient system\index{Gradient systems} we write
$G_{1}(-t_{1};y)=\hat{G}(\hat{p};x)$ for $y=G(p;x)$.A
straight computation shows
\begin{equation}\label{eq:1.116}
\frac{\partial {G_{1}}(-t_{1};y)}{\partial y}\frac{\partial
y}{\partial t_{j}}=\frac{\partial
\hat{{G}}(\hat{p};y)}{\partial t_{j}},\,\hbox{for} j=2,...,m
\end{equation}
which is equivalent with writing\\$\frac{\partial y}{\partial
t_{j}}=H_{1}(-t_{1};y)\hat{X_{j}}(\hat{p_{j}};G_{1}(-t_{1};y)),j=2,...,m$.\\
Here the matrix $H_{1}(\sigma;y)=[\frac{\partial
{G_{1}}(\sigma;y)}{\partial y}]^{-1}$ satisfies\\\\
$\frac{d H_{1}}{d\sigma}=-H_{1}\frac{\partial Y_{1}}{\partial
x}(G_{1}(\sigma;y)),H_{1}(0;y)=I_{n}$(see Theorem \ref{th:1.6}). Denote
$$X_{2}(t_{1};y)=H_{1}(-t_{1};y)Y_{2}(G_{1}(-t_{1};y))$$
$$X_{j}(t_{1},...j_{j-1};y)=H_{1}(-t_{1};y)\hat{X_{j}}(t_{2},...,t_{j-1};G_{1}(-t_{1};y)),j=3,...,m$$
for $y\in V(x_{0})\subseteq \hat{V}(x_{0})$ and $p=(t_{1},...,t_{m})\in \prod_{1}^{m}(-a_{i},a_{i})$. With these notations, the system \eqref{eq:1.116} is written as follows
\begin{equation}\label{eq:1.117}
\frac{\partial y}{\partial t_{1}}=Y_{1}(y),\frac{\partial
y}{\partial t_{2}}=X_{2}(t_{1};y),...\frac{\partial y}{\partial
t_{m}}=X_{m}(t_{1},...,t_{m-1};y)
\end{equation}
and to prove that \eqref{eq:1.117} stands for a gradient system\index{Gradient systems} we need to show
that
\begin{equation}\label{eq:1.119}
\frac{\partial X_{j}}{\partial
t_{i}}(p_{j};y)=[X_{i}(p_{i},.),X_{j}(p_{j},.)](y),1\leq
i<j=2,...,m
\end{equation}
For $j=2$ by direct computation we obtain
\begin{equation}\label{eq:1.120}
\frac{\partial X_{2}}{\partial
t_{1}}(t_{1};y)=H_{1}(-t_{1};y)[Y_{1},Y_{2}](G(-t_{1};y))
\end{equation}
and assuming that we know (see the exercise which follows)
\begin{eqnarray}\label{eq:1.11c}
H_{1}(-t_{1};y)[Y_{1},Y_{2}](G_{1}(t_{1},y))
&=&[H_{1}(-t_{1};.)Y_{1}(G_{1}(t_{1};.)),H_{1}(-t_{1};.)Y_{2}(G_{1}(t_{1};.))](y)\nonumber\\
&=&[Y_{1}(.),X_{2}(t_{1};.)](y)
\end{eqnarray}
we get
\begin{equation}\label{eq:1.121}
\frac{\partial X_{2}}{\partial
t_{1}}(t_{1};y)=[Y_{1}(.),X_{2}(t_{1};.)](y)
\end{equation}
The equation \eqref{eq:1.121} stands for \eqref{eq:1.119} when $i=1$and $j=2,3,...m$. It
remains to show \eqref{eq:1.119} for $j=3,...,m$ and $2\leq i<j$. Using \eqref{eq:1.5c} and\\
$\frac{\partial X_{j}}{\partial
t_{i}}(p_{j};y)=H_{1}(-t_{1};y)\frac{\partial\hat{X_{j}
}}{\partial t_{i}}(\hat{p_{j}};G_{1}(-t_{1};y))$,\\we obtain
\begin{equation}\label{eq:1.13c}
\frac{\partial X_{j}}{\partial
t_{i}}(p_{j};y)=H_{1}(-t_{1};y)[\hat{X_{i}}(\hat{p_{i}};.),\hat{X_{j}}(\hat{p_{j}};.)](G_{1}(-t_{1};y))
\end{equation}
if $2\leq i<i<j=3,...,m$. The right side in \eqref{eq:1.13c} is similar to that in \eqref{eq:1.119} and the same
argument used above applied to \eqref{eq:1.13c} allow one to write
\begin{equation}\label{eq:1.123}
\frac{\partial X_{j}}{\partial
t_{i}}(p_{j};y)=[X_{i}(p_{i};.),X_{j}(p_{j};.)](y)
\end{equation}
for any $2\leq i<j=3,...,m$, and the proof is complete.
\end{proof}
\noindent\textbf{Exercise 1}. Let $X,Y_{1},Y_{2}\in
\mathcal{C}^{2}(\mathbb{R}^{n},\mathbb{R}^{n})$ and consider the local flow
$G(\sigma;x),\sigma\in(-a,a),x\in
V(x_{0})\subseteq\mathbb{R}^{n}$,and the matrix
$H(\sigma;x)=[\frac{\partial G}{\partial x}(\sigma;x)]^{-1}$
determined by the vector field $X$(see Theorem \ref{th:1.6}). Then
$$H(-t;x)[Y_{1},Y_{2}](G(-t;x))= =[H(-t;.)Y_{1}(G(-t;.)),H(-t;.)Y_{2}(g(-t;.))](x)$$
for any $t\in(-a,a),x\in V(x_{0})\subseteq\mathbb{R}^{n}$.\\
\textbf{Solution}. The Lie bracket in the right hand side of the
conclusion is computed using $H(-t;y)=\frac{\partial G}{\partial
x}(t;G(-t;y))$ and using the symmetry of the matrices
$\frac{\partial G_{i}}{\partial x^{2}}(t;x),i\in\{1,...,n\}$,\\
we get the conclusion, where $G=(G_{1},...,G_{n})$.\\
\textbf{Exercise 2}. Under the same conditions as in Exercise
1, prove that
$$H(-t;y)Y_{1}(G(-t;y))=Y_{1}(y),\,\, \hbox{provided} X=Y_{1}$$
\textbf{Solution}. By definition $H(0;y)=I_{n}$ and
$$X_{1}(t;y)=H(-t;y)Y_{1}(G(-t;y))$$ satisfies $X_{1}(0;y)=Y_{1}(y)$. In
addition, applying the standard derivation of $X_{1}(t;y)$ we get
$\frac{d}{dt}X_{1}(t;y)=0$ which shows
$X_{1}(t;y)=Y_{1}(y),t\in(-a,a)$ and the verification is
complete.
\subsection{Frobenius Theorem\index{Frobenius! Theorem}}
Let $X_{j}\in
\mathcal{C}^{2}(\mathbb{R}^{n},\mathbb{R}^{n}),j\in\{1,...,m\}$ be given and
consider the following system of differential
equations
\begin{equation}\label{eq:1.124}
\frac{\partial y}{\partial t_{j}}=X_{j}(y),j=1,...,m,y(0)=x\in
V\subseteq\mathbb{R}^{n}
\end{equation}
\begin{definition} A solution for \eqref{eq:1.124} means a function \\
$G(p;x):D_{m}\times V\rightarrow\mathbb{R}^{n}$ of class $\mathcal{C}^{2}(G\in
\mathcal{C}^{2}(D_{m}\times V;\mathbb{R}^{n}))$ fulfilling \eqref{eq:1.1c} for any
$p=(t_{1},...,t_{m})\in D_{m}=\prod_{1}^{m}(-a_{i},a_{i})$ and $x\in
V(open set)\subseteq\mathbb{R}^{n}$. The system \eqref{eq:1.124} is completely
integrable on $\theta(open set)\subseteq\mathbb{R}^{n}$ if for any
$x_{0}\in\theta$ there exists a neighborhood
$V(x_{0})\subseteq\theta$ and a unique solution
$G(p;x),(p;x)\in D_{m}\times V(x_{0})$ of \eqref{eq:1.124} fulfilling
$G(0;x)=x\in V(x_{0})$
\end{definition}
\begin{theorem}\label{th:1.9}(Frobenius theorem)\index{Frobenius! theorem} Let $X_{j}\in
\mathcal{C}^{2}(\mathbb{R}^{n},\mathbb{}^{n}),j\in\{1,...,m\}$ be given. Then
the system \eqref{eq:1.124} is completely integrable on $$\theta \,\hbox{(open
set)} \,\subseteq\mathbb{R}^{n} iff [Y_{i},Y_{j}](x)=0$$
for any $i,j\in\{1,...,m\},x\in\theta$, where $[Y_{i},Y_{j}]=$ Lie
bracket. In addition, any local solution $y=G(p;x),(p,x)\in D_{m}\times
V(x_{0})$ is given by\\
$G(p;x)=G_{1}(t_{1})\circ...\circ G_{m}(t_{m})(x)$,where
$G_{j}(\sigma)(x)$ is the local flow generated by\index{Flow!local}
$X_{j},j\in\{1,...,m\}$.
\end{theorem}
\begin{proof} For
given $\{Y_{1},...,Y_{m}\}\subseteq
\mathcal{C}^{2}(\mathbb{R}^{n};\mathbb{R}^{n})$ and
$x_{0}\in\theta\subseteq\mathbb{R}^{n}$ fixed we associate the
corresponding gradient system
\begin{equation}\label{eq:1.125}
\frac{\partial y}{\partial t_{1}}=Y_{1}(y),\frac{\partial
y}{\partial t_{2}}=\hat{X_{2}}(t_{1};y),...,\frac{\partial
y}{\partial t_{m}}=\hat{X_{m}}(t_{2},...,t_{m-1};y)
\end{equation}
 and its solution
\begin{equation}\label{eq:1.126}
y=G(p;x)=G_{1}(t_{1})\circ...\circ
G_{m}(t_{m})(x),p\in D_{m}=\prod_{1}^{m}(-a_{j},a_{j})
\end{equation}
satisfying $G(0;x_{0})=x_{0}\in V(x_{0}) $(see Theorem \ref{th:1.8}) and
$G(p;x)\in\theta$. By definition $\frac{\partial y}{\partial
t_{j}}=X_{j}(t_{1},...t_{j-1};y),j\in\{1,...,m\}$, if $y=G(p;x_{0})$
and the orbit given in \eqref{eq:1.126} is a solution of the system \eqref{eq:1.124}$iff$
\begin{equation}\label{eq:1.127}
X_{j}(t_{1},...t_{j-1};y)=Y_{j}(y),j\in\{2,...,m\},y\in
V(x_{0})\subseteq\theta
\end{equation}
To prove \eqref{eq:1.127} we notice that
\begin{equation}\label{eq:1.128}
X_{2}(t_{1};y)=H_{1}(-t_{1};y)Y_{2}(G_{1}(-t_{1};y)),\,\hbox{for any}
\,y\in V(x_{0})
\end{equation}
where $H_{1}(-t_{1};y)=[\frac{\partial G_{1}}{\partial
y}(-t_{1};y)]^{-1}$ satisfies linear equation
\begin{equation}\label{eq:1.129}
\left\{
       \begin{array}{ll}
         \frac{dH_{1}}{d\sigma}(\sigma;y)= & -H_{1}(\sigma;y)\frac{\partial Y_{1}}{\partial y}(G_{1}(\sigma;y)),\sigma\in(-a_{1},a_{1}) \\
         H_{1}(0;y)= & I_{n}
       \end{array}
     \right.
\end{equation}
(see Remark \eqref{re:1.2} ). Using \eqref{eq:1.128} and \eqref{eq:1.129}, by a direct computation we get $$\left\{
   \begin{array}{ll}
     X_{2}(0;y)= & Y_{2}(y),y\in V(x_{0}) \\
     \frac{dX_{2}(t_{1};y)}{dt_{1}}= &
H_{1}(-t_{1};y)[Y_{1},Y_{2}](G_{1}(-t_{1};y))
   \end{array}
 \right.
$$\\
 for any $y\in V(x_{0})\subseteq\theta,t_{1}\in(-a_{1},a_{1})$. In
particular
\begin{equation}\label{eq:1.130}
X_{2}(t_{1};y)=Y_{2}\,\hbox{for any} \,y\in V(x_{0})
\hbox{iff} [Y_{1},Y_{2}](y)=0\,\hbox{for any } \,y\in V(x_{0})
\end{equation}
A similar argument can be used for proving that\\
\begin{eqnarray} 
X_{j}(t_{1},...,t_{j-1};y)&=&Y_{j}(y),\, \hbox{for any} \,y\in
V(x_{0}),j\geq 2 \,iff
 [Y_{i},Y_{j}](y)=0 \nonumber\\
 &&\hbox{ for any} \,1\leq i\leq
j-1,y\in V(x_{0})
\end{eqnarray}
and the conclusion of the system \eqref{eq:1.124} implies
\begin{equation}\label{eq:1.131}
[Y_{i},Y_{j}](y)=0,\forall y\in
V(x_{0})\subseteq\theta,i,j\in\{1,...,m\}
\end{equation}
The reverse implication, $\{Y_{1},...Y_{m}\}\subseteq
\mathcal{C}^{2}(\mathbb{R}^{n},\mathbb{R}^{n})$ are commuting on
$\theta(open)\subseteq\mathbb{R}^{n}$ implies that the system
is completely integrable on $\theta$ will be proved noticing t\eqref{eq:1.124} hat\\
\begin{equation}
G_{i}(t_{i})\circ G_{j}(t_{j})(y)=G_{j}(t_{j})\circ
G_{i}(t_{i})(y),\hbox{for any} \,i,j\in\{1,...,m\},y\in V(x_{0})
\end{equation}
if \eqref{eq:1.129}is
assumed. Here $G_{j}(\sigma)(y)$ is the local flow \index{Flow!local}generated by the
vector field $Y_{j},j\in\{1,...,m\}$. It remains to check that the
orbit \eqref{eq:1.126} is the unique solution of the system \eqref{eq:1.124} and in this
respect we notice that
\begin{equation}
\varphi(\theta;x)=G(\theta p;x),\theta\in[0,1]
\end{equation}
satisfies the following system of ordinary differential equations\index{Ordinary differential equations}
\begin{equation}\label{eq:1.132}
\left\{
       \begin{array}{ll}
         \frac{d\varphi}{d\theta}(\theta;x)= & \Sigma_{j=1}^{m}t_{j}Y_{j}(\varphi(\theta;x))= g(p,\varphi(\theta;x))\\
         \varphi(0;x)= & x,p=(t_{1},...,t_{m})\in D_{m}
       \end{array}
     \right.
\end{equation}
Here $g(p,y)$ is a locally Lipschitz continuous function with
respect to $y\in\theta$ and $\{\varphi(\theta;x):\theta\in[0,1]\}$
is unique. The proof is complete.
\end{proof}
\noindent\textbf{Comment}(total differential equation and gradient
system)\index{Gradient systems}\\
For a continuously differentiable
function $y(p):D_{m}=\prod_{1}^{m}(-a_{k},a_{k})\rightarrow
\mathbb{R}^{n}$ define the corresponding differential as follows
$$dy(p)=\Sigma_{j=1}^{m}\frac{\partial y(p)}{\partial
t_{j}}dt_{j}$$
where
$$p=(t_{1},...,t_{m})\in D_{m}$$
Let$\{Y_{1},...,Y_{m}\}\subseteq
C_{2}(\mathbb{R}^{n},\mathbb{R}^{n})$ be given and consider the
following equation (using differentials)\\
$dy(p)=\Sigma_{j=1}^{m}Y_{j}(y(p))dt_{j},p\in D_{m},y(0)=x\in
V(x_{0})\subseteq\mathbb{R}^{n}$.$\theta$  solution for the last
equation implies that \\
$\frac{\partial y}{\partial
y}=Y_{j}(y),j\in\{1,...,m\},p=(t_{1},...,t_{m})\in D_{m}$ is a
gradient system for $y\in V(x_{0})$ and it can be solved using
Frobenius theorem\index{Frobenius! theorem}(\eqref{th:1.9}). In addition, for an analytic
function of $z\in D\subseteq\mathbb{C}$,\\
$w(z):D\rightarrow \mathbb{C}^{n}$ defines the corresponding
differential \\
$dw(z)=dw_{1}(x,y)+i dw_{2}(x,y),z=a+iy\in D$\\where
$w(z)=w_{1}(x,y)+iw_{2}(x,y)$and\\ $v(x,y)=\left(
          \begin{array}{c}
            w_{1} \\
            w_{2} \\
          \end{array}
        \right)
(x,y):D_{2}=\prod_{1}^{2}(-a_{k},a_{k})\rightarrow \mathbb{R}^{2n}$
is an analytic function.Let
$f(w):\Omega\subseteq\mathbb{C}^{n}\rightarrow \mathbb{C}^{n}$ be
given analytic function and associate the following equation (using
differentials)\\
$dw(z)=f(w(z))dz=[f_{1}(v(x,y))dx-f_{2}(v(x,y))dy]+\\
+i[f_{2}(v(x,y))dx+f_{2}(v(x,y))dy]$ where \\
$f(w)=f_{1}(v)+if_{2}(v)$ and $f_{1},f_{2}$ are real analytic
functions of $v=\left(
          \begin{array}{c}
            w_{1} \\
            w_{2} \\
          \end{array}
        \right)\in D_{2n}=\prod_{k=1}^{2n}(-a_{k},a_{k})$. To solve
the last complex equation\index{Complex!equation} we need to show that $Y_{1}(v)=\left(
          \begin{array}{c}
            f_{1} \\
            f_{2} \\
          \end{array}
        \right)(v)$ and $Y_{2}(v)=\left(
          \begin{array}{c}
            -f_{1} \\
            f_{2} \\
          \end{array}
        \right)(v)$are commuting on $v\in D_{2n}$ which shows the
system $\frac{\partial v}{\partial x}=Y_{1}(v),\frac{\partial
v}{\partial y}=Y_{2}(v)$ is completely integrable on $D_{2}$.
\section{Appendix}
\textbf{$(a_{1})$}Assuming that $f(x,y):I\times G\subseteq
\mathbb{R}\times\mathbb{R}^{n}\rightarrow\mathbb{R}^{n}$ is a
continuous function, Peano proved the following existence theorem
\begin{theorem}(Peano)
Consider the following system of $ODE$
$$\frac{dy}{dx}=f(x,y),y(x_{0})=y_{0}\,\, \hbox{where}\,\, (x_{0},y_{0})\in I\times
G$$ is fixed and $\{f(x,y):(x,y)\in I\times G\}$ is a continuous
function. Then there exist an interval $J\subseteq I$ and a
continuously derivable function $y(x):J\rightarrow G$ satisfying
$y(x_{0})=y_{0}$ and $\frac{dy}{dx}(x)=f(x,y(x)),x\in
J$.\end{theorem}
 \begin{proof}
  It relies on Arzela-Ascoli theorem
which is refering to so called compact sequence\\$\{y_{k}(x),x\in
D\}_{k\geq 1}\subseteq C(D;\mathbb{R}^{n})$\\
 satisfying
$$(i)\,\,\, \hbox{A boundedness condition } \,\mid y_{k}(x)\mid\leq M,\forall x\in
D,k\geq
0$$
$$(ii)\,\,\,\,\hbox{ they are equally uniformly continuous, i.e for any } \,\epsilon>0 \,\hbox{there exists }\delta(\epsilon)>0$$
$$\hbox{ such that} \,\mid
x''-x'\mid<\delta(\epsilon)\, \hbox{ implies} \,\mid
y_{k}(x'')-y_{k}(x')\mid<\epsilon\, \forall\, x',x''\in D,k\geq 1$$
If a sequence $\{y_{k}(x):x\in D\}_{k\geq 1}$ satisfy (i) and (ii) then
Arzela-Ascoli theorem allows to find a subsequence
$\{y_{k_{j}}(x):x\in D\}_{j\geq 1}$ which is uniformly convergent on
the compact set $D$. Let $a>0$ and $b>0$ such that
$I_{a}=[x_{0}-a,x_{0}+a]\subseteq I$ and $B(y_{0},b)\subseteq
G$. Define $M=\max\{\mid f(x,y)\mid:(x,y)\in I_{a}\times B(y_{0},b)\}$
and let $\alpha=min(a,\frac{b}{M})$. Construct the following sequence
of continuous functions
$$y_{k}(x)=[x_{0},x_{0}+\alpha]\rightarrow G,y_{k}(x)=y_{0},
\hbox{for} x\in[x_{0},x_{0}+\frac{\alpha}{k}]$$
(iii)$$y_{k}(x)=y_{0}+\int_{x_{0}+\frac{\alpha}{k}}^{x}f(t,y_{k}(t-\frac{\alpha}{k})), \hbox{if } \,
x\in[x_{0}+\frac{\alpha}{k},x_{0}+\alpha] \hbox{for each } \,k\geq
1$$
$\{y_{k}(x):x\in[x_{0},x_{0}+\alpha]\}$ is well defined because
it is constructed on the interval
$[x_{0}+m\frac{\alpha}{k},x_{0}+(m+1)\frac{\alpha}{k}]$ using its
definition on the preceding interval
$[x_{0}+(m-1)\frac{\alpha}{k},x_{0}+m\frac{\alpha}{k}]$. It is easily
seen that$\{y_{k}(x):x\in[x_{0},\alpha]\}_{k\geq 1}$ is uniformly
bounded and $\mid y_{k}(x)-y_{0}\mid\leq M\alpha\leq b$.\\
In addition, $\mid y_{k}(x'')-y_{k}(x')\mid=0$, if
$x',x''\in[x_{0},x_{0}+\frac{\alpha}{k}]$ and $\mid
y_{k}(x'')-y_{k}(x')\mid\leq M\mid x''-x'\mid$ if
$x',x''\in[x_{0}+\frac{\alpha}{k},x_{0}+\alpha]$ which implies $\mid
y_{k}(x'')-y_{k}(x')\mid\leq M\mid x''-x'\mid$ for any
$x',x''\in[x_{0},x_{0}+\alpha],k\geq 1$.\\
Let $\hat{y}(x):[x_{0},x_{0}+\alpha]\rightarrow B(y_{0},b)\subseteq
G$ be the continuous function obtained using a
subsequence$\{y_{k_{j}}(x):x\in[x_{0},x_{0}+\alpha]\}_{j\geq
1}, \hat{y}(x)=\mathop{lim}\limits_{j\rightarrow\infty}y_{k_{j}}(x)$ uniformly with
respect to $x\in[x_{0},x_{0}+\alpha]$. By passing $j\rightarrow\infty$
into the equation (iii) written for
$k=k_{j}$ we get
$$\hat{y}(x)=y_{0}+\int_{x_{0}}^{x}f(t,\hat{y(t)})dt,\forall
x\in[x_{0},x_{0}+\alpha]$$
and the proof is complete.
\end{proof}
\noindent\textbf{($a_{2})$.\,Ordinary differential equations with delay}\index{Ordinary differential equations}\\
Let $\sigma>0$ be fixed and consider the following system of $ODE$
\begin{equation}\label{eq:1.132}
\left\{
      \begin{array}{ll}
        \frac{dy}{dx}(x)= & f(x,y(x),y(x-\sigma)),x\in[0,a] \\
        y(\sigma)= & \varphi(\sigma),\sigma\in[-\sigma,\sigma]
      \end{array}
    \right.
\end{equation}
where $\varphi\in C([-\sigma,0];\mathbb{R}^{n})$ is fixed as a
continuous function(Cauchy condition),and
$f(x,y,z):I\times\mathbb{R}^{n}\times\mathbb{R}^{n}\rightarrow\mathbb{R}^{n}$
is a continuous function admitting first order continuous partial
derivatives
$\partial_{y_{i}}f(x,y,z):I\times\mathbb{R}^{n}\times\mathbb{R}^{n}\rightarrow\mathbb{R}^{n},i\in\{1,...,n\}$.\\
The  above given assumptions allow us to construct a unique local
solution satisfying \eqref{eq:1.132} and it is done starting with the system \eqref{eq:1.132}
on the fixed interval $x\in[o,\sigma]$
\begin{equation}\label{eq:1.133}
\left\{
      \begin{array}{ll}
        \frac{dy}{dx}= & f(x,y(x),\varphi(x-\sigma)),x\in[0,\sigma] \\
        y(0)= & \varphi(0)
      \end{array}
    \right.
\end{equation}
which satisfies the condition of Cauchy-Lipschitz theorem\index{Cauchy!Lipschitz theorem}. In order
to make sure that the solution of \eqref{eq:1.133} exists for any
$x\in[0,\sigma]$ we need to assume a linear growth of $f$ with
respect to the variable $y$ uniformity  of $(x,z)$ in compact
sets,i.e\\
(*)$\mid f(x,y,z)\mid\leq C(1+\mid y\mid)\,\forall\,\in\mathbb{R}^{n}$
and $(x,z)\in J\times K(compact)\subseteq
I\times\mathbb{R}^{n}$,where the constant $C>0$ depends on the
arbitrary fixed compact $J\times K$.\\
Any function $f$ satisfying\\
(**)$f(x,y,z)=A(x,z)y+b(x,z),(x,z)\in I\times\mathbb{R}^{n}$ where
the matrix $A(x,z)\in M_{n\times n}$ and the vector
$b(x,z)\in\mathbb{R}^{n}$ are continuous functions will satisfy the
necessary conditions to extend the solution of \eqref{eq:1.133} on any interval
$[m\sigma,(m+1)\sigma], m\geq 0$
\newpage
\noindent\textbf{Exercises(Linear integrable system)}\\
(1) Formulate Frobenius theorem \index{Frobenius! theorem}(\eqref{th:1.9}). Rewrite the
content of Frobenius theorem\index{Frobenius! theorem} using linear vector fields
$Y_{i}(y)=A_{i}y$, where $A_{i}\in M_{n\times
n}(\mathbb{R}),i\in\{1,...,n\}$.\\
(2) For any matrices $A,B\in M_{n\times n}(\mathbb{R})$ define the
meaning of the exponential mapping\index{Exponential !mapping} \\
$(\exp tad_{A})(B)$, where $ad_{A}$(adjoint mapping \index{Adjoint!mapping}associated with
$A$)is the linear mapping $ad_{A}:M_{n\times
n}(\mathbb{R})\rightarrow M_{n\times n}(\mathbb{R})$, defined by
$ad_A(B)=AB-BA=[A,B]$(Lie bracket of $(A,B)$).\\
(3) Show that $(\exp tA)B(\exp-tA)=(\exp tad_{A})(B)$, for any $A,B\in
M_{n\times n}(\mathbb{R}),\\t\in\mathbb{R}$ where
$(\exp tad_{A})$ is the linear mapping defined in (2).\\
(4)Rewrite the gradient system (\eqref{th:1.8})when the vector
fields $Y_{i}(y)=A_{i}y,i\in\{1,...,m\}$,are linear. Show that,in
this case the vector fields with parameters are the following \\
$Y_{1}(y)=A_{1}y,X_{2}(t_{1};y)=(\exp t_{1}ad_{A_{1}})(A_{2})y,...,X_{m}(t_{1},...,t_{m-1};y)=[\exp t_{1}ad_{A_{1}}\circ(\exp t_{2}ad_{A_{2}})...\circ(\exp t_{m-1}ad_{A_{m-1}})(A_{m})]y$,such
that the corresponding gradient system\index{Gradient systems} \\
$\frac{\partial y}{\partial t_{1}}=Y_{1}(y),\frac{\partial
y}{\partial t_{2}}=X_{2}(t_{1},y),...,\frac{\partial y}{\partial
t_{m}}=X_{m}(t_{1},...,t_{m-1};y)$ has the unique solution \\
$y(t_{1},...,t_{m};y_{0})=(\exp t_{1}A_{1})...(\exp t_{m}A_{m})y_{0}$.\\
\section*{Bibliographical Comments }
The entire Chapter 1 is presented with minor changes as in the \cite{12}.

\chapter[First Order $PDE$]{First Order Partial Differential Equation}\index{Partial differential equations}
Let $D\subseteq\mathbb{R}^{n}$ be an open set and
$f(y):D\rightarrow\mathbb{R}^{n}$ is a continuous function.Consider
the following nonlinear $ODE$\\
\begin{equation}\label{eq:2.1}
\frac{dy}{dt}=f(y)
\end{equation}
\begin{definition} Let $D_{0}\subseteq D$ be an open
set. A function $u(y):D_{0}\rightarrow \mathbb{R}$ is called first
integral for \eqref{eq:2.1} on $D_{0}$ if
\begin{enumerate}
  \item $u$ is nonconstant on $D_{0}$,
  \item $u$ is continuously differentiable on $D_{0} (u\in
\mathcal{C}^{1}(D_{0};\mathbb{R}))$,
  \item for each solution $y(t):I\subseteq\mathbb{R}\rightarrow D_{0}$
of the system \eqref{eq:2.1} there exists a constant $c\in\mathbb{R}$ such that
$U(y(t))=c,\forall \,t\in I$.
\end{enumerate}
\end{definition}
\begin{example} A Hamilton system is described by the
following system of$ODE$
\begin{equation}\label{eq:2.2}
\frac{dp}{dt}=\frac{\partial H}{\partial
q}(p,q),\frac{dq}{dt}=-\frac{\partial H}{\partial
p}(p,q),p(0)=p_{0},q(0)=q_{0}
\end{equation}
where
$H(p,q):G_{1}\times
G_{2}\subseteq\mathbb{R}^{n}\times\mathbb{R}^{n}\rightarrow\mathbb{R}$
is continuously differentiable and $G_{i}\subseteq
\mathbb{R}^{n},i\in\{1,2\}$, are open sets. Let
$\{(p(t),q(t)):t\in[0,T]\}$ be solution of the Hamilton system
\eqref{eq:2.1} and by a direct computation, we get
$$\frac{d}{dt}[H(p(t),q(t))]=0
\hbox{ for any }t\in[0,T]$$
It implies $H(p(t),q(t))=H(p_{0},q_{0})\forall\, t\in[0,T]$ and
$\{H(p,q):(p,q)\in G_{1}\times G_{2} \}$  is a first integral for
\eqref{eq:2.2}
\end{example}
\begin{theorem}\label{th:2.1}
Let $f(y):D\rightarrow\mathbb{R}^{n}$ be a
continuous function,$D_{0}\subseteq D$ an open subset and consider a
nonconstant continuously differentiable function $U\in
\mathcal{C}^{1}(D_{0},\mathbb{R})$. Then $\{U(y):y\in D_{0}\}$ is
first integral on $D_{0}$ of $ODE$ \,iff the following
differential equality
\begin{equation}\label{eq:2.3}
<\frac{\partial u(y)}{\partial
y},f(y)>=\Sigma_{i=1}^{n}\frac{\partial u}{\partial
y_{i}}(y)f_{i}(y)=0,\forall y\in D_{0}
\end{equation}is satisfied, where
$$f=(f_{1},...,f_{n}),\frac{\partial u}{\partial y}=(\frac{\partial
u}{\partial y_{1}},...,\frac{\partial u}{\partial y_{n}})$$
\end{theorem}
\begin{proof} Let $\{u(y):y\in D_{0}\}$ be a continuously
differentiable first integral of$ODE$\eqref{eq:2.1} and consider that
$y(t,y_{0}):(-\alpha,\alpha)\rightarrow D_{0}$ is a solution of \eqref{eq:2.1}
verifying $y(0,y_{0})=y_{0}\in D_{0}$. By hypothesis,
$u(y(t,y_{0}))=c=u(y_{0})$, for any $t\in(-\alpha,\alpha)$
and by derivation, we get
\begin{equation}\label{eq:2.4}
0=\frac{d}{dt}[u(y(t,y_{0}))]=<\frac{\partial u}{\partial
y}(y(t,y_{0})),f(y(t,y_{0}))>,\,\hbox{for any}\, t\in (-\alpha,\alpha)
\end{equation}
In particular, for $t=0$ we obtain the conclusion \eqref{eq:2.3}
for an arbitrary fixed $y_{0}\in D_{0}$. The reverse implication
uses the equality\eqref{eq:2.3} and define $\varphi(t)=u(y(t)),t\in
(-\alpha,\alpha)$ where $u\in \mathcal{C}^{1}(D_{0},\mathbb{R})$
fulfils \eqref{eq:2.3} and $\{y(t):t\in
(-\alpha,\alpha)\}$ is a solution of \eqref{eq:2.1}. It follows that
$$\frac{d\varphi}{dt}(t)=<\frac{\partial u}{\partial
y}(y(t)),f(y(t))>=0,t\in (-\alpha,\alpha)$$
and
$$\varphi(t)=constant,t\in (-\alpha,\alpha)$$
The proof is complete.
\end{proof}
Let $\Omega\subseteq\mathbb{R}^{n+1}$ be an open set and consider
that
\begin{equation}\label{eq:2.5}
f(x,u):\Omega\rightarrow\mathbb{R}^{n},L(x,u):\Omega\rightarrow\mathbb{R}
\end{equation}
are given continuously differentiable function, $f\in
\mathcal{C}^{1}(\Omega;\mathbb{R}^{n}),L\in \mathcal{C}^{1}(\Omega;\mathbb{R})$.
\begin{definition}
A causilinear first order $PDE$ is defined by the following
differential equality
\begin{equation}\label{eq:2.6}
<\frac{\partial u(x)}{\partial x},f(x,u(x))>=L(x,u(x)),x\in
D(open)\subseteq \mathbb{R}^{n}
\end{equation}
where
$f\in\mathcal{C}^{1}(\Omega;\mathbb{R}^{n}), L\in
\mathcal{C}^{1}(\Omega;\mathbb{R})$ are fixed and
$u(x):D\rightarrow\mathbb{R}$ is an unknown continuously
differentiable function such that $(x,u(x))\in \Omega,\forall x\in
D$\end{definition}
A solution for \eqref{eq:2.6} means a function $u\in
\mathcal{C}^{1}(D;\mathbb{R})$ such that  $(x,u(x)\in\Omega)$ and
\eqref{eq:2.6}is
satisfied for any $x\in D$
\begin{remark}\label{re:2.4} Assuming that $L(x,u)=0$ and $f(x,u)=f(x)$ for
any $(x,u)\in D\times\mathbb{R}=\Omega$ then $PDE$ defined in \eqref{eq:2.6}
stands for the differential equality\eqref{eq:2.3} in Theorem \ref{th:2.1} defining first
integral for $ODE$ \eqref{eq:2.1}.
In this case, the differential system \eqref{eq:2.1}is called Cauchy
characteristic system \index{Cauchy!characteristic system}
associated with
linear $PDE$
\begin{equation}\label{eq:2.7a}
<\frac{\partial u(x)}{\partial x},f(x)>=0,x\in
D\subseteq\mathbb{R}^{n}
\end{equation}
\end{remark}
\begin{proposition}\label{pr:2.1}
Let $D_{0}\subseteq D$ be an open subset and consider a nonconstant
continuously differentiable scalar function
$u(x):D_{0}\rightarrow\mathbb{R}$. Then $\{u(x):x\in D_{0}\}$ is a
solution for the linear $PDE $ \eqref{eq:2.7a} iff $\{u(x):x\in D_{0}\}$ is a
first integral for $ODE$ \eqref{eq:2.1} on $D_{0}$.
\end{proposition}
\begin{definition} Let $a\in D\subseteq\mathbb{R}^{n}$ be
fixed and consider $m$ first integrals
\\$\{u_{1}(y),...,u_{m}(y):y\in V(a)\subseteq D\}$ for $ODE$ \eqref{eq:2.1}
which are continuously differentiable where $V(a)$ is a neighborhood
of $a$.We say that $u_{i}\in
\mathcal{C}^{1}(V(a);\mathbb{R}),i\in\{1,...,m\}$ are independent
first integrals of \eqref{eq:2.1} if $rank(\frac{\partial u_{1}}{\partial
y}(a),...,\frac{\partial u_{m}}{\partial y}(a))=m$
\end{definition}
\begin{remark}\label{th:2.2}
Let $f\in \mathcal{C}^{1}(D;\mathbb{R}^{n})$ and $a\in D$
be fixed such that $f(a)\neq 0$. Then  there exist a neighborhood
$V(a)\subseteq D$ and $(n-1)$ independent first integrals
$\{u_{1}(y),...,u_{n-1}(y):y\in V(a)\}$ of the system \eqref{eq:2.1}.
\end{remark}
\begin{proof}
Using a permutation of the coordinates $(f_{1},...,f_{n})=f$ we say
that $f_{n}(a)\neq 0$ (see $f(a)\neq 0$). Denote
$\Lambda=\{(\lambda_{1},...,\lambda_{n-1})\in
\mathbb{R}^{n-1}:(\lambda_{1},...,\lambda_{n-1},a_{n})\in D\}$ and
consider $\{F(t,\lambda):(t,\lambda)\in[0,T]\times \Sigma\}$ as the
local flow\index{Flow!local} associated with $ODE$ \eqref{eq:2.1}
 \begin{equation}\label{eq:2.8}
 \left\{
      \begin{array}{ll}
        \frac{dF}{dt}(t,\lambda)= & f(F(t,\lambda)),t\in[-\alpha,\alpha],\lambda\in\Sigma\subseteq\Lambda \\
        F(0,\lambda)= & (\lambda_{1},...,\lambda_{n-1},a_{n})
      \end{array}
    \right.
\end{equation}
where $\Sigma\subseteq\Lambda$ is a compact subset and
$(a_{1},...,a_{n-1})\in int\Sigma$. Using the differentiability
properties of the local flow
$\{F(t,\lambda):(t,\lambda)\in[-\alpha,\alpha]\times\Sigma\}$, we
get
$$det \left(
       \begin{array}{cccccc}
        \frac{\partial F}{\partial t}  & \frac{\partial F}{\partial \lambda_{1}} & . & . & . & \frac{\partial F}{\partial
\lambda_{n-1}} \\
       \end{array}
     \right).
(0,a_{1},...,a_{n-1})=$$
$$=det\left(
                                                  \begin{array}{ccccccc}
                                                    f_{1}(a) & 1 & 0 & . & . & . & 0 \\
                                                     & 0 & 1 &  &  &  & . \\
                                                     & . & 0 &  & &  & . \\
                                                    . & . & . &  &  &  & . \\
                                                    . & . & . &  &  &  & 0 \\
                                                    . &  & . &  &  &  & 1 \\
                                                     f_{n}(a) & 0 & 0 & . & . & . & 0 \\
                                                  \end{array}
                                                \right)
=(-1)^{n+1}f_{n}(a)\neq 0$$
and the conditions for applying an implicit functions theorem are
fulfilled when the algebraic
equations
\begin{equation}\label{eq:2.9}
F(t,\lambda)=y,\,\,\,y\in V(a)\subseteq
D,F(0,a_{1},...,a_{n-1})=a
\end{equation}
are considered. It implies
that there exist
$t=u_{0}(y),\lambda_{i}=u_{i}(y),i\in\{1,...,n-1\},y\in V(a)$, which
are continuously differentiable such that
\begin{equation}\label{eq:2.10}
F(u_{0}(y),u_{1}(y),...,u_{n-1}(y))=y,\forall y\in V(a)
\end{equation}
 where
$u_{0}(y)\in(-\alpha,\alpha)$ and
$\lambda(y)=(u_{1}(y),...,u_{n-1}(y)),\lambda\in\Sigma$, for any
$y\in V(a)$.\\
In addition, $\{u_{0}(y),...,u_{n-1}(y):y\in V(a)\}$ is the unique
solution satisfying \eqref{eq:2.10} and
\begin{equation}\label{eq:2.11}
u_{0}(F(t,\lambda))=t,u_{i}(F(t.\lambda))=\lambda_{i},i\in\{1,...,n-1\}
\end{equation}
Using \eqref{eq:2.11} we get easily
that $\{u_{1}(y),...,u_{n-1}(y):y\in V(a)\}$ are $(n-1)$ first
integrals for $ODE$ \eqref{eq:2.1} which are independent noticing
that$$<\frac{\partial u_{i}}{\partial y}(a),\frac{\partial
F}{\partial \lambda_{j}}(0,\lambda)>=\delta_{ij}$$ \\
where $$\frac{\partial F}{\partial
\lambda_{j}}(0,\lambda)=e_j\in\mathbb{R}^{n},\delta_{ij}=\left\{
                                                                   \begin{array}{ll}
                                                                     1, & i=j \\
                                                                     0, &
i\neq 0, \,j\in\{1,...,n-1\}
                                                                   \end{array}
                                                                 \right.
$$\\
and $\{e_{1},...,e_{n}\}\subseteq\mathbb{R}^{n}$ is the
canonical basis. In conclusion, $$<\frac{\partial u_{i}}{\partial
y}(a),e_{j}>=\delta_{ij}$$ for any $$i,j\in\{1,...,n-1\}$$ and
$$rank\left(
       \begin{array}{ccccc}
         \frac{\partial u_{1}}{\partial y}(a) & . & . & . & \frac{\partial u_{n-1}}{\partial y}(a) \\
       \end{array}
     \right)
=n-1$$\\
The proof is complete.
\end{proof}
\begin{theorem}\label{th:2.3}
Let $f(y):D\rightarrow\mathbb{R}^{n}$ be a continuously
differentiable function such that $f(a)\neq 0$ for fixed $a\in
D$. Let $\{u_{1}(y),...,u_{n-1}(y):y\in V(a)\subseteq D\}$ be the
$(n-1)$ first integrals constructed in Theorem \ref{th:2.2}. Then for any
first integral of \eqref{eq:2.1} $\{u(y):y\in V(a)\},u\in
\mathcal{C}^{1}(V,\mathbb{R})$, there exists an open set
$\theta\subseteq\mathbb{R}^{n-1}$and $h\in
\mathcal{C}^{1}(\theta,\mathbb{R})$
such that $u(y)=h(u_{1}(y),...,u_{n-1}(y)),y\in V_{1}(a)\subseteq V(a)$.
\end{theorem}
\begin{proof} By hypothesis $\{u(y):y\in V(a)\}$ is a
continuously differentiable first integral of $ODE$ \eqref{eq:2.1} and we get
\begin{equation}\label{eq:2.12}
u(F(t,\lambda_{1},...,\lambda_{n-1}))=h(\lambda_{1},...,\lambda_{n-1}),\forall
t\in[-\alpha,\alpha]
\end{equation}
provided that
$\{F(t,\lambda):t\in[-\alpha,\alpha],\lambda\in\Sigma\}$ is the
local flow \index{Flow!local}associated with $ODE$ \eqref{eq:2.1} satisfying
$F(0,\lambda_{1},...,\lambda_{n-1})=h(\lambda_{1},...,\lambda_{n-1},a_{n})$(see \eqref{eq:2.8} of
Theorem \ref{th:2.2}). As far as $F\in
\mathcal{C}^{1}((-\alpha,\alpha)\times\Sigma;\mathbb{R}^{n})$ and
$u\in \mathcal{C}^{1}(V(a),\mathbb{R})$ we get $h\in
\mathcal{C}^{1}(\theta,\mathbb{R})$ fulfilling \eqref{eq:2.12} where
$(a_{1},...,a_{n-1})\in\theta(open set)\subseteq\Sigma$. Using \eqref{eq:2.10} in Theorem \ref{th:2.2} we obtain\\
$$F(u_{0}(y),...,u_{n-1}(y))=y\in V(a)$$
and rewrite \eqref{eq:2.12} for
$$t=u_{0}(y),\lambda_{i}=u_{i}(y).i\in\{1,...,n-1\}$$
we get
\begin{equation}\label{eq:2.13}
u(y)=h(u_{1}(y),...,u_{n-1}(y))\,\hbox{for some} \,y\in
V_{1}(a)\subseteq V(a)
\end{equation}
 where
$(u_{1}(y),...,u_{n-1}(y))\in\theta, \,y\in V_{1}$. The proof is
complete.
\end{proof}
\begin{remark}\label{re:2.5}The above given consideration can be extended to
non autonomous $ODE$
\begin{equation}\label{eq:2.14}
\frac{dy}{dt}=f(t,y),t\in\in I\subseteq \mathbb{R},y\in
G\subseteq\mathbb{R}^{n}
\end{equation}
where $f(t,y):I\times
G\rightarrow\mathbb{R}^{n}$ is a continuous
function.\\
In this respect, denote $$z=(t,y)\in\mathbb{R}^{n+1},D=I\times
G\subseteq \mathbb{R}^{n+1}$$ and consider
$$g(z):D\rightarrow\mathbb{R}^{n+1}$$ defined by
\begin{equation}\label{eq:2.15}
g(z)=\hbox{coloumn}(1,f(t,y))
\end{equation}
With these notations, the system \eqref{eq:2.14} can be written as
an autonomous
$ODE$
\begin{equation}\label{eq:2.16}
\frac{dz}{dt}=g(z)
\end{equation}
where $g\in \mathcal{C}^{1}(D;\mathbb{R}^{n+1})$
provided $f$ is continuously
differentiable on $(t,y)\in I\times G$.\\
Using Proposition \ref{pr:2.1} we restate the conclusion of Theorem \ref{th:2.3} as a
result for linear $PDE$ \eqref{eq:2.7a}.
\end{remark}
\begin{proposition}\label{pr:2.2}
Let $a\in D\subseteq\mathbb{R}^{n}$ be
fixed such that $f(a)\neq 0$, where $f\in
\mathcal{C}^{1}(D;\mathbb{R}^{n})$ defines the characteristic system
\index{Characteristic! system }\eqref{eq:2.1}. Let $\{u_{i}(y):y\in
V(a)\subseteq D\},i\in\{1,...,n-1\}$ be the $(n-1)$ first integral
for \eqref{eq:2.1} constructed in Theorem \ref{th:2.2}. Then an arbitrary solution
$\{u(x):x\in V(a)\subseteq D\}$ of the linear $PDE$ \eqref{eq:2.7a} can be
represented
$$u(x)=h(u_{1}(x),...,u_{n-1}(x)),x\in V_{1}(a)\subseteq V(a)$$
where
$h\in \mathcal{C}^{1}(\theta,\mathbb{R})$ and
$\theta\subseteq\mathbb{R}^{n-1}$ is open set. In particular, each
$\{u_{i}:x\in V(a)\subseteq D\}$ is a solution for the linear
PDE \eqref{eq:2.7a}, $i\in\{1,...,n-1\}$
\end{proposition}
\begin{remark}\label{re:.6}A strategy based on the corresponding Cauchy
characteristic system \index{Characteristic!system
}\index{Cauchy!characteristic system}of $ODE$ can be used for
solving the qausilinear PDE given in \eqref{eq:2.6}.
\end{remark}
\section{Cauchy Problem for the Hamilton-Jacobi\\ Equations\index{Cauchy!Problem for the Hamilton-Jacobi Equations}} A
linear Hamiltonion-Jacobi (H-J) is defined by the following first
order $PDE$\\
\begin{equation}\label{eq:2.17}
\partial_{t}S(t,x)+<\partial_{x}S(t,x),g(t,x)>=L(t,x),t\in
I\subseteq\mathbb{R},x\in G\subseteq \mathbb{R}^{n}
\end{equation}
where
$S\in \mathcal{C}^{1}(I\times G;\mathbb{R})$ is the unknown function
$$\partial_{t}S=\frac{\partial S}{\partial
t},\partial_{x}S=\frac{\partial S}{\partial
x}$$
and
$$g(t,x):I\times G\rightarrow \mathbb{R}^{n}, L(t,x):I\times
G\rightarrow \mathbb{R}$$
are given continuously differentiable
functions. Here $I\subseteq \mathbb{R}$ and
$G\subseteq\mathbb{R}^{n}$ are open sets and for each $t_{0}\in
I,h\in\mathcal{C}^{1}(G;\mathbb{R})$ fixed. The following Cauchy
problem can be defined. Find a solution
$S(t,x):(t_{0}-\alpha,t_{0}+\alpha)\times D\rightarrow\mathbb{R}$
verifying (H-J) equation \eqref{eq:2.17} $(t,x)\in
(t_{0}-\alpha,t_{0}+\alpha)\times D,D(open)\subseteq G$ such that
$S(t_{0},x)=h(x),x\in D$, where $t_{0}\in I$ and $h\in
\mathcal{C}^{1}(G;\mathbb{R})$ are fixed, and
$(t_{0}-\alpha,t_{0}+\alpha)\subseteq I$. A causilinear Hamiltonion-Jacobi (H-J) equation is defined by the
following first order $PDE$
{\small
\begin{equation}\label{eq:2.18}
\partial_{t}S(t,x)+<\partial_{x}S(t,x),g(t,x,S(t,x))>=L(t,x,S(t,x)),t\in
I\subseteq\mathbb{R},x\in G\subseteq \mathbb{R}^{n}
\end{equation}}
where $g\in \mathcal{C}^{1}(I\times G\times
\mathbb{R};\mathbb{R}^{n}),L\in \mathcal{C}^{1}(I\times G\times
\mathbb{R};\mathbb{R})$ are fixed and $I\subseteq
\mathbb{R},G\in\mathbb{R}^{n}$ are open sets. A Cauchy problem for
causilinear (H-J)\eqref{eq:2.18}is defined as follows:let $t_{0}\in I$and
$h\in \mathcal{C}^{1}(G;\mathbb{R})$ be given and find a solution
$S(t,x):(t_{0}-\alpha,t_{0}+\alpha)\times D\rightarrow \mathbb{R}$
verifying \eqref{eq:2.18}for any $t\in (t_{0}-\alpha,t_{0}+\alpha)\subseteq
I$ and $x\in D(open)\subseteq G$ such that $S(t_{0},x)=h(x),x\in
D$. A unique Cauchy problem solution of \eqref{eq:2.17} and \eqref{eq:2.18} are
constructed using the Cauchy method of characteristics
\index{Characteristic}which relies on the corresponding Cauchy
characteristic  system\index{Cauchy!characteristic  system} of
$ODE$. In this respect, we present this algorithm for \eqref{eq:2.18} and  let
$\{F(t,\lambda):t\in(t_{0}-\alpha,t_{0}+\alpha),\lambda\in\Sigma\subseteq
G\}$ be the local flow generated by the following characteristic
system\index{Characteristic! system } associated
with\eqref{eq:2.18}, $F(t,\lambda)=(\varphi(t,\lambda),u(t,\lambda))\in\mathbb{R}^{n+1}$
\begin{equation}\label{eq:2.19}
\left\{
       \begin{array}{ll}
         \frac{dF}{dt}(t,\lambda)= &f(t,F(t,\lambda)),(t,\lambda)\in (t_{0}-\alpha,t_{0}+\alpha)\times\Sigma \\
         F(t_{0},\lambda)=&(\lambda,
u(t_{0},\lambda))=(,\lambda,h(\lambda))\,,\lambda\in\Sigma\subseteq
G
 \end{array}
     \right.
\end{equation}
where $f(t,x,u)=(g(t,x,u),L(t,x,u))\in\mathbb{R}^{n+1}$ and $h\in
\mathcal{C}^{1}(G;\mathbb{R})$ is fixed. Using $det\left(
                                          \begin{array}{c}
                                            \frac{\partial\varphi}{\partial\lambda}(t,\lambda) \\
                                          \end{array}
                                        \right)
\neq 0$ for  any $t\in (t_{0}-\alpha,t_{0}+\alpha)$ and
$\lambda\in\Sigma(compact)\subseteq G$ (see
$\frac{\partial\varphi}{\partial\lambda}(t_{0},\lambda)=I_{n}$ and
we may assume that $\alpha>0$ is sufficiently small) an implicit
function theorem can be applied to the algebraic equation
\begin{equation}\label{eq:2.20}
\varphi(t,\lambda)=x\in V(x_{0})\subseteq G
\end{equation}
We find the unique solution of \eqref{eq:2.20}
\begin{equation}\label{eq:2.21}
\lambda=(\psi_{1}(t,x),...,\psi_{n}(t,x)):(t_{0}-\alpha,t_{0}+\alpha)\times
V(x_{0})\rightarrow\Sigma
\end{equation}
such that
$$\psi_{i}\in \mathcal{C}^{1}(t_{0}-\alpha,t_{0}+\alpha)\times
V(x_{0});\mathbb{R}),i\in\{1,...,n\}$$ fulfills
\begin{equation}\label{eq:2.22}
\varphi(t,\psi_{1}(t,x),...,\psi_{n}(t,x))=x\in
V(x_{0})\subseteq G
\end{equation}
for any $t\in
(t_{0}-\alpha,t_{0}+\alpha)$ and $x\in V(x_{0})$
\begin{equation}
\label{eq:2.23}
\psi_{i}(t_{0},x)=x_{i},i=1,...,n,\psi=(\psi_{1},...,\psi_{n})
\end{equation}
w Define $S(t,x)=u(t,\psi(t,x))$ and it will be a
continuously differentiable function on
$(t_{0}-\alpha,t_{0}+\alpha)\times V(x_{0})$ satisfying
\begin{equation}\label{eq:2.24}
S(t_{0},x)=u(t_{0},\psi(t_{0},x))=u(t_{0},x)=h(x),x\in
V(x_{0})
\end{equation}
In addition using
$\psi(t,\varphi(t,\lambda))=\lambda ,t\in
(t_{0}-\alpha,t_{0}+\alpha)$, and taking the derivative with respect
to the variable $t$, we get $S(t,\varphi(t,\lambda))=u(t,\lambda)$
and
{\small
\begin{equation}\label{eq:2.25}
\partial_{t}S(t,\varphi(t,\lambda))+<\partial_{x}S(t,\varphi(t,\lambda)),g(t,\varphi(t,\lambda),u(t,\lambda))>
=L(t,\varphi(t,\lambda),u(t,\lambda))
\end{equation}}
$\forall t\in (t_{0}-\alpha,t_{0}+\alpha),\lambda\in V_{1}(x_{0})$
where $V_{1}(x_{0})\subseteq V(x_{0})$ is taken such that
$\varphi(t,\lambda)\in V(x_{0})$ if
$(t,\lambda)\in(t_{0}-\alpha,t_{0}+\alpha)\times V_{1}(x_{0})$. In
particular, for $\lambda=\psi(t,x)$ into (1.25), we obtain
\begin{equation}\label{eq:2.26}
\partial_{t}S(t,x)+<\partial_{x}S(t,x),g(t,x,S(t,x))>=L(t,x,S(t,x))
\end{equation}
where $\varphi(t,\psi(t,x))=x$ and $u(t,\psi(t,x))=S(t,x)$ are used.
In addition,the Cauchy problem solution for \eqref{eq:2.18} is unique and
assuming that another solution $\{v(t,x)\}$ of \eqref{eq:2.18} satisfies
$v(t_{0},x)=h(x)$for $x\in D\subseteq G$ then $v(t,x)=S(t,x),\forall
t\in(t_{0}-\widetilde{\alpha},t_{0}+\widetilde{\alpha}),x\in
V(x_{0})\bigcap D$ where $\widetilde{\alpha}>0$. It relies on the
unique Cauchy problem associated with$ODE$ \eqref{eq:2.19}. The proof is
complete.
\section{Nonlinear First Order $PDE$}
\subsection{Examples of Scalar Nonlinear$ODE$}
We consider a simple scalar equation given implicitly by
\begin{equation}\label{eq:2.27}
F(x,y(x),y'(x))=0,y'(x)=\frac{dy}{dx}(x)
\end{equation}
where
$F(x,y,z):D(open)\subseteq\mathbb{R}^{3}\rightarrow\mathbb{R}$
be second order continuously differentiable satisfying\\
\begin{equation}\label{eq:2.28}
F(x_{0},y_{0},z_{0})=0,\partial_{z}F(x_{0},y_{0},z_{0})\neq
0 \end{equation}
 for some $(x_{0},y_{0},z_{0})\in D$ fixed. Notice that for a smooth
curve
$$\{\gamma(t)=(x(t),y(t),z(t)))\in D:t\in[0,a]\}$$ with
$$x(0)=x_{0},y(0)=y_{0},z(0)=z_{0}$$we get

\begin{equation}\label{eq:2.29}
F(\gamma(t))=0,t\in[0,a]
\end{equation}
provided
{\small
\begin{equation}\label{eq:2.30}
0=\frac{d}{dt}[F(\gamma(t))]=\partial_{x}F(\gamma(t))\frac{dx}{dt}(t)+\partial_{y}F(\gamma(t))\frac{dy}{dt}(t)+
\partial_{z}F(\gamma(t))\frac{dz}{dt}(t),t\in[0,a]
\end{equation}}
Using \eqref{eq:2.30} we may and do define a corresponding
characteristic
system\index{Characteristic! system } associated with \eqref{eq:2.27}
\begin{equation}\label{eq:2.31}
\left\{
       \begin{array}{ll}
         \frac{dx}{dt}= & \frac{\partial F}{\partial z}(x,y,z),\frac{dy}{dt}=z\frac{\partial F}{\partial z}(x,y,z) \\
         \frac{dz}{dt}= & -[\frac{\partial F}{\partial x}(x,y,z)+\frac{\partial F}{\partial
y}(x,y,z)z],x(0)=x_{0},y(0)=y_{0},z(0)=z_{0}
       \end{array}
     \right.
\end{equation}
Notice that each solution of \eqref{eq:2.31} satisfies \eqref{eq:2.29}. Define
$x=x(t),y=y(t),z=z(t),t\in[-a,a]$, the unique Cauchy problem
solution of \eqref{eq:2.31} and by definition
$\frac{dx}{dt}(0)=\partial_{z}F(x_{0},y_{0},z_{0})\neq 0$ allows
to apply an implicit function theorem for solving the following
scalar
equation
\begin{equation}\label{eq:2.32}
x(t)=x\in V(x_{0})\subseteq I\hbox{(interval)}
\end{equation}
We find a unique continuously derivable function
$t=\widetilde{\tau}(x):V(x_{0})\rightarrow(-a,a)$
\begin{equation}\label{eq:2.33}
x(\widetilde{\tau}(x))=x \,\hbox{and} \,
\widetilde{\tau}(x(t))=t(\frac{d\widetilde{\tau}}{dx}(x)\frac{dx}{dt}(\widetilde{\tau}(x))=1)
\end{equation}
Denote
$$\hat{y}(x)=y(\hat{\tau}(x)), \hat{z}(x)=z(\hat{\tau}(x)),x\in
V(x_{0})\subseteq I$$and it is easily seen that
$$\frac{d\hat{y}}{dx}=\frac{dy}{dt}(\hat{\tau}(x)).\frac{d\hat{\tau}}{dx}(x)=\hat{z}(x)$$It
implies that $$F(x,\hat{y}(x),\hat{y'}(x))=0,\forall x\in V(x_{0})$$
provided $$F(x(t),y(y),z(t))=0,t\in(-a,a)$$ is used. It shows that
$\{\hat{y}(x):x\in V(x_{0})\}$ is a solution of the scalar nonlinear
differential equation \eqref{eq:2.27}.
\begin{remark}\label{re:2.7}In getting the characteristic system \index{Characteristic! system }\eqref{eq:2.31}we must
confine ourselves to the following constraints
$y(t)=\hat{y}(x(t)),\frac{dy}{dt}(t)=\frac{d\hat{y}}{dx}(x(t)).\frac{dx}{dt}(t)=z(t)\frac{dx}{dt}$
where $z(t)=\frac{d\hat{y}}{dx}(x(t))$ and $\{\hat{y}(x),x\in
I\subseteq \mathbb{R}\}$ is a solution of \eqref{eq:2.27}. The following two
examples can be solved using the algorithm of the characteristic
system \index{Characteristic !system }used for the scalar
equation\eqref{eq:2.27}.
\end{remark}
\begin{example}(Clairant and Lagrange equations)
\begin{equation}\label{eq:2.34}
\left\{
       \begin{array}{ll}
        y= & xa(y')+b(y')\,\,\, \hbox{(Clairaut equation)} \\
         y= & xy'+b(y')\,\,\, \hbox{(Lagrange equation,a(z)=z)}
       \end{array}
     \right.
\end{equation}
Here $F(x,y,z)=xa(z)+b(z)-y$ and the corresponding
characteristic system\index{Characteristic! system} is given by
\begin{equation}\label{eq:2.35}
\left\{
  \begin{array}{ll}
    \frac{dx}{dt}= & xa'(z)+b'(z),\frac{dz}{dt}=-a(z)+z,\frac{dz}{dt}=z(xa'(z)+b'(z))
    \\
    x(0)= & x_{0}                       z(0)=z_{0}
y(0)=y_{0}
  \end{array}
\right.
\end{equation}
where $$x_{0}a(z_{0})+b(z_{0})-y_{0}$$
and
$$x_{0}a'(z_{0})+b'(z_{0})\neq 0$$
\end{example}
\begin{example}(Total differential equations)
\begin{equation}\label{eq:2.36}
\frac{dy}{dx}=\frac{g(x,y)}{h(x,y)}
\end{equation}
where
$g,h:D\subseteq\mathbb{R}^{2}\rightarrow\mathbb{R}$ are continuously
differentiable functions and $h(x,y)\neq 0,\forall(x,y)\in
D$. Formally, \eqref{eq:2.36} can be written as
\begin{equation}\label{eq:2.37}
-g(x,y)dx+h(x,y)dy=0
\end{equation}
and \eqref{eq:2.36} is a total differential equation if a second
order continuously differentiable function $F:D\rightarrow\mathbb{R}$
exists such that
\begin{equation}\label{eq:2.38}
\frac{\partial F}{\partial x}(x,y)=-g(x,y),\frac{\partial
F}{\partial y}(x,y)=h(x,y)\neq 0 \end{equation}
If $y=y(x),x\in I$,
is a solution of \eqref{eq:2.36} then $F(x,y(x))=constant,x\in I$, provided $F$
fulfils \eqref{eq:2.38}. In conclusion,assuming\eqref{eq:2.38}, the nonlinear first
order equation \eqref{eq:2.36} is
solved provided the corresponding algebraic equation
\begin{equation}\label{eq:2.39}
F(x,y(x))=c
\end{equation} is satisfied, where the constant $c$ is parameter.
\end{example}
\subsection{Nonlinear Hamilton-Jacobi E quations}
A Hamilton-Jacobi equation is a first order $PDE$ of the
following form
\begin{equation}\label{eq:2.40}
\partial_{t}u(t,x)+ H(t,x,u(t,x),\partial_{x}u(t,x))=0,t\in
I\subseteq \mathbb{R},x\in D\subseteq \mathbb{R}^{n}
\end{equation}
where $H(t,x,u,p):I\times
D\times\mathbb{R}\times\mathbb{R}^{n}\rightarrow\mathbb{R}$ is a
second order continuously differentiable function.
\begin{definition}
A solution for H-J equation \eqref{eq:2.40} means a
first order continuously differentiable function
$u(t,x):I_{a}(x_{0})\times B(x_{0},\rho)$ where
$B(x_{0},\rho)\subseteq I$is a ball centered at $x_{0}\in D$ and
$I_{a}(x_{0})=(x_{0}-a,x_{0}+a)\subseteq I$. A Cauchy problem for
the H-J equation\eqref{eq:2.40} means to find a solution $u\in
\mathcal{C}^{1}(I_{a}(x_{0})\times B(x_{0},\rho))$ of \eqref{eq:2.40} such
that $u(t_{0},x)=u_{0}(x),x\in B(x_{0},\rho)$ where $t_{0}\in I$ and
$u_{0}\in \mathcal{C}^{1}(D;\mathbb{R})$ are fixed.
\end{definition}
\noindent solution for Cauchy problem associated with H-J equation \eqref{eq:2.40} is
found using the corresponding characteristic system\index{Characteristic!system }
\begin{equation}\label{eq:2.41}
\left\{
    \begin{array}{ll}
     \frac{dx}{dt}= & \partial_{p}H(t,x,u,p),x(t_{0})=\xi\in B(x_{0},\rho)\subseteq D \\
      \frac{dp}{dt} =& -[\partial_{x}H(t,x,u,p)+p\partial_{u}H(t,x,u,p)], p(t_{0})=p_{0}(\xi) \\
      \frac{du}{dt}= &
-H(t,x,u,p)+<p, \partial_{p}H(t,x,u,p)>, u(t_{0})=u_{0}(\xi)
    \end{array}
  \right.
\end{equation}
where $p_{0}(\xi)=\partial_{\xi}u_{0}(\xi)$ and
$u_{0}\in
\mathcal{C}^{1}(D;\mathbb{R})$ are fixed. Consider that
\begin{equation}\label{eq:2.42}
\{(\hat{x}(t,\xi)),\hat{p}(t,\xi),\hat{u}(t,\xi):t\in
I_{a}(x_{0}),\xi\in B(x_{0},\rho)\}
\end{equation}
is the unique
solution fulfilling $ODE$ \eqref{eq:2.40}. By definition
$\partial_{\xi}\hat{x}(0,\xi)=I_{n}$ and assuming that $a>0$ is
sufficiently small, we admit
\begin{equation}\label{eq:2.43}
\partial_{\xi}\hat{x}(t,\xi) \,\hbox{is nonsingular for any}
(t,x)\in I_{a}(x_{0})\times B(x_{0},\rho)
\end{equation}
Using \eqref{eq:2.47}
we may and do apply the standard implicit functions theorem for
solving the algebraic equation
\begin{equation}\label{eq:2.44}
\hat{x}(t,\xi)=x\in B(x_{0},\rho_{1}),t\in
I_{\alpha}(x_{0})
 \end{equation}
We get a continuously differentiable mapping
$$\xi=\psi(t,x):I_{\alpha}(x_{0})\times B(x_{0},\rho_{1})\rightarrow
B(x_{0},\rho)$$ such that
\begin{eqnarray*}
  \hat{x}(t,\psi(t,x)) &=& x,\psi(t_{0},x)=x\, \hbox{\,\,\,\,\,\,\,\,\,\,\,\,\,\,\,and}\\
  \psi(t,\hat{x}(t,\xi))&=& \xi \,\,\,\,\,\,\,\,\,\,\,\,\,\,\,\,\,\,\,\,\,\,\,\,\,\,\,\,\,\,\,\,\,\,\,\,\,\,\,\,\,\,\,\,\,\,\,\,\,\,\hbox{for any}
 \end{eqnarray*}
\begin{equation}\label{eq:2.45}
t\in I_{\alpha}(x_{0}),x\in
B(x_{0},\rho_{1})
\end{equation}
Define the following continuously
differentiable function $$u\in
\mathcal{C}^{1}(I_{\alpha}(x_{0})\times
B(x_{0},\rho_{1});\mathbb{R}),p\in
\mathcal{C}^{1}(I_{\alpha}(x_{0})\times B(x_{0},\rho_{1});\mathbb{R}^{n})$$\\
\begin{equation}\label{eq:2.46}
u(t,x)=\hat{u}(t,\psi(t,x)),p(t,x)=\hat{p}(t,\psi(t,x))
\end{equation}
By definition $u(t_{0},x)=u_{0}(x),x\in B(x_{0},\rho_{1})\subseteq
D$ and to show that $\{u(t,x):t\in I_{\alpha}(x_{0}),x\in
B(x_{0},\rho_{1})\}$ is a solution for \eqref{eq:2.40} satisfying
$u(t_{0},x)=u_{0}(x),x\in B(x_{0},\rho_{1})\subseteq D$, we need to
show
\begin{equation}\label{eq:2.47}
\left\{
       \begin{array}{ll}
         \partial_{t}u(t,x)= &-H(t,x,u(t,x),p(t,x)) \\
       p(t,x)= & \partial_{x}u(t,x),(t,x)\in I_{\alpha}(x_{0})\times
B(x_{0},\rho_{1})
       \end{array}
     \right.
\end{equation}
The second equation of \eqref{eq:2.47}is valid if
$$\partial_{\xi}\hat{u}(t,\xi)=\hat{p}(t,\xi)\partial_{\xi}\hat{x}(t,\xi),(\hat{p}\in\mathbb{R}^{n})$$
is a row vector holds for each
\begin{equation}\label{eq:2.48}
t\in I_{\alpha}(x_{0}),\xi\in B(x_{0},\rho)
\end{equation}
which will be proved in the second form
\begin{equation}\label{eq:2.49}
\partial_{k}\hat{u}(t,\xi)=<\hat{p}(t,\xi),\partial_{k}\hat{x}(t,\xi)>,k\in\{1,...,n\}
\end{equation}
where
$$\partial_{k}\varphi(y,\xi)=\frac{\partial\varphi(t,\xi)}{\partial\xi_{k}},\xi=(\xi_{1},...,\xi_{n})$$\\
Using the characteristic system\index{Characteristic! system } \eqref{eq:2.40}
we notice that
$$\lambda_{k}(t,\xi)=\partial_{k}\hat{u}(t,\xi)-<\hat{p}(t,\xi),\partial_{k}\hat{x}(t,\xi)>$$
fulfils
\begin{equation}\label{eq:2.50}
\partial_{k}(t_{0},\xi)=\partial_{k}u_{0}(\xi)-<\partial_{\xi}u_{0}(\xi),e_{k}>=0,k\in\{1,...,n\}
\end{equation}
where
 $$\{e_{1},...,e_{n}\}\subseteq\mathbb{R}^{n}$$
is the canonical basis. In addition, by a direct computation, we obtain
\begin{eqnarray}\label{eq:2.51}
  \frac{d\lambda_{k}}{dt}(t,\xi) &=& \partial_{k}[\frac{d\hat{u}}{dt}(t,\xi)]-<\frac{d\hat{p}(t,\xi)}{dt} \nonumber\\
  \partial_{k}\hat{x}(t,\xi) &>-<& \hat{p}(t,\xi),\partial_{k}[\frac{d\hat{x}}{dt}(t,\xi)]>,t\in
I_{\alpha}(x_{0})
\end{eqnarray}
Notice that
\begin{eqnarray}\label{eq:2.52}
  \partial_{k}[\frac{d\hat{u}}{dt}(t,\xi)] &=& -\partial_{k}[H(t,\hat{x}(t,\xi),\hat{u}(t,\xi),\hat{p}(t,\xi))]
+<\partial_{k}\hat{p}(t,\xi) \nonumber\\
  \frac{\hat{x}}{dt}(t,\xi) &>+<& \hat{p}(t,\xi),\partial_{k}[\frac{d\hat{x}}{dt}(t,\xi)]>
\end{eqnarray}
and
{\small
\begin{equation}\label{eq:2.53}
\frac{d\hat{p}}{dt}(t,\xi)=-[\partial_{x}H(t,\hat{x}(t,\xi)),\hat{u}(t,\xi),\hat{p}(t,\xi_{1})+\hat{p}(t,\xi)\partial_{u}H(t,\hat{x}(t,\xi)),\hat{u}(t,\xi),\hat{p}(t,\xi)]
\end{equation}}
Combining \eqref{eq:2.52}and \eqref{eq:2.53} we obtain
\begin{eqnarray}\label{eq:2.54}
  \frac{d\lambda_{k}}{dt}(t,\xi) &=& -\partial_{u}H(t,\hat{x}(t,\xi),\hat{u}(t,\xi),\hat{p}(t,\xi)).\partial_{k}\hat{u}(t,\xi)+\partial_{u}H(t,\hat{x}(t,\xi),\hat{u}(t,\xi),\hat{p}(t,\xi)) \nonumber\\
   &<& \hat{p}(t,\xi),\partial_{k}\hat{x}(t,\xi)>
=-\partial_{u}H(t,\hat{x}(t,\xi),\hat{u}(t,\xi),\hat{p}(t,\xi)).\lambda_{k}(t,\xi)
\end{eqnarray}
\noindent which is a linear scalar equation for the unknown$\{\lambda_{k}\}$
satisfying $\lambda_{k}(t_{0},\xi)=0,k\in \{1,...,n\}$(see\eqref{eq:2.51}).
It follows $\lambda_{k}(t,\xi)=0,\forall(t,\xi)\in
I_{\alpha}(x_{0})\times B(x_{0},\rho_{1})$ and for any
$k\in\{1,...,n\}$, which proves
\begin{equation}\label{eq:2.55}
p(t,x)=\partial_{x}u(t,x),(t,x)\in I_{\alpha}(x_{0})\times
B(x_{0},\rho_{1})
\end{equation}
standing for the second equation
in\eqref{eq:2.47}. Using\eqref{eq:2.55}, $\hat{u}(t,\xi)=u(t,\hat{x}(t,\xi))$and the
third equation of the characteristic system\index{Characteristic!
system } \eqref{eq:2.40} we see easily that
\begin{eqnarray}\label{eq:2.56}
  \partial_{t}u(t,\hat{x}(t,\xi))+<\hat{p}(t,\xi),\frac{d\hat{x}}{dt}(t,\xi)&>=-& H(t,\hat{x}(t,\xi),\hat{u}(t,\xi),\hat{p}(t,\xi)) \nonumber\\
  &+<& \hat{p}(t,\xi),\frac{d\hat{x}}{dt}(t,\xi)>
\end{eqnarray}
which lead us to the first equation of \eqref{eq:2.47}
\begin{equation}\label{eq:2.57}
\partial_{t}u(t,\hat{x}(t,\xi))=-H(t,\hat{x}(t,\xi),\hat{u}(t,\xi),\hat{p}(t,\xi))
\end{equation}
for any $(t,\xi)\in I_{\alpha}(x_{0})\times
B(x_{0},\rho)$. In particular ,taking $\xi=\psi(t,x)$ in \eqref{eq:2.57}, we
get
\begin{equation}\label{eq:2.57}
\partial_{t}u(t,x)=-H(t,x,u(t,x),\partial_{x}u(t,x)),u(t_{0},x)=u_{0}(x)
\end{equation}
$$\forall(t,x)\in(t_{0}-\alpha,t_{0}+\alpha)\times
B(x_{0},\rho_{1})$$
which stands for the existence of a Cauchy problem solution. The
uniqueness of the Cauchy problem solution for H-J equation\eqref{eq:2.40} can
be easily proved using the fact that any other solution
$\hat{u}(t,x),(t,x)\in I_{\alpha}(x_{0})\times
B(x_{0},\rho_{2})$ satisfying \eqref{eq:2.49} and
$\hat{u}(t_{0},x)=u_{0}(x),x\in B(x_{0},\rho_{2})$ induces a
solution of the same $ODE$ \eqref{eq:2.40}. The conclusion is that the
uniqueness property  for the Cauchy problem solution of $ODE$ \eqref{eq:2.41}
implies that the H-J equation \eqref{eq:2.40} has a unique Cauchy problem
solution. The above given computations and considerations regarding
the H-J equation \eqref{eq:2.40} will be stated as
\begin{proposition}\label{pr:2.3}
Let $H(t,x,u,p):I\times
D\times\mathbb{R}\times\mathbb{R}^{n}\rightarrow\mathbb{R}$ be a
second order continuously differentiable scalar function, where
$I\subseteq\mathbb{R},D\subseteq\mathbb{R}^{n}$ are open sets. Let
$(t_{0},x_{0})\in I\times D$and $u_{0}\in
\mathcal{C}^{1}(D;\mathbb{R})$ are
fixed. Then the following nonlinear H-J equation
\begin{equation}\label{eq:2.58}
\partial_{t}u(t,x)+H(t,x,u(t,x),\partial_{x}u(t,x))=0
\end{equation}
This has a unique Cauchy problem solution satisfying
$$u(t_{0},x)-u_{0}(x),x\in B(x_{0},\rho_{0})\subseteq D$$In
addition,the unique Cauchy problem
solution$$\{u(t,x):t\in(t_{0}-\alpha,t_{0}+\alpha),x\in
B(x_{0},\rho_{0})\}$$ of \eqref{eq:2.58} is defined by
$u(t,x)=\hat{u}(t,\psi(t,x))$where $\{\varphi(t,x):t\in
I_{\alpha}(x_{0}),x\in B(x_{0},\rho_{0})\}$ is the unique solution
of the algebraic \eqref{eq:2.44} and
$$\{(\hat{x}(t,\xi),\hat{p}(t,\xi),\hat{u}(t,\xi)):t\in
I_{\alpha}(x_{0}),\xi\in B(x_{0},\rho) \}$$ is the unique solution of
the characteristic system \index{Characteristic! system }\eqref{eq:2.41}
\end{proposition}
Starting with the H-J equation \eqref{eq:2.58} (see\eqref{eq:2.40}) we may associate
the following system of H-J equation for the unknown
$p(t,x)=\partial_{x}u(t,x)\in \mathbb{R}^{n}$
\begin{eqnarray}\label{eq:2.59}
  \partial_{t}p(t,x) &+<& \partial_{p}H(t,x,u(t,x),p(t,x))\partial_{x}p(t,x)> \\\nonumber
  &=-& [\partial_{x}H(t,x,u(t,x),p(t,x))+p(t,x)\partial_{u}H(t,x,u(t,x),p(t,x))]
\end{eqnarray}
where $p=(p_{1},...,p_{n}),\partial_{p}H$ and
$\partial_{x}H$ are row vectors. Using
$\partial_{x_{i}}p(t,x)=\partial_{x}p_{i}(t,x),i\in\{1,...,n\}$ we
notice that \eqref{eq:2.59} can be written as follows
\begin{eqnarray}\label{eq:2.60}
  \partial_{t}p_{i}(t,x) &+<& \partial_{x}p_{i}(t,x),\partial_{p}H(t,x,u(t.x),p(t,x))> \\\nonumber
  &=-& [\partial_{x_{i}}H(t,x,u(t,x),p(t,x))+p_{i}(t,x)\partial_{u}H(t,x,u(t,x),p(t,x))]
\end{eqnarray}
for each $i\in \{1,...,n\}$, and
\begin{equation}\label{eq:2.61}
\{(\partial_{x}(t,\xi)),\partial_{p}(t,\xi)=p(t,\partial_{x}(t,\xi)),\partial_{u}(t,\xi)=u(t,\partial_{x}(t,\xi):t\in
I_{\alpha}(x_{0}))\}
\end{equation}
satisfies the characteristic
system\index{Characteristic! system } \eqref{eq:2.42}
provided
\begin{equation}\label{eq:2.62}
\left\{
       \begin{array}{ll}
         \frac{d\hat{x}(t,\xi)}{dt}=&\partial_{p}H(t,\hat{x}(t,\xi),\hat{u}(t,\xi),\hat{p}(t,\xi)),t\in I_{\alpha}(x_{0}) \\
         \hat{x}(t_{0},\xi)= & \xi
       \end{array}
     \right.
\end{equation}
The additional system of H-J equation \eqref{eq:2.60} stands for a
causilinear system of evolution equations
\begin{equation}\label{eq:2.63}
\left\{
       \begin{array}{ll}
        \partial_{t}v_{i}(t,x)+<\partial_{x}v_{i}(t,x),X(t,x,v(t,x))>=&L_{i}(t,x,v(t,x))\\
         v_{i}(t_{0},x)= v_{6}^{0},i\in\{1,...n\}
       \end{array}
     \right.
\end{equation}
Which allows to use the corresponding characteristic
system\index{Characteristic! system } suitable for a scalar
equation. It relies on the unique vector field
$X(t,x,v)$ deriving each scalar equation in the system \eqref{eq:2.63}.
\begin{remark}\label{re:2.8}
In the case that the unique vector field
$X(t,x,v)$ is replaced by some $X_{i}(t,x,v)$, for each
$i\in\{1,...,n\}$, which are not commuting with respect to the Lie
bracket $[X_{i}(t,x,.),X_{i}(t,x,.)](v)\neq o$, for some $i\neq j$, then the integration of the system \eqref{eq:2.63}changes drastically.
\end{remark}
\subsection{Exercises}
$(E_{1}).$Using the characteristic system\index{Characteristic!
system } method, solve the following
Cauchy problems
\begin{equation}\label{eq:2.64}
\partial_{t}u(t,x)=(\partial_{x}u(t,x))^{2},u(0,x)=cosx,x\in\mathbb{R},u\in
\mathbb{R}
\end{equation}
\begin{equation}\label{eq:2.65}
\left\{
         \begin{array}{ll}
          \partial_{t}u(t,x_{1},x_{2})= & x_{1}\partial_{x_{1}}u(t,x_{1},x_{2})+(\partial_{x_{2}}u(t,x_{1},x_{2}))^{2} \\
           u(0,x_{1},x_{2})= &
x_{1}+x_{2},(x_{1},x_{2})\in\mathbb{R}^{2},u\in\mathbb{R}
         \end{array}
       \right.
\end{equation}
$(E_{2})$. Let
$f(t,x):\mathbb{R}\times\mathbb{R}^{n}\rightarrow\mathbb{R}$ and
$\varphi(x):\mathbb{R}\rightarrow\mathbb{R}$ be some first order
continuously differentiable functions. Find the Cauchy problem
solution of the following linear H-J equations
\begin{equation}\label{eq:2.66}
\left\{
      \begin{array}{ll}
       \partial_{t}u(t,x)+&\Sigma_{i=1}^{n}a_{i}(t)\partial_{x_{i}}u(t,x)= f(t,x) \\
       u(0,x)= & \varphi(x)
      \end{array}
    \right.
\end{equation} where
$a(t)=(a_{1}(t),...,a_{n}(t)):\mathbb{R}\rightarrow\mathbb{R}^{n}$
is a continuous function.\\
\begin{equation}\label{eq:2.67}
\left\{
      \begin{array}{ll}
      \partial_{t}u(t,x)+<A(t)x.\partial_{x}u(t,x)> = f(t,x) \\
        u(0,x) =\varphi(x)
      \end{array}
    \right.
\end{equation}
where $A(t):\mathbb{R}\rightarrow M_{n\times n}$ is a continuous
mapping.
\section[Stationary Solutions for Nonlinear Ist Order
$PDE$]{Stationary Solutions for Nonlinear First \\Order
$PDE$}
\subsection{Introduction}
We consider  a nonlinear equation
\begin{equation}\label{eq:2.68}
H_{0}(x,\partial_{x}u(x),u(x))=constant \, \forall x\in
D\subseteq\mathbb{R}^{n}
\end{equation}
where
$$H_{0}(x,p,u):\mathbb{R}^{n}\times\mathbb{R}^{n}\times\mathbb{R}\rightarrow\mathbb{R}$$
is a scalar continuously differentiable function
$$H_{0}\in
\mathcal{C}^{1}(\mathbb{R}^{n}\times\mathbb{R}^{n}\times\mathbb{R})$$
and $\partial_{x}u=(\partial_{1}u,...,\partial_{n}u)$ stands for the
gradient of a scalar function. A standard solution for \eqref{eq:2.68} means to find
$u(x):B(x_{0},\rho)\subseteq D\rightarrow\mathbb{R},u\in
\mathcal{C}^{1}(B(x_{0},\rho))$ such
that $H_{0}(x,\partial_{x}u(x),u(x))=\hbox{constant}\,  \forall \,x\in
B(x_{0},\rho)\subseteq D$. The usual method of solving \eqref{eq:2.68} uses the
associated characteristic system\index{Characteristic! system }
\begin{equation}\label{eq:2.69}
\frac{d\hat{z}}{dt}=Z_{0}(\hat{z}),\hat{z}(0,\lambda)=\hat{z}(\lambda),\lambda\in\Lambda\subseteq\mathbb{R}^{n-1},t\in(-a,a)
\end{equation} where
$Z_{0}(z):\mathbb{R}^{2n+1}\rightarrow\mathbb{R}^{2n+1},z=(x,p,u)$, is
the characteristic vector\index{Characteristic!vector} field
corresponding to $H_{0}(z)$
\begin{equation}\label{eq:2.70}
Z_{0}(z)=(X_{0}(z),P_{0}(z),U_{0}(z)),X_{0}(z)=\partial_{p}H_{0}(z)\in\mathbb{R}^{n}
\end{equation}
$$U_{0}(z)=<p,X_{0}(z)>,P_{0}(z)=-(\partial_{x}H_{0}(z)+p\partial_{u}H_{0}(z))\in\mathbb{R}^{n}$$\\
For a fixed Cauchy condition
$\hat{z}(\lambda)=(\hat{x}(\lambda),\hat{p}(\lambda),\hat{u}(\lambda))$
given on a domain $\lambda\in\Lambda\subseteq\mathbb{R}^{n-1}$, a
compatibility condition
\begin{equation}\label{eq:2.71}
\partial_{i}\hat{u}(\lambda)= <\hat{p}(\lambda),\partial_{i}\hat{x}(\lambda)>,i\in\{1,...,n-1\}
\end{equation}
is necessary. In addition, both vector fields $Z_{0}(z)$ and the
parametrization $\{\hat{z}(\lambda):\lambda\in\Lambda\}$ must
satisfy a nonsingularity condition
\begin{eqnarray}\label{eq:2.72}
\hbox{the vectors in }
\mathbb{R}^{n},X_{0}(\hat{z}(\lambda)),\partial_{1}\hat{x}(\lambda),...,\partial_{n-1}\hat{x}(\lambda)\\\nonumber
\hbox{are linearly independent for any} \lambda\in\Lambda\in\mathbb{R}^{n-1}
\end{eqnarray}
By definition, the
characteristic vector\index{Characteristic!vector} field
$\{\partial_{z}H_{0}(z):z\in\mathbb{R}{2n+1}\}$ and we get
{\small
\begin{equation}\label{eq:2.73}
0=\left\{
  \begin{array}{ll}
    <\partial_{z}H_{0}(z),Z_{0}(z)> \\
    <\partial_{x}H_{0}(z),X_{0}(z)>+<\partial_{p}H_{0}(z),P_{0}(z)>+\partial_{u}H_{0}(z)U_{0}(z)
  \end{array}
\right.
\end{equation}}
\noindent for any $z=(x,p,u)\in\mathbb{R}^{2n+1}$. Using \eqref{eq:2.73} for the local solution $\{\hat{z}(t,\lambda):t\in(-a,a)\}$ of
the characteristic system\index{Characteristic! system } \eqref{eq:2.69} we
obtain
$H_{0}(\hat{z}(t,\lambda))=H_{0}(\hat{z}(\lambda)),t\in(-a,a)$, for
each $\lambda\in\Lambda\subseteq\mathbb{R}^{n-1}$. Looking for a
stationary solution of the equation \eqref{eq:2.68} we need to impose the
following constaint
\begin{equation}\label{eq:2.74}
H_{0}(\hat{z}(t,\lambda))=H_{0}(z_{0}) \hbox{for any}
(t,\lambda)\in(-a,a)\times\Lambda
\end{equation}
where
$z_{0}=\hat{z}(0)=(x_{0},p_{0},u_{0})(0\in int\Lambda \hbox{ for
simplicity})$. The condition \eqref{eq:2.72} allows to apply the standard implicit
function theorem and to solve the algebraic equation
\begin{equation}\label{eq:2.75}
\hat{x}(t,\lambda)=\lambda\in B(x_{0},\rho)\subseteq D
\end{equation}
We get smooth functions $t=\tau(x)\in(-a,a)$ and $\lambda=\psi(x)\in\Lambda$ such that
\begin{equation}\label{eq:2.76}
\hat{x}(\tau(x),\psi(x))=x\in B(x_{0},\rho)\subseteq
D,\tau(x_{0})=0,\psi(x_{0})=0\in\Lambda
\end{equation}
A solution for the nonlinear
equation \eqref{eq:2.68} is obtained as follows
\begin{equation}\label{eq:2.77}
u(x)=\hat{u}(\tau(x),\psi(x)),p(x)=\hat{p}(\tau(x),\psi(x)),x\in
B(x_{0},\rho)\subseteq\mathbb{R}^{n}
\end{equation}
where $p(x)$
fulfils
$$p(x)=\partial_{x}u(x)$$
\textbf{Obstruction}. We need explicit condition \index{Explicit!
condition }to compute $\{\hat{z}(\lambda):\lambda\in\Lambda\}$ such
that \eqref{eq:2.71}, \eqref{eq:2.72} and \eqref{eq:2.74} are verified.
\subsection{The Lie Algebra of Characteristic Fields\index{Characteristic!fields}}
Denote
$\mathcal{H}=\mathcal{C}^{\infty}(\mathbb{R}^{2n+1};\mathbb{R})$ the
space consisting of the scalar
functions $H(x,p,u):\mathbb{R}^{n}\times\mathbb{R}^{n}\times\mathbb{R}\rightarrow\mathbb{R}$
which are differentiable of any order. For each pair $H_{1},H_{2}\in$ define the Poisson bracket
\begin{equation}\label{eq:2.79}
\{H_{1},H_{2}\}(z)=<\partial_{z}H_{2}(z),Z_{1}(z)>,z=(x,p,u)\in\mathbb{R}^{2n+1}
\end{equation}
where $\partial_{z}H_{2}(z)$ stands for the gradient of a scalar function $H_{2}\in \mathcal{H}$ and $Z_{1}(z)=(X_{1}(z),P_{1}(z),U_{1}(z))\in\mathbb{R}^{2n+1}$ is the characteristic field  corresponding to $H_{1}\in \mathcal{H}$. We recall that $Z_{1}$ is obtained from $H_{1}\in \mathcal{H}$ such that the following equations
\begin{equation}\label{eq:2.80}
X_{1}(z)=\partial_{p}H_{1}(z),P_{1}(z)=-(\partial_{x}H_{1}(z)+p\partial_{u}H_{1}(z)),U_{1}(z)=<p,\partial_{p}H_{1}(z)>
\end{equation}are satisfied. The linear mapping connecting an arbitrary $H\in\mathcal{H}$ and its characteristic field can be represented by
\begin{equation}\label{eq:2.81}
Z_{H}(z)=T(p)(\partial_{z}H)(z),\,z=(x,p,u)\in\mathbb{R}^{2n+1}
\end{equation}
where the real $(2n+1)\times(2n+1)$ matrix $T(p)$ is defined by
\begin{equation}\label{eq:2.82}
T(p)=\left(
            \begin{array}{ccc}
              O & I_{n} & \theta \\
              -I_{n} & O & -p \\
              \theta^{*} & p^{*} & 0 \\
            \end{array}
          \right)
\end{equation}
where $\hbox{O-zero matrix of}\,\, M_{n\times n},\,  I_{n} \hbox{unity matrix of}\,\,  M_{n\times n} \hbox{and} \,\theta\in\mathbb{R}^{n}$ is the null column vector. We notice that\ $T(p)$ is a skew symmetric matrix
\begin{equation}\label{eq:2.83}
[T(p)]^{*}=-T(p)
\end{equation}
and as a consequence, the Poisson bracket satisfies a skew symmetric property
\begin{equation}\label{eq:2.84}
  \{H_{1},H_{2}\}(z)=\left\{
  \begin{array}{ll}
    &<\partial_{z}H_{1}(z),Z_{2}(z)>\\
    &<\partial_{z}H_{1}(z),T(p)\partial_{z}H_{2}(z)>\\
   &<[T(p)]^{*}\partial_{z}H_{1}(z),\partial_{z}H_{2}(z)>\\
   &-\{H_{2},H_{1}\}
  \end{array}
\right.
\end{equation}
In addition, the linear space of characteristic fields\index{Characteristic!fields} $K\subseteq\mathcal{C}^{\infty}(\mathbb{R}^{2n+1},\mathbb{R}^{2n+1})$ is the image of a linear mapping $S:D\mathcal{H}\rightarrow K$ where $D\mathcal{H}=\{\partial_{z}H:H\in\mathcal{H}\}$. In this respect, using \eqref{eq:2.81} we define
\begin{equation}\label{eq:2.85}
S(\partial_{z}H)(z)=T(p)(\partial_{z}H)(z),z\in\mathbb{R}^{2n+1}
\end{equation}
where the matrix $T(p)$ is given in \eqref{eq:2.82}. The linear space of characteristic fields $K=S(d\mathcal{H})$ is extended to a Lie algebra
 $$L_{k}\subseteq \mathcal{C}^{\infty}(\mathbb{R}^{2n+1},\mathbb{R}^{2n+1})$$
using the standard Lie bracket of vector fields
\begin{equation}\label{eq:2.86}
[Z_{1},Z_{2}]=[\partial_{z}Z_{2}(z)]Z_{1}-[\partial_{z}Z_{1}]Z_{2}(z),\,Z_{i}\in K,i=1,2
\end{equation}
On the other hand, each $H\in\mathcal{H}$ is associated with a
linear mapping
\begin{equation}\label{eq:2.87}
\overrightarrow{H}(\varphi)(z)=\{H_{1}\varphi\}(z)=<\partial_{z}\varphi(z),Z_{H}(z))>,z\in\mathbb{R}^{2n+1}
\end{equation}
for each $\varphi\in\mathcal{H}$, where $Z_{H}\in K$ is the characteristic vector field corresponding to
$H\in\mathcal{H}$ obtained from $\partial_{z}H$ by $Z_{H}(z)=T(p)(\partial_{z}H)(z)$ (see \eqref{eq:2.81}). Define a linear space consisting of linear mappings
\begin{equation}\label{eq:2.88}
\overrightarrow{\mathcal{H}}=\{\overrightarrow{H}:H\in\mathcal{H}\}
\end{equation}
and extend $\overrightarrow{\mathcal{H}}$ to a Lie algebra $L_{H}$ using the Lie bracket of linear mappings
\begin{equation}\label{eq:2.89}
[\overrightarrow{H_{1}},\overrightarrow{H_{2}}]=\overrightarrow{H_{1}}\circ\overrightarrow{H_{2}}-\overrightarrow{H_{2}}\circ\overrightarrow{H_{1}}
\end{equation}
The link between the two Lie algebras $L_{K}$(extending $K$) and $L_{H}$(extending $\overrightarrow{\mathcal{H}}$) is given by a homomorphism of Lie algebras
\begin{equation}\label{eq:2.90}
A:L_{H}\rightarrow L_{K}\, \hbox{satisfying} \,A(\overrightarrow{\mathcal{H}})=K
\end{equation}
and
\begin{equation}\label{eq:2.91}
A([\overrightarrow{H_{1}},\overrightarrow{H_{2}}])=[Z_{1},Z_{2}]\in
L_{K} \hbox{where}  Z_{i}=A(\overrightarrow{H_{i}}),\, i\in\{1,2\}
\end{equation}
\begin{remark}\label{re:2.9}
The Lie algebra $L_{H}(\supseteq \overrightarrow{\mathcal{H}})$ does not coincide with the linear space
$\overrightarrow{\mathcal{H}}$ and as a consequence,the linear
space $K\subseteq L_k$. It relies upon the fact the linear mapping
$\{\overrightarrow{H_{1},H_{2}}\}$ generated by the Poisson bracket
$\{H_{1},H_{2}\}\in\mathcal{H}$ does not coincide with the Lie
bracket $[\overrightarrow{H_{1}},\overrightarrow{H_{1}}]$ defined in
\eqref{eq:2.89}.
\end{remark}
\begin{remark}\label{re:2.10}In the particular case when the equation \eqref{eq:2.68} is replaced by $H_{0}(x,p):\mathbb{R}^{2n}\rightarrow\mathbb{R}$ is continuously differentiable  then the above given analysis will
be restricted to the space
$\mathcal{H}=\mathcal{C}^{\infty}(\mathbb{R}^{2n};\mathbb{R})$. If
it is the case then the corresponding linear mapping
$S:D\mathcal{H}\rightarrow K$ is determined by a simplectic matrix $T\in M_{2n\times 2n}$
\begin{equation}\label{eq:2.92}
 T=\left(
            \begin{array}{cc}
              O & I_{n} \\
              -I_{n} & O \\
            \end{array}
          \right)
, D\mathcal{H}=\{\partial_{z}H:H\in
\mathcal{C}^{\infty}(\mathbb{R}^{2n};\mathbb{R})\}
\end{equation}
In addition, the linear spaces $\overrightarrow{H}$ and $K\subseteq\mathcal{C}^{\infty}(\mathbb{R}^{2n};\mathbb{R}^{2n})$ coincide with their Lie algebra $L_{H}$ and correspondingly $L_{K}$ as the following direct computation shows
\begin{equation}\label{eq:2.93}
[Z_{1},Z_{2}](z)=T\partial_{z}H_{12},\,z\in\mathbb{R}^{2n}
\end{equation}
where
$$Z_{i}=T\partial_{z}H_{i},\,i\in\{1,2\}$$
and
$$H_{12}=\{H_{1},H_{2}\}(z)=<\partial_{z}H_{2}(z),Z_{1}(z)>$$
is the Poisson bracket associated with two scalar functions
$H_{1},H_{2}\in\mathcal{H}$. We get
\begin{equation}\label{eq:2.94}
T\partial_{z}H_{12}(z)=
\left\{
   \begin{array}{ll}
  &T[\partial_{z}^{2}H_{2}(z)]Z_{1}(z)+T(\partial_{z}Z_{1}^{*}(z))\partial_{z}H_{2}\\
   &[\partial_zZ_2(z)]Z_1(z)+T[\partial_z(\partial_{z}^{*}H_{1}^{*}T^{*})]\partial_{z}H_2(z)\\
     &[\partial_{z_{2}}(z)]Z_{1}(z)-T(\partial_{z}^{2}H_{1}(z))T\partial_{z}H_{2}(z) \\
  & [\partial_{z}Z_{2}(z)]Z_{1}(z)-[\partial_{z}Z_{1}(z)]Z_{2}(z)
   \end{array}
  \right.
  \end{equation}
and the conclusion $\{L_{H}=\overrightarrow{H},L_{K}=K\}$ is
proved.
\end{remark}
\subsection{Parameterized Stationary Solutions}Consider a nonlinear first order equation
\begin{equation}\label{eq:2.95}
H_{0}(x,p(x),u(x))=\hbox{const}\,\,\forall\,x\in D\subseteq\mathbb{R}^{n}
\end{equation}
With the same notations as in section $2.4.2$, let $\mathcal{H}=\mathcal{C}^{\infty}(\mathbb{R}^{2n+1};\mathbb{R})\},\,z=(x,p,u)\in\mathbb{R}^{2n+1}$ and define the skew-symmetric matrix
\begin{equation}\label{eq:2.96}
T(p)=\left(
             \begin{array}{ccc}
               O & I_{n} & \theta \\
              -I_{n} & O & -p\\
              \theta^{*} & p^{*} & 0 \\
             \end{array}
           \right)
\,\,\,,p\in\mathbb{R}^{n}
\end{equation}
By definition, the linear spaces $K$ and $D\mathcal{H}$ are
\begin{equation}\label{eq:2.97}
 K=S(D\mathcal{H}),\,D\mathcal{C}=\{\partial_{z}H:H\in\mathcal{H}\}
\end{equation}
where the mapping $S:D\mathcal{H}\rightarrow K$ satisfies
\begin{equation}\label{eq:2.98}
S(\partial_{z}H)(z)=T(p)\partial_{z}H(z),\,z\in\mathbb{R}^{2n+1}
\end{equation}
Let $z_{0}=(x_{0},p_{0},u_{0})\in\mathbb{R}^{2n+1}$ and the ball
$B(x_{0},2\rho)\subseteq\mathbb{R}^{2n+1}$ be fixed. Assume that
\begin{equation}\label{2.99a}
\hbox{there exist} \,\{Z_{1},...,Z_{m}\}\subseteq K
\end{equation}
such that the smooth vector fields $$\{Z_{1}(z),...,Z_{m}(z):z\in B(z_{0},2\rho)\}$$ are in involution over $\mathcal{C}^{\infty}(B(z_{0},2\rho))$
$$\hbox{(the Lie bracket)}  [Z_{i},Z_{j}](z)=\mathop{\sum}\limits_{k=1}^{m}\alpha_{ij}^{k}(z)Z_{k}(z),\,i,j\in\{1,...,m\}$$
 where $$\alpha_{ij}^{k}\in\mathcal{C}^{\infty}(B(z_{0},2\rho))$$\\
A parameterized solution of the equation \eqref{eq:2.95} is given by the following orbit
\begin{equation}\label{eq:2.99}
\hat{z}(\lambda,z_{0})=G_{1}(t_{1})\circ...\circ G_{m}(t_{m})(z_{0}),\,\,\lambda=(t_{1},...,t_{m})\in\Lambda=\prod_{1}^{m}[-a_{i},a_{i}]
\end{equation}
where
$$G_{i}(\tau)(y),\,y\in
B(z_{0},\rho),\,\,\tau\in[-a_{i},a_{i}]$$
is the local flow
\index{Flow!local}generated by $Z_{i}$. Notice that
\begin{equation}\label{eq:2.100}
 N_{z_{0}}=\{z\in B(z_{0},2\rho):z=\hat{z}(\lambda,z_{0}),\,\lambda\in\Lambda\}
\end{equation}
is a smooth manifold and
\begin{equation}\label{eq:2.101}
\dim M_{z_{0}}=\dim L(Z_{1},...,Z_{m})(z_{0})=m\leq n\,
if\, Z_{1}(z_{0}),...,Z_{m}(z_{0})\in
\mathbb{R}^{2n+1}
\end{equation}
are linearly
independent\index{Independent!linearly}.
\begin{remark}\label{re:2.11}
Under the conditions of the hypothesis \eqref{2.99a} and using the algebraic representation of a gradient system \index{Gradient systems}associated with
$\{Z_{1},...,Z_{m}\}\subseteq K$
a system $\{q_{1},...,q_{m}\}\subseteq\mathcal{C}^{\infty}(\Lambda,\mathbb{R}^{m})$ exists such that
\begin{equation}\label{**}
q_{1}(\lambda),...,q_{m}(\lambda)\in\mathbb{R}^{m}  are
linearly independent
\end{equation}
and
$\partial_{\lambda}\hat{z}(\lambda,z_{0})q_{i}(\lambda)=Z_{i}(\hat{z}(\lambda,z_{0}))$ for
any $\lambda\in\Lambda,\,i\in\{1,...,m\}$
\end{remark}
\begin{definition}A parameterized solution associated with \eqref{eq:2.95}is defined by orbit \eqref{eq:2.99}provided the equality $H_{0}(\hat{z}(\lambda,z_{0}))=H_0(z_0),\,\forall\,\lambda\in\Lambda$, is satisfied.
\end{definition}
\begin{remark}\label{re:2.12}
The conclusion \eqref{**} does not depend on the manifold structure given in \eqref{eq:2.100}. Using \eqref{eq:2.102}, we rewrite the equation $H_{0}(\hat{z}(\lambda,z_{0}))=H_0(z_0),\,\forall\,\lambda\in\Lambda$,
in the following equivalent form\\
\begin{equation}\label{eq:2.102}
0=\{H_{i},H_{0}\}(\hat{z}(\lambda.z_{0}))=<\partial_{z}H_{0}(\hat{z}(\lambda,z_{0})),Z_{i}(\hat{z}(\lambda,z_{0}))>,\,\lambda\in\Lambda
\end{equation}
for each $\,i\in\{1,...,m\}$,\,where $\,\{H_{i},H_{0}\}$ is the
Poisson bracket associated with
$H_{i},H_{0}\in\mathcal{C}^{1}(\mathbb{R}^{2n+1},\mathbb{R})$, and
$Z_{i}=S(\partial_{z}H_{i}),\,i\in\{1,...,m\}$. The equations \eqref{eq:2.102}
are directly computed from the scalar equation
$H_{0}(\hat{z}(\lambda,z_{0}))=H_0(z_0),\,\forall\,\lambda\in\Lambda$, by
taking the corresponding Lie derivatives where $\{q_{1},...,q_{m}\}$  from \eqref{**} are used. As a consequence, the orbit defined in \eqref{eq:2.99}, under the conditions \eqref{**} and \eqref{eq:2.101}, will determine a
parameterized solution of  \eqref{eq:2.102}.
\end{remark}
\begin{remark}\label{re:2.13}
A classical solution for \eqref{eq:2.95} can be deduced from a parameterized solution $\{\hat{z}(\lambda,z_{0}):\lambda\in\Lambda\}$ if we take $m=n$ and assume
\begin{equation}\label{eq:2.103}
\hbox{the matrix }[\partial_{\lambda}\hat{x}(0,z_{0})]\in M_{n\times
n}\hbox{ is nonsingular }
\end{equation}
where the components
$(\hat{x}(\lambda,z_{0}),\hat{p}(\lambda,z_{0}),\hat{u}(\lambda,z_{0}))=\hat{z}(\lambda,z_{0})$ define
the parameterized solution $\{\hat{z}(\lambda,z_{0})\}$.
\end{remark}
\begin{proposition}\label{pr:2.4}
Assume the orbit$\{\hat{z}(\lambda,z_{0}):\lambda\in\Lambda\}$ given in \eqref{eq:2.99} is a parameterized solution of \eqref{eq:2.95} such that the condition \eqref{eq:2.103} is satisfied.Let $\lambda=\psi(x):S(x_{0},\rho)\rightarrow int\Lambda$ be the smooth mapping satisfying $\hat{x}(\psi(x),z_{0})=x\in B(x_{0},\rho),\,\psi(x_{0})=0$. Denote $u(x)=\hat{u}(\psi(x),z_{0})$ and $p(x)=\hat{p}(\psi(x),z_{0})$.Then
\begin{equation}\label{eq:2.104}
 p(x)=\partial_{x}u(x) \hbox{and} \, H_{0}(x,p(x),u(x))=H_{0}(z_{0}), \hbox{for any} \,x\in B(x_{0},\rho)\subseteq\mathbb{R}^{n}
\end{equation}
where $p(x)\,,x\in B(x_{0},\rho))$ verifies the following cuasilinear system of first order equations
\begin{eqnarray}\label{eq:2.105}
   [\partial_{x}H_{0}(x,p(x),u(x))+p(x)\partial_{u}H_{0}(x,p(x),u(x))]\nonumber\\
+[\partial_{x}p(x)]^{*}\partial_{p}H_{0}(x,p(x),u(x))=0 ,\,x\in B(x_{0},\rho)
\end{eqnarray}
\end{proposition}
\begin{proof}
The component $\{\hat{x}(\lambda,z_{0}):\lambda\in\Lambda\}$ fulfills the condition of the standard implicit functions theorem (see \eqref{eq:2.103}) and, by definition, the matrix
$$\partial_{\lambda}\hat{x}(0.z_{0})=\parallel X_{1}(z_{0},...,X_{n}(z_{0}))\parallel$$ is composed by first vector-component of
$Z_{i}(z)=(X_{i}(z),P_{i}(z),U_{i}(z))\,\,i\in\{1,...,n\}$, where $\{Z_{1}(z),...,Z_{n}(z)\}$ define the orbit $\{\hat{z}(\lambda,z_{0})\}$.
Let $\lambda=\psi(x):B(x_{0},\rho)\subseteq\mathbb{R}^{n}\rightarrow
int\Lambda\subseteq\mathbb{R}^{n}$ be such that
$\psi(x_{0})=0$ and $\hat{x}(\psi(x),z_{0})=x$. Denote
$p(x)=\hat{p}(\psi(x),z_{0}),\,u(x)=\hat{u}(\psi(x),z_{0})$ and
using the equation
$$H_{0}(\hat{z}(\lambda,z_{0}))=H_{0}(z_{0}),\,\lambda\in\Lambda$$
we get
$$H_{0}(z(x))=H_{0}(z_{0}),\,\forall\,x\in
B(x_{0},\rho)\subseteq\mathbb{R}^{n}$$
where $z(x)=(x,p(x),u(x))$. Using \eqref{**} written on corresponding
components we find that $\partial_{\lambda}\hat{u}(\lambda,z_{0})q_{i}(\lambda)=\hat{p}(\lambda,z_{0})
\partial_{\lambda}\hat{x}(\lambda,z_{0})q_{i}(\lambda)),\,i\in\{1,...,n\}$ and\\
\begin{equation}\label{eq:2.106}
\partial_{\lambda}{u}(\lambda,z_{0}) =\hat{p}(\lambda,z_{0})\partial_{\lambda}\hat{x}(\lambda,z_{0}),\,\lambda\in\Lambda
\end{equation}
is satisfied, where $\hat{p}(\lambda)\in\mathbb{R}^{n}$ is a row
vector. On the other hand, a direct computation applied to $u(\hat{x}(\lambda,z_{0}))=\hat{u}(\lambda,z_{0})$ leads us to
\begin{equation}\label{eq:2.107}
\partial_{x}u(\hat{x}(\lambda,z_{0})).\partial_{\lambda}\hat{x}(\lambda,z_{0})=\partial_{\lambda}\hat{u}(\lambda,z_{0})
\end{equation}
and using \eqref{eq:2.103} we may and do multiply by the inverse matrix $[\partial_{\lambda}\hat{x}(\lambda,z_{0})]^{-1}$ in both equations \eqref{eq:2.106} and \eqref{eq:2.107}.
We get $\hat{p}(\lambda,z_{0})=\partial_{x}u(\hat{x}(\lambda,z_{0}))$, for any $\lambda\in B(z_{0},2\rho)$\,($\rho>0$ sufficiently small) which stands for\\
$$\partial_{x}u(x)=p(x),\,x\in B(z_{0},\rho)\subseteq\mathbb{R}^{n}$$\\
provided $\lambda=\psi(x)$ is used. The conclusions \eqref{eq:2.105}
tell us
that the gradient $\partial_{x}[H_{0}(x,p(x),u(x))]$ is vanishing
and the proof is complete.
\end{proof}
\begin{remark}\label{re:2.14}
Taking $\{Z_{1},...,Z_{m}\}\subseteq K$ is involution we get the property \eqref{**} fulfilled (see \S 5 of ch II) 
\begin{equation}\label{eq:2.108}
 0= \left\{
  \begin{array}{ll}
   & \{H_{i},H_{0}\}(\hat{z}(z,z_{0}))\\
     & <\partial_{z}H_{0}(\hat{z}(\lambda,z_{0})),Z_{i}(\hat{z}(\lambda,z_{0}))> \\
     & -<T(p)\partial_{z}H_{0}(\hat{z}(\lambda,z_{0})),\partial_{z}H_{i}(\hat{z}(\lambda,z_{0}))>\\
     &-<\partial_{z}H_{i,}(\hat{z}(\lambda,z_{0})),Z_{0}(\hat{z}(\lambda,z_{0}))>,\,i\in\{1,...,m\}
  \end{array}
\right.
\end{equation}
It shows that $\{H_{1},...,H_{m}\}\subseteq\mathcal{H}=\mathcal{C}^{\infty}(\mathbb{R}^{2n+1},\mathbb{R})$ defining $\{Z_{1},...,Z_{m}\}\subseteq K$ can be found
as first integrals for a system of $ODE$
$$\frac{dz}{dt}=Z_{0}(z),Z_{0}(z)=T(p)\partial_{z}H_{0}(z)\in
K$$
corresponding to $H_{0}(z)$.
\end{remark}
\subsection{The linear Case:$H_{0}(x,p)=<p,f_{0}(x)>$}
With the same notations as in $\S 4.3$ we define $z=(x,p)\in\mathbb{R}^{2n}$ and $\mathcal{H}=\{H(z)=<p,f(x)>,p\in\mathbb{R}^{2n},f\in\mathcal{C}^{\infty}(\mathbb{R}^{2n},\mathbb{R}^{2n})\}$.
The linear space of characteristic fields $K\subseteq \mathcal{C}^{\infty}(\mathbb{R}^{2n},\mathbb{R}^{2n})$ is the image of a linear mapping $S:D\mathcal{H}\rightarrow K$, where
\begin{equation}\label{eq:2.109}
D\mathcal{H}=\{\partial_{z}H:H\in\mathcal{H}\},\,S(\partial_{z}H)(z)=T\partial_{z}H(z),\,z\in\mathbb{R}^{2n}
\end{equation}
\begin{equation}\label{eq:2.110}
T=\left(
           \begin{array}{cc}
             O & I_{n} \\
             -I_{n} & O \\
           \end{array}
         \right)
\,for \,each\, Z\in K
\end{equation}
is given by\\
\begin{equation}\label{eq:2.111}
Z(z)=\left(
              \begin{array}{c}
                \partial_{p}H(z) \\
                -\partial_{x}H(z) \\
              \end{array}
            \right)
=\left(
   \begin{array}{c}
     f(x) \\
     -\partial_{x}<p,f(x)> \\
   \end{array}
 \right)
\end{equation}
Let $z_{0}=(p_{0},x_{0})\in\mathbb{R}^{2n}$ and $B(z_{0},2\rho)\subseteq \mathbb{R}^{2n}$ be fixed and consider the following linear equation of $p\in\mathbb{R}^{n}$
\begin{equation}\label{eq:2.112}
 H_{0}(x,p)=<p,f_{0}(x)>=H_{0}(z_{0})\hbox{for any} \,z\in D\subseteq B(z_{0},2\rho)
\end{equation}
where $f_{0}\in\mathcal{C}^{1}(\mathbb{R}^{2n},\mathbb{R}^{2n})$. We are looking for \,$D\subseteq B(z_{0},2\rho)\subseteq \mathbb{R}^{2n}$ as an orbit.
\begin{equation}\label{eq:2.113}
\hat{z}(\lambda,z_{0})=G_{1}(t_{1})\circ...\circ G_{m}(t_{m})(z_{0}),\,\,\lambda=\{t_{1},...,t_{m}\}\in\Lambda=\prod_{1}^{m}[-a_{i},a_{i}]
\end{equation}
where $G_{i}(\tau)(y),\,y\in B(z_{0},\rho),\,\tau\in[-a_{i},a_{i}]$, is the local flow generated by some $Z_{i}\in K,\,i\in\{1,...,m\},\,m\leq n$. Assuming that
\begin{equation}\label{eq:2.114}
\{Z_{1},...,z_{m}\}\subseteq K \, \hbox{are in involution over} \,\, \mathcal{C}^{\infty}(B(z_{0},2\rho),\mathbb{R})
\end{equation}
where $$Z_{i}(z)=\left(
            \begin{array}{c}
              f_{i}(x) \\
             -\partial_{x}<p,f_{i}(x)>\\
            \end{array}
          \right)
\hbox{and}
\{f_{1},...,f_{m}\}\subseteq\mathcal{C}^{\infty}(\mathbb{R}^{2n},\mathbb{R}^{2n})$$
are in involution over $\mathcal{C}^{\infty}(B(x_{0}),\mathbb{R})$ we define
\begin{equation}\label{eq:2.115}
D=\{z\in B(z_{0},2\rho):z=\hat{z}(\lambda,z_{0}),\lambda\in\Lambda\}\subseteq{R}^{2n}
\end{equation}
where the orbit $\{\hat{z}(\lambda,z_{0}):\lambda\in\Lambda\}$ is given in \eqref{eq:2.113}. Notice that the orbit \eqref{eq:2.113} is represented by
\begin{equation}\label{eq:2.116}
\hat{z}(\lambda,z_{0})=(\hat{x}(\lambda,z_{0}),\hat{p}(\lambda,z_{0})),\,\lambda\in\Lambda,\,z_{0}=(x_{0},p_{0})\subseteq{R}^{2n}
\end{equation}
where the orbit $\hat{x}(\lambda,z_{0}),\lambda\in\Lambda,in\,\mathbb{R}^{n}$ verifies
\begin{equation}\label{eq:2.117}
\hat{x}(\lambda,x_{0})=F_{1}(t_{1})\circ...\circ F_{m}(t_{m})(x_{0}),\,\lambda=\{t_{1},...,t_{m}\}\in\Lambda
\end{equation}
Here $F_{i}(\tau)(x),\,x\in
B(x_{0},\rho)\subseteq{R}^{n},\,\tau\in[-a_{i},a_{i}]$, is the local
flow\index{Flow!local}generated by the vector field $f_{i}\in
\mathcal{C}^{\infty}(\mathbb{R}^{n},\mathbb{R}^{n})$ and $\{f_{1},...,f_{m}\}\subseteq\mathcal{C}^{\infty}(\mathbb{R}^{n},\mathbb{R}^{n})$
are in involution over $\mathcal{C}^{\infty}(B(x_{0},\rho))$.
\begin{remark}\label{re:2.15}
Denote by $N_{x_{0}}\subseteq\mathbb{R}^{n}$ the set consisting of all points $\{\hat{x}(\lambda,x_{0}):\lambda\in\Lambda\}$ using the orbit defined in \eqref{eq:2.117}. Assuming that$\{f_{1}(x_{0}),...,f_{m}(x_{0})\}\subseteq{R}^{n}$\,are linearly independent then
\,$\{f(x_{0}),...,f_{m}(x)\}\subseteq{R}^{n}$ are linearly independent for any $x\in B(x_{0},\rho)$($\rho$ sufficiently small) and the subset $N_{x_{0}}\subseteq\mathbb{R}^{n}$ can be structured as an m-dimensional smooth manifold. In addition,the set $D\subseteq\mathbb{R}^{2n}$ defined in \eqref{eq:2.115}
 can be verified as an image of smooth mapping $z(x)=(x,p(x)):N_{x_{0}}\rightarrow
D$, if $p(x)=\hat{p}(\psi(x),z_{0})$ and $\lambda=\psi(x):B(x_{0},\rho)\rightarrow
\Lambda$ is unique solution of the algebraic equation $\hat{x}(\lambda,x_{0})=x\in B(x_{0},\rho)$.
\end{remark}
\begin{definition}
The orbit $\{\hat{z}(\lambda,z_{0}):\lambda\in \Lambda\}$ defined
in \eqref{eq:2.113} is a parameterized solution of the linear equation \eqref{eq:2.112} if
$H_{0}(\hat{z}(\lambda,z_{0}))=H(z_{0})\,,\forall\,\lambda\in\Lambda$.
\end{definition}
\begin{proposition}\label{pr:2.5}
Let $z_{0}=(x_{0},p_{0})\subseteq{R}^{2n}$ and $B(z_{0},2\rho)\subseteq{R}^{2n}$ be
fixed such that the hypothesis \eqref{eq:2.114} is fulfilled.For
$\{Z_{1},...,Z_{m}\}\subseteq K$ given in \eqref{eq:2.114} assume in
addition, that $\{f_{1},...,f_{m}\}$ are commuting with
$\{f_{0}\}$ i.e $[f_{i},f_{0}](x)=0,\,x\in
B(x_{0},2\rho),\,i\in\{1,...,m\}$. Then the orbit
$\{\hat{z}(\lambda,z_{0}):\lambda\in \Lambda\}$\,defined in \eqref{eq:2.113}
is a parameterized solution of equation \eqref{eq:2.112}, where $D\in\mathbb{R}^{2n}$ is given in \eqref{eq:2.115}.
\end{proposition}
\begin{proof}
By hypothesis ,the Lie algebra $L(Z_{1},...,Z_{m})\subseteq\mathcal{C}^{\infty}(\mathbb{R}^{2n},\mathbb{R}^{2n})$ is
of finite type (locally) and $\{Z_{1},...,Z_{m}\}$ is a system of
generators in $L(Z_{1},...,Z_{m})$. As a consequence $L(Z_{1},...,Z_{m}) $ is $(f.g.\circ;z_{0})$ and using the algebraic representation of the gradient system associated with $\{Z_{1},...,Z_{m}\}$ we get $\{q_{1},...,q_{m}\}\subseteq\mathcal{C}^{\infty}(\Lambda;\mathbb{R}^{n})$ such that
\begin{equation}\label{eq:2.118}
\partial_{\lambda}\hat{z}(\lambda,z_{0}).q_{i}(\lambda) = Z_{i}(\hat{z}(\lambda,z_{0})),\,\lambda\in\Lambda,\,i\in\{1,...,m\},\,
    q_{1}(\lambda),...,q_{m}(\lambda)\in\mathbb{R}^{m}\
\end{equation}
are linearly independent for any $\lambda\in\Lambda.$\\
The meaning of $(f.g.\theta;z_{0})$ Lie algebra and the conclusion \eqref{eq:2.118} are explained in the next section.
For the time being we use \eqref{eq:2.118} and taking Lie derivatives of the scalar equation $H_{0}(\hat{z}(\lambda,z_{0}))=H_{0}(z_{0}),\,\lambda\in\Lambda$, we get
{\small
\begin{equation}\label{eq:2.119}
0=\partial_{\lambda}H(\hat{z}(\lambda,z_{0})).q_{i}(\lambda)=<\partial_{z}H(\hat{z}(\lambda,z_{0})),Z_{i}(\hat{z}(\lambda,z_{0}))>,\hbox{ for any }i\in\{1,...,m\}
\end{equation}}
\noindent and $\lambda\in\Lambda$. By definition
$$\partial_{z}H_{0}=\left(
                                                                                     \begin{array}{c}
                                                                                       \partial_{x}<p,f_{0}(x)> \\
                                                                                       f_{0}(x) \\
                                                                                     \end{array}
                                                                                   \right)
,\,and \,Z_{i}(z)=\left(
                    \begin{array}{c}
                      f_{i}(x) \\
                      -\partial_{x}<p,f_{0}(x)> \\
                    \end{array}
                  \right)
,\,i\in\{1,...,m\}$$
which allows us to rewrite \eqref{eq:2.119} as follows
\begin{equation}\label{eq:2.120}
<\hat{p}(\lambda,z_{0}),[f_{i},f_{0}](\hat{x}(\lambda,x_{0}))>=0,\,i\in\{1,...,m\},\,\lambda\in\Lambda
\end{equation}
Using $[f_{i},f_{0}](x)=0,\,x\in B(x_{0},2\rho),\,i\in\{1,...,m\}$ we obtain that \eqref{eq:2.119} is fulfilled and it implies
\begin{equation}\label{eq:2.121}
\partial_{\lambda}H_0(\hat{z}(\lambda,z_{0}))=0,\, \forall\,\lambda\in\Lambda
\end{equation}
provided \eqref{eq:2.118} is used. In conclusion, the scalar
equation $H_{0}(z)=H_{0}(z_{0})$ for any $z\in
D\subseteq{R}^{2n}$, is satisfied, where $D$ is defined in\eqref{eq:2.115} and
the proof is complete.
\end{proof}
\begin{remark}\label{re:2.16}
The involution condition of $\{Z_{1},...,Z_{m}\}\subseteq K$  is
satisfied if
$$\{f_{1},...,f_{m}\}\subseteq\mathcal{C}^{\infty}(B(x_{0},2\rho)
;\mathbb{R}^{n})$$
are in involution over $\mathbb{R}$. In this respect, using the particular form of
\begin{equation}\label{eq:2.122}
Z_{i}(z)=\left(
 \begin{array}{c}
 f_{i}(x) \\
  A_{i}(x)p \\
   \end{array}
   \right)
,A_{i}(x)=-[\partial_{x}f_{i}(x)]^{*}\,i\in\{1,...,m\}$$
we compute a Lie bracket $[Z_{i},Z_{j}]$ as follows
$$[Z_{i},Z_{j}](z)=\left(
   \begin{array}{c}
     [f_{i},f_{j}](x) \\
     P_{ij}(z) \\
   \end{array}
 \right)
,\,\,\,i,j\in\{1,...,m\}
\end{equation}
Here $[f_{i},f_{j}]$ is the Lie bracket from $\mathcal{C}^{\infty}(\mathbb{R}^{2n},\mathbb{R}^{2n})$ and
\begin{equation}\label{eq:2.123}
P_{ij}(z)\left\{
  \begin{array}{ll}
    =[\partial_{x}(A_{j}(x)p)].f_{i}(x)+A_{j}(x)A_{i}(x)p-[\partial_{x}(A_{i}(x)p)]f_{j}(x)  \\
    = -A_{i}(x)A_{j}(x)p\\
    =\partial_{x}[<A_{j}(x)p,f_{i}(x)>-<A_{i}(x)p,f_{j}(x)>]\\
  = -\partial_{x}<p,[f_{i},f_{j}](x)>
  \end{array}
\right.
\end{equation}
where\\
\begin{equation}\label{eq:2.124}
P_{ij}(z)=[\partial_{z}(A_{j}(x)p)]Z_{i}(z)-[\partial_{z}(A_{i}(x)p)]Z_{j}(x),\,z\in\mathbb{R}^{2n}
\end{equation}
is used. Notice that \eqref{eq:2.123} allows one to write \eqref{eq:2.122} as a vector field from $K$\\
\begin{equation}\label{eq:2.125}
[Z_{i},Z_{j}](z)=T\partial_{z}H_{ij}(z),\,z=(x,p)\in \mathbb{R}^{2n}
\end{equation}
where $H_{ij}=<p,[f_{i},f_{j}](x)>,\, i,j\in\{1,...,m\}$. As a consequence, assuming that $\{f_{1},...,f_{m}\}$ are in
involution over $\mathbb{R}$ we get that $L(f_{1},...,f_{m})$ and $L(Z_{1},...,Z_{m})$ are finite
dimensional with $\{Z_{1},...,Z_{m}\}$ in involution over
$\mathbb{R}$.
\end{remark}
\subsection{The Case $H_{0}(x,p,u)=H_{0}(x,p)$; Stationary Solutions}
Denote $\mathcal{C}^{\infty}(\mathbb{R}^{2n},\mathbb{R}^{n})\, ,z=(x,p)\in\mathbb{R}^{2n}$ and define the linear space of characteristic fields \index{Characteristic!fields} $K\subseteq\mathcal{C}^{\infty}(\mathbb{R}^{2n},\mathbb{R}^{n})\,,z=(x,p)\in\mathbb{R}^{2n}$\,by
\begin{equation}\label{eq:2.126}
K=S(D\mathcal{H}),\,D\mathcal{H}=\{\partial_{z}H:H\in\mathcal{H}\}
\end{equation}
Here the linear mapping $:D\mathcal{H}\rightarrow K$ is given by
\begin{equation}\label{eq:2.127}
S(\partial_{z}H)(z)=T\partial_{z}H(z),\,H\in\mathcal{H},\,z\in\mathbb{R}^{2n}
\end{equation}
and $T$ is the simplectic matrix
\begin{equation}\label{eq:2.128}
T=\left(
         \begin{array}{cc}
          O & I_{n} \\
           -I_{n} & O\\
         \end{array}
       \right)
,T^{2}=\left(
         \begin{array}{cc}
           -I_{n} & O \\
           O & -I_{n} \\
         \end{array}
       \right)
\end{equation}
Let $z_{0}=(x_{0},p_{0})\in\mathbb{R}^{2n}$ and $B(z_{0},2\rho)\subseteq\mathbb{R}^{2n}$ be fixed and consider the following nonlinear equation
\begin{equation}\label{eq:2.129}
H_{0}(z)=H_{0}(z_{0}),\,\forall\,z\in D\subseteq B(z_{0},2\rho)
\end{equation}
where $H_{0}\in \mathcal{C}^{1}(B(z_{0},2\rho);\mathbb{R})$ is given and $D\subseteq B(z_{0},2\rho)$ has to be found. A solution for \eqref{eq:2.129} uses the following assumption
\begin{equation}\label{eq:2.130}
\hbox{there exist }Z_{1},...,Z_{m}\in K\hbox{ such that }
\{Z_{1}(z),...,z_{m}(z):z\in B(z_{0},2\rho)\}
\end{equation}
in involution over $\mathcal{C}^{\infty}( B(z_{0},2\rho);\mathbb{R})$. Assuming that \eqref{eq:2.130} is fulfilled then a parameterized solution for \eqref{eq:2.129} uses the following orbit
\begin{equation}\label{eq:2.131}
\hat{z}(\lambda,z_{0})=G_{1}(t_{1})\circ...\circ G_{m}(t_{m})(z_{0}),\,\lambda=(t_{1},...,t_{m})\in\Lambda=\prod_{1}^{m}[-a_{i},a_{i}]
\end{equation}
where
$$G_{i}(\tau)(y),\,y\in
B(z_{0},\rho),\,\tau\in[-a_{i},a_{i}]$$
is the local flow generated by $Z_{i}\in K$ given in \eqref{eq:2.130}. By definition
\begin{equation}\label{eq:2.132}
Z_{i}(z)=T\partial_{z}H_{i}(z)=\left(
                                      \begin{array}{c}
                                        \partial_{p}H_{i}(x,p) \\
                                        -\partial_{x}H_{i}(x,p) \\
                                      \end{array}
                                    \right)
,\,i\in\{1,...,m\}
\end{equation}
and we recall that, in this case, the linear space of characteristic fields \index{Characteristic!fields}$K$ is closed under the lie bracket. As a consequence, each Lie product $[Z_{i},Z_{j}]$ can be computed by
\begin{equation}\label{eq:2.133}
[Z_{i},Z_{j}](z)=T\partial_{z}H_{ij}(z),\,i,j\in\{1,...,m\},z\in\mathbb{R}^{2n}
\end{equation}
where $H_{ij}(z)=\{H_{i},H_{j}\}(z)=<\partial_{z}H_{ij}(z),Z_{i}(z)>=-<\partial_{z}H_{i},Z_{j}(z)>$\\
stands for Poisson bracket associated with two $\mathcal{C}^{1}$ scalar functions.
In addition, assuming \eqref{eq:2.130} we get that the Lie algebra
$L(Z_{1},...,Z_{m})\subseteq\mathcal{C}^{\infty}(B(z_{0},2\rho),\mathbb{R}^{n})$ is of the finite type (see in next section)
allowing one to use the algebraic representation of the
corresponding gradient system\index{Gradient systems}. We get that$\{q_{1},...,q_{m}\}\subseteq\mathcal{C}^{\infty}(\Lambda,\mathbb{R}^{m})$ will exist such that
\begin{equation}\label{eq:2.134}
\partial_{\lambda}\hat{z}(\lambda,z_{0})q_{i}(\lambda)=Z_{i}(\hat{z}(\lambda,z_{0})),\,i\in\{1,...,m\}\,,\lambda\in\Lambda,
q_{1}(\lambda),...,q_{m}(\lambda)\in\mathbb{R}^{m}
\end{equation}
are linearly independent, $\lambda\in \Lambda$, where $\{\hat{z}(\lambda,z_{0}):\lambda\in\Lambda\}$ is the fixed orbit given in \eqref{eq:2.131}. Denote $D=\{z\in B(z_{0},2\rho):z=\hat{z}(\lambda,z_{0}),\lambda\in\Lambda\}$ and the orbit given in \eqref{eq:2.131} is a parameterized solution of the equation \eqref{eq:2.129} $iff$
\begin{equation}\label{eq:2.135}
\partial_{\lambda}[H_{0}(\hat{z}(\lambda,z_{0}))]q_{i}(\lambda)=0,\,i\in\{1,...,m\}
\end{equation}
Using \eqref{eq:2.134} we rewrite \eqref{eq:2.135}as follows
{\small
\begin{equation}\label{eq:2.136}
0=<\partial_{\lambda}H_{0}(\hat{z}(\lambda,z_{0})),Z_{i}(\hat{z}(\lambda,z_{0}))>=\{H_{i},H_{0}\}(\hat{z}(\lambda,z_{0})),\,\lambda\in\Lambda,\forall\, \,i\in\{1,...,m\}
\end{equation}}
It is easily seen that \eqref{eq:2.136} is fulfilled provided $$H_{1},...,H_{m}\in\mathcal{C}^{\infty}(B(z_{0},2\rho),\mathbb{R}^{n})$$ can be found  such that
\begin{equation}\label{eq:2.137}
0=\{H_{i},H_{0}\}(z)=<\partial_{z}H_{0}(z),T\partial_{z}H_{i}(z)>=<Z_{0}(z),\partial_{z}H_{i}(z)>
\end{equation}
for each $i\in\{1,...,m\}\, z\in B(z_{0},2\rho)$ and $Z_{i}(z)=T\partial_{z}H_{i}(z)\,,i\in\{1,...,m\}$ are in involution for $z\in B(z_{0},2\rho)$. In particular ,the equations in \eqref{eq:2.137} tell us that $\{H_{1},...,H_{m}\}$ are first integrals for the vector field $Z_{0}(z)=T\partial_{z}H_{0}(z)$.\\
\begin{remark}\label{re:2.17}
Taking $H_{0}\in\mathcal{C}^{2}(B(z_{0},2\rho;\mathbb{R}))$(instead
of $H_{0}
\in\mathcal{C}^{1}$) we are in position to rewrite the equation \eqref{eq:2.137} using Lie product $[Z_{0},Z_{j}]$ (see \eqref{eq:2.133})
as follows
\begin{equation}\label{eq:2.138}
0=\left\{
       \begin{array}{ll}
          [Z_{0},Z_{j}](z),\,\forall\,j\in\{1,...,m\},\,z\in  B(z_{0},2\rho)\\
         \{H_{i},H_{0}\}(z_{0})=<Z_{0}(z_{0}),\partial_{z}H_{i}(z_{0})>,\,i\in\{1,...,m\}
       \end{array}
     \right.
\end{equation}
\end{remark}
We conclude these considerations.
\begin{proposition}\label{pr:2.6}
Let
$z_{0}=(x_{0},p_{0})\in\mathbb{R}^{2n}$ and $B(z_{0},2\rho)\subseteq\mathbb{R}^{2n}$ be fixed
such that the hypothesis \eqref{eq:**} and \eqref{eq:2.136} are fulfilled.
Then $\{\hat{z}(\lambda,z_{o}):\lambda\in\Lambda\}$ defined in \eqref{eq:2.131}is a
parameterized stationary solution of the nonlinear equation \eqref{eq:2.129}.
\end{proposition}
\begin{remark}\label{re:2.18}
A standard solution for the nonlinear equation
$$H_{0}(x,p(x))=\hbox{const},\,\forall\,x\in D\subseteq
\mathbb{R}^{n}$$ can be found from a parameterized solution
provided the hypothesis  \eqref{eq:2.130} is stated with $m=n$ and assuming, in
addition, that $X_{1}(z_{0}),...,X_{n}(z_{0})\in\mathbb{R}^{n}$
from $$Z_{i}(z)=\left(
                                                                   \begin{array}{c}
                                                                     X_{i}(z) \\
                                                                     P_{i}(z) \\
                                                                   \end{array}
                                                                 \right)
,\,i\in\{1,...,n\}$$ are linearly independent. It leads us to the equations\\
\begin{equation}\label{eq:2.139}
H_{0}(\hat{x}(\lambda,z_{0}),\hat{p}(\lambda,z_{0}))=0\,\,\forall\,\lambda\in\Lambda=\prod_{1}^{n}[-a_{i},a_{i}]
\end{equation}
where the orbit $\hat{z}(\lambda,z_{0})=(\hat{x}(\lambda,z_{0}),\hat{p}(\lambda,z_{0}))$
satisfies the implicit functions theorem when the algebraic equations
$$\hat{x}(\lambda,z_{0})=x\in B(x_{0},\rho)\subseteq\mathbb{R}^{n}$$
are involved,Find a smooth mapping
$$\lambda=\psi(x):B(z_{0},2\rho)\rightarrow \lambda$$ such that $$\hat{x}(\psi(x),z_{0})=x$$ and define $$p(x)=\hat{p}(\psi(x),z_{0}),\,x\in B(x_{0},\rho)\subseteq\mathbb{R}^{n}$$
Then \eqref{eq:2.139} written for $\lambda=\psi(x)$ becomes
\begin{equation}\label{eq:2.140}
H_{0}(x,p(x))=const=H_{0}(x_{0},p_{0}),\,\forall\,x\in B(x_{0},\rho)
\end{equation}
and $p(x):B(x_{0},\rho)\rightarrow \mathbb{R}^{n}$ is the standard solution. Assuming that \eqref{eq:2.140}is established it implies that $\{p(x):x\in B(x_{0},\rho)\}$ is a
smooth solution of the following system of the first order causilinear equations
\begin{equation}\label{eq:2.141}
\partial_{x}H_{0}(x,p(x))+[\partial_{x}p^{*}(x)]^{*}.\partial_{p}H_{0}(x,p(x))=0,\,x\in
B(x_{0},\rho)
\end{equation}
whose solution is difficult to be obtained using characteristic system\index{Characteristic! system } method. In addition,if $[\partial_{x}p(x)]$
is a symmetric matrix $(p(x)=\partial_{x}u(x),u\in\mathcal{C}^{2}(D;\mathbb{R}))$ then the first order system \eqref{eq:2.141} becomes\\
\begin{equation}\label{eq:2.142}
\left(
       \begin{array}{c}
         \partial_{x_{1}}H_{0}(x,p(x)) \\
         . \\
         . \\
         . \\
          \partial_{x_{n}}H_{0}(x,p(x)) \\
       \end{array}
     \right)
+\left(
       \begin{array}{c}
         <\partial_{x}p_{1}(x)\partial_{p}H_{0}(x,p(x))> \\
         . \\
         . \\
         . \\
          <\partial_{x}p_{n}(x)\partial_{p}H_{0}(x,p(x))> \\
       \end{array}
     \right)=0
\end{equation}
for any $x\in B(x_{0},\rho)\subseteq\mathbb{R}^{n}$, where $p(x)=(p_{1}(x),...,p_{n}(x))$. In this case, the smooth mapping $y=p(x),\,x\in B(x_{0},\rho)$,
can be viewed as a coordinate transformation in $\mathbb{R}^{n}$ such that any local solution $(\hat{x}(t,\lambda),\hat{p}(t,\lambda))$ of the Hamilton system
\begin{equation}\label{eq:2.143}
\left\{
       \begin{array}{ll}
         \frac{dx}{dt} =& \partial_{p}H_{0}(x,p)\,\,\,x(0)=\lambda\in B(x_{0},a)\subseteq B(x_{0},\rho)\\\
        \frac{dp}{dt} =& -\partial_{x}H_{0}(x,p)\,\,\,p(0)=p(\lambda);t\in(-\alpha,\alpha)=I_{\alpha}
       \end{array}
     \right.
\end{equation}
has the property
$\hat{p}(t,\lambda)=p(\hat{x}(t,\lambda)),\,(t,\lambda)\in
I_{\alpha}\times B(x_{0},a)$.
\end{remark}
\subsection{Some Problems}
\textbf{Problem 1}. For the linear equation $H_{0}(x,p)=<p,f_{0}(x)>=c$ define an extended parameterized solution
$$\hat{Z}^{e}(\lambda,z_{0})=(\hat{x}(\lambda,z_{0}),\hat{p}(\lambda,z_{0}),\hat{u}(\lambda,z_{0}))\in\mathbb{R}^{2n+1},\,\lambda\in\Lambda=\prod_{1}^{m}[-a_{i},a_{i}]$$ such that
\begin{equation}\label{eq:2.144}
\partial_{\lambda}\hat{u}(\lambda,z_{0})=[\hat{p}(\lambda,z_{0})]^{*}\partial_{\lambda}\hat{x}(\lambda,z_{0})
\end{equation}
where $z_{0}=(x_{0},p_{0},u_{0})$.\\
\textbf{Hint}. Define the extended vector fields $$Z_{i}^{e}\in\mathcal{C}^{\infty}(\mathbb{R}^{2n+1},\mathbb{R}^{2n+1})\\
Z_{i}^{e}(z^{e})=\left(
                   \begin{array}{c}
                     Z_{i}(x,p) \\
                     H_{i}(x,p) \\
                   \end{array}
                 \right)
,z_{e}=(x,p,u)=(z,u)$$\\
where $H_{i}(x,p)=<p,f_{i}(x)>$ and $Z_{i}(x,p)=T\partial_{z}H_{i}(x,p),\,z=(x,p),\,i\in\{1,...,m\}$, where the matrix $T$ is defined in \eqref{eq:2.110}. Assume the hypothesis \eqref{eq:2.114} of
section $2.4.4$ fulfilled and define the orbit
\begin{equation}\label{eq:2.145}
\hat{Z}^{e}(\lambda,z_{0})=G_{1}^{e}(t_{1})\circ...\circ G_{m}^{e}(t_{m})(z_{0}),\,\,\lambda=(t_{1},...,t_{m})\in\Lambda=\prod_{1}^{m}[-a_{i},a_{i}]
\end{equation}
where $G_{i}^{e}(\tau)(y^{e})$ is the flow \index{Flow!local}generated by $Z_{i}^{e}$. Prove that under the hypothesis assumed in proposition \eqref{pr:2.5}
we get
 $\{\hat{z}^{e}(\lambda,z_{0}):\lambda\in\Lambda\}$ as a parameterized solution of the linear equation $H_{0}(x,p)=const$ satisfied the conclusion \eqref{eq:2.144}\\
\textbf{Problem 2}. For the nonlinear equation $H_{0}(x,p)=\hbox{const}$ define an extended parameterized solution
$$\hat{z}^{e}(\lambda,z_{0}^{e})=(\hat{x}(\lambda,z_{0}^{e}),\hat{p}(\lambda,z_{0}^{e})),\hat{u}(\lambda,z_{0}^{e})\in\mathbb{R}^{2n+1},
\lambda\in\Lambda=\prod_{1}^{m}[-a_{i},a_{i}]$$
such that
\begin{equation}\label{eq:2.146}
\partial_{\lambda}\hat{u}(\lambda,z_{0})=[\hat{p}(\lambda,z_{0})]^{*}\partial_{\lambda}\hat{x}(\lambda,z_{0}),\,\lambda\in\Lambda\,
\hbox{where} \,Z_{0}^{e}=(x_{0},p_{0},u_{0})\in\mathbb{R}^{2n+1}
\end{equation}
is fixed.\\
\textbf{Hint}\begin{equation}\label{eq:2.147}
Z_{i}^{e}(z^{e})=\left(
                                 \begin{array}{c}
                                   Z_{i}(z) \\
                                   <p,X_{i}(z)> \\
                                 \end{array}
                               \right)
\,where\, Z_{i}(z)\left(
                    \begin{array}{c}
                      X_{i}(z) \\
                      P_{i}(z) \\
                    \end{array}
                  \right)
\end{equation}
is obtained from a smooth $H_{i}\in\mathcal{C}^{\infty}(\mathbb{R}^{2n},\mathbb{R})$ as follows
$$Z_{i}(z)=T\partial_{z}H_{i}(z),\,i\in\{1,...,m\}$$  given in \eqref{eq:2.128}
defined the extended orbit
\begin{equation} \label{eq:*}
\hat{z}^{e}(\lambda,z_{0}^{e})=G_{1}^{e}(t_{1})\circ...\circ G_{m}^{e}(t_{m})(z_{0}^{e}),\,\lambda=(t_{1},...,t_{m})\in\Lambda=\prod_{1}^{m}[-a_{i},a_{i}]
\end{equation}
where $G_{i}^{e}(\tau)(y^{e})$ is the local flow \index{Flow!local}$\{Z_{i}^{e}\}$. Assume that
\begin{equation}\label{eq:**}
\{Z_{1}^{e}(z^{e}),...,Z_{m}^{e}(z^{e}):z^{e}\in B(z_{0},2\rho)\subseteq\mathbb{R}^{2n+1}\}
\end{equation}
are in involution over $\mathcal{C}^{\infty}(B(z_{0},2\rho);\mathbb{R})$ and then prove
the corresponding proposition \eqref{pr:2.6} of assuming that \label{eq:2.149} replaces
the condition \eqref{eq:2.130}.
\section[{\small Overdetermined System of 1st O.PDE \& Lie Algebras}]{Overdetermined system\index{Overdetermined system} of First Order PDE and Lie Algebras of Vector Fields}
As one may expect, a system of first order partial differential equations\index{Partial differential equations}
\begin{equation}\label{eq:2.148}
<\partial_{x}S(x),g_{i}(x)>=0,\,i\in\{1,...,m\}\,g_{i}\in\mathcal{C}^{\infty}(\mathbb{R}^{n},\mathbb{R}^{n})
\end{equation}
has a nontrivial solution $S\in\mathcal{C}^{2}(B(z_{0},\rho)\subseteq\mathbb{R}^{2n})$ provided $S$ is nonconstant on the ball $B(z_{0},\rho)\subseteq\mathbb{R}^{2n+1}$ and
\begin{equation}\label{eq:2.149}
\dim L(g_{1},...,g_{m}(x_{0}))=k<n
\end{equation}
Here $L(g_{1},...,g_{m})\subseteq\mathcal{C}^{\infty}(\mathbb{R}^{n},\mathbb{R}^{n})$
is the Lie algebra determined by the given vector fields $\{g_{1},...,g_{m}\}\subseteq\mathcal{C}^{\infty}(\mathbb{R}^{n},\mathbb{R}^{n})$, it
can be viewed as the smallest Lie algebra $\Lambda\subseteq\mathcal{C}^{\infty}(\mathbb{R}^{n},\mathbb{R}^{n})$ containing $\{g_{1},...,g_{m}\}$.
\begin{definition} A real algebra $\Lambda\subseteq \mathcal{C}^{\infty}(\mathbb{R}^{n},\mathbb{R}^{n})$ is Lie algebra if the multiplication operation among vector fields
 $X,Y\subseteq\mathcal{C}^{\infty}(\mathbb{R}^{n},\mathbb{R}^{n})$ is given by the corresponding Lie bracket $[X,Y](x)=\frac{\partial Y}{\partial x}(x)X(x)-\frac{\partial X}{\partial x}(x)Y(x),\,x\in\mathbb{R}^{n}$. By a direct computation we may convince ourselves that the following two properties are valid\\
$$[X,Y]=-[Y,X],\,\forall\,X,Y\in\Lambda$$ and
$$[X,[Y,Z]]+[Z,[X,Y]]+[Y,[Z,X]]=0,\,\forall\,X,Y,Z\in \Lambda$$(Jacobi
s' identity)
\end{definition}
\begin{remark}\label{re:2.19}
The system \eqref{eq:2.148} is overdetermined when $m>1$. Usually a nontrivial
solution $\{S(x),x\in B(x_{0},\rho)\}$  satisfies  the system \eqref{eq:2.148}
only on a subset $M_{x_{0}}\subseteq
B(x_{0},\rho)\subseteq\mathbb{R}^{n}$, which can be structured as a
smooth manifold satisfying $\dim M_{x_{0}}=k$ (see \eqref{eq:2.149}); it will be
deduced from an orbit of local flows.
\end{remark}
\begin{definition}
Let $\Lambda\subseteq\mathcal{C}^{\infty}(\mathbb{R}^{n},\mathbb{R}^{n})$ be a Lie  algebra and $x_{0}\in\mathbb{R}^{n}$ fixed. By an orbit of the origin $x_{0}$ of $\Lambda$ we mean a mapping (finite composition of flows)
\begin{equation}\label{eq:2.150}
G(p,x_{0})= G_{1}(t_{1})\circ ....\circ
G(t_{k})(x_{0})\,,p=(t_{1},t_{k})\in
D_{k}=\prod_{1}^{k}[-a_{i},a_{i}]
\end{equation}
where $G_{i}(t)(x),
t\in(-a_{i},a_{i}),x\in V(x_{0})$, is the local
flow\index{Flow!local} generated by some $g_{i}\in\Lambda,
i\in\{1...k\}.$
\end{definition}
\begin{definition}
We say that
$\Lambda\subseteq\mathcal{C}^{\infty}(\mathbb{R}^{n},\mathbb{R}^{n})$is
finitely generated with respect to the  orbits of the origin
$x_{0}\in\mathbb{R}^{n},(f,g,0,x_{0})$ if
$\{Y_{i},...,Y_{m}\}\subseteq\Lambda$ will exist such that any
$Y\in\Lambda$ along on arbitrary orbit $G(p, x_{0}), p\in
D_{k}$ can be written
\begin{equation}\label{eq:2.151}
 Y(G(p,x_{0}))=\Sigma_{j=1}^{M}a_{j}(p)Y_{j}(G(p,x_{0}))
\end{equation}
with
$a_{j}\in\mathcal{C}^{\infty}(D_{k},\mathbb{R})$depending on $Y$ and
$G(p,x_{0}),p\in D_{k};\{Y_{I},...Y_{M}\}\subseteq\Lambda$ will be
called a system of generators.
\end{definition}
\begin{remark}\label{re:2.20}It is easily seen that $\{g_1,..g_m\}\subseteq\mathcal{C}^{\infty}(\mathbb{R}^{n}),\mathbb{R}^{n})$ in involution determine a Lie algebra
$L(g_{I},..g_{m})$ which is $(f.g.o.x_{0})$ for any $x_{o}\in\mathbb{R}^{n}$, where the involution properly means
 $[g_{i},...,g_{j}](x)=\mathop{\sum}\limits_{k=1}^{m}a_{ij}^{k}(x)g_{k}(x) \,\forall,i,j\in\{1,...m\}$ for some,\\
$a_{ij}^{k}in\mathcal{C}^{\infty}(\mathbb{R}^{n}),\{g_{1},...,g_{m}\}$ will
be a system of generators. A nontrivial solution of the system $(1)$
will be constructed assuming that
\begin{equation}\label{eq:2.152}
L(g_{1},...g_{m})\,is\, a\, (f,g,o;x_{o})\,\hbox{Lie algebra}
\end{equation}
\begin{equation}\label{eq:2.153}
\dim L(g_{1},...g_{m})(x_{0})=k<n
\end{equation}
In addition, the domain $V(x_{0})\subseteq \mathbb{R}^{n}$ on which a non trivial solution satisfies\eqref{eq:2.148} will be defined as an orbit starting the origin $x_{o}\in\mathbb{R}^{n}$
\begin{equation}\label{eq:2.154}
y(p)=G_{1}(t_{1}\circ...\circ G_{M}(t_{M})(x_{0}),\,p=(t_{1},...,t_{M})\in
D_{M}=\prod_{1}^{M}(-a_{i},a_{i})
\end{equation}
where $G_{i}(t)(x),t\in(-a_{i},a_{i})\,\,,x\in B(x_{0},\rho)\subseteq\mathbb{R}^{n}$ is the local flow generated by $Y_{i}\in L(g_{1},...,g_{m}),\,\,i\in\{1,...,M\}$.Here $\{Y_{1},...,Y_{M}\}\subseteq L(g_{1},...,g_{m})$ is fixed system of generators such that
\begin{equation}\label{eq:2.155}
\{Y_{1}(x_{0}),...,Y_{k}(x_{0})\}\subseteq \mathbb{R}^{n}
\end{equation}
are linearly independent and
$Y_{j}(x_{0})=0,\,j\in\{1,...,M\}$
\end{remark}
\subsection{Solution for Linear Homogeneous Over Determined System\index{Homogeneous!overdetermined system}.}
\begin{theorem}\label{th:2.4}
Assume that $g_{1},...g_{m}\subseteq\mathcal{C}^{\infty}(\mathbb{R}^{n}),\mathbb{R}^{n})$ are given such that the hypotheses \eqref{eq:2.152} and \eqref{eq:2.153} are satisfied. Let $\{Y_{1},..,Y_{M}\}\subseteq L(g_{1},..g_{m})$ be a system of generators fulfilling\eqref{eq:2.155} and define
\begin{equation}\label{eq:2.156}
V(x_{o})=\{x\in B(x_{0},\rho)\subseteq \mathbb{R}^{n} :x=y(p),p\in D_{M}\}\subseteq \mathbb{R}^{n}
\end{equation}
 where $\{y(p):\in D_{m}\}$ is the orbit define in\eqref{eq:2.154}. Then there exists a smooth nontrivial function\\
 $ S (x):B(x_{o},\rho)\rightarrow \mathbb{R} $ such that the system \eqref{eq:2.148} is satisfied for any $ x\in V(x_{0})$given in \eqref{eq:2.156},i.e
\begin{equation}\label{eq:2.157}
<\partial_{x}S(x),g_{i}(x)>=0 \,\,\forall x\in
V(x_{0})\,\,i\in\{1,...m\}
\end{equation}
\end{theorem}
\begin{proof}
 To prove the conclusion \eqref{eq:2.157} we need to show the gradient system associated with the  orbit $\{y(p): p \in D_{M}\}$ defined in \eqref{eq:2.154}  has a nonsingular algebraic representation
\begin{eqnarray}\label{eq:2.158}
  \frac{\partial y(p)}{\partial p} &=& \{Y_{1}(.),X_{2}(t_{1},.),...,X_{M}(t_{1},...,t_{M-i};.)\}(y(p)) \nonumber\\
  &=& \{Y_{1}(.),Y_{2}(.),...,Y_{M}\}(y(p))A(p) p\in D_{M}
\end{eqnarray}
where the smooth matrix
$A(p)=(a_{ij}(p))_{i,j\in\{1,M\}}$
 satisfies
\begin{equation}\label{eq:2.159}
A(0)=I_{M}, a_{ij}\in \mathcal{C}^{\infty}(D_{M})
\end{equation}
(seem next theorem). On the other hand, we notice
that according to the fixed system of generators
$\{Y_{1},...,Y_{M}\}\subset L(g_{1},...,g_{m})$ (see \eqref{eq:2.155}),
we may and rewrite the orbit\eqref{eq:2.154} as follows
\begin{equation}\label{eq:2.160}
\{y(p):p\in D_{M}\}=\{y(\hat{p})=G_{1}(t_{1}), ...,
G_{k}(t_{k})(x_0):\hat{p}=(t_{1},...,t_{k})\in D_{k}\}
\end{equation}
Using a standard procedure we redefine
\begin{equation}\label{eq:2.161}
M_{x_{0}}=\{y(p):p\in B(0,\alpha)\}\subseteq
B(x_{0},\varrho)\subseteq \mathbb{R}^{n},\,\hbox{where}\,B(0,\rho)\subseteq
D_{k}\}
\end{equation}
as a smooth manifold satisfying $\dim M_{x_{0}}=k<n$ for which there exist $n-k$ smooth functions
$\varphi_{j}\in
\mathcal{C}^{\infty}(B(x_{0},\varrho))$ fulfilling
\begin{equation}\label{eq:2.162}
\left\{
      \begin{array}{ll}
        \varphi_{j}(x)=0,\,j\in\{k+1,...,n\},x\in
M_{x_{0}} \\
        \{\partial_{x}\varphi_{k+1}(x_{0}),...,\partial_{x}\varphi_{n}(x_{0})\}\subseteq\mathbb{R}^{n}\hbox{ are linearly independent }
      \end{array}
    \right.
\end{equation}
Using \eqref{eq:2.159} and $\det A(p)\neq 0$ for $p$ in a ball
$B(0,\alpha)\subseteq D_{M}$ we find smooth $q_{j}\in
C^{\infty}(B(0,\alpha);\mathbb{R}^{M}),j\in \{1,..,M\}$, such that
\begin{equation}\label{eq:2.163}
\left\{
      \begin{array}{ll}
        A(p)q_{j}(p)=e_j,\,p\in
B(0,\alpha),\,\{l_{1},...,l_{M}\}\subseteq\mathbb{R}^{M} \hbox{ is the canonical basis }\\
        \frac{\partial y(p)}{\partial p}q_{j}(p)=Y_{j}(y(p)),\,p\in B(0,\alpha),\,j\in\{1,...,M\}
      \end{array}
    \right.
\end{equation} Using \eqref{eq:2.162} and \eqref{eq:2.163} we get $\varphi_{j}(y(p))=0,\,p\in
B(0,\alpha)\subseteq D_{M},\,j\in\{k+1,...,n\}$ and taking the Lie
derivatives in the directions $q_{1}(p),...,q_{M}(p)$ we obtain\\
{\small
\begin{equation}\label{eq:2.164}
<\partial_{x}\varphi_{j}(y(\hat{p})),\,Y_{i}(y(\hat{p}))>=0\,\hat{p}\in
B(0,\alpha)\subseteq D_{k},\,i\in\{1,...,M\} \,\forall\, j\in
\{k+1,...,n\}
\end{equation}}
By hypothesis, $L(g_{1},...,g_{m})$ is
a $(f\cdot g\cdot o;x_{0})$ Lie Algebra and each
$g\in\,L\{g_{1},...,g_{m}\}$ can be written
\begin{equation}\label{eq:2.165}
g(y(\hat{p}))=\sum\limits_{i=1}^{M}\alpha_{i}(\hat{p})Y_{i}(y(\hat{p})),\,\hat{p}\in
B(0,\alpha)\subseteq \mathbb{R}^{k},\,\alpha_{i}\in
\mathcal{C}^{\infty}(B(0,\alpha))
\end{equation}
 Using \eqref{eq:2.164} and \eqref{eq:2.165} we get
\begin{eqnarray}\label{eq:2.166}
  <\partial_{x}\varphi_{j}(x),\,g_{i}(x)>&=& 0\,\forall\,x\in
M_{x_{0}}\subseteq B(0,\varrho) \hbox{ for each } \nonumber\\
  && i\in\{1,...,m\} \hbox{ and }j\in\{k+1,...,n\}
\end{eqnarray}
Here $\{\varphi_{k+1}(x),...,\varphi_{n}(x):x\in B(0,\varrho)\}$ are $(n-k)$ smooth solutions satisfying the overdetermined
system\index{Overdetermined system} \eqref{eq:2.148} along the $k$-dimensional
manifold $M_{x_{0}}$ in \eqref{eq:2.161}. The proof is complete.
\end{proof}
\begin{remark}\label{re:2.21}
The nonsingular algebraic representation of a gradient
system\index{Gradient systems}
\begin{equation}\label{eq:2.167}
\frac{\partial y}{\partial t_{1}}=Y_{1}(y),\,\frac{\partial y}{\partial t_{2}}=X_{2}(t_{1};y),...,\frac{\partial y}{\partial t_{M}}=X_{M}(t_{1},...,t_{M-1};y)
\end{equation}
associated with the system of generators
$\{Y_{1},...,Y_{M}\}\subseteq L(g_{1},...,g_{M})$ relies on the
assumption that $L(g_{1},...,g_{M})$ is a $(f\cdot g\cdot o;x_{0})$
Lie algebra and the orbit \eqref{eq:2.154} is a solution of \eqref{eq:2.167} satisfying
$y(0)=x_{0}$. The algorithm of defining a gradient system for which
a given orbit is its solution was initiated in Chapter $I$ of these
lectures. Here we shall use the following formal power series
\begin{equation}\label{eq:2.168}
\left\{
   \begin{array}{ll}
     X_{2}(t_{1};y)=(\exp t_{1}ad Y_{1})(Y_{2})(y) \\
     X_{3}(t_{1},t_{2};y)=(\exp t_{1}ad Y_{1})\cdot(\exp t_{2}ad Y_{2})(Y_{3})(y) \\
     \cdot \\
     \cdot \\
     \cdot \\
     X_{M}(t_{1},t_{2},...,t_{M-1};y)=(\exp t_{1}ad Y_{1})\cdot...\cdot(\exp t_{M-1}ad Y_{M-1})(Y_{M})(y)
   \end{array}
 \right.
\end{equation}
where the linear mapping $ad
Y:L(g_{1},...,g_{M})\rightarrow L(g_{1},...,g_{M})$ is defined by
{\small
\begin{equation}\label{eq:2.169}
(ad Y)(Z)(y)=\frac{\partial Z}{\partial y}(y)\cdot
Y(y)-\frac{\partial Y}{\partial
y}Z(y),y\in\mathbb{R}^{n}\,for\,each\,Y,Z\in L(g_{1},...,g_{M})
\end{equation}}
Actually, the vector fields in the left hand side of
\eqref{eq:2.168} are well
defined by the following mappings
\begin{equation}\label{eq:2.170}
\left\{
   \begin{array}{ll}
     X_{2}(t_{1};y)=H_{1}(-t_{1};y_{1})(Y_{2})(G_{1}(-t_{1};y_{1})),y_{1}=y, \\
     X_{3}(t_{1},t_{2};y)=H_{1}(-t_{1};y_{1})\cdot H_{2}(-t_{2};y_{2})(Y_{3})(G_{2}(-t_{2};y_{2})) \\
     \cdot \\
     \cdot \\
     \cdot \\
     X_{M}(t_{1},t_{2},...,t_{M-1};y)=H_{1}(-t_{1};y_{1})\cdot H_{2}(-t_{2};y_{2})\cdot...\cdot
     H_{M-1}(-t_{M-1};y_{M-1})(Y_{M})(y_{M})
    \end{array}
 \right.
\end{equation} where $$G_{i}(-t_{i};y_{i}),\,t\in(-a_{i},a_{i}),\,x\in
V(x_{0})$$ is the local flow \index{Flow!local}generated by
$$Y_{i},H_{i}(t,y)=[\frac{\partial G_{i}}{\partial y}(t;y)]^{-1}$$ and
$$y_{i+1}=G_{i}(-t_{i};y_{i}),\,\in\{1,...,M-1\},\hbox{ where }
y_{1}=y.$$ The formal writing \eqref{eq:2.168} is motivated by the explicit
computation we can perform using exponential formal
series\index{Exponential !formal series} and noticing that the
Taylor series associated with \eqref{eq:2.168} and \eqref{eq:2.170} coincide.
\end{remark}
\section[Nonsingular Representation of a Gradient System]{Nonsingular Algebraic Representation of a \\Gradient System}
\begin{theorem}\label{th:2.5}
Assume that $L(g_{1},...,g_{m})\subseteq
\mathcal{C}^{\infty}(\mathbb{R}^{n};\mathbb{R}^{n})$ is a $(f\cdot
g\cdot o;x_{0})$ Lie algebra and consider
$\{Y_{1},...,Y_{M}\}\subseteq L(g_{1},...,g_{M})$ as a fixed system
of generators. Define the orbit $\{y_{p}:p\in D_{M}\}$ as in \eqref{eq:2.154} and
associate the gradient system \eqref{eq:2.167} where
$\{X_{2}(t_{1};y),...,X_{M}(t_{1},t_{2},...,t_{M-1};y)\}$ are
defined in \eqref{eq:2.170}. Then there exist an $(M\times M)$ nonsingular
matrix $A(p)=(a_{ij}(p))_{i,j\in\{1,...,M\}},a_{ij}\in
\mathcal{C}^{\infty}(B(0,\alpha)\subseteq\mathbb{R}^{n})$ such that
{\small
\begin{equation}\label{eq:2.171}
\{Y_{1}(\cdot)X_{2}(t_{1};\cdot),...,X_{M}(t_{1},...,t_{M-1};\cdot)\}(y(p))=\{Y_{1}(\cdot),Y_{2}(\cdot),...,Y_{M}(\cdot)\}(y(p))A(p)
\end{equation}}
 for any $p\in
B(0,\alpha)\subseteq\mathbb{R}^{n},\hbox{ and } A(0)=I_{M}$
\end{theorem}
\begin{proof}
Using the $(f\cdot g\cdot o;x_{0})$ property of Lie algebra
$L(g_{1},...,g_{M}$ we fix the following $(M\times M)$ smooth
matrices $B_{i}(p),\,p\in D_{M}$, such that
\begin{equation}\label{eq:2.172}
adY_{i}\{Y_{1},...,Y_{M}\}(y(p))=\{Y_{1},...,Y_{M}\}(y(p))\cdot
B_{i}(p),\,i\in\{1,...,M\}
\end{equation}
where
$\{Y_{1},...,Y_{M}\}\subseteq L(g_{1},...,g_{m})$ is a fixed system
of generators. Let $\{Z_1(t_{1},...,t_{M}):t\in[0,t_{1}]\}$ be the
matrix solution
\begin{equation}\label{eq:2.173}
\frac{dZ_{1}}{dt}=Z_{1}B_{1}(t_{1}-t,t_{2},...,t_{M}),Z_{1}(0)=I_{M}\,\hbox{(identity)}
\end{equation}
The standard Picard's method of approximations applied
in \eqref{eq:2.173} allows one to write the solution
$\{z,(t_{1},...,t_{M}):t\in[0,t_{1}]\}$ as a convergent Volterra
series, and by a direct computation we obtain
\begin{equation}\label{eq:2.174}
X_{2}(t_{1};y(p))=\{Y_{1},...,Y_{M}\}(y(p))z_{1}(t_{1},t_{2},...,t_{M})e_{2}
\end{equation}
where $X_{2}(t_{1};y)$ is defined in \eqref{eq:2.170} and
$\{e_{1},...,e_{M}\}\subset\mathbb{R}^{n}$ is the canonical basis.
More precisely, denote
\begin{equation}\label{eq:2.175}
N_{1}(t;y)=H_{1}(-t;y)\{Y_{1},...,Y_{M}\}(G_{1}(-t;y)) \hbox{and} \,X_{2}(t_{1};y)
\end{equation}
 can be written
\begin{equation}\label{eq:2.176}
X_{2}(t_{1};y)=N_{1}(t_{1};y)e_{2}
\end{equation} Then we noticed that $N_{1}(t;y(p)),\,t\in[0,t_{1}]$ satisfies
\begin{equation}\label{eq:2.177}
\frac{dN_{1}}{dt}=N_{1}B_{1}(t_{1}-t,t_{2},...,t_{M}),N_{1}(0)=\{Y_{1},...,Y_{M}\}(y(p))
\end{equation} where the matrix $B_{1}(p)$ is fixed in \eqref{eq:2.119}
and
satisfies
\begin{equation}\label{eq:2.178}
adY_{i}\{Y_{1},...,Y_{M}\}(G_{1}(-t;y))=\{Y_{1},...,Y_{M}\}(G_{1}(-t;y))\cdot
B_{1}(t_{1}-t,t_{2},...,t_{M})
\end{equation}
for $y=y(p)$. The same method of Approximation shows that a convergent
sequence $\{N_{1}^{k}(t):t\in[0,t_{1}]\}_{k\geq 0}$ can be
constructed such
that
\begin{equation}\label{eq:2.179}
\left\{
           \begin{array}{ll}
             N_{1}^{0}(t)=\{Y_{1},...,Y_{M}\}(y(p))=N_{1}(0)  \\
             N_{1}^{k+1}(t)=N_{1}(0)+,
\int\limits_{0}^{t}N_{1}^{k}(s)B_{1}(t_{1}-s,t_{2},...,t_{M})ds,k=0,1,2,...
           \end{array}
         \right.
\end{equation}
It is easily seen that
\begin{equation}\label{eq:2.180}
N_{1}^{k+1}(t)=N_{1}(o)Z_{1}^{k+1}(t;t_{1},...,t_{M}),\,t\in[0,t_{1}]
\end{equation}
where $\{Z_{1}^{k}(t;t_{1},...,t_{M}):t\in[0,t_{1}]\}_{k\geq 0}$ defines a solution in \eqref{eq:2.173} and
\begin{equation}\label{eq:2.181}
\mathop{lim}\limits_{k\rightarrow\infty}Z_{1}^{k}(t;t_{1},...,t_{M})=Z_{1}(t;t_{1},...,t_{M})
\end{equation}
uniformly in$t\in[0,t_{1}]$. Therefore using \eqref{eq:2.180} and \eqref{eq:2.181} we get that
\begin{equation} \label{eq:2.182}
N_{1}(t)=\mathop{lim}\limits_{k\rightarrow\infty}N_{1}^{k}(t)\,\,,t\in[0,t_{1}]
\end{equation}
is the solution in \eqref{eq:2.177} fulfilling
\begin{equation}\label{eq:2.183}
N_{1}(t;y)=\{Y_{1},...,Y_{M}\}(y)Z_{1}(t;t_{1},...,t_{M}),\,\,t\in[0,t_{1}]
\end{equation}
if $y=y(p)$, and \eqref{eq:2.174} holds. The next vector field $X_{3}(t_{1},t_{2};y)$ in \eqref{eq:2.117}
can be represented similarly and rewriting it as
\begin{equation}\label{eq:2.184}
X_{3}(t_{1},t_{2};y)=H_{1}(-t_{1};y)\hat{X_{3}}(t_{2};y_{2}),\,\hbox{for} \,y_{1}=y(p)=y
\end{equation}
$$y_{2}=G_{1}(-t_{1};y_{1})=G_{2}(t_{2})\circ...\circ G_{M}(t_{M})$$
and
\begin{equation}\label{eq:2.185}
\hat{X_{3}}(t_{2};y_{2})=H_{2}(-t_{2};y_{2})Y_{3}(G_{2}(-t_{2},y_{2}))
\end{equation}
can be represented using the same algorithm as in \eqref{eq:2.174} and we get
\begin{equation}\label{eq:2.186}
\hat{X_{3}}(t_{2};y_{2})=\{Y_{1},...,Y_{M}\}(y_{2})Z_{2}(t_{2};t_{2},...,t_{M})e_{3}
\end{equation}
Here the nonsingular and smooth $M\times M \,\{Z_{2}(t;t_{2},...,t_{M}:t\in[0,t_{2}]\}$ is the unique solution of the following linear matrix system
\begin{equation}\label{eq:2.187}
\frac{dZ_{2}}{dt}=Z_{2}B_{2}(t_{2}-t,t_{3},...,t_{M}),\,Z_{2}(0)=I_{M}
\end{equation}
where $B_{2}(t_{2},...,t_{M})\,p\in D_{M}$ is fixed such that
\begin{equation}\label{eq:2.188}
adY_{2}\{Y_{1},...,Y_{M}\}(y_{2})=\{Y_{1},...,Y_{M}\}(y_{2})B_{2}(t_{2},...,t_{M}),\,p\in D_{M}
\end{equation}
Denote
\begin{equation}\label{eq:2.189}
N_{2}(t,y_{2})=H_{(-t,y_{2})}\{Y_{1},...,Y_{M}\}G_{2}(-t;y_{2}),\,t\in[0,t_{2}]
\end{equation}
and we obtain
\begin{equation}\label{eq:2.190}
\hat{x}_{3}(t_{2};y_{2})=N_{2}(t_{2};y_{2})e_{3}=\{Y_{1},...,Y_{M}\}(y_{2})Z_{2}(t_{2};t_{2},...,t_{M})e_{3}
\end{equation}
That is to say, $\{N_{2}(t;y_{2}):t\in[0,t_{2}]\}$ as a solution of
\begin{equation}\label{eq:2.191}
\frac{dN_{2}}{dt}=N_{2}B_{2}(t_{2}-t,t_{3},...,t_{M}),\, N_{2}(0)=\{Y_{1},...,Y_{M}\}(y_{2})
\end{equation}
can be represented using a standard iterative procedure and we get
\begin{equation}\label{eq:2.192}
N_{2}(t;y_{2})=\{Y_{1},...,Y_{M}\}(y_{2})Z_{2}(t;t_{2},...,t_{M})\,\, ,t\in[0,t_{2}]
\end{equation}
where the matrix $Z_{2}$ is the solution in\eqref{eq:2.187}. Now we use \eqref{eq:2.190} into \eqref{eq:2.184} and taking into account that
$y_{2}=G_{1}(-t_{1},y_{1})$ we rewrite
\begin{equation}\label{eq:2.193}
H_{1}(-t_{1};y_{1})\{Y_{1},...,Y_{M}\}(G_{1}(-t_{1},y_{1}))=N_{1}(t_{1};y_{1})\,\hbox{for}\,y_{1}=y(p)
\end{equation}
where $\{N_{1}(t;y(p)):t\in[0,t_{1}]\}$ fulfills\eqref{eq:2.183}. Therefore the vector field $X_{3}(t_{1},t_{2};y)$ in \eqref{eq:2.170} satisfies
\begin{equation}\label{eq:2.194}
X_{3}(t_{1},t_{2};y)=\{Y_{1},...,Y_{M}\}(y)Z(t_{1},t_{1},...,t_{M})\times Z_{2}(t_{2};t_{2},...,t_{M})e_{3}
\end{equation}
for $y=y(p)$. An induction argument will complete the proof and each $X_{j}(p_{j};y)$ in \eqref{eq:2.170} gets the corresponding representation
\begin{equation}\label{eq:2.195}
X_{j+1}(p_{j+1};y)=\{Y_{1},...,Y_{M}\}(y)Z(t_{1},t_{1},...,t_{M})\times ...\times Z_{j}(t_{j};t_{j},...,t_{M})e_{j+1}
\end{equation}
if $y=y(p)$ and $\{Z_{j}(t,t_{j},...,t_{M}:t\in[0,t_{j}])\}\,j\in\{2,...,M-1\}$ is the solution of the linear matrix equation
\begin{equation}\label{eq:2.196}
\frac{dZ_{j}}{dt}=Z_{j}B_{j}(t_{j}-t,t_{j+1},...,t_{M}),\,Z_{j}(0)=I_{M}
\end{equation}
Here the smooth matrix $B_{j}(t_{j},....,t_{M})\,\, ,p\in D_{M}$,
is fixed such that
\begin{equation}\label{eq:2.197}
adY_{j}\{Y_{1},...,Y_{M}\}(y_{j})=\{Y_{1},...,Y_{M}\}(y_{j})B_{j}(t_{j},...,t_{M})
\end{equation}
where $y_{j}=G_{j}(t_{j})\circ...\circ G_{M}(t_{M})(x_{0})\,\,,j=1,...,M$. Now, for simplicity, denote the $M\times M$ smooth matrix
\begin{equation}\label{eq:2.198}
Z_{j}(p)=Z_{1}(t_{1};t_{1},...,t_{M})\times...\times  Z_{j}(t_{j};t_{j},...,t_{M}),\,p=(t_{1},...,t_{M})\in D_{M}
\end{equation}
$$j\in\{1,...,M-1\}\,and\,X_{j+1}(p_{j+1};y)$$
in \eqref{eq:2.195} is written as
\begin{equation}\label{eq:2.199}
X_{j+1}(p_{j+1};y)=\{Y_{1},...,Y_{M}\}(y)Z_{j}(p)e_{j+1}\, ,if\,y=y(p),\,p\in\, D_{M}
\end{equation}
In addition, define the smooth $(M\times M)$
\begin{equation}\label{eq:2.200}
A(p)=(e_{1},Z_{1}(p)e_{2},...,Z_{M-1}(p)e_{M}),\, p\in D_{M}
\end{equation}
which fulfills $A(0)=I_{M}$ and represent the gradient system \eqref{eq:2.167} as follows
\begin{equation}\label{eq:2.201}
\{Y_{1}(.),X_{2}(t_{1},.),...,X_{M}(t_{1},...,t_{M-1},.)\}(y)=\{Y_{1},...,Y_{M}\}(y)A(p)
\end{equation} if $y=y(p)$, and the smooth $(M\times M)$ matrix $\,A(p)$ is defined in \eqref{eq:2.200}. The proof is complete.
 \end{proof}
 \section[1st Order Evolution System of PDE and C-K Theorem]{First Order Evolution System of PDE and Cauchy-Kowalevska Theorem\index{Cauchy!Kowalevska theorem}}
 We consider an evolution system of PDE of the following form
\begin{equation}\label{eq:1}
\left\{
   \begin{array}{ll}
    \partial_{t}u_{j}= & \mathop{\sum}\limits_{i=1}^{n}\mathop{\sum}\limits_{k=1}^{N}a_{jk}^{i}(t,x,u)\partial_{i}u_{k}+b_{j}(t,x,u),\,j=1,...,N\\
   u(0,x) =& u_{0}(x)
   \end{array}
 \right.
\end{equation}
where $u=(u_{1},...,u_{N})$ is an unknown vector function depending on the time variable $t\in\mathbb{R}$ and $x\in\mathbb{R}^{n}$. The system \eqref{eq:1}
will be called first order evolution system of $PDE$ and without losing generality we may assume $u_{0}=0$
(see the transformation $\widetilde{u}=u-u_{0}$) and the coefficients $a_{ij}^{k},\,b_{j}$ not depend on the variable  $t$(see $u_{N+1}=t,\,\partial_{t}u_{N+1}=1$). The notations used in \eqref{eq:1} are the usual ones $\partial_{t}u=\frac{\partial u}{\partial t}\,,\partial_{i}u=\frac{\partial u}{\partial x_{i}},\,i\in\{1,...,n\}$. Introducing a new
 variable $z=(x,u)\in B(0,\rho)\subseteq\mathbb{R}^{n+N}$, we rewrite the system \eqref{eq:1} as follows
\begin{equation}\label{eq:2}
\left\{
   \begin{array}{ll}
    \partial_{t}u_{j}= & \mathop{\sum}\limits_{i=1}^{n}\mathop{\sum}\limits_{k=1}^{N}a_{jk}^{i}(z)\partial_{i}u_{k}+b_{j}(z),\,j=1,...,N\\
   u(0) =& 0
   \end{array}
 \right.
\end{equation}
Everywhere in this section we assume that the coefficients
$a_{jk}^{i},b_{j}\in\mathcal{C}^{w}(B(0,\rho)\subseteq
\mathbb{R}^{n+N})$ are analytic functions, i.e the corresponding
Taylor's
 series at $z=0$ is convergent in the ball $B(0,\rho)\subseteq \mathbb{R}^{n+N}$ (Brook Taylor 1685-1731)\\
\subsection{Cauchy-Kowalevska Theorem\index{Cauchy!Kowalevska theorem}}
\begin{theorem}\label{th:2.6}
Assume that $$a_{jk}^{i},b_{j}\in \mathcal{C}^{w}(B(o,\rho))\subseteq
\mathbb{R}^{n+N}$$ for any $$i\in\{1,...,n\}, k,j\in\{1,...,n+N\}$$
Then the evolution system \eqref{eq:2} has an analytical solution
$u\in\mathcal{C}^{w}(B(o,a))\subseteq \mathbb{R}^{n+1}$ and it is
unique with respect to analytic functions.
\end{theorem}
\begin{proof}
Without restricting generality we assume $n=1$ and write $a_{jk}=a_{jk}^{1}$. The main argument of the proof uses the property that the coefficients $c_{lk}^{i}$
from the associated Taylor's series
\begin{equation}\label{eq:3}
u_{i}(t,x)=\mathop{\sum}\limits_{l,k=0}^{\infty}c_{lk}^{i}t^{l}x^{k}=\mathop{\sum}\limits_{l.k=0}^{\infty}\frac{t^lx^k}{l!k!}[\frac{\partial^{k+l}}{\partial t^{l}\partial x^{k}}u_{i}(t,x)]_{t=0,x=0}
\end{equation}
are uniquely determined by the evolution system\eqref{eq:2} and its analytic coefficients. Then using an upper bound for the coefficients $c_{lk}^{i}$, we prove that the series \eqref{eq:3} is convergent. In this respect we notice that
\begin{equation}\label{eq:4}
[\frac{\partial^{m}}{\partial_{x^{m}}}u_{i}(t,x)]_{t=0}=0\,\hbox{for any}\,m\geq 0
\end{equation}
(see $u_{o}(x)=0$) using \eqref{eq:2} we find the partial derivatives
\begin{equation}\label{eq:5}
\frac{\partial u_{i}}{\partial t}(o,x),\frac {\partial^{2}u_{i}}{\partial t\partial x}(o,x),\,\frac {\partial ^{2}u_{i}}{\partial t^2}(o,x),\,\frac {\partial^{3}u_{i}}{\partial t^{2}\partial x}(o,x),\,\frac {\partial ^{3}u_{i}}{\partial x^2}(o,x)
\end{equation}and so
 on which are entering in \eqref{eq:3}, for each $i\in\{i,...N\}$. If
\begin{equation}\label{eq:6}
a_{jk}(z)=\mathop{\sum}\limits_{\mid\alpha\mid=0}^{\infty}g_{\alpha}^{jk}z^{\alpha} \,\hbox{and}  \,bj\mid z\mid=\mathop{\sum}\limits_{\mid\alpha\mid=0}^{\infty}h_{x}^{j}z^{\alpha}
\end{equation}
for $\mid z\mid
\leq\rho\,\,,\alpha=(\alpha_{i},..\alpha_{N+1}),\mid\alpha\mid=\alpha_{i}+....\alpha_{N+1},$
and $\alpha_{i}\geq 0,\alpha_{i}$ integer, the
\begin{equation}\label{eq:7}
c_{lk}^{i}=P_{lk}^{i}[(g_{\alpha}^{jm})_{\alpha,j,m},(h_{\alpha}^{j})_{\alpha,j}]
\end{equation}
when $P_{lk}^{i}$ is a polynomial with positive coefficient($\geq 0$). It uses the coefficients $e_{lk}^{i}$(\eqref{eq:3}) and noticing that by derivation of a product we get
 only positive coefficient. Now we are constructing the upper bound coefficient $C_{lk}^{i}$ for $C_{lk}^{i}$ given in \eqref{eq:7} such that
\begin{equation}\label{eq:8}
\nu_{i}(t,x)=\mathop{\sum}\limits_{l,k=0}^{\infty}C_{lk}^{i}t^{l}x^{k},\,i=1,2,...,N
\end{equation}
is a local solution of the Cauchy problem.
\begin{equation}\label{eq:9}
\partial_{t}\nu{j}=\mathop{\sum}\limits_{k=1}{N}A_{jk}(z)\frac{\partial}{\partial
x}\nu(k)+B_{j}(z),\,\nu_{j}(0,x)=0,\,j\in\{1,...,N\}
\end{equation}
Here $A_{jk}$ and $B_{j}$ must be determined such that
\begin{equation}\label{eq:10}
A_{jk}(z)=\mathop{\sum}\limits_{\mid\alpha\mid=0}^{\infty}G_{\alpha}^{jk}z^{\alpha},\,B_{j}(z)=\mathop{\sum}\limits_{\mid\alpha\mid=0}^{\infty}H_{\alpha}^{j}z^{\alpha}
\end{equation}
where $\mid g_{\alpha}^{jk}\mid\leq G_{\alpha}^{jk},\,\mid h_{\alpha}^{j}\mid\leq H_{\alpha}^{j}$\\
Using \eqref{eq:10}, we get
\begin{equation}\label{eq:11}
C_{lk}^{i}=P_{lk}^{i}[(G_{\alpha}^{jm})_{\alpha,j,m},\mid
H_{\alpha}^{j}\mid_{\alpha,j}]\geq\mid P_{lk}^{i}[\mid
g_{\alpha}^{jm}\mid_{\alpha,j,m},(\mid
h_{\alpha}^{j})_{\alpha,j}]\mid\geq\mid c_{lk}^{i}\mid
\end{equation}and the upper bound coefficients $\{C_{lk}^{i}\}$ are found provided $A_{jk}$ and $B_{j}$ are determined such that the system\eqref{eq:9} is satisfied by the analytic
solution ${v(t,x)}$ given in \eqref{eq:10}. In this respect, we define the constants $G_{\alpha}^{jk}$ and $H_{\alpha}^{jk} $and$H_{\alpha}^{j}$ using  the following estimates
\begin{equation}\label{eq:12}
M_{1}=\max \mid a_{jk}(z)\mid,M=\max\mid
b_{j}(z)\mid j,k=1,...N\mid Z\mid\leq \rho
\end{equation}
we get
\begin{equation}\label{eq:13}
\left\{
       \begin{array}{ll}
        \mid g_{\alpha}^{jk}\mid\leq \frac{M_{1}}{\rho^{\mid\alpha\mid}}\leq \frac{M_{1}}{\rho^{\mid\alpha\mid}}\frac{\mid\alpha\mid!}{\alpha!}= & G_{\alpha}^{jk} \\
        \mid h_{\alpha}^{j}\mid\leq \frac{M_{2}}{\rho^{\mid\alpha\mid}}\leq \frac{M_{2}}{\rho^{\mid\alpha\mid}}\frac{\mid\alpha\mid!}{\alpha!}= & H_{\alpha}^{j}
       \end{array}
     \right.
\end{equation}
\begin{eqnarray}\label{eq:14}
  A_{jk}(z) &=& \mathop{\sum}\limits_{\mid\alpha\mid=0}^{\infty}M\frac{\mid\alpha\mid!}{\alpha!}(\frac{z}{\rho})^{\alpha}
  = M\mathop{\sum}\limits_{\mid\alpha\mid=0}^{\infty}\frac{1}{\rho^{\mid\alpha\mid}}\frac{\mid\alpha\mid!}{\alpha!}z^{\alpha} \nonumber\\
   &=& M[1-\frac{z_{1}+...+z_{N+1}}{\rho}]^{-1},\,\hbox{for} \mid z_{1}\mid+...+\mid z_{N+1}\mid <\rho
\end{eqnarray}
\begin{equation}\label{eq:15}
B_{j}(z)=\mathop{\sum}\limits_{\mid\alpha\mid=0}^{\infty}H_{\alpha}^{jk}z_{\alpha}=M[1-\frac{z_{1}+...+z_{N+1}}{\rho}]^{-1},
\,\hbox{for}\,\mid z_{1}\mid+...+\mid z_{N+1}\mid <\rho
\end{equation}
Here we have used
\begin{equation}\label{eq:16}
\mathop{\sum}\limits_{\mid\alpha\mid=k}\frac{k!}{\alpha!}Z^{\alpha}=(\mathop{\sum}\limits_{i=1}^{N+1}z_{i})^{k},\,\alpha!=(\alpha_{1}!)...(\alpha_{N+1}!),
z^{\alpha}=z_{1}^{\alpha_{1}}...Z_{N+1}^{\alpha_{N+1}}
\end{equation}
Notice that the coefficients $A_{jk},B_{j}$ do not depend on $(j,k)$ and each component $\nu_{i}(t,x)$ can be taken equals to $w(t,x)\,,i\in\{1,...,N\}$, where $\{w(t,x)\}$ satisfies the following scalar equation
\begin{equation}\label{eq:17}
\partial_{t}w=\frac{M\rho}{\rho-x-Nw}(1+N\partial_{x}w),\,w(0,x)=0\,
\hbox{for} \,\mid x\mid+N\mid w\mid<\rho
\end{equation}
by a direct inspection we notice that for $x\in\mathbb{R}^{n}$ we need to replace $x\in\mathbb{R}$ in \eqref{eq:17} by $x\rightarrow y=\mathop{\sum}\limits_{1}^{n}x_{i}$ where $x=(x_{1},...x_{n})$ and $w(t,x)\rightarrow w(t,\mathop{\sum}\limits_{1}^{n}x_{i})$
The solution of \eqref{eq:17} can be represented by
\begin{equation}\label{eq:18}
W(t,x)=\frac{1}{2N}[\rho-x-\sqrt{(\rho-x)^{2}-4MN\rho t}],\,if\,\mid x\mid<\sqrt{\rho}
\end{equation}
and
$$t<\frac{\rho}{16MN}=T(M,\rho)$$
In the case $x\in\mathbb{R}^{n}$ we use $y=(\mathop{\sum}\limits_{1}^{n}x_{i})$ instead of $x\in\mathbb{R}$ and
$w(t,y)$ satisfy both the equation \eqref{eq:17} and the representation formula\eqref{eq:18}.
\end{proof}
\subsection[1st Order Evolution System of Hyperbolic \& Elliptic Equations]{1st Order Evolution System of Hyperbolic \& Elliptic Equations\index{Elliptic equation}}
In the following we shall rewrite some hyperbolic and elliptic equation appearing in Mathematical Physics as an evolution system of first order equation. It suggest that assuming only analytic cofficients we may and do solve these equations applying the Cauchy-Kowalevska theorem (C-K)\index{Cauchy!Kowalevska theorem}\\
\textbf{$(E_{1})$. Klein-Gordon equations(D.B Klein 1894-1977, W.Gordon 1893-1939)}\\
They are describing the evolution of a wave function $y$ associated with a \textbf{vanishinq} \lq\lq spin\rq\rq particle and in differential form is expressed  using standard notations, as follows
\begin{equation}\label{.1}
\left\{
      \begin{array}{ll}
        \partial_{t}^{2}y =\Delta y-my+f(y,\partial_{t}y,\partial_{x}y),\,m>0,t\geq 0,y\in\mathbb{R} \\
        y(0,x)=y_{0}(x) =\partial_{t}y(0,x)=y_{1}(x),\,x\in\mathbb{R}^{n}
      \end{array}
    \right.
\end{equation}
where $\partial_{t}y=\frac{\partial y}{\partial t},\partial_{x}y=(\partial_{1}y,...,\partial_{n}y),\,\partial_{i}y=\frac{\partial y}{\partial x_{i}}$ and\\
the laplacian $\Delta y=\mathop{\sum}\limits_{1}^{n}\partial_{i}^{2}y$. Denote
$$u=(\partial_{x}y,\partial_{t}y,y)\in\mathbb{R}^{n+2},\,u^{0}(x)=(\partial_{x}y_0(x),y_{1}(x),y_{0}(x))$$
and\,$F(u)=f(y,\partial_{t}y,\partial_{x}y)$. Then \eqref{.1} is written as an evolution system
\begin{equation}\label{.2}
\left(
      \begin{array}{ll}
        \partial_{t}u_{j}= & \partial_{j}u_{n+1},\,j\in\{1,...,n\}\\
        \partial_{t}u_{n+1} & \mathop{\sum}\limits_{i=1}^{n}\partial_{i}u_{i}+F(u) \\
         \partial_{t}u_{n+2}=& u_{n+1} \\
        u(0,x)= & u^{0}(x)
      \end{array}
\right)
\end{equation}
\textbf{$(E_{2})$. Maxwell equations(J.C.Maxwell 1831-1879)}\index{Maxwell equation}\\
Denote the electric field $E(t,x)\in\mathbb{R}^{3},x\in\mathbb{R}^{3},t\in\mathbb{R}$, the magnetic field $H(t,x)\in\mathbb{R}^{3},x\in\mathbb{R}^{3},t\in\mathbb{R}$
and the Maxwell equations are the following\\
\begin{equation}\label{1}
\left(
      \begin{array}{ll}
        \frac{\partial \epsilon(E)}{\partial E} \partial_{t}E-\Delta\times H= & 0,E(0,x)=E^{0}(xs) \\
        \frac{ \partial \mu(H)}{ \partial H} \partial_{t} H+\Delta\times E=& 0,H(0,x)=H^{0}(x)
      \end{array}
\right)
\end{equation}
where $D=\epsilon(E)$ and $B=\mu(H)$ are some nonlinear mappings satisfying
\begin{equation}\label{2}
\epsilon(E),\mu(H):\mathbb{R}^{3}\rightarrow\mathbb{R}
\end{equation}
are analytic functions with uniformly positive definite matrices $ \frac{\partial \epsilon(E)}{\partial E} , \frac{ \partial \mu(H)}{ \partial H} $ for $E,H$ in bounded sets. In addition, we assume
\begin{eqnarray}
  \epsilon(E) &=& \epsilon_{0}E+O(\mid E\mid^{3}),\,\mu(H)\nonumber \\
  &=& \mu_{0}H+O(\mid H\mid^{3})\,with\, \epsilon_{0}>0,\mu_{0}>0
\end{eqnarray}
\begin{equation}\label{3}
\epsilon(E)=\epsilon_{0}E+O(\mid E\mid^{3}),\,\mu(H)=\mu_{0}H+O(\mid H\mid^{3})\,\hbox{with}\, \epsilon_{0}>0,\mu_{0}>0
\end{equation}
\begin{equation}\label{4}
\div D(t,x)=0,\div B(t,x)=0,D(0,x)=D^{0}(x),B(0,x)=B^{0}(x)
\end{equation}
where $D(t,x)=\epsilon(E(t,x))$ and $B(t,x)=\mu(H(t,x))$. The vectorial products $\Delta\times H$ and $\Delta\times E$ are given formally by the corresponding determinant
\begin{equation}\label{5}
\Delta\times E=det\left(
                     \begin{array}{ccc}
                       i & j & k \\
                       \partial_{1} & \partial_{2} & \partial_{3} \\
                       H_{1} & H_{2} & H_{3}\\
                     \end{array}
                   \right)
=i(\partial_{2}H_{3}-\partial_{3}H_{2})+j(\partial_{3}H_{1}-\partial_{1}H_{3})+k(\partial_{1}H_{2}-\partial_{2}H_{1})
\end{equation}
\begin{equation}\label{6}
\Delta\times E=det\left(
                     \begin{array}{ccc}
                       i & j & k \\
                       \partial_{1} & \partial_{2} & \partial_{3} \\
                       E_{1} & E_{2} & E_{3}\\
                     \end{array}
                   \right)
=i(\partial_{2}E_{3}-\partial_{3}E_{2})+j(\partial_{3}E_{1}-\partial_{1}E_{3})+k(\partial_{1}E_{2}-\partial_{2}E_{1})
\end{equation}
where $(i,j,k)\in\mathbb{R}^{3}$ is the canonical basis. Denote $u=(E,H)\in\mathbb{R}^{6}$ and the system \eqref{1} for which the assumptions \eqref{2}, \eqref{3} and  \eqref{4} are valid will be written as an evolution system
\begin{equation}\label{7}
\partial_{t}u=\mathop{\sum}\limits_{j=1}^{3}[A^{0}(u)]^{-1}A^{j}\partial_{j}u,\,u(0,x)=u^{0}(x)=(E^{0}(x),H^{0}(x))
\end{equation}
where $(6\times 6)$ matrices $A^{i}$ are given by
$$A^{0}(u)=\left(
            \begin{array}{cc}
              \frac{\partial\epsilon(E)}{\partial E} & O_{3} \\
              O_{3} & \frac{\partial\mu(H)}{\partial H} \\
            \end{array}
          \right)
,\,A^{1}=\left(
           \begin{array}{cccccc}
              &  &  & 0 & 0 & 0 \\
              & O_3 &  & 0 & 0 & -1 \\
              &  &  & 0 & 1 & 0 \\
             0 & 0 & 0 &  &  &  \\
             0 & 0 & 1 &  & O_3 &  \\
             0 & -1 & 0 &  & &  \\
           \end{array}
         \right)
$$
$$A^2=\left(
           \begin{array}{cccccc}
              &  &  & 0 & 0 & 1 \\
              & O_3 &  & 0 & 0 & -0 \\
              &  &  & -1 & 0 & 0 \\
             0 & 0 & -1 &  &  &  \\
             0 & 0 & 1 &  & O_3 &  \\
             1 & 0 & 0 &  & &  \\
           \end{array}
         \right),\,A^3=\left(
           \begin{array}{cccccc}
              &  &  & 0 & -1 & 0 \\
              & O_3 &  & 1 & 0 & 0 \\
              &  &  & 0 & 0 & 0 \\
             0 & 1 & 0 &  &  &  \\
             -1 & 0 & 1 &  & O_3 &  \\
             0 & 0 & 0 &  & &  \\
           \end{array}
         \right)$$
where $O_{3}$ is $(3\times 3)$ zero matrix, $A^{0}(u)$ is
strictly positive definite matrix and each $A^{i}$  is a symmetric
matrix. Recalling that Maxwell equation\index{Maxwell equation}
means the evolution system with analytic coefficients \eqref{7} and, in
addition the constraints \eqref{4}, we notice that \eqref{4} are satisfied, provide
\begin{equation}\label{8}
\div D^{0}(x)\,=\,0\,\, \hbox{and}\, \,\div B^0(x)=0
\end{equation}
In this respect, we rewrite the original system \eqref{1} as
\begin{equation}\label{9}
\partial_{t} D\Delta\times H=0,\ \partial_{t}
B\,+\,\Delta\times E0,\,D(0,x)\,=\,D^{0}(x),\,\,B(0,x)\,=B^{0}(x)
\end{equation}
and by a direct computation we get
\begin{equation}\label{10}
\left\{
       \begin{array}{ll}
        \partial_{t}[\div D(t,x)]=\div[\partial_{t} D](t,x)=\div(\Delta\times
H)(t,x)=0 \\
       \partial_{t}[\div B(t,x)]= div[\partial_{t}B](t,x)= \div(\Delta\times
E)(t,x)
       \end{array}
     \right.
\end{equation}
which show that \eqref{8} implies the constraints \eqref{4}
\section{$(E_3)$ Plate Equations}\index{Plate equation}
(they cannot be solved by (C-K) theorem). They are described by the
following equation
\begin{equation}\label{..1}
\partial_{t}^{2}y+\Delta ^{2}y=f(\partial_{t}y,\partial_{x} ^{2}
y)+\mathop{\sum}\limits_{i=1}^{n} b_{i}(\partial_{t} y,\partial_{x}
^{2} y)\partial_{i}(\partial_{t} y),\,\,x\in\mathbb{R}^{n}
\end{equation}
where $f$ and $b_i$ are analytic functions, and
$y(t,x)\in \mathbb{R}$ satisfies Cauchy conditions.
\begin{equation}\label{..2}
y(0,x)=y_{0}(x),\, \partial_{t }y(0,x)\,=\,y_{1}(x)
\end{equation}
Here \lq\lq$\Delta$\rq\rq is the standard laplacian operator, \lq\lq$\partial_{x}$\rq\rq is the gradient acting as linear mappings for which \lq\lq$\Delta^{2}$\rq\rq
and \lq\lq$\partial_{x}^{2}$\rq\rq have the usual meaning.
Denote
$$u=(\partial_{x}^{2}y,\partial_{x}(\partial_{t}y),\,\partial_{t}y)=(u_{1},\ldots,u_{n^{2}},u_{n^{2}+1},\ldots,u_{n^{2}+n},u_{n^{2}+n+1})$$
$$=(y_{11},y_{12},\ldots,y_{n1},\ldots,y_{1n},y_{2n},\ldots,y_{nn},u_{n^{2}+1},\ldots,u_{n^{2}+n},u_{n^{2}+n+1})\in
R^{n^{2}+n+1}$$

as the unknown vector function, where $(n\times n)$ matrix
$\partial_{x}^{2}y$ is denoted as a rowvector
$(y_{11},\ldots,y_{nn})\in R^{n^{2}}$. By a direct inspection we see
that
\begin{equation}\label{..3}
\left\{
        \begin{array}{ll}
        \partial_{t}y_{ij}= \partial_{i}\partial_{j}
u_{n}{^2+n+1}\,=   \partial_{i}u_{n^{2}+j},\,\,i,j\in\{1,\ldots,n\} \\
       \partial_{t}u_{n^2+n+1}\,=   -\mathop{\sum}\limits_{i=1}^{n}\triangle(y_{ii})+F(u) \\
           \partial_{t}\partial_{t}u_{n^2+j}=\partial_{j}(\partial_{t}u_{n^{2}+n+1})= -\mathop{\sum}\limits_{i=1}^{n}\partial_{j}\triangle(y_{ii})+\mathop{\sum}\limits_{k=1}^{n^{2}+n+1}\frac{\partial}{\partial
u_k}F(u)\partial_{j} u_{k} \\
         u(0,x) =(\partial_{x}^{2}y_{0}(x),\partial_{x},\,y_{1}(x),y_{1}(x))=u^{0}(x)
        \end{array}
      \right.
\end{equation}
where
$F(u)\,=\,f(u_{n^{2}+n+1},\partial_{x}^{2}y)+\mathop{\sum}\limits_{i=1}^{n}b_{i}(u_{n^{2}+n+1},\partial_{x}^{2}y)u_{n^{2}+i}.$
As far as the system \eqref{..3} contains laplacian \lq\lq$\Delta$\rq\rq in the
right hand side it cannot be asimilated with a first order evolution
system. Actually, it is well known that a parabolic
equation\index{Parabolic equation} has no rewriting as a first order
evolution system and show that the simplest plate
equations\index{Plate equation}.
\begin{equation}\label{..4}
\partial_{t}^{2}y+\Delta^{2}y\,=\,0\,\,\,y(0,x)\,=\,y_{0}(x),\,\,\partial_{t}y(0,x\,=\,y_{1}(x))
\end{equation} are equivalent with a parabolic type
equation (Schr\"{o}dinger)
\begin{equation}\label{..5}
\left\{
      \begin{array}{ll}
       \partial_{t}w =i\Delta
w (or\,-i\partial_{t}w(t,x)= \Delta_{x}w(t,x))\\
       w(o,x)=y_{1}(x)+i\Delta y_{0}(x)= w^0(x)
      \end{array}
    \right.
\end{equation}
\begin{remark}\label{re:2.22}Let $\{y(t,x)\}$ be the analytic solution of the plate equation \index{Plate equation}
$(4)$. Define the analytic function
$w(t,x)\,=\,\partial_{t}y(t,x)+i\Delta_{x}y(t,x) \in\mathbb{C} $.
Then $\{w(t,x):\,(t,x)\in D\subseteq\mathbb{R}^{n+1}\}$ satisfies
Schr\"{o}dinger $(S)$ equation \eqref{..5} and initial condition
$w(o,x)\,=\,y_{1}(x)+i\Delta y_{0}(x)\,=\,w^{0}(x).$ Conversely, if
$w(t,x)\,=\,y(t,x)+iz(t,x)$ satisfies \eqref{..5} then
$Re w(t,x)\,4=\,y(t,x)$ and $Im w(t,x)\,=\,z(t,x)$ are real analytic
solution for the homogeneous equation \index{Homogeneous!equation}of
the plate\eqref{..4} with Cauchy conditions $y(0,x)\,=\,y_0(x)$ and
$z(0,x)\,=\,z_0(x)$ correspondingly. Moreover, the solution of the
Schr\"{o}dinger equation \eqref{..5} agrees with the following
representation $w(t,x)\,=\,u(it,x)$ where
$u(\tau,x)\,=\,(4\pi\tau)^{\frac{n}{2}}\int_{\mathbb{R}^{n}}[\exp(-)]u_{0}(y)dy$
and we get $-i\partial_{t}w(t,x)\,=\,\partial_{\tau}
u(it,x)\,=\,\Delta_{x}u(it,x)\,=\,\Delta_{x}w(t,x)$ (S-equation).
\end{remark}
\subsection{Exercises}
\textbf{(1)} Using the algorithm for Klein-Gordon equation, write the evolution system corresponding to the following system of hperbolic equations\\
$$\partial_{t}^{2}y_i=\mathop{\sum}\limits_{m,j,k=1}^{n}C_{imjk}(\partial_xy)\partial_m\partial_k y_j,\,\,i=1,...,n\,,y=(y_1,...,y_n)$$\\
$$y_i(0,x)=y_{i}^{0}(x)\,,\,\partial_t y_i(0,x)=y_{i}^{1}(x)\,,\,i=1,...,n\,,x\in\mathbb{R}^n$$\\
(it appears in elasticity field)\\
\textbf{Hint}. $\partial_xy=(\partial_xy_1,...,\partial_xy_n)\,,\,\partial_xy_i=(\partial_1y_i,...,\partial_ny_i)\,,j\in\{1,...,n\}$\\
Denote
$u_{ik}y=\partial_ky_i(t,x)\,,\,k\in\{1,...,n\}\,,\,i\in\{1,...,n\}$
$$u_{n^2+i}=\partial_ty_i\,,\,i\in\{1,...,n\}$$
$$u_{n^2+n+i}=y_i\,,\,i\in\{1,...,n\}$$
we get\\
$$\left\{
    \begin{array}{ll}
     \partial_t\,u_{ik}=\partial_k\partial_t\,y_i=\partial_{x_k}\,u_{n^2+i}\,\,i\in\{1,...,n\}\\
     \partial_t\,u_{n^2+i}=\partial_{t}^{2}=\mathop{\sum}\limits_{m,j,k=1}^{n}C_{imjk}(u_{11},...,u_{nn})\partial_{x_m}(u_{kk})\,,\,i\in\{1,...,n\}\\
     \partial_tu_{n^2+n+i}=  \partial_t\,y_i=u_{n^2+i}\,,\,i\in\{1,...,n\}
    \end{array}
  \right.
$$
\textbf{(2)} Using (C-K) theorem solve the following elliptic
equation\index{Elliptic equation}
$$  \partial_{x}^{2}+  \partial_{y}^{2}=x^2+y^2\,,u(0,y)=  \partial_x\,u(0,y)=0$$
\textbf{Hint}. Denote $t=x,  \partial_yu=u_1,\,  \partial_xu=u_2\,,u=u_3$ and we get\\
$$\left\{
    \begin{array}{ll}
       \partial_tu_1=  \partial_y(  \partial_tu)=  \partial_yu_2\\
     u_1(0,y)=0
    \end{array}
  \right.
\,,\left\{
     \begin{array}{ll}
       \partial_tu_2=  \partial_{x}^{2}u=-  \partial_yu_1+t^2+y^2 \\
      u_2(0,y)=0
     \end{array}
   \right.
\,,\left\{
     \begin{array}{ll}
         \partial_tu_3=u_2 \\
   u_3(0,y)=0
     \end{array}
   \right.
$$\\
We are looking for $$u_3(t,y)=u_3(0,y)+\frac{t}{1!}  \partial_tu_3(0,y)+\frac{t^2}{2!}  \partial_{t}^{2}u_3(0,y)+...$$\\
Now  compute $$u_3(0,y)=0\,,  \partial_tu_3|_{t=0}=u_2|_{t=0}=0\,,\,  \partial_{t}^{2}u_3|_{t=0}=  \partial_tu_2|_{t=0}=y^2\,,\,  \partial_{t}^{k}u_3|_{t=0}$$
for $k\geq 3$. We obtain $u(x,y)=u_3(x,y)=\frac{1}{2}x^2y^2$\\
\textbf{(3)} by the same method solve
$$ \partial_{x}^{2}+  \partial_{y}^{2}=y^2\,,u(0,y)=  \partial_x\,u(0,y)=0$$
and get
$$u(x,y)=\frac{x^2y^2}{2}-\frac{x^4}{12}$$
\subsection{The Abstract Cauchy-Kawalewska Theorem \index{Cauchy!Kowalevska theorem}}
(L.Nirenberg, J.Diff.Geometry 6(1972)pp 561-576)\\
For $0<s<1$, let $X_{s}$ be the space of vectorial functions $\nu(x)\in\mathbb{C}^{n}$ which are holomorphic and bounded on
 $D_{s}=\prod_{j=1}^{\alpha}\{\mid x_{j}\mid<sR\}\subseteq\mathbb{C}^{d}$. Denote $\parallel\nu \parallel_{s}=\sup_{D_{s}}\mid\nu(x)\mid$. Then $X_{s}$ is a Banach space and the natural inspection $X_{s}\subseteq
X_{s'}(s'\leq s)$ has the norm $\parallel i\parallel\leq 1$. By the
standard  estimates of the derivatives (see Cauchy
formula\index{Cauchy!formula}), we get $\parallel
\partial_{j}\nu\parallel_{s'}\leq
R^{-1}\parallel\nu\parallel_{s}(s-s')$ for any $0<s'<s<1$ and for
a linear equation the following holds true
\begin{theorem}\label{th:2.7}
Let $A(t):X_{s}:\rightarrow X_{s'}$ be linear and continuous of $\mid t\mid<\eta$ for any $0<s'<1$ fulfilling
$$(i_{1})\parallel A(t)v\parallel_{s'}\leq C\parallel
v\parallel_{s-s'},\,\,\forall\,0<s'<s<1$$
Let $f(t)$\,be a continuous mapping of $\mid t\mid<\eta$ with values in $X_{s},\,0<s<1$ fulfilling
$$(i_{2})\parallel f(t)\parallel_{s}\leq \frac{K}{a(1-s)},\,\hbox{for}\,0<s<1$$
where $0<a<\frac{1}{8K}$ is fixed and $K>0$ is given constant. Then there exists a unique function $u(t)$ which is continuously differentiable of $\mid t\mid<a(1-s)$ with values in $X_{s}$ for each $0<s<1$, fulfilling
$$(C_{1})\frac{du}{dt}(t)=\,A(t)u(t),u(0)=0$$
$(C_{2}) \parallel u(t)\parallel_{s}\leq 2K(\frac{a(1-s)}{\mid t\mid}-1)^{-1}$ for any $\mid t\mid <a(1-s)$.
\end{theorem}

\begin{theorem}\label{th:2.8}
 The nonlinear case is refering to the following Cauchy problem
$$(\alpha) \frac{du}{dt}=F(u(t),t)\,\mid t\mid<\eta,\,\,u(0)=0$$
Let the condition $(\alpha_{1})$ and$(\alpha_{5})$  be
fulfilled. Then there exists $a>0$ and a unique function $u(t)$ which is continuously differentiable of $\mid t\mid<a(1-s)$ with values in $X_{s}$ fulfilling the equation
$(\alpha)$ and $\parallel u(t)\parallel_{s}<R\,\,\forall\,\mid
t\mid<a(1-s)$, for each $0<s<1$.
\end{theorem}
\section{Appendix Infinitesimal Invariance}
{J.R.Olver (Application 0f Lie algebra to Diff.Eq, Springer, 1986,
Graduate texts in mathematics; 107)}\\
\textbf{1}. One can replace the complicated, nonlinear conditions
for the invariance of a subset or function under a group of
transformations by an equivalent linear condition of infinitesimal
invariance under the corresponding infinitesimal generators of the
group action. It will provide the key to
 the explicit determination\index{Explicit! determination} of the symmetry groups of systems of differential equations. We begin with the simpler case of
 an invariant function under the flow generated by a vector field which can be expressed as follows
\begin{equation}\label{1.}
f(G(t;x))=f(x),\,t\in(-a,a)\,x\in D\subseteq \mathbb{R}^{n}\Longleftrightarrow
\end{equation}
\begin{equation}\label{2.}
g(f)(x)=<\partial_{x}f(x),g(x)>=0,\,\forall\,x\in D.
\end{equation}
where
$G(0;x)=x,\,\frac{dG(t;x)}{dt}=g(G(t;x)),\,t\in(-a,a),\,x\in D$
\begin{theorem}\label{th:2.9}
Let $G$\, be a group of transformation acting on a domain
$D\subseteq \mathbb{R}^{n}$. Let $F:D\rightarrow
\mathbb{R}^{m},\,m\leq n$, define a system of algebraic equations of
maximal rank
\begin{equation}\label{3.}
(F_{j}(x)=0,\,j=1,...,m)\,;rank\frac{\partial F}{\partial x)}x)=m,\,\forall \,x\in D,F(x)=0
\end{equation}
Then $G$  is a symmetry group of system $iff$
\begin{equation}\label{4.}
g(F_{j})(x)=0,\,\forall\,x\in\{y\in D:F(y)=0\},\,j=1,2,...,m
\end{equation}
where $g$\,is the infinitesimal generators of $G$.
\end{theorem}
\begin{proof}
The necessity of \eqref{4.} follows by differentiating the identity
$F(G(t;x))=0,\,t\in(-a,a)$ in which $x$ is solution of\eqref{3.} and
$G(t;x)\,,t\in(-a,a)$, is the flow generated by the vector field
$g$. To prove the sufficiency, let $x_{0}$ be a solution of the
system using the maximal rank condition we can choose
a coordinate transformation $y=(y^{1},...,y^{n})$ such that $x_{0}=0$ and $F$ has the simple form $F(y)=(y^{1},...,y^{m})$. Let $g(y)=(g^{1}(y),...,g^{n}(y))\in\mathbb{R}^{n}$ be any infinitesimal generators of $G$ expressed in the new coordinates and rewrite \eqref{4.} as follows
\begin{equation}g^{i}(y)=0\,\,,i=1,2,...,m,\,\forall\,y\in\mathbb{R}^{n},\,y^{1}=...=y^m=0
\end{equation} Now the flow\index{Flow!}
$G(t)(x_{0}),\,t\in(-a,a)$ generated by the vector field $g$ and
passing through $x_{0}=0$ satisfies the system of ordinary
differential equations\index{Ordinary differential equations}
$\frac{dG^{i}}{dt}(t)(x_{0})=g^{i}(G(t)(x_{0})),\,G(0)(x_{0})=0,\,G^{i}(t)(x_{0})=0,\,t\in(-a,a),\,i=1,...,m$ which
means $F(G(t)(x_{0}))=0,\,t\in(-a,a)$. The proof is complete.
\end{proof}
\begin{example}Let $G=SO(2)$ be the rotation group in the plane, with infinitesimal generator ($g=-y\partial_{x}+x\partial_{y}$)\\
$g(x,y)=col(-y,x)$. The unit circle $S_{1}=\{x^{2}+y^{2}=1\}$ is an
invariant subset of $SO(2)$ as it is the solution set of the
invariant function\,$f(x,y)=x^{2}+y^{2}-1$.
Indeed, $g(f)(x,y)=-2xy+2xy=0,\,\forall (x,y)\in\mathbb{R}^{2}$\,so
the equations (4) are satisfied on the unit circle itself. The
maximal rank condition does hold for $f$ since its gradient
$\partial f(x,y)=col(2x,2y)$ does not vanish on $S^{1}$. As a less
trivial example,consider the function
$f(x,y)=(x^{2}+1)(x^{2}+y^{2}-1)$ and notice that
$g(f)(x,y)=-2xy(x^{2}+1)^{-1}f(x,y)$ which shows that
$g(f)(x,y)=0$ whenever $f(x,y)=0$. In addition $\partial
f(x,y)=(4x_{3}+2xy^{2},2x^{2}y+2y)$ vanishes only when
$x=y=0$ which is not a solution to $f(x,y)=0$. We conclude that the
solution set $\{(x,y):(x^{2}+1)(x^{2}+y^{2}-1)=0\}$ is a
rotationally-invariant subset of $\mathbb{R}^{2}$.
\end{example}
\begin{remark}\label{re:2.23}
For a given one-parameter group of transformations $y=G(t)(x),\,x\in D\subseteq\mathbb{R}^{n},\,t\in(-a,a)$\,generated by the infinitesimal generator
$$\overrightarrow{g}=g^{1}(x)\partial_{x^{1}}+...+g^{n}(x)\partial_{x^{n}}(\hbox{see } \frac{dG(t)(x)}{dt}=g(G(t)(x)),\,t\in(-a,a),G(0)(x)=x$$
$g(x)=col(g^{1}(x),...,g^{n}(x))$. We notice that the corresponding invariant functions are determined
by the standard first integrals associated with the $ODE$
$\frac{dx}{dt}=g(x)$.
\end{remark}
\subsection{Groups and Differential Equations}
Suppose we are considering a system $S$ of differential equation
involving $p$ independent variables $x=(x^{1},...,x^{p})$ and
$q$ dependent variables$u=(u^{1},...,u^{q})$. The solution of the
system will be of the form
$u^{\alpha}=f^{\alpha}(x),\,\alpha=1,...,q$. Denote
$X=\mathbb{R}^{p},U=\mathbb{R}^{q}$ and a symmetry group of the
system $S$ will be a local group of transformations,$G$, acting on
some open subset $M\subseteq X\times U$ in such a way that"
$G$ transforms solutions of $S$ to other solutions of $S$".
To proceed rigorously, define the graph of $u=f(x)$,
$$\Gamma_{f}=\{(x,f(x)):x\in\Omega\}\subseteq X\times U$$
where $\Omega\subseteq X$ is the domain of definition of $f$. Note
that $\Gamma_{f}$ is a certain $p-$dimentional submanifold
of $X\times U$.
If $\Gamma_{f}\subseteq M_{G}$(domain of definition of the group transformations $G$) then the transform of $\Gamma_{f}$ by $G$ is just
$$G.\Gamma_{f}=\{(\hat{x},\hat{u})=G(x,u):(x,u)\in\Gamma_{f}\}$$
The set $G\Gamma_{f}$ is not necessarily the graph of another single valued function $\hat{u}=\hat{f}(\hat{x})$ but if it is the case then we write $\hat{f}=G.f$ and $\hat{f}$ the transform of $f$ by $G$.
\subsection{Prolongation}\index{Prolongation}
The infinitesimal methods for algebraic equations can be extended
for "systems of differential equations". To do this we need to
prolong the basic space $X\times U$ to a space which also
represents the various partial derivatives occurring in the
system.If $f:X\rightarrow U$ is a smooth
function,$u=f(x)=(f^{1}(x),...,f^{q}(x))$ there are $q.p_{k}$
numbers $u_{j}^{\alpha}=\partial_{j}f^{\alpha}(x)$ needed to
represent all different $k-th$ order derivatives of the components
of $f$ at point $x$. We let $U_{k}=\mathbb{R}^{q}.p_{k}$ be the
Euclidean space of this dimension,endowed with coordinates
$u_{j}^{\alpha}$ corresponding to $\alpha=1,...,q$, and all
multi-indices $J=(j_{1},...,j_{k})$ of order k, designed so as to
 represent $\partial_{J}f^{\alpha}(x)=\frac{\partial^{k}f^{\alpha}(x)}{\partial_{x}^{j_{1}}\partial_{x^{j_{k}}}}$.\\
Set $U^{n}=U\times U_{1}\times...\times U_{n}$ to be the product
space.whose coordinates represents all the derivatives of functions
$u=f(x)$ of all orders from 0 to n.
Note that $U^{n}$ is Euclidean space of dimension
$$q+qp_{1}+...+qp_{n}=qp^{n}$$
Atypical point in $U^{n}$ will be denoted by $u^{n}$ so $u^{n}$ has $q.p^{n}$ different components $u_{j}^{\alpha}$ where $\alpha=1,...,q$ has $J$ sums over all unordered multi-indices
$J=(j_{1},...,j_{k})$ with $1\leq j_{k}\leq p,j_{1}+...+j_{k}=k$ and $0\leq k\leq n$. (By convention for $k=0$ there is just one such multi-index,denoted by O, and $u_{0}^{\alpha}$\,
just replace to the component $u^{\alpha}$ of $u$ itself)
\begin{example}
 $p=2,q=1$.Then $X=\mathbb{R}^{2},\,(x^{1},x^{2})=(x,y)$ and $U=\mathbb{R}$ has the single coordinate $u$. The space $U_{1}$ isomorphic to $\mathbb{R}^{2}$ with coordinates
$(u_{x},u_{y})$ since these represents all the first order partial
derivatives of $u$ with respect to $x$ and $y$. Similarly
$U_{2}=\mathbb{R}^{3}$ has coordinates
$(u_{xx},u_{xy},u_{yy})$ representing the second order partial
derivatives of $u$, namely $\frac{\partial^{2}u}{\partial
x^{i}\partial y_{2-i}}\,\,,i=0,1,2$. In general, $U_{k}=\mathbb{R}^{k+1}$ , since there are (k+1) k-th order partial derivatives of $u$, namely\\
$\frac{\partial^{k}u}{\partial x^{i}\partial y_{k-i}}\,\,,i=0,1,..,k$. Finally, the space $U^{2}=U\times  U_{1}\times U_{2}=\mathbb{R}^{6}$, with coordinates $u^{2}=(u;u_{x},u_{y},u_{xx},u_{xy},u_{yy})$ represents all derivatives of $u$  with respect to$ x$ and $ y$ of order at most 2. Given a smooth function $u=f(x)\,f:X\rightarrow U$ there is an induced function $u^{n}=pr^{(n)}f(x)$ called the $n-th$ prolongation of $f$ which is defined by the equations
$$u_{j}^{\alpha}=\partial_{J}f^{\alpha}(x),\,J=(j_{1},...,j_{k}),\,1\leq j_{k}\leq p,\,j_{1}+...+j_{k}=k,\,0\leq k\leq n$$
for each $\alpha=1,\dots,q\,(see \,X=\mathbb{R}^{p},\,U=\mathbb{R}^{q})$.\\
The total space $X\times U^{n}$ whose coordinates represents the
independent variables,the dependent\index{Dependent!variables}
variables and the derivatives of the dependent variables up to order
$n$\, is called the $n-th$ order jet space of the underlying space
$X\times U$ (it   comes from viewing $pr^{(n)}f$ as a
corresponding polynomial degree $n$ associated with its Taylor
series ar the point $x$)If the differential equations are defined in
some open subset $M\subset X\times U$ then we define the n-jet
space $M^{(n)}=M\times U_{1}\times,\dot,\times U_{n}$ of $ M$.
\end{example}
\subsection{Systems of Differential Equations}
A system $S$\ of $n-th$ order differential equation in $p$ independent and $q$ dependent variables is given as a system of equations\\
\begin{equation}\label{z1}
\Delta_{\nu}(x,u^{(n)})=0, nu=1,\dots,b
\end{equation}
involving $x=(x^{1},\dots,x^{p}), u=(u^{1},\dots,u^{q})$ and the
derivatives of $u$ with respect to $x$ up to order $n$.
The function $\Delta(x,u^{n})=(\Delta_{1}(x,u^{n})),\dots,\Delta_{l}(x,u^{n})$, will be assumed to be smooth in their arguments so $\Delta$ can be viewed as a smooth map from the jet space $X\times U^{n}$ to some $l-$ dimensional
 Euclidean space $\Delta:X\times U^{n}\rightarrow \mathbb{R}^{l}$.\\
The differential equations themselves tell where the given map $\Delta$ variables on $X\times U^{n}$ and thus determine a subvariety\\
\begin{equation}\label{z2}
S_{\Delta}=\{(x,u^{n}):\Delta(x,u^{n})=0\}\subseteq X\times
U^{n}\end{equation}
0n the total jet space.\\
From this point of view, a smooth solution of the given system of differential equations is a smooth function $u=f(x)$ such that
$$\Delta_{\nu}(x,pr^{(n)}f(x))=0,\nu=1,...,l$$
whenever $x$ lies in the domain of $f$. This condition is equivalent to the statement that the graph of the prolongation $pr^{(n)}f(x)$ must lies entirely within the subvariety $S_{\Delta}$ determined by the system\\
\begin{equation}\label{z3}
\Gamma_{f}^{(n)}=\{(x,pr^{(n)}f(x))\}\subseteq S_{\Delta}=\{\Delta(x,u^{(n)})=0\}
\end{equation}
We can thus take an $n-th$ order system of differential equations to be a subvariety $S_{\Delta}$ in the $n=jet$ space $X\times U^{n}$ and a solution to be a function $u=f(x)$ such that the graph of the $n-th$ prolongation $pr^{(n)}f$ is contained in the subvariety $S_{\Delta}$.
\begin{example} Consider Laplace equation\index{Laplace equation} in the plane
\begin{equation}
u_{xx}+u_{yy}=0
\end{equation}
Here $p=2, q=1, n=2$ coordinates $(x,y,u,u_{x},u_{y},u_{xx},u_{xy},u_{yy})$ of $X\times U^{n}$ (a hyperplane)there, and this is the set $_{\Delta}$ for Laplac's equation.A solution must satisfy
$$\frac{\partial^{2}f}{\partial x^{2}}+\frac{\partial^{2}f}{\partial y^{2}}=o\,\,\,\forall(x,y)$$
This is clearly the same as requiring that the graph of the second prolongation $pr^{2}f$ lie in $S_{\Delta}$. For example, if $$f(x,y)=x^{3}-3xy^{2}$$then
$$pr^{2}f(x,y)=(x^{3}-3xy^{2};3x^{2}-3y^{2},-6xy;6x,-6y,-6x)$$
which lies in $$S_{\Delta}(\hbox{ see } 6x+(-6x)=0)$$
\end{example}
\subsection{Prolongation of Group Action and Vector Fields}\index{Prolongation}
We begin by considering a simple first order scalar differential equation
\begin{equation}\label{b1}
\frac{du}{dx}=F(x,u),\,\,\,(x,u)\in D\subseteq \mathbb{R}^{2}
\end{equation}
This condition is invariant under a one-parameter
group of transformations\\ $G(\epsilon)=\exp
\epsilon g$ defined on an open subset $D\subseteq \times U=\mathbb{R}^{2}$ where
\begin{equation}\label{b2}
g=\xi(x,u)\partial_{x}+\phi(x,u)\partial_{u}
\end{equation}
is the infinitesimal generator.\\
Let $u=u(x;x_{0},u_{0})$ be s solution of \eqref{b1} satisfying $u(x;x_{0},u_{0})=u_{0}$ and $(x_{0},u_{0})\in D$ arbitrarily fixed. Then $u(\hat{x_{0}}(\epsilon);x_{0},u_{0})=\hat{u_{0}}(\epsilon), \epsilon\in(-a,a)$\\
where $$(\hat{x_{0}}(\epsilon),\hat{u_{0}}(\epsilon))=G(\epsilon)(x_{0},u_{0})$$
It leads us to the following equations
\begin{equation}\label{b3}
\frac{du}{dx}(\hat{x_{0}}(\epsilon)).\frac{d\hat{x_{0}}}{d\epsilon}(\epsilon)=\frac{d\hat{u_{0}}}{d\epsilon}(\epsilon)\Longleftrightarrow F(x_{0},u_{0}).\xi(x_{0},u_{0})
=\phi(x_{0},u_{0})
\end{equation}
for any $(x_{0},u_{0})\in D\subseteq \mathbb{R}^{2}$.\\
On the other hand, the equation \eqref{b1} can be viewed as an algebraic constraint
\begin{equation}\label{b4}
F(x,u)-u_{x}=0
\end{equation}
on the variables $(x,u,u_{x})\in\mathbb{R}^{3}$ and we may do ask to find a prolonged and parameter group of transformations $G{(1)}(\epsilon)+\exp \epsilon g^{(1)}$ acting on a prolonger subvariety $M^{(1)}=X\times U\times U^{(1)}=\mathbb{R}^{3}$ such that the set solution \eqref{b4}is invarient. In this respect, notice that the new vector field $ g^{(1)}$ is a prolongation \index{Prolongation}of the vector field $g$   given in \eqref{b2}, $pr^{(1)}g=g^{(1)}$
\begin{equation}\label{b5}
g^{(q)}(x,u,u_{x})=\xi(x,u)\partial_{x}+\phi(x,u)\partial_{u}+\eta(x,u)\partial_{u_{x}}\hbox{where} (\xi(x,u),\phi(x,u),)
\end{equation}
satisfies \eqref{b3} $\forall \,(x,u)\in D\subseteq{R}^{2}$. The set solution\eqref{b4} is invariant under the infinitesimal generator \eqref{b5} $iff (f(x,u,u_{x})=F(x,u)-u_{x})$
\begin{equation}\label{b6}
(\partial_{x}f)\xi+(\partial_{u}f)\phi+(\partial_{u_{x}}f)\eta=0
\end{equation}
for any $(x,u,u_{x})$ verifying \eqref{b4}. An implicit computation of \eqref{b6} show us that $(\eta(x,u,u_{x}),\phi(x,u),\xi(x,u))$ must satisfy the following first order partial differential equation\index{Partial differential equations}
\begin{equation}\eta(x,u,u_{x})=\partial_{x}\phi(x,u)+[\partial_{u}\phi(x,u)-\partial_{x}\xi(x,u)]u_{x}-\partial_{u}\xi(x,u).u_{x}^{2}
\end{equation} Once we found a symmetry group $G$, the integration
of the equation \eqref{b1} may become a simplex one using elementary
operations like integration of a scalar function.
\subsection{ Higher Order Equation}
Consider a single $n-{th}$ order differential equation involving a single dependent variable $u$
\begin{equation}\label{g1}
\Delta(x,u^{(n)})=\Delta(x,u,\frac{du}{dx},\dots,\frac{d^{n}u}{dx^{n}}=0\, \,x\in \mathbb{R}
\end{equation}
If we assume that\eqref{g1} does not depend either of
$u$ or $x$ then the order of the equation can be reduced by one.
In this respect consider that $\Delta$ in \eqref{g1} satisfied
$\partial_{u}\Delta=0 i.e$ the vector field $g=\partial_{u}$ has
a trivial prolongation\index{Prolongation} $pr^{n}g=g=\partial_{u}$ generating a corresponding prolonged
group of symmetry $G^{(n)}$ for the equation
\begin{equation}\label{g2}
\widetilde{\Delta}(x,u_{1},...u_{n}=0)\,, \hbox{where} \,u_{i}=\frac{ d^{i}u}{dx^{i}},\,i\in\{1,\dots,n\}{1,...n}
\end{equation}
Denote $z=\frac{ du}{dx}$ and rewrite \eqref{g2} as
\begin{equation}\label{g3}
\hat{\Delta}(x,z,\frac{ dz}{dx},\dots,\frac{ d^{n-1}z}{dx})=\widetilde{\Delta}(x,z^{(n-1)})=0
\end{equation}
Whose solutions provide the general solution for \eqref{g2} and
$u(x)=\int_{0}^{x}h(y)dy +c$ is a solution for \eqref{g2} provided
$z=h(x)$ is a solution for \eqref{g3}. The second elementary group of
symmetry for \eqref{g1} is obtained assuming that $\Delta$ does not
depend on $x$ and write\eqref{g1} as
\begin{equation}\label{g4}
\widetilde{\Delta}(u^{(n)})=\widetilde{\Delta}(u,\frac{du}{dx},\dots,\frac{d^{n}u}{dx^{n}})=0
\end{equation}
This equation is clearly invariant under the group of transformations in the\\ $x-$direction,
with infinitesimal generator $g=\partial_{x}$. In order to change this into the vector field
$g=\partial_{\nu}$, corresponding to translations of the dependent variable,it suffices to reverse the  rules of dependent\index{Dependent!} and independent variable; we set $y=u,\,\nu=x$. Then compute $\frac{ du}{dy}=\frac{1}{\frac{ d\nu}{dy}}$, using $u(\nu(y))=y$.
Similarly, we get
$$\frac{ d^{2}u}{dx^{2}}=-\frac{\nu'' y}{(\nu'y)^{3}},\dots,\frac{ d^{n}u}{dx^{n}}=\delta_{n}(\nu_{y}^{(1)},\dots,\nu_{y}^{(n)})$$
and rewrite \eqref{g4} as follows
\begin{equation}\label{g5}
\widetilde{\Delta}(u^{(n)})=\widetilde{\Delta}(u,\frac{du}{dx},\dots,\frac{d^{n}u}{dx^{n}})=\hat{\Delta}(y,\nu_{y}^{(1)},\dots,\nu_{y}^{(n)})=\Delta
\end{equation}
where $\nu_{y}^{(k)}=\frac{d^{k}\nu}{dy^{k}}\,\,,k=1,\dots,n$.
The equation \eqref{g5} is transformed into a $(n-1)-th$ order
differential equation as above denoting $\nu_{y}^{(1)}=z$ as the
new unknown function.
\section*{Bibliographical Comments}
The first part till to Section 2.3 is written following the references \cite{8},\cite{12} and \cite{13}. Section 2.3 contained in a ASSMS preprint (2010). Sections 2.4
and 2.5 are written following the book in the reference \cite{11}. Sections 2.6 and 2.7 are using more or less the same presentation as in the reference \cite{12}.

\chapter[Second Order $PDE$]{Second Order Partial Differential Equations}\index{Partial differential equations}
\section{Introduction}
PDE of second order are written using symbols $\partial_t,
\partial_x, \partial_y, \partial_z, \partial_t^2, \partial_x^2, \partial_y^2,
\partial_z^2$ where $t\in \mathbb{R}$ is standing for the time variable, $(x,y,z)\in
R^3$ are space coordinates, and a symbol $\partial_s(\partial_s^2)$
represent first partial derivative with respect to $s\in \{t, x, y,
z\}$ (second partial derivative). There are three types of second
order PDE we are going to analyze here and they are illustrated by
the following examples
$$\left\{
  \begin{array}{ll}
\partial_x^2 u - \partial_y^2 u \,=&\, 0 \hbox{ \,\,\,\,\,\,(hyperbolic) }
\,\,\,\,(u(x,y)\in\mathbb{R}, (x,y)\in \mathbb{R}^2) \\
\triangle u\,=\,\partial_x^2 u +
\partial_y^2 u \,=&\, 0 \hbox{ \,\,\,\,\,\,(elliptic) } \,\,\,\,\,\,\,\,\,\,\,\,\,(u(x,y)\in \mathbb{R},
(x,y)\in \mathbb{R}^2)\\
\partial_t u\,=\,\partial_x^2 u + \partial_y^2 u \,=&\,\triangle u
\hbox{ (parabolic) } \,\,\,\,\,\,\,(u(t,x,y)\in R, t\in R,(x,y)\in \mathbb{R}^2)
 \end{array}
\right.$$
PDE of second order have a long tradition and we recall the
Laplace equation\index{Laplace equation} (Pierre Simon Laplace 1749-1827)
\begin{equation}\label{3:1.1}
\partial_x^2 u + \partial_y^2 u + \partial_z^2
u\,=\,\triangle u\,=\,0 \hbox{ (elliptic, linear, homogeneous) }
\end{equation}
and its nonhomogeneous version\index{Homogeneous!version}
\begin{equation}\label{3:1.2}
\triangle u\,=\,\partial_x^2 u + \partial_y^2 u +
\partial_t^2 u\,=\,f(x,y,z) \hbox{ (Poisson\, equation) }
\end{equation}
originate in the Newton universal attraction law (Isaac Newton 1642-1727).
Intuitively, it can be explained as follows. An attractive body
induces a field of attraction where intensity at each point
$(x,y,z)\in R^3$ is calculated using Newton's formula
$$u\,=\,\gamma \frac{\mu}{\sqrt{(x-x_0)^2+(y-y_0)^2+(z-z_0)^2}}$$
where $\gamma$ is a constant, $\mu=$ mass of the body, considering
that the attractive body is reduced to the point $(x_0,y_0,z_0)\in
R^3.$
In the case of several attractive bodies which are placed at the
points $(x_i,\,y_i,\,z_i),\,\,i\in\{1,\ldots,N\}$, we compute the
corresponding potential function
$$u\,=\,\gamma \sum_{i=1}^N \frac{\mu_i}{\sqrt{(x-x_i)^2+(y-y_i)^2
+(z-z_i)^2}}\,=\,\gamma \sum_{i=1}^N\frac{\mu_i}{r(P,P_i)}$$ where
$$P\,=\,(x,y,z), P_i\,=\,(x_i,y_i,z_i)$$ and
$$r(P,P_i)\,=\,\sqrt{(x-x_i)^2+(y-y_i)^2
+(z-z_i)^2}.$$ It was Laplace who proposed to study the
corresponding PDE satisfied by the potential function $u$ and in
this respect denote $u_i\,=\,\gamma \frac{\mu_i}{r(P,P_i)}$ and
compute its partial derivatives. We notice that $\partial_x
r\,=\,\frac{x-x_i}{r}$, $\partial_y
r\,=\,\frac{y-y_i}{r}\,=\,\partial_z r\,=\,\frac{z-z_i}{r}$ and
\begin{equation}\label{3:1.3}
\partial_x
u_i\,=\,-\gamma\mu_i\frac{(x-x_i)}{r^3},\,\partial_y
u_i\,=\,-\gamma\mu_i\frac{(y-y_i)}{r^3},\,\partial_z
u_i\,=\,-\gamma\mu_i\frac{(z-z_i)}{r^3}
\end{equation}
Using \eqref{3:1.3} we see easily that
\begin{equation}\label{3:1.4}
\left\{
  \begin{array}{ll}
\partial_x^2
u_i\,=\,\gamma\mu_i[-\frac{1}{r^3}\,+\,3\frac{{x-x_i}^2}{r^5}]\\
\partial_y^2
u_i\,=\,\gamma\mu_i[-\frac{1}{r^3}\,+\,3\frac{{y-y_i}^2}{r^5}]\\
\partial_z^2
u_i\,=\,\gamma\mu_i[-\frac{1}{r^3}\,+\,3\frac{{z-z_i}^2}{r^5}]
 \end{array}
\right.
\end{equation}
and by adding we obtain\\
\begin{equation}\label{3:1.5}
\triangle u_i\,=\,\partial_x^2 u_i + \partial_y^2 u_i + \partial_z^2
u_i\,=\,0,\,\,\,\,i=1,2,\ldots,N
\end{equation}
which implies\\
\begin{equation}\label{3:1.6}(u=\sum_{i=1}^N u_i)
\triangle u\,=\,\partial_x^2 u + \partial_y^2 u + \partial_z^2
u\,=\,0 \hbox{ \textbf{(Laplace Equation)} }\index{Laplace equation}
\end{equation}
\section{Poisson Equation}\index{Poisson equation}
It may occur that we need to consider a body with a mass distributed
in a volume having the density $\rho\,=\,\rho(a, b, c)$ at the point
$x=a, y=b, z=c$ and  vanishing outside of the ball
$a^2+b^2+\mathcal{C}^2\leq R^2.$ In this case the potential function will be
computed as follows\\
\begin{equation}\label{3:1.7}
u(P)=\mathop{\iiint}\limits_{a^2+b^2+\mathcal{C}^2\leq R^2}\frac{\rho(a, b, c)\,da\,
db\, dc}{r(P,P(a, b, c))}
\end{equation}
where $P=(x, y, z),\,P(a, b, c)=(a, b, c)$  and the constant $\gamma$
is included in the function $\rho.$ By a direct computation we will
prove that if $\rho(a, b, c),\,(a, b, c)\in B(0,R)$, is first order
continuously differentiable then the potential defined in \eqref{3:1.7}
satisfies Poisson equation.\index{Poisson equation}
\begin{equation}\label{3:1.8}
\triangle u \,=\,\partial_x^2 u + \partial_y^2 u + \partial_z^2
u\,=\,-4\pi\rho(x, y, z),\,(x, y, z)\in B(0,R)
\end{equation}
and\\
\begin{equation}\label{3:1.9}
\triangle u\,=\,0,\,\,(x, y, z)\not\in
B(0,R) \hbox{ (Laplace Equation) }
\end{equation}
We recall that a similar law of interaction between electrical
particles is valid and Coulomb law is described by
\begin{equation}\label{3:1.10}
\mu\,=\,\gamma\frac{m_1m_2}{r^2},\,\,(m_1,m_2)- \hbox{ electric charges of }(P_1,P_2)
\end{equation}
$\gamma\,=\,\frac{1}{\epsilon}$ a constant, $$r^2\,=\,(x_1-x_2)^2+(y_1-y_2)^2 +(z_1-z_2)^2\,P_1=(x_1, y_1,
z_1)\\P_2=(x_2, y_2, z_2).$$ The associated electrostatic field
$(E_x,E_y,E_z)$ is defined by $$E_x\,=\,\partial_x
u,\,E_y\,=\,\partial_y u,\,E_z\,=\,\partial_zu$$ and in this case we
get the following Poisson equation\index{Poisson equation}
\begin{equation}\label{3:1.11}
\partial_xE_x+\partial_yE_y+\partial_zE_z\,=\,\frac{4\,\pi\,\rho}{\epsilon}
\end{equation}
when a density $\rho$ is used and the electrical potential function
has the corresponding integral form.\\
\underline{
Proof of the equation \eqref{3:1.8} for $u$  defined in \eqref{3:1.7}}. Rewrite the potential function on the whole space
\begin{equation}\label{3:1.12}
u(x, y, z)\,=\,\iiint \frac{\rho(a, b,
c)}{\sqrt{(x-a)^2+(y-b)^2+(z-c)^2}}\,da\,db\,dc
\end{equation}
and
making a translation of coordinates
$a-x\,=\,\xi,\,b-y\,=\,\eta,\,c-z\,=\,\tau$, we get
\begin{equation}\label{3:1.13}
u(x, y,
z)\,=\,\iiint \frac{\rho(x+\xi,\,y+\eta,\,z+\tau)}{\sqrt{\xi^2+\eta^2+\tau^2}}\,d\xi\,d\eta\,d\tau
\end{equation}
where the integral is singular and $\xi=\eta=\tau=0$ is the singular
point. The integral in \eqref{3:1.13} is uniformly convergent with respect
to the parameters $(x,y,z)$ because the function under integral has
an integrable upper bound
$\frac{\rho^*}{\sqrt{\xi^2+\eta^2+\tau^2}},$ where $\rho^*=
\max|\rho|$. In this respect, we notice that if $(x,y,z)\in
B(0,R)\subseteq R^3$ then $B(0,2R)$ can be taken as a domain $D$
where the integration of \eqref{3:1.13} is performed. (see
$\xi^2+\eta^2+\tau^2\leq (2R)^2$). Taking a standard coordinates
transformation (spherical coordinates) $$\xi=r\cos\varphi\sin\psi,\,\eta=r\sin\varphi\sin\psi,\,\sigma=r\cos\psi,\,0\leq\varphi\leq2\pi,\,0\leq\psi\leq\pi,\,0\leq
r\leq2R$$
we rewrite \eqref{3:1.13} on $D=B(0,2R)$ as follows
\begin{equation}\label{3:1.14}
\mathop{\iiint}\limits_{D}\frac{\rho^{*}d\xi d\eta d\sigma}{\sqrt{\xi^{2}+\eta^2+\sigma^2}}=4\pi \rho^*\int_{0}^{2R}\frac{r^2dr}{r}=4\pi\rho^*\frac{(2R)^2}{2}
\end{equation}
where
$$\mathop{\iint}\limits_{D}\frac{1}{r}d \xi d \eta d\sigma=(\int_{0}^{2R}\frac{r^2}{r}dr(\int_{0}^{\pi}sin\psi d\psi).2\pi=4\pi\int_{0}^{2R}r dr$$
and  $$det\left(
           \begin{array}{c}
             \partial_{\nu}\xi \\
             \partial_{\nu}\eta \\
             \partial_{\nu}\sigma \\
           \end{array}
         \right)
=-r^{2}sin\psi (\nu=(r,\varphi, \psi)) \hbox{ are used }.$$
In addition by formal derivation of \eqref{3:1.13} with respect to$(x,y,z)$
we get the following  uniformly convergent integrals
\begin{eqnarray}\label{3:1.15}
  \partial_{x}u &=& \iiint\frac{\partial_{x}[\rho(x+\xi,y+\ eta,z+\sigma)]} {\sqrt{\xi^{2}+\eta^2+\sigma^2}} d\xi d\eta d\sigma \nonumber\\
   &==& \mathop{\iiint}\limits_{D}\frac{\partial_{\xi}[\rho(x+\xi,y+\ eta,z+\sigma)]} {\sqrt{\xi^{2}+\eta^2+\sigma^2}} d\xi d\eta d\sigma
\end{eqnarray}
Using $\rho(x+\xi ,y+\eta,z+\sigma)=0$ on the sphere  $\xi^2+\eta^2+\sigma^2=(2R)^2$  and \\
{\small
$$\frac{\partial_{\xi}[\rho(x+\xi,y+\eta,z+\sigma)]}{\sqrt{\xi^2+\eta^2+\sigma^2}}=\partial_{\xi}
[(\rho(x+\xi,y+\eta,z+\sigma))]
{\sqrt{\xi^2+\eta^2+\sigma^2}}+\frac{\xi \rho(x+\xi+\eta,z+\sigma)}{(\xi^2+\eta^2+\sigma^2)^{3\setminus 2}}$$}
we get
\begin{equation}\label{3:1.16}
\partial_{x}u=\iiint_{d}\frac{\xi \rho(x+\xi+\eta,z+\sigma)}{(\xi^2+\eta^2+\sigma^2)^{3\setminus 2}} d\xi d\eta
\end{equation}
provided $$0=\mathop{\iint}\limits_{D}[\xi\phi(\xi,\eta,\sigma)]d\xi d \eta d\sigma=\mathop{\iint}\limits_{D1}[\phi(\xi_{2},\eta_{1},\sigma)-\phi(\xi_{1},\eta_{1},\sigma)]d\eta d\sigma$$
is used  where
$$\xi_{2}=+\sqrt{(2R^2)-\eta^2-\sigma^2}\,,\,\xi_1=-\sqrt{(2R^2)-\eta^2-\sigma^2}$$ and $\phi(\xi_{i},\eta,\sigma)=0 \,\,i\in\{1,2\}$.
The integral \eqref{3:1.16}is uniformly convergent with respect to$(x,y,z)\in D(0,R)$
and noticing $\frac{\xi}{\sqrt{\xi^2+\eta^2=\sigma^2}}\leq1$, we get the following integrable upper bound\\
\begin{equation}\label{3:1.17}
|\partial_x\,u|\leq \mathop{\iiint}\limits_{D}\frac{\rho^{*}d\xi d\eta d\sigma}{\sqrt{\xi^{2}+\eta^2+\sigma^2}}=4\pi \rho^*\int_{0}^{2R}\frac{r^2dr}{r^2}=4\pi\rho^2(2R)
\end{equation}
which proves that \eqref{3:1.16} is  valid. Similar  arguments are used to show that   $\partial_{y} u$ and $\partial_{z} u$ exist fulfilling \\
\begin{equation}\label{3:1.18}
\left\{
       \begin{array}{ll}
       \partial_{y} u=\mathop{\iint}\limits_{D}\frac{\eta \rho(x+\xi_{1}y+\eta_{1}z+\sigma)}{(\xi^2+\eta^2+\sigma^2)^{3\setminus2}}d\xi d\eta d\sigma \\
       \partial_{z}u =\mathop{\iint}\limits_{D}\frac{\sigma \rho(x+\xi_{1}y+\eta_{1}z+\sigma)}{(\xi^2+\eta^2+\sigma^2)^{3\setminus2}}d\xi d\eta d\sigma
     \end{array}
\right.
\end{equation}
Applying $\partial_{x}$  to $ (\partial_{x}u)$  in \eqref{3:1.16},$\partial_{y}$  to $\partial_{y}u$  and $\partial_{z}$  to $ (\partial_{z} u)$ in \eqref{3:1.18} we get convergent integrals
\begin{equation}\label{3:1.19}
\left\{
  \begin{array}{ll}
    \partial_{x}^{2} u=\mathop{\iiint}\limits_{D}\frac{\xi}{(\xi^2+\eta^2+\sigma^2)^{3\setminus2}} \partial_{\xi}[\rho(x,\xi,y,\eta,z+\sigma)]d\xi d\eta d\sigma \\
    \partial_{y}^{2} u= \mathop{\iiint}\limits_{D}\frac{\eta}{(\xi^2+\eta^2+\sigma^2)^{3\setminus2}} \partial_{\eta}[\rho(x,\xi,y,\eta,z+\sigma)]d\xi d\eta d\sigma\\
   \partial_{z}^{2} u =\mathop{\iiint}\limits_{D}\frac{\sigma}{(\xi^2+\eta^2+\sigma^2)^{3\setminus2}} \partial_{\sigma}[\rho(x,\xi,y,\eta,z+\sigma)]d\xi d\eta d\sigma
   \end{array}
\right.
\end{equation}
where $\partial_{s}[\rho(x+\xi,y+\eta,z+\sigma)](s\in{\xi, \eta,\sigma})$  is a continuous and bounded function on D(see $\rho$  is first order continuously differentiable on $D$)
using \eqref{3:1.19} we get the expression of the laplacian
\begin{equation}\label{3:1.20}
\left\{
  \begin{array}{ll}
   \Delta u= \partial_{x}^{2}u+ \partial_{y}^{2}u+ \partial_{z}^{2}u\\
   =\mathop{\iiint}\limits_{D}\frac{\xi\partial_\xi+\eta\partial_\eta+\sigma\partial_\sigma[\rho(x+\xi,y+\eta,z+\sigma)]}{(\xi^2+\eta^2+\sigma^2)^{\frac{}3}{2}}d\xi d\eta d\sigma  \\
  =\int_{0}^{2R}\{\mathop{\iint}\limits_{S_{r}}\frac{\partial_r[\rho(x+\xi,y+\eta,z+\sigma)]}{r^2}dS_r\}dr
  \end{array}
\right.
\end{equation}
where $S_{r}$ is a sphere with radius $$ r=(\xi^2+\eta^2+\sigma^2)^{\frac{1}{2}}$$ and $$\partial_r\rho=<(\partial_\xi\rho,\partial_\eta\rho,\partial_\sigma\rho),(\frac{\xi}{r},\frac{\eta}{r},\frac{\sigma}{r})>$$
 Here each $(\xi,\eta,\sigma)\in S_r$ can be represented as $$(\xi,\eta,\sigma)=(r\xi_0,r\eta_0,r\sigma_0)$$ where $$(\xi_0,\eta_0,\sigma_0)\in S_1$$ and using $d\,S_{r}=r^2d\,S_1$ we rewrite \eqref{3:1.20} as follows
\begin{equation}\label{3:1.21}
\Delta u
\left\{
  \begin{array}{ll}
   = \int_{0}^{2R}\{\mathop{\iint}\limits_{S_{1}}\partial_r[\rho(x+\xi_0r,y+\eta_0r,z+\sigma_0r)]dr\}dS_{1}\\
   =\mathop{\iint}\limits_{S_{1}}[\rho(x+2R\xi_0,y+2R\eta_0,z+2R\sigma_0)-\rho(x,y,z)]dS_{1} \\
    -\mathop{\iint}\limits_{S_{1}}\rho(x,y,z)dS_{1}=-4\pi\rho(x,y,z)
  \end{array}
\right.
\end{equation}
It shows that the equality \eqref{3:1.8} (Poisson equation)\index{Poisson equation} is for any $(x,y,z)\in B(0,R)$ where $ R>0$  was arbitrarily fixed and the proof is complete.\,\,\,\,\,\,\,$\Box$\\
We conclude this introduction by showing that the potential function $u(x,y,z)$ defined in \eqref{3:1.7} satisfied the following asymptotic behaviour
{\small
\begin{equation}\label{3:1.22}
\mathop{lim}\limits_{r\rightarrow \infty}u(x,y,z)=0, \hbox{ or }\mathop{lim}\limits_{r\rightarrow\infty}\sqrt{x^2+y^2+z^2}u(x,y,z)=\mathop{\iint\int}\limits_{D}\rho(a,b,c)da\,db\,dc
\end{equation}}
In this respect ,denote $Q=(x,y,z)\,,P=(a,b,c),\,dv=da\,db\,dc$ and  rewrite equation \eqref{3:1.22} as follows
$$u(Q)=\mathop{\iiint}\limits_{D}\frac{\rho(P)dv}{r(P,Q)}\,\,\,(D:r(P,Q)\leq R)$$
{\small
\begin{equation}\label{3:1.23}
\mathop{lim}\limits_{r(0,Q)\rightarrow \infty}r(0,Q)u(Q)=\mathop{lim}\limits_{r(0,Q)\rightarrow \infty} \mathop{\iiint}\limits_{D}\frac{\rho(P)dv}{[1-\frac{r(0,Q)-r(P,Q)}{r(0,Q)}]}=\mathop{\iiint}\limits_{D}\rho(P)dv
\end{equation}}
Here we notice that $a=\frac{r(0,Q)-r(P,Q)}{r(0,Q)}$ satisfies $\mid a\mid<1\,,\frac{1}(1-a)=1+a+a^2+,\dots$ and $\mid r(0,Q)-r(P,Q)\mid\leq r(0,P)$ leads us to \eqref{3:1.23} and \eqref{3:1.22} are valid. We conclude the above given considerations by
\begin{theorem}\label{th:3.1.1}
Let $\rho(x,y,z):\mathbb{R}^{3}\rightarrow\mathbb{R}$ be a continuously differentiable function in a open neighborhood $V(x_0,y_0,z_0)\subseteq\mathbb{R}^{3}$ and vanishing outside of the fixed ball $B(0,L)\subseteq\mathbb{R}^{3}$. Then $$u(x,y,z)=\mathop{\iiint}\limits_{a^2+b^2+\mathcal{C}^2\leq L^2}\frac{\rho(a,b,c)da\,db\,dc}{r(P,P(a,b,c))}\,\,,P=(x,y,z)\in\mathbb{R}^{3}$$
satisfies the following Poisson equations
\index{Poisson equation}
\begin{equation}\label{3:1.24}
\Delta u(x,y,z)=(\partial_{x}^2+\partial_{y}^2+\partial_{z}^2)u(x,y,z)=-4\pi\rho(x,y,z)
\end{equation}
for any $(x,y,z)\in V(x_0,y_0,z_0)
, \Delta u(x,y,z)=0\,\,\forall\, (x,y,z)\,not \,\,in \, \,B(0,L)$. In addition
\begin{equation}\label{3:1.25}
\mathop{lim}\limits_{r\rightarrow \infty}u(x,y,z)=0\,(or\,\mathop{lim}\limits_{r\rightarrow \infty}ru(x,y,z))=\mathop{\iiint}\limits_{B(0,L)}\rho(a,b,c)da\,db\,dc
\end{equation}
\end{theorem}
\begin{remark}
The result in Theorem \ref{th:3.1.1} holds true when replacing $\mathbb{R}^3$ with  $\mathbb{R}^n(n\geq 3)$ and if it is the case then the corresponding newtonian potential function is given by
$$u(x)=\mathop{\int\dots\int}\limits_{n\,times}\frac{\rho(y)dy_1\,\dots\,dy_n}{r^{n-2}}$$
where $\rho(y):\mathbb{R}^n\rightarrow \mathbb{R}$ is a continuously differentiable function vanishing outside of the ball $B(0,L)\subseteq\mathbb{R}^n$ and $r=\mid x-y\mid=(\mathop{\sum}\limits_{i=1}^{n}(x_i-y_i)^2)^{\frac{1}{2}}$. The corresponding Poisson equation \index{Poisson equation}is given by
$$\mathop{\sum}\limits_{i=1}^{n}\partial_{i}^2u(x)=\Delta u(x)=\rho(x)(-\sigma_1)\,\,\forall\,x\in\mathbb{R}^n, \hbox{ where }\sigma_1=meas S(0,1)$$
\end{remark}
\section{Exercises}
Using the same algorithm as in Theorem \ref{th:3.1.1} prove that the corresponding potential function(logarithm)in $\mathbb{R}^2$ is given by\\
$$u(x,y)=\iint\rho(a,b)\{ln\frac{1}{((x-a)^2+(y-b)^2)^{\frac{1}{2}}}\}da\,db$$
and satisfies the following Poisson equation \index{Poisson equation}
$$\Delta u(x,y)=(\partial_{x}^2+\partial_{y}^2)u(x,y)=-2\pi\rho(x,y),\,\,\forall\,(x,y)\in V(x_0,y_0)$$
Here $\rho:\mathbb{R}^2\rightarrow\mathbb{R}$ is first order continuously differentiable in the open neighborhood $V(x_0,y_0)\subseteq\mathbb{R}^2$ and vanishes out side of the disk $B(0,L)\subseteq\mathbb{R}^2$. By a direct computation show that the following Lapace equation
$$\Delta u(x)=\mathop{\sum}\limits_{i=1}^n\partial_{i}^2u(x)=0\,\,\,\forall\,x\in\mathbb{R}^n,\,(n\geq 3),\,x\neq 0$$is valid, where $$u(x)=\frac{1}{r^{n-2}},r=(\mathop{\sum}\limits_{i=1}^nx_{i}^2)^{\frac{1}{2}},\,x=(x_{1},...x_{n}) \partial_{i}^2 u (x)=\frac{\partial^2u(x)}{\partial x_{i}^2}$$

\section{Maximum Principle for Harmonic Functions\index{Harmonic functions }}
Any solution of the Lapace equation $\Delta u(x,y,z)=0$ will be called harmonic function\\
\textbf{Maximum principle}\\
A harmonic function $u(x,y,z)$  which is continuous in a bounded closed domain $\overline{G}=G\sqcup \Gamma$ and admitting second order continuous partial derivatives in the open set
$G\subseteq\mathbb{R}^3$ satisfy
$$\mathop{\max}\limits_{(x,y,z)\in \overline{G}}u(x,y,z)=\mathop{\max}\limits_{(x,y,z)\in\Gamma}u(x,y,z)\,and\,\mathop{\min}\limits_{(x,y,z)\in \overline{G}}u(x,y,z)=\mathop{\min}\limits_{(x,y,z)\in\Gamma}u(x,y,z)$$
where  $\Gamma=\partial G$ ( boundary of $G$ )
\begin{proof}
Denote $m=\max\{u(x,y,z):(x,y,z)\in\Gamma\}$ and assume that $\max\{u(x,y,z):(x,y,z)\in \overline{G}\}=M=u(x_0,y_0,z_0)>m$, where $(x_0,y_0,z_0)\in G$. Define the auxiliary function
$$\nu(x,y,z)=u(x,y,z)+\frac{M-m}{2d^2}[(x-x_0)^2+(y-y_0)^2+(z-z_0)^2]$$ and $$d=\max\{r(P,Q):P,Q\in\overline{G}\} \hbox{ and }r(P,Q)=\mid P-Q\mid$$
(distance between two points).
By definition,$$r^2(P,P_0)=(x-x_0)^2+(y-y_0)^2+(z-z_0)^2,(P=(x,y,z),P_0=(x_0,y_0,z_0))$$ and using $r^2(P,P_0)\leq d^2$\\
we get
$$\nu(x,y,z)\leq m+\frac{M-m}{2}=\frac{M+m}{2},\,\forall\,(x,y,z)\in\Gamma$$
On the other hand, $\nu(x_0,y_0,z_0)=u(x_0,y_0,z_0)=M$ and it implies $\max\{\nu(x,y,z):(x,y,z)\in\overline{G}\}=V(\overline{P})$ is achieved in the open set $\overline{P}\in G$. As a consequence $$\partial_x\nu(\overline{x},\overline{y},\overline{z})=0,\,\partial_y\nu(\overline{x},\overline{y},\overline{z})=0,
\,\partial_z\nu(\overline{x},\overline{y},\overline{z}),\,\partial_{x}^2\nu(\overline{x},\overline{y},\overline{z})\leq 0$$
$$\partial_{y}^2\nu(\overline{x},\overline{y},\overline{z})\leq 0\hbox{ and }\partial_{z}^2\nu(\overline{x},\overline{y},\overline{z})\leq 0$$
where $\overline{P}=(\overline{x},\overline{y},\overline{z})$. In addition
{\small
$$\Delta\nu(\overline{P})\leq 0\hbox{ and }\Delta\nu(\overline{P})=\Delta u(\overline{P})+\frac{M-m}{2d^2}[\Delta r^2(P,P_0)]_{P=\overline{P}}=\frac{M-m}{2d^2}(2+2+2)>0$$}
which is a contradiction. Therefore\\
$$u(x,y,z)\leq m=\max\{u(x,y,z):(x,y,z)\in\Gamma\},\,\forall\,(x,y,z)\in\overline{G}$$
To prove the inequality
$$u(x,y,z)\geq \min\{u(x,y,z):(x,y,z)\in\Gamma\},\,\forall\,(x,y,z)\in\overline{G}$$
we apply the above given result to $\{-u(x,y,z)\}$. The proof is complete.
\end{proof}
\begin{remark}
With the same proof we get that a harmonic function \index{Harmonic functions }in the plane
$$\partial _{x}^2 u(x,y)+\partial_{y}^2 u(x,y)=0(x,y)\in D\subseteq \mathbb{R}^2$$
which  is continuous on
$$\overline{D}=D \,U \partial D \,\,\,\,\,\,\,\,\, satisfy$$
$$\max\{u(x,y):(x,y)\in\overline{D}\}=\max\{u(x,y):(x,y)\in\,\partial D\}$$
$$\min\{u(x,y):(x,y)\in\overline{D}\}=\min\{u(x,y):(x,y)\in\,\partial D\}$$
\end{remark}
\begin{remark}
The Poisson equation \index{Poisson equation}
\begin{equation}\label{3:2.1}
\Delta u(x,y,z)=-4\pi\rho(x,y,z)
\end{equation}
has a unique solution under the restriction.
$$\mathop{lim}\limits_{r\rightarrow \infty} u(x,y,z)=0\hbox{ where }\rho(x,y,z):\mathbb{R}^3\rightarrow\mathbb{R}$$ is a continuous function. Consider that $u_1,u_2$ are solutions of the Poisson equation \eqref{3:2.1} fulfilling
$$\mathop{lim}\limits_{r\rightarrow \infty} u_i(x,y,z)=0, i\in\{1,2\}$$
Then $u=u_1-u_2$ satisfies $\Delta u(x,y,z)=0$ for any $(x,y,z)\in\mathbb{R}^3$
and $\mathop{lim}\limits_{r\rightarrow \infty} u(x,y,z)=0$. In particular, $\{u(x,y,z):(x,y,z)\in B(0,R)\subseteq\mathbb{R}^3\}$ satisfies, maximum principle, where $B(0,R)=\overline{G}$ and $\Gamma=\{(x,y,z)\in\mathbb{R}^3:x^2+y^2+z^2=R^2\}$.
We get $u(x_0,y_0,z_0)\leq \max\{u(x,y,z):(x,y,z)\in\Gamma\}$ and
$u(x_0,y_0,z_0)\geq \min\{u(x,y,z):(x,y,z)\in\Gamma\}$
for any $(x_0,y_0,z_0)\in int\,B(0,R)$. In particular, for $(x_0,y_0,z_0)\in int\,B(0,R)$ fixed and passing $R\rightarrow \infty$ we get $u(x_0,y_0,z_0)=0$ and
$u_1(x,y,z)=u_2(x,y,z)$ for any $(x,y,z)\in\mathbb{R}^3$.
\end{remark}
\subsection[The Wave Equation;Kirchhoff,D'Alembert \& Poisson Formulas]{The Wave Equation;\index{Wave equation} Kirchhoff, D'Alembert\index{D'Alembert!} and \\Poisson Formulas}
Consider the waves equation
\begin{equation}\label{3:4.1}
\partial _{t}^2 u(t,x,y,z)= C_{0}^2(\partial_{x}^2 u +\partial _{y}^2 u+\partial_{z}^{2}u)(t,x,y,z)
\end{equation}
for $(x,y,z)\in D(domain) \subseteq R^3$ and initial condition
\begin{equation}\label{3:4.2}
u(0,x,y,z)= u_{0}(0,x,y,z),\,\partial_t u(0,x,y,z)= u_{1}(x,y,z)
\end{equation}
The integral representation of the solution satisfying \eqref{3:4.1} and \eqref{3:4.2}is called Kirchhoff  formula. In particular, for 1- dimensional case the equation and initial conditions are described by
\begin{equation}\label{3:4.3}\partial_{t}^2 u (t,x)= C_{0}^2 \partial _{x}^2 u (t,x), u(0,x)=u_0(x),\partial_{t} u (0,x)=u_{1}(x)
\end{equation}
and d'Alembert formula\index{D'Alembert!formula} gives the following representation
\begin{equation}\label{3:4.4}
u(t,x)=\frac{u_{0}(x+c_{0}t)+u_{0}(x-c_{0}t)}{2}+\frac{1}{2c_{0}}\int _{x-c_{0}t}^{c_{0}t+x}u_{1}(\sigma)d\sigma
\end{equation}
In this simplest case, the equation \eqref{3:4.3}is decomposed as follows
\begin{equation}\label{3:4.5}
(\partial_{t}+c_{0}\partial _{x})(\partial_{t}-c_{0}\partial_{x}) u (t,x)=0
\end{equation}
and find the general solution of the linear first order equation
\begin{equation}\label{3:4.6}
\partial _{t}\nu(t,x)+ c_{0}\partial_{x}\nu(t,x)=0
\end{equation}
By a direct computation we get
\begin{equation}\label{3:4.7}
\nu(t,x)=\nu_{0}(x-c_{0}t),\,(t,x)\in\mathbb{R}\times\mathbb{R}
\end{equation}
where $\nu_{0}(\lambda): \mathbb{R}\rightarrow \mathbb{R}$ is an arbitrary first order continuously differentiable function. Then solve the following equation\\
\begin{equation}\label{3:4.8}
\left\{
     \begin{array}{ll}
      \partial_tu(t,x)-c_0\partial_xu(t,x)=\nu_0(x-c_0t) \\
   u(0,x)=u_0(x),\,\partial_tu(0,x)=u_1
(x)     \end{array}
   \right.
\end{equation}
From the equation $  u(0,x)=u_0(x),\,\partial_tu(0,x)=u_1(x)$ we find $\nu_0$ such that
\begin{equation}\label{3:4.9}
u_1(x)-c_0\partial_xu_0(x)=\nu_0(x),\,x\in\mathbb{R}
\end{equation}
and using the characteristic system \index{Characteristic! system }associated with \eqref{3:4.8} we obtain
\begin{equation}\label{3:4.10}
u(t,x)=u_0(x+c_0t)+\mathop{\int}\limits_{0}^{t}\nu_0(x+c_0t-2c_0s)ds
\end{equation}
Using \eqref{3:4.9} into \eqref{3:4.8} we get the corresponding D'Alembert formula\index{D'Alembert!formula}
\begin{equation}\label{3:4.11}
u(t,x)=\frac{u_0(x+x_0t)+u_0(x-c_0t)}{2}+\frac{1}{2c_0}\mathop{\int}\limits_{x-c_0t}^{x+c_0t}u_1(\sigma)d\sigma
\end{equation}
given in \eqref{3:4.4} In the two-dimensional case $(x,y)\in\mathbb{R}^2$\,we recall that the Poisson formula is expressed as follows
\begin{eqnarray}\label{3:4.12}
  u(t,x,y) &=& \frac{1}{2\pi c_0}[\partial_t(\mathop{\int}\limits_{0}^{2\pi}\mathop{\int}\limits_{0}^{c_0t}\rho\frac{u_0(x+\rho cos\varphi,y+\rho sin\varphi)}{\sqrt{c_{0}^2t^2-\rho^2}}d\rho d\varphi) \nonumber\\
  &+& (\mathop{\int}\limits_{0}^{2\pi}\mathop{\int}\limits_{0}^{c_0t}\rho\frac{u_1(x+\rho cos\varphi,y+\rho sin\varphi)}{\sqrt{c_{0}^2t^2-\rho^2}}d\rho d\varphi)]
\end{eqnarray}
The general case ,$(x,y,z)\in\mathbb{R}^3$, will be treated reducing the equation \eqref{3:4.1} to an wave equation \index{Wave equation}analyzed in the one-dimensional case provided adequate coordinate transformations are used. In this respect, for a $P_0=(x_0,y_0,z_0)\in D$ fixed and $B(P_0,r)\subseteq D$ we define ($\overline{u}$ is the mean value of u on $S_r$)
\begin{equation}\label{3:4.13}
\overline{u}(r,t)=\frac{1}{4\pi r^2}\mathop{\iint}\limits_{S_{r}}u(t,x,y,z)dS_{r}=\frac{1}{4 \pi}\mathop{\iint}\limits_{S_{1}}u(t,x,y,z)dS_{1}
\end{equation}
where $S_{r}=\partial B(P_0,r)$ is the boundary of the ball $B(P_0,r)$, and $dS_r=r^2dS_1$ is used.The integral in \eqref{3:4.13} can be computed as a two dimensional integral provided we notice that each $(x,y,z)\in S_r$ can be written as
\begin{equation}\label{3:4.14}
(x,y,z)=(x_0+\alpha r,y_0+\beta r,z_0+\gamma r)=P_0+rw
\end{equation}
where $(\alpha,\beta,\gamma)=\omega$\,are the following
\begin{equation}\label{3:4.15}
\left\{
   \begin{array}{ll}
   \alpha=sin\theta cos\varphi\,\,\,0\leq\theta\leq \pi \\
     \beta= sin\theta sin\varphi\,\,\,0\leq\varphi\leq \pi\\
   \gamma =cos\theta
   \end{array}
 \right.
\end{equation}
Here $dS_r=r^2 sin\theta d\varphi d\theta\,$ and rewrite \eqref{3:4.13} using a two dimensional integral
\begin{equation}\label{3:4.16}
\overline{u}(r,t)=\frac{1}{4 \pi}\mathop{\int}\limits_{0}^{2\pi}\mathop{\int}\limits_{0}^{\pi}u(t,P_0+rw)sin\theta d\varphi d\theta
\end{equation}
assuming\,\,$w=(\alpha,\beta,\gamma)$\,given in \eqref{3:4.16}.The explicit expression\index{Explicit! expression} of $\overline{u}(r,t)$\, in (14) will be deduced taking into consideration the corresponding wave equation \index{Wave equation}satisfied by $\overline{u}(r,t)$ when $u(t,x,y,z)$ is a solution of \eqref{3:4.1}. In this respect, integrate in both sides of \eqref{3:4.1} using the three dimensional domain $D_r=B(P_0,r),P_0=(x_0,y_0,z_0)$, and the spherical transformation of the coordinates $(x,y,z)$
\begin{equation}\label{3:4.17}
(x,y,z)=(x_0+\rho\alpha,y_0+\rho\beta+z_0+\rho\gamma)=P_0+\rho\omega
\end{equation}
where $0\leq \rho\leq r$ and $\alpha(\theta,\varphi),\beta(\theta,\varphi),\gamma(\theta)$ satisfy \eqref{3:4.15}. We get
\begin{equation}\label{3:4.18}
\mathop{\iiint}\limits_{D_r}\partial_{t}^{2}u(t,x,y,z)dxdydz=\mathop{\int}\limits_{0}^{r}(\mathop{\iint}\limits_{S_1}\partial_{t}^{2}
u(t,P_0+\rho\omega)\rho^2dS_1)d\rho
\end{equation}
where $dS_1=sin\theta d\theta d\varphi$\,and $S_1=S(P_0,1)$ is the sphere centered at $P_0$. Denote $\Delta=\partial_{x}^{2}+\partial_{y}^{2}+\partial_{z}^{2}$ and for the integral in the right hand side we apply Gauss-Ostrogradsky \index{Gauss-Ostrogadsky}formula. We get
\begin{eqnarray}\label{3:4.19}
  C_{0}^{2}\mathop{\iiint}\limits_{D_r}\Delta u(t,x,y,z)dxdydz &=& C_{0}^2\mathop{\iint}\limits_{S_r}[\alpha\partial_xu(t,P_0+r\omega) \nonumber\\
  &+&\beta\partial_yu(t,P_0+r\omega)+\gamma\partial_zu(t,P_0+r\omega)]dS_r \nonumber\\
   &=& C_{0}^{2}\mathop{\iint}\limits_{S_1}\{\partial_r[u(t,P_0+r\omega)]r^2dS_1\}\nonumber\\
   &=& C_0^2r^2\partial_r[\mathop{\iint}\limits_{S_1}u(t,P_0+r\omega)dS_1]
\end{eqnarray}
where $(\alpha,\beta,\gamma)=\omega$ are the coordinates of the unit outside normal vector at $S_r$. Using \eqref{3:4.13} we rewrite \eqref{3:4.19} as
\begin{equation}\label{3:4.20}
C_{0}^{2}\mathop{\iiint}\limits_{D_r}\Delta u(t,x,y,z)dxdydz=C_{0}^{2}r^2(4\pi)\partial_r\overline{u}(r,t)
\end{equation}
In addition, notice that \eqref{3:4.18} can be written as
\begin{eqnarray}\label{3:4.21}
  \mathop{\iiint}\limits_{D_r}\partial_{t}^{2}u(t,P_0+\rho\omega)dxdydz &=& \int_{0}^{r}\rho^2 [\partial_{t}^2\mathop{\iiint}\limits_{S_1}u(t,P_0+\rho\omega)dS_1]d\rho \nonumber\\
  &=& \int_{0}^{r}4\pi\rho^2\partial_{t}^{2}\overline{u}(\rho,t)d\rho
\end{eqnarray}
and deriving with respect to $r$ in \eqref{3:4.20}and \eqref{3:4.21} we obtain
\begin{equation}\label{3:4.22}
\partial_{t}^{2}\overline{u}(r,t)=\frac{C_{0}^{2}}{r^2}\partial_r[r^2\partial_{r}\overline{u}(r,t)]
\end{equation}
Denote $\overline{\overline{u}}(r,t)=r\overline{u}(r,t)$\,and using \eqref{3:4.22} compute
\begin{equation}\label{3:4.23}
\partial_{t}^{2}\overline{\overline{u}}(r,t)\left\{
  \begin{array}{ll}
    =& r\partial_{t}^{2}\overline{u}(r,t) \\
    =&\frac{C_{0}^{2}}{r^2}\partial_r[r^2\partial_{r}\overline{u}(r,t)]  \\
    = & \frac{C_{0}^{2}}{r^2}[2r\partial_{r}\overline{u}(r,t)+r^2\partial_{r}^{2}\overline{u}(r,t)] \\
    = & C_{0}^{2}[2\partial_{r}\overline{u}(r,t)+r\partial_{r}^{2}\overline{u}(r,t)]
  \end{array}
\right.
\end{equation}
Notice that
\begin{equation}\label{3:4.24}
\partial_{r}^{2}\overline{\overline{u}}(r,t)=2\partial_{r}\overline{u}(r,t)+r\partial_{r}^{2}\overline{u}(r,t)
\end{equation}
and rewrite \eqref{3:4.23} as follows
\begin{equation}\label{3:4.25}
\partial_{t}^{2}\overline{\overline{u}}(r,t)=C_{0}^{2}\partial_{r}^{2}\overline{\overline{u}}(r,t),\,r\geq 0
\end{equation}
where $\overline{\overline{u}}$ satisfies the boundary condition
\begin{equation}\label{3:4.26}
\overline{\overline{u}}(0,t)=0
\end{equation}
Extend $\overline{\overline{u}}(r,t)$\,for $r<0$\,by $\overline{\overline{u}}(-r,t)=-\overline{\overline{u}}(r,t)$ where $r\geq 0$. In addition, the initial conditions for $\overline{\overline{u}}(r,t)$ are the following
\begin{equation}\label{3:4.27}
\left\{
       \begin{array}{ll}
        \overline{\overline{u}}(r,0)=r\overline{u}(r,0)=\frac{1}{4\pi r}\mathop{\iint}\limits_{S_r}u_0(x,y,z)dS_r \\
         \partial_t\overline{\overline{u}}(r,0)=\frac{1}{4\pi r}\mathop{\iint}\limits_{S_r}u_1(x,y,z)dS_r
       \end{array}
     \right.
\end{equation}
The solution $\{\overline{\overline{u}}(r,t)\}$ fulfilling \eqref{3:4.25})and \eqref{3:4.27} are expressed using Poisson formula (see\eqref{3:4.4})
\begin{equation}\label{3:4.28}
\overline{\overline{u}}(r,t)=\frac{\varphi(r+C_0t)+\varphi(r-C_0t)}{2}+\frac{1}{2C_0}\int_{x-C_0}^{x+C_0}\psi(\sigma)d\sigma
\end{equation}
where\\
\begin{equation}\label{3:4.29}
\varphi(\xi)=\overline{\overline{u}}(\xi,0)=\frac{1}{4\pi\xi}\mathop{\iint}\limits_{S_\xi}u_1(x,y,z)dS_\xi
\end{equation}
and\\
\begin{equation}\label{3:4.30}
\psi(\xi)=\partial_{t}\overline{\overline{u}}(\xi,0)=\frac{1}{4\pi\xi}\mathop{\iint}\limits_{S_\xi}u_0(x,y,z)dS_\xi
\end{equation}
Using \eqref{3:4.29} and \eqref{3:4.30} we get the Kirchhoff formula for the solution of the wave equation \index{Wave equation}\eqref{3:4.1} satisfying Cauchy conditions \eqref{3:4.2} and it can be expressed as follows
\begin{eqnarray}\label{3:4.31}
  u(t,P_0) &=& \frac{d}{dr}[\frac{1}{4\pi r}\mathop{\iint}\limits_{S_r(P_0)}u_0(x,y,z)dS_r]_{r=C_{0}t}+
[\frac{1}{4\pi C_{0}}\mathop{\iint}\limits_{S_r(P_0)}u_1(x,y,z)dS_r]_{r=C_{0}t} \nonumber\\
   &=& \frac{1}{4\pi C_{0}}[\frac{d}{dr}(\frac{1}{C_0t}\mathop{\iint}\limits_{S_{c_{0}t}(P_0)}u_0(x,y,z)dS)+\frac{1}{C_0t}
\mathop{\iint}\limits_{S_{c_{0}t}(P_0)}u_1(x,y,z)dS]
  \end{eqnarray}
where \\
$$\mathop{\iint}\limits_{S_{c_{0}t}(P_0)}u_0(x,y,z)dS=\int_{0}^{2\pi}\int_{0}^{\pi}u(P_0+(C_0t)\omega)(C_0t)^2sin\theta d\theta d\varphi$$
and $\omega=(\alpha,\beta,\gamma)$\,is defined in \eqref{3:4.15}.
Passing $P_0\rightarrow P=(x,y,z)$\,in the Kirchhoff formula \eqref{3:4.31} we get the integral representation of a Cauchy problem solution satisfying \eqref{3:4.1} and \eqref{3:4.2}
\begin{equation}\label{3:4.32}
u(t,x,y,z)=\frac{1}{4\pi C_{0}}
[\frac{d}{dt}(\frac{1}{C_0t}\mathop{\iint}\limits_{S_{c_{0}t}(P_0)}u_0(x,y,z)dS)+\frac{1}{C_0t}\mathop{\iint}\limits_{S_{c_{0}t}(P_0)}u_1(x,y,z)dS]
\end{equation}
where the integral in \eqref{3:4.32} is computed as follows
{\small
\begin{equation}\label{3:4.33}
\frac{1}{C_0t}\mathop{\iint}\limits_{S_{c_{0}t}(P_0)}u(\xi,\eta,\sigma)dS=C_0t\int_{0}^{2\pi}\int_{0}^{\pi}u(x+(C_0t)\alpha,y+(C_0t)\beta,z+(C_0t)\gamma)sin\theta d\theta d\varphi
\end{equation}}
and $(\alpha,\beta,\gamma)$ are defined in \eqref{3:4.15}
$\left\{
   \begin{array}{ll}
   \alpha=sin\theta cos\varphi\,\,\,0\leq\theta\leq\pi \\
     \beta= sin\theta sin\varphi\,\,\,0\leq\varphi\leq 2\pi\\
   \gamma =cos\theta
   \end{array}
 \right.
$
\begin{proposition}
Assume that $u_0,u_1\in\mathcal{C}^4(\mathbb{R}^3)$\,are given. Then $\{u(t,x,y,z):t\in\mathbb{R}_{+},(x,y,z)\in\mathbb{R}^3\}$\, defined in \eqref{3:4.32} is
 a solution of the hyperbolic equation \eqref{3:4.1} satisfying the Cauchy condition \eqref{3:4.2}.
\end{proposition}
\begin{proof}
Notice that if $u(t,x,y,z)$\,satisfies the wave equation \index{Wave equation}\eqref{3:4.1} then $\nu(t,x,y,z)=\partial_tu(t,x,y,z)$ verifies also the wave equation \eqref{3:4.1}. It allows us to get the conclusion and it is enough to prove that
\begin{equation}\label{3:4.34}
u(t,x,y,z)=C_0t\int_{0}^{2\pi}\int_{0}^{\pi}u(x+(C_0t)\alpha,y+(C_0t)\beta,z+(C_0t)\gamma)sin\theta d\theta d\varphi
\end{equation}
fulfils the wave equation \eqref{3:4.1}.In this respect, using \eqref{3:4.34}, compute the corresponding derivatives involved in equation \eqref{3:4.1} and we get
\begin{eqnarray}\label{3:4.35}
  \partial_tu &=& \frac{u}{t}+(C_0t)\partial_t \int_{0}^{2\pi}\int_{0}^{\pi}u(x+(C_0t)\alpha,y+(C_0t)\beta,z+(C_0t)\gamma)sin\theta d\theta d\varphi \nonumber\\
  &=& \frac{u}{t}+\frac{1}{t}\int_{0}^{2\pi}\int_{0}^{\pi}(\partial_\xi u)\alpha+(\partial_\eta u)\beta+(\partial_\sigma)\gamma  C_{0}^2t^2 sin\theta d\theta d\varphi
\end{eqnarray}
Notice that $(\alpha,\beta,\gamma)=n$\,represent the unit vector oriented outside of the sphere $S_{C_0}$\,and $dS=(C_0t)^2sin\theta\, d\theta\, d\varphi$ we rewrite \eqref{3:4.35}
\begin{equation}\label{3:4.36}
\partial_tu=\frac{u}{t}+\frac{I}{t}
\end{equation}
where
\begin{eqnarray}\label{3:4.37}
  I &=& \int_{0}^{2\pi}\int_{0}^{\pi}(\partial_\xi u)\alpha+(\partial_\eta u)\beta+(\partial_\sigma)\gamma  C_{0}^2t^2 sin\theta d\theta d\varphi \nonumber\\
  &=& \mathop{\iint}\limits_{S_{C_0t}}\frac{\partial u}{\partial n}dS \nonumber\\
  &=& \mathop{\iiint}\limits_{B(x,y,z;C_0t)}(\partial_{\xi}^{2}u+\partial_{\eta}^{2}u+\partial_{\sigma}^{2}u)d\xi d\eta d\sigma
\end{eqnarray}

is written using Gauss-Ostrogrodsky \index{Gauss-Ostrogadsky}formula,and $B(x,y,z;C_0t)$\,is the ball in $\mathbb{R}^3$ centered at $P=(x,y,z)$. From \eqref{3:4.36}, deriving again we get
\begin{equation}\label{3:4.38}
\partial_{t}^{2}=\frac{1}{t}\partial_{t}u-\frac{1}{t^2}u+\frac{1}{t}\partial_{t}I-\frac{1}{t^2}I=\frac{1}{t}(\frac{u}{t}+\frac{I}{t})-\frac{u}{t^2}-\frac{I}{t^2}+\frac{1}{t}\partial_{t}I=\frac{1}{t}\partial_{t}I
\end{equation}
On the other hand,using \eqref{3:4.37} and applying the spherical coordinate transformation in the three-dimensional integral we get\
\begin{equation}\label{3:4.39}
\left\{
       \begin{array}{ll}
         I=\int_{0}^{C_0t}[\int_{0}^{2\pi}\int_{0}^{\pi}(\partial_{\xi}^{2}u+\partial_{\eta}^{2}u+\partial_{\sigma}^{2}u)sin\theta d\theta d\varphi]r^2dr \\
         \partial_tI=C_0\mathop{\iint}\limits_{S_{C_0}t}(x,y,z)(\partial_{\xi}^{2}u+\partial_{\eta}^{2}u+\partial_{\sigma}^{2}u)ds
       \end{array}
     \right.
\end{equation}
and \eqref{3:4.38} becomes
\begin{equation}\label{3:4.40}
\partial_{t}^2 u=\frac{C_0}{t}\mathop{\iint}\limits_{S_{C_)t}(x,y,z)}(\Delta u)(\xi,\eta,\sigma)dS
\end{equation}
Denote $\Delta u(t,x,y,z)=(\partial_{x}^{2}u+\partial_{y}^{2}u+\partial_{z}^{2}u)(t,x,y,z)$ and a direct computation applied in \eqref{3:4.34} allows to get
\begin{equation}\label{3:4.41}
\Delta u(t,x,y,z)=\frac{1}{(C_0)^2}.\frac{C_0}{t}\mathop{\iint}\limits_{S_{C_)t}(x,y,z)}(\Delta u)(\xi,\eta,\sigma)dS
\end{equation}
and each term entering in the definition \eqref{3:4.32} will satisfy the wave equation \eqref{3:4.1}\index{Wave equation}. In addition, using \eqref{3:4.33}, we see easily that the Cauchy condition given in \eqref{3:4.2} are satisfied by the function $u(t,x,y,z)$ defined in \eqref{3:4.32}. The proof is complete.
\end{proof}
\section{Exercises}
$(a_{1})$ consider the wave equation in plane
$$\partial_{t}^2 u (t,x,y)=c_{0}^2[\partial_{x}^2 u(t,x,y) +\partial_{y}^2 u(t,x,y)],u(0,x,y)= y_{0}(x,y),\, \partial_{t} u (0,x,y)= u_{1}(x,y)$$
Find the integral representation of its solution.
(Poisson formula see \eqref{3:4.42}).\\
\textbf{Hint}. It will be deduced from the Kirchhoff formula (see proposition 1) considering that the variable $z$ does not appear. We get
\begin{eqnarray}\label{3:4.42}
  u (t,x,y) &=& \frac{1}{2\pi C_{0}}\Big[\partial_{t}\Big(\int_{0}^{2\pi}\int_{0}^{C_0t}\frac{\rho}{\sqrt{c_{0}^{2} t^2-\rho^2}}u_0(x+\rho cos\phi,y+\rho sin \phi)) d\rho d\phi\Big) \nonumber\\
  &+& \int_{0}^{2\pi}\int_{0}^{C_0t}\frac{\rho}{\sqrt{c_{0}^{2} t^2-\rho^2}}u_1(x+\rho cos\phi,y+\rho sin \phi) d\rho d\phi]
\end{eqnarray}
In the Kirchhoff formula (see\eqref{3:4.32}) the function $u(\xi,\eta)$ does not depend on $\sigma$ and $a$ direct computation  gives the following
$$(\theta \in[0,\pi],\phi \in [0,2\pi])$$
\begin{eqnarray*}
  E &=& \int_{0}^{2\pi} \int_{0}^{\pi} u(x+(C_0t)sin\theta cos\phi,y+(C_0t)sin\theta sin \phi)(C_{0}^2 t^2)sin\theta d\theta d\phi \\
  &=& \int_{0}^{2\pi} \int_{0}^{\frac{\pi}{2}} u()(C_{0}^2 t^2)sin\theta d\theta d \phi+\int_{0}^{2 \pi}\int_{\frac{\pi}{2}}^ \pi u ()C_{0}^2 t^2 sin\theta d\theta d\phi \\
  &=&2 \int_{0}^{2\pi} \int_{0}^{\frac{\pi}{2}} u (x,cot sin\theta cos\phi,y+(cot)sin\theta sin\phi)(C_{0}^2 t^2)sin\theta d\theta d\phi
\end{eqnarray*}
Change the coordinate $\rho =(C_0t)sin\theta $ and  get $d\rho=(C_0t)cos\theta d\theta $ for $0 \leq \theta\leq\frac{\pi}{2}$ where
$$cos\theta=\sqrt{1-sin^2\theta}=\sqrt{1-\frac{\rho^2}{C_0t}^2}=\frac{1}{C_0t}\sqrt{(C_0t)^2 -\rho^2}$$
We get
$$E =2\int_{0}^{2\pi}\int_{0}^{C_0t}\frac{\rho}{\sqrt{c_{0}^{2} t^2-\rho^2}}u(x+\rho cos\varphi,y+\rho sin \varphi) d\rho d\phi$$ and Kirchhoff's formula\index{Kirchhoff's formula } becomes Poison formula given in \eqref{3:4.42}.\\
$(a_{2})$ Prove that the solution of the non homogeneous waves equation\index{Homogeneous!wave equation}.
\begin{equation}\label{3:4.43}
\left\{
       \begin{array}{ll}
         \partial_{t}^2 u- C_{0}^2(\partial_{x}^2 u +\partial_{y}^2 y+\partial_{z}^2 u)=f(x,y,z,t) \\
         u(0,x,y,z)=0 \,\,\,\partial_{t} u(0,x,y,z)=0
         \end{array}
     \right.
\end{equation}
is presented by the following formula.
\begin{equation}\label{3:4.44}
u (t,x,y,z)=\frac{1}{4\pi C_{0}^2} \mathop{\iiint}\limits_{B(x,y,z;C_0t)} \frac{ f(\xi,\eta,\sigma,t-\frac{\sqrt{(x- \xi)^2+(y- \eta)^2+(z-\sigma)^2}}{C_{0}})}{\sqrt{(x-\xi)^2+(y-\eta)^2+(z-\sigma)^2}}d\xi d\eta d\sigma
\end{equation}
\textbf{solution}.\\Denote $p=(x,y,z),q=(\xi,\eta, \sigma)$ \,and\,$\mid p-q\mid=[\mathop{\sum}\limits_{i=1}^{3}(p_{i}-q_{i})^2]^{\frac{1}{2}}\,;B(p,C_0t)$ is the ball centered at $p$ with radius $C_0t$. To get\eqref{3:4.44} as a solution of \eqref{3:4.43} we use the solution of the following homogeneous equation\index{Homogeneous!equation}
\begin{equation}\label{3:4.45}
\partial _{t}^2\nu= c_{0}^2\Delta \nu,\nu|_{t=\sigma}=\nu_0(p),\partial_t\nu|_{t=\sigma}=f(p,\sigma)=\nu_1(p)
\end{equation}
where the initial moment $t=\sigma (replacing \,\,t=0)$ is a parameter.Recall the Kirchhoff formula for the solution satisfying \eqref{3:4.45}
\begin{equation}\label{3:4.46}
\nu(t,p)=\frac{1}{4\pi C_{0}}[\frac{d}{dt}(\mathop{\iint}\limits_{S_r(P)}\frac{\nu_0(q)}{r}dS_r)+\mathop{\iint}\limits_{S_r(P)}\frac{\nu_1(q)}{r}dS_r]
\end{equation}
 where $r=C_0t$ and $dS_r=r^2dS_1=r^2sin\theta d\theta d\varphi$\,for$\theta\in[0,2\pi]$, where $S_1$ is the unit sphere. Making use of the spherical coordinates transformation
\begin{equation}\label{3:4.47}
q=p+r\omega,\,\omega=(sin\theta cos\varphi,sin\theta sin\varphi,cos\theta),\,r=C_0t
\end{equation}
we rewrite \eqref{3:4.46} as follows
\begin{equation}\label{3:4.48}
\nu(t,p)=\frac{1}{4\pi}\partial_t[t\mathop{\iint}\limits_{S_1(P)}\nu_0(p+r\omega)dS_1]+\frac{t}{4\pi}\mathop{\iint}
\limits_{S_1(P)}\nu_1(p+r\omega)dS_1
\end{equation}
where $dS_1=sin\theta \,d\theta \,d\varphi$. Using the Cauchy conditions of \eqref{3:4.45} from \eqref{3:4.48} we get
\begin{equation}\label{3:4.49}
\nu(t,p;\sigma)=\frac{t-\sigma}{4\pi}\mathop{\iint}\limits_{S_1(P)}f(p+C_0(t-\sigma)\omega;\sigma)dS_1,\,\,0\leq\sigma\leq t
\end{equation}
 where $t-\sigma$ representing $t$ and $\nu|_{t=\sigma}=0$ are used. We shall show that $u(t,p)$ defined by (Duhamel integral)
\begin{equation}\label{3:4.50}
u(t,p)=\int_{0}^{t}\nu(t,p;\sigma)\,d\sigma
\end{equation}
is a solution of the wave equation \index{Wave equation}\eqref{3:4.43}. The initial conditions $u(0,p)=0=\partial_tu(0,p)$\,are easily verified by direct inspection. Computing lapacian operator $\Delta_p=\partial_{x}^2+\partial_{y}^2+\partial_{z}^2$ applied in both sides of \eqref{3:4.50} we obtain
\begin{equation}\label{3:4.51}
\Delta_pu(t,p)=\int_{0}^{t}\Delta_p\nu(t,p;\sigma)\,d\sigma
\end{equation}
Notice that $"\partial_t"$ applied in \eqref{3:4.50} lead us to
 \begin{equation}\label{3:4.52}
\left\{
        \begin{array}{ll}
          \partial_tu(t,p)= \nu(t,p;\sigma)_{\sigma=t}+ \int_{0}^{t}\partial_t\nu(t,p;\sigma)\,d\sigma \\
          \partial_{t}^{2}u(t,p)=\partial_t\nu(t,p;\sigma)_{\sigma=t}+\int_{0}^{t}\partial_{t}^{2}\nu(t,p;\sigma)\,d\sigma=\\
          =f(p,t)+C_0\int_{0}^{t}\Delta_{p}\nu(t,p;\sigma)d\sigma\\
          =f(p,t)+C_{0}^{2}\Delta_{p}u(t,p)
        \end{array}
     \right.
\end{equation}
and $\{u(t,p)\}$\,defined in \eqref{3:4.50} verifies the wave equation \index{Wave equation}\eqref{3:4.43}. Using \eqref{3:4.49}
 we rewrite \eqref{3:4.50} as in conclusion \eqref{3:4.44}. In this respect, substituting \eqref{3:4.49} into \eqref{3:4.50} we get following formula ($dS_1=sin\theta \,d\theta \,d\varphi$)
\begin{equation}\label{3:4.53}
u(t,p;\sigma)=\frac{1}{4\pi}\int_{0}^{t}(t-\sigma)[\mathop{\iint}\limits_{S_1(P)}f(p+C_0(t-\sigma)\omega;\sigma)dS_1]d\sigma
\end{equation}
Making a change of variables $C_0(t-\sigma)=r$, we get
\begin{equation}\label{3:4.54}
u(t,p)=\frac{1}{4\pi C_{0}^{2}}\int_{0}^{C_0t}\int_{0}^{2\pi}\int_{0}^{\pi}f(p+r\omega,t-\frac{r}{C_0})r\,sin\theta\,d\theta\,d\varphi\,dr
\end{equation}
and noticing that $r=\frac{r^2}{r},\,\mid \omega\mid^2=1$\,we rewrite \eqref{3:4.54} as follows
\begin{equation}\label{3:4.55}
u(t,p)=\frac{1}{4\pi C_{0}^{2}}\mathop{\iint\int}\limits_{B(p,C_0t)}f(q,t-\frac{\mid p-q\mid}{C_0})\frac{1}{\mid p-q\mid}dq
\end{equation}
and the proof is complete.\,\,\,\,\,\,$\Box$
\subsection{System of Hyperbolic and Elliptic Equations(\index{Elliptic equation}Definition)}
\begin{definition}
A first order system of n equations
\begin{equation}\label{3:4.56}
A_{0}\partial_{t} u+ \mathop{\sum}\limits_{i=1}^{m} A_{i}(x,t) \partial_{i} u = f(x,y,u),u\in R^{n},\times \in R^{m}
\end{equation}
is t-hyperbolic if the corresponding characteristic equation \index{Characteristic!equation}$det\mid\mid\sigma A_{0}(x,t)+\\ \sum_{i=1}^{m}\xi A_{i} (x,t)\mid\mid=o$ has n real distinct roots for the variable $\sigma$  at each point $(t,x)\in[0,T]\times D,\,\,D\subseteq R^{m},$ \,for any \,$\xi \in R^{m},\xi \neq 0$. There is a particular first order system for which a verification of this property is simple
\end{definition}
\begin{definition}
The first order system \eqref{3:4.56} is symmetric t- hyperbolic (defined by Friedreich)if all matrices $A_{j}(x,t)j\in {0,1,...,m}$ are symmetric and $A_{0}(x,t)$ is strictly positive definite in the domain $(t,x) \in[0,T]\times D$
\end{definition}
\begin{definition}
The first order system \eqref{3:4.56} is called elliptic if the corresponding characteristic equation\index{Characteristic! equation} $\det \parallel\sigma A_{0}(x,t)+\sum_{i=1}^{m} \xi A_{i}\parallel=0$ has no real solution $(\sigma,\xi)\in\mathbb{R}^{m+1},\,(\sigma,\xi)\neq 0,\,\forall\,\,(x,t)\in D\times[0,T]$.
\end{definition}
\begin{example}
In the case of a second order PDE
{\small
\begin{equation}\label{3:4.57} a(x,t)\partial_{t}^{2}u+2b(x,t)\partial_{(t,x)}^{2}u+c(x,t)\partial_{x}^{2}u+d(x,t)\partial_{t}u+e(x,t)\partial_{x}u=f(x,t,u),x\in\mathbb{R}
\end{equation}}
We get the following characteristic equation\index{Characteristic! equation}
\begin{equation}\label{3:4.58}
a(x,t)\sigma^2+2b(t,x)\sigma\xi+c(t,x)\xi^2=0
\end{equation}
\end{example}
\noindent\textbf{Hint} Using a standard procedure we rewrite the scalar equation \eqref{3:4.57} as an evolution system for the unknown vector $u_1(t,x)=\partial_xu(t,x),\,u_2(t,x)=\partial_tu(t,x),u_3(t,x)=u(t,x)$. We get
\begin{equation}\label{3:4.59}
\left\{
       \begin{array}{ll}
         \partial_tu_1=\partial_xu_2 \\
         a(x,t)\partial_tu_2 =-2b(x,t)\partial_xu_2 -c(x,t)\partial_xu_1+f_1(x,t,u_1,u_2,u_3)\\
        \partial_tu_3 =u_2\,,\,\,t\in\mathbb{R},\,x\in\mathbb{R}
       \end{array}
     \right.
\end{equation}
where$f_1(x,t,u_1,u_2,u_3)=f(x,t,u_3)-e(x,t)u_1-d(x,t)u_2$. The system \eqref{3:4.59} can be written as \eqref{3:4.56} using the following $(3\times 3)$\,matrices\\
$$A_0(x,t)=\left(
            \begin{array}{ccc}
              1 & 0 & 0 \\
              0 & a(x,t) & 0 \\
              0 & 0 & 1 \\
            \end{array}
          \right)
,\,\,A_0(x,t)=\left(
            \begin{array}{ccc}
              0 & -1 & 0 \\
              c(x,t) & 2b(x,t) & 0 \\
              0 & 0 & 0 \\
            \end{array}
          \right)$$\\
Compute
\begin{eqnarray*}
  det\parallel \sigma A_0(x,t)+\xi A_1(x,t)\parallel &=& det\left(
                                                                 \begin{array}{ccc}
                                                                   \sigma & -\xi& 0 \\
                                                                  \xi.c & \sigma a+2b \xi& 0 \\
                                                                   0 & 0 & \sigma \\
                                                                 \end{array}
                                                               \right) \\
  &=& \sigma[a(x,t)\sigma^2+2b(x,t)\sigma\xi+c(x,t)\xi^2]
\end{eqnarray*}
and the characteristic equation\index{Characteristic!equation} \eqref{3:4.58} is obtained.
\section[Adjoint 2nd order operator;Green formulas]{Adjoint Second Order Differential Operator\index{Adjoint!second order differential operator}; Riemann and Green Formulas\index{Green!formula}}
We consider the following linear second order differential operator
\begin{equation}\label{3:6.1}
L\,\nu=\mathop{\sum}\limits_{i=1}^{n}\mathop{\sum}\limits_{j=1}^{n}A_{ij}\partial_{ij}^{2}\nu+
\mathop{\sum}\limits_{i=1}^{n}B_i\partial_i\nu+C\nu
\end{equation}
where the coefficients $A_{ij},\,B_{i},\,and\, C$\, are second order continuously differentiable scalar functions of the variable $x=(x_{1},...,x_{n})\in \mathbb{R}^ {n}$\, and $ \partial_{i} \nu =\frac{d\nu}{dx_{i}},\, \partial_{i j}^{2} \nu = \frac{\partial^2 \nu}{\partial x_{i}\partial x_{j}}$. Without restricting generality we may assume that $A_{ij}=A_{ij} (see A_{ij}\rightarrow \frac{1}{2}(A_{ij}+A_{ij}))$ and define the adjoint operator\index{Adjoint!operator} associated with $L(L^{\infty }=M)$
\begin{equation}\label{3:6.2}
M\nu \equiv\mathop{\sum}\limits_{i=1}^{n}\mathop{\sum}\limits_{j=1}^{n} \partial_{ij}^{2} (A_{i j}\nu)- \mathop{\sum}\limits_{i=1}\partial _{i} (B_{i}\nu)+C\nu
\end{equation}
By a direct computation we convince ourselves that the following formula is valid
\begin{equation}\label{3:6.3}
\nu L_{u}-uM\nu=\mathop{\sum}\limits_{i=1}^{n}\partial_{i}P_{i},\,u,\nu\in \mathcal{C}^2(\mathbb{R}^n)
 \end{equation}
where
$$P_{i}=\mathop{\sum}\limits_{j=i}^{n}[\nu A_{ij}\partial_{j}u-u\partial_{j}(A_{ij}\nu)]+B_{i} u \nu $$
In this respect notice that
\begin{eqnarray*}
  \mathop{\sum}\limits_{i=1}^{n}\partial_{i}P_{i} &=& \Big[\mathop{\sum}\limits_{i=1}^{n}\mathop{\sum}\limits_{j=1}^{n}\nu A_{ij}\partial_{ij}^{2}u+ \mathop{\sum}\limits_{i=1}^{n}\nu B_{i}\partial_{i}u+C u\nu\Big]\\
&-&\Big[\mathop{\sum}\limits_{i=1}^{n}\mathop{\sum}\limits_{j=1}^{n}u\partial_{ij}^{2}(A_{ij}\nu)
  - \mathop{\sum}\limits_{i=1}^{n}u\partial_{i}(B_{i}\nu)+Cu\nu\Big]\\
  &+&\mathop{\sum}\limits_{i=1}^{n}\mathop{\sum}\limits_{j=1}^{n} [\partial_{j}u\partial_{i}(A_{ij}\nu)-\partial_{i}u\partial_{j}(A_{ij}\nu)]\\
&&
\end{eqnarray*}
where the last term is vanishing.
\subsection{Green Formula\index{Green!formula} and Applications}
Consider a bounded domain $\Omega\subseteq\mathbb{R}^{n}$and assume that its boundary $S=\partial \Omega$ can be expressed using piecewise first order continuously differentiable functions. Then using Gauss-Ostrogradsky \index{Gauss-Ostrogadsky}formula associated with \eqref{3:6.3} we get
\begin{equation}\label{3:6.4}
\mathop{\int}\limits_{\Omega}(\nu Lu-uM\nu)dx_{1},...dx_{n}=+\mathop{\int}\limits_{S}(\mathop{\sum}\limits_{i=1}^{n}P_{i}cos nx_{i})dS
\end{equation}
(Green formula) where $cos \,n\,x_{i},\,i\in\{1,...,n\},$\, are the coordinates of the unit orthogonal vector\,$\overrightarrow{n}$ \,at $S$\, oriented outside of $S$.\\
\begin{example} Assume $Lu=\Delta u=\mathop{\sum}\limits_{i=1}^{3}\partial_{i}^{2}u\hbox{(laplacian)}$ associate the adjoint\index{Adjoint!} $M\nu=\Delta\nu\hbox{ (laplacian)}$ and $P_{1}=\nu\partial_{i}u-u\partial_{1}\nu,P_{2}=\nu\partial_{2}u-u\partial_{2}\nu,P_{3}=\nu\partial_{3} u-u\partial_{3}u$ satisfying \eqref{3:6.3}. Applying \eqref{3:6.4} we get  the Green formula for Laplace operator\index{Laplace operator}.
\begin{eqnarray}\label{3:6.5}
  \mathop{\iiint}\limits_{\Omega}(\nu\delta u-u\delta\nu)dx_{1}dx_{2}dx_{3} &=& \mathop{\iint}\limits_{S}[P_{1}(cosnx_{1})+P_{2}(cosnx_{2})+P(cosnx_{3})]dS \nonumber\\
  &=& \mathop{\iint}\limits_{S}[\nu\frac{\partial u}{\partial n}-u\frac{\partial u}{\partial n}]dS
\end{eqnarray}
Where $\frac{\partial u}{\partial n}(\frac{\partial \nu}{\partial n})$ stands for the derivative in direction of the vector $$\overrightarrow{n}\frac{\partial u}{\partial n}=\mathop{\sum}\limits_{i=1}^{3}(\partial_{i}u)cosnx_{i}=<\partial_{x}u,\overrightarrow{n}>\overline{n}=(cosnx_{1},(cosnx_{2},(cosnx_{3}))$$
\end{example}
\begin{example} Consider $$L\equiv\partial_{xy}^{2} u+a(x,y)\partial_{x}u+b(x,y)\partial_{y}u+c(x,y)u$$
and define
 $$M\nu\equiv\partial_{xy}^{2}\nu-\partial_{x}(a\nu)-\partial_{y}(b\nu)+c\nu, P_{1}=\frac{1}{2}(\nu\partial_y u-u\partial_y\nu)+a\,u\,\nu$$ and
 $$ P_{1}=\frac{1}{2}(\nu\partial_x u-u\partial_x\nu)+b\,u\,\nu$$
 The corresponding Green formula\index{Green!formula} lead us to
\begin{eqnarray}\label{3:6.6}
  \mathop{\iint}\limits_{\Omega}[\nu\,L\,u-u\,M,\nu]dx\,dy &=& \int_{S}\{[ \frac{1}{2}(\nu\partial_y u-u\partial_y\nu) \\
  &=& a\,u\,\nu]cos(n,x)+[\frac{1}{2}(\nu\partial_x u-u\partial_x\nu)+a\,u\,\nu]cos(n,y)\}ds \nonumber
\end{eqnarray}
\end{example}
\subsection{Applications of Riemann Function}
 Using \eqref{3:6.6} we will find the Cauchy problem solution associated  with the equation
 \begin{equation}\label{3:6.7}
L\,u=F
\end{equation}
 and the initial conditions
 \begin{equation}\label{3:6.8}
u|_{y=\mu(x)}=\varphi_0(x),\,\partial_y\,u|{y=\mu(x)}=\varphi_1(x),\,x\in[a,b]
\end{equation}
 Here the curve $\{y=y(x)\}$\,is first order continuously differentiable satisfying $\frac{d\mu(x)}{dx}<0\,,x\in[a,b]$. The algorithm belongs to Riemann and using \eqref{3:6.8} we find $\partial_x\,u|_{y=\mu(x)}$ as follows. By direct derivation of $u|_{y=\mu(x)}=\varphi_0(x)$, we get
 \begin{equation}\label{3:6.9}
\left\{
      \begin{array}{ll}
        \frac{d\varphi_0}{dx}(x)=\partial_x\,u(x,\mu(x))+\partial_y\,u(x,\mu(x)) \frac{d\mu(x)}{dx}\\
      \partial_x\,u(x,\mu(x))=\phi'_0(x)-\varphi_1(x)\mu'(x)
      \end{array}
    \right.
\end{equation}
Rewrite the Green formula \eqref{3:6.6} using a domain $\Omega\subseteq\mathbb{R}^{2}$ as follows.
where $A$ and $B$ are located on the curve $y=\mu(x)$ intersected with the coordinates lines of the fixed point $P=(x_0,y_0)$. Applying the corresponding Green formula \eqref{3:6.6} on the domain $PAB$ (triangle) we get
\begin{eqnarray}\label{3:6.10}
  \mathop{\iint}\limits_{\Omega}(\nu L u-u M \nu)dx dy &=&\int_{A}^{B}[\frac{1}{2}(\nu\partial_y u-u\partial_y\nu)+a\,u\,\nu]dy \nonumber\\
  &-& \int_{A}^{B}[\frac{1}{2}(\nu\partial_x u-u\partial_x\nu)+b\,u\,\nu]dx \nonumber\\
  &+&\int_{B}^{P}[\frac{1}{2}(\nu\partial_y u-u\partial_y\nu)+a\,u\,\nu]dy \nonumber\\
  &+& \int_{A}^{P}[\frac{1}{2}(\nu\partial_x u-u\partial_x\nu)+a\,u\,\nu]dx
\end{eqnarray}

Where $dy=cos(n,x)dS$ and $dx=-cos(n,y)dS$ and considering that $dS$ is a positive measure. Rewrite the last two integrals from \eqref{3:6.10}
\begin{eqnarray}\label{3:6.11}
  \int_{B}^{P}[\frac{1}{2}(\nu\partial_y u-u\partial_y\nu)+a\,u\,\nu]dy &=& \frac{1}{2} u\nu|_{B}^{P}+\int_{B}^{P}(-u\partial_y\nu)+a\,u\,\nu]dy \nonumber\\
  &=& \frac{1}{2} u\nu|_{B}^{P}+\int_{B}^{P}u(-\partial_y\nu)+a\,\nu]dy
\end{eqnarray}
and
\begin{equation}\label{3:6.12}
\int_{A}^{P}[\frac{1}{2}(\nu\partial_x u-u\partial_x\nu)+a\,u\,\nu]dx=\frac{1}{2} u\nu|_{A}^{P}+\int_{A}^{P}u(-\partial_x\nu)+b\,\nu]dx
\end{equation}
these formulas lead us to the solution provided the following Riemann function $\nu(x,y;x_0,y_0)$ is used
\begin{equation}\label{3:6.13}
M\nu=0,\nu|_{x=x_0}=\exp\int_{y_0}^{y}a(x_0,\sigma)d\sigma,\,\nu|_{y=y_0}=\exp\int_{x_0}^{x}b(\sigma,y_0)d\sigma
\end{equation}
From \eqref{3:6.13} we see easily that
\begin{equation}\label{3:6.14}
\left\{
       \begin{array}{ll}
        \nu(x_0,y_0;x_0,y_0)=1\,and \,\partial_y\nu|_{x=x_0}=a(x_0,y)\nu|_{x=x_0},\\
         \partial_x\nu|_{y=y_0}=b(x,y_0)\nu|_{y=y_0}
       \end{array}
     \right.
\end{equation}
Using \eqref{3:6.13} and \eqref{3:6.14} into \eqref{3:6.10} we get
\begin{equation}\label{3:6.15}
\mathop{\iint}\limits_{\Omega}\nu(x,y;x_0,y_0)F(x,y)dxdy=(u\nu)(P)+\Gamma+I=u(x_0,y_0)+\Gamma+I
\end{equation}
where
\begin{equation}\label{3:6.16}
\Gamma=+\int_{A}^{B}[\frac{1}{2}(\nu\partial_y u-u\partial_y\nu)+a\,u\,\nu]dy-\int_{A}^{B}[\frac{1}{2}(\nu\partial_x u-u\partial_x\nu)+b\,u\,\nu]dx
\end{equation}
and
\begin{equation}\label{3:6.17}
I=-\frac{1}{2}[(u\nu)(B)+(u\nu)(A)]+\int_{B}^{P}u(-\partial_y\nu)+a\,\nu]dy+\int_{A}^{P}u(-\partial_x\nu)+b\,\nu]dx
\end{equation}
Notice that the integral in \eqref{3:6.17} are vanishing provided that the initial conditions \eqref{3:6.14} are used for Riemann function $\nu$\,(see\eqref{3:6.13} and \eqref{3:6.14}) and the equation \eqref{3:6.15} is written as
\begin{equation}\label{3:6.18}
u(x_0,y_0)=\frac{1}{2}(u\nu)(B)+(u\nu)(A)-\Gamma+\mathop{\iint}\limits_{\Omega}\nu\,F(x,y)dxdy
\end{equation}
where $\Gamma$ is defined in \eqref{3:6.16}. The Cauchy problem solution defined in \eqref{3:6.7} and \eqref{3:6.8} has a more explicit expression if the Cauchy conditions are vanishing and it can be achieved by modifying the right hand side $F$ as follows
\begin{eqnarray}\label{3:6.19}
  F_1(x,y) &=& F(x,y)-\varphi'_1(x)-a(x,y)\{\varphi'_0(x)+[y-\nu(x)]\varphi'_1(x)-\nu'(x)\varphi_1(x)\}\nonumber\\
  &=& b(x,y)\varphi_1(x)-c(x,y)\{\varphi_0(x)+(y-\mu(x))\varphi_1(x)\}
\end{eqnarray}
\begin{theorem}
The Cauchy problem solution defined in \eqref{3:6.7} and \eqref{3:6.8} is given by
$$u(x_0,y_0)=\mathop{\iint}\limits_{\Omega}\nu(x,y;x_0,y_0)F_1(x,y)dxdy$$\\
where $\nu$ is the corresponding Riemann function defined in \eqref{3:6.13} and \eqref{3:6.14} and $F_1$ is \eqref{3:6.19}.
\end{theorem}
\begin{remark}
In the case we consider a rectangle $PASB=\Omega$ and look for a Cauchy problem solution
\begin{equation}\label{3:6.20}
L\,u=F \hbox{ (where L and F are defined in \eqref{3:6.7})}
\end{equation}
\begin{equation}\label{3:6.21}
u|_{x=x_{1}} =\phi_{1}(y), u|_{y=y_{1}}=\phi_{2}(x), S=(x_{1},y_{1})
\end{equation}
then Green formula\index{Green!formula} lead as to the following solution
\begin{equation}\label{3:6.22}
u(x_{0},y_{0})=u|_{p}=u\nu|_{S}+\int_{B}^{B}\nu(\varphi'_2+b\varphi_2)dx+\int_{S}^{A}\nu(\varphi'_1+a\varphi_1)dy+
\mathop{\iint}\limits_{\Omega}\nu\,F(x,y)dxdy
\end{equation}
where $\nu(x,y;x_0,y_0)$\,is associated Riemann function satisfying \eqref{3:6.13} and \eqref{3:6.14}.
Let $u(x,y;x_1,y_1)$ be the Riemann function for the  adjoint equation \index{Adjoint!equation }$M\nu=0$ satisfying
\begin{equation}\label{3:6.23}
L u=0 , u|_{y=y_{1}}=\exp\mathop{\int}\limits_{x_{1}}^{x}b(\sigma_{1},y_{1})d\sigma,u|x=x_{1}=\exp\mathop{\int}\limits_{y_{1}}^{y} a(x_{1},\sigma)d\sigma
\end{equation}
In this particular case , the corresponding solution $u$ verify \eqref{3:6.20} and \eqref{3:6.21} with $F=0,\varphi'_2+by\varphi_{2}=\varphi'_{1}a\varphi_{1}=0$, and the representation formula \eqref{3:6.22} becomes
\begin{equation}\label{3:6.24}
u(x_{0},y_{0},x_{1},y_{1})=\nu(x_{1},y_{1},x_{0},y_{0})(see \,\,u (S)=u(x_{1},y_{1};x_{1},y_{1})=1)
\end{equation}
The equation \eqref{3:6.24} tell us that the Riemann function is symmetric: if $u(x_0, y_0, y_{1};\\ x_{1}, y_{1})$ is the Riemann function for$(x_{0},y_{0})\in \mathbb{R}^2$\,satisfying $(L\,u)(x_0,y_0)=0$ and initial
condition given at $x=x_{1},y=y_{1}$ then $u(x_{0},y_{0}; x_{1},y_{1})$ as a function of $((x_{1},y_{1})$ verifies the adjoint
equation\index{Adjoint!equation } $Mv=0$.
\end{remark}
\begin{remark}
The equation \eqref{3:6.19} reflects the continuous dependence of a solution for the Cauchy problem defined in \eqref{3:6.7} and \eqref{3:6.8} with respect to the  right hand side $F$ and initial conditions $\varphi_{0},\varphi_{1} $ restricted to  the triangle $PAB$. A solution is determined inside of this  triangle using the fixed $ F,\varphi_{0},\varphi_{1}$ and the curve $y=\mu(x)$ and outside of this triangle. We may get different solutions according to the modified initial conditions prescribed outside of this triangle.
\end{remark}
\begin{remark}
A unique solution (Riemann function) verifying the Cauchy problem defined in \eqref{3:6.23} can be obtained directly using the standard method of approximation for a system of integral equations \index{Integral equations}associated with\eqref{3:6.13}. In the rectangle \\
$x_{0}\leq x\leq a\,,y_{0}\leq y\leq b.$ we are given $\nu|_{x=x_{0}}=\varphi_{1}(y),\nu|_{y=y_{0}}
=\varphi_{2}(x)$\, with $\varphi_{1}(y_{0})=\varphi_{2}(x_{0})$ and denote $\partial_{x} \nu =\lambda , \partial_{y}\nu=w.$\, Rewrite \eqref{3:6.13} as follows
$$\frac{\partial \lambda}{\partial y}=\frac{\partial W}{\partial x}=a(x,y)\lambda+b(x,y)W+\widetilde{C}(x,y)\nu$$
where $\widetilde{C}(x,y)=-c(x,y)+\partial_{x} a(x,y)+\partial_{y}b(x,y)$  and we get a system of integral equations
$$\left\{
   \begin{array}{ll}
     \lambda(x,y)=\lambda(x,y_{0})+ \mathop{\int}\limits_{y_{0}}^{y}[a(x,\sigma)\lambda(x,\sigma)+b(x,\sigma)\omega(x,\sigma)+\widetilde{C}(x,\sigma)v(x,\sigma)]d\sigma\\
     w(x,y)=w(x_{0},y)+ \mathop{\int}\limits_{x_{0}}^{x}[a(x,\sigma)\lambda(x,\sigma)+b(x,\sigma)\omega(x,\sigma)+\widetilde{C}(x,\sigma)v(x,\sigma)]d\sigma\\
  \nu(x,y)=\varphi_{2}(x)+\mathop{\int}\limits_{y_0}^{y} \omega(x,y\sigma) d\sigma\\
   \end{array}
 \right.
$$
where $\lambda(x,y_{0})= \varphi'_{2}(x)\,,w(x_0,y)=\varphi'_{1}(y)$.
\end{remark}
\begin{example}(Green functions\index{Green!functions} for Lapace operator)\\
Let $\Omega\in\mathbb{R}^3$\,be a bounded domain whose boundary $S=\partial\Omega$ is represented by  smooth functions.We are looking for a smooth function $u:\Omega\rightarrow \mathbb{R}$ satisfying Poisson equation
\begin{equation}\label{3:6:1.1}
\Delta\,u=f(P),\,P\in\Omega (\hbox{ $\Delta$-Laplace operator})\index{Laplace operator}
\end{equation}
and one of the following boundary conditions
\begin{equation}\label{3:6:1.2}
u|_{S}=F_0(S)\,(\hbox{Dirichlet problem})
\end{equation}
\begin{equation}\label{3:6:1.3}
\frac{\partial u}{\partial n}|_{S}=F_1(S)\,(\hbox{Neumann problem})
\end{equation}
It is known that Laplace operator\index{Laplace operator} is self adjoint and using Green formula\index{Green!formula}
\begin{equation}\label{3:6:1.4}
\mathop{\iiint}\limits_{\Omega}(\nu\Delta u-u\Delta\nu)d\Omega=\mathop{\iint}\limits_{S}(u\frac{\partial \nu}{\partial n}-\nu\frac{\partial u}{\partial n})dS
\end{equation}
we see that a vanishing Dirichlet (Neumann) condition is selfadjoint as the Green formula \eqref{3:6:1.4} shows.
\end{example}
\subsection{Green function\index{Green!function} for Dirichlet problem}
\textbf{Definition 1}\\
Let $G(P,P_0),\,P,P_0\in\Omega$, be a scalar function satisfying
\begin{enumerate}
  \item $G(P,P_0)$\,is a harmonic function \index{Harmonic functions }$(\Delta\,G=0)$\,with respect to $P\in\Omega,\,P\neq P_0$
  \item $G(P,P_0)|_{P\in S}=0,\,S=\partial\Omega$
  \item $G(P,P_0)=\frac{1}{4\pi r}+g(P,P_0)$,\,where $r=|P-P_0|$
\end{enumerate}
and $g$ is a second order continuously differentiable function of $P\in\Omega$\,verifying $\Delta\,g=0,\,\,\forall\,P\in \Omega$.\\
A function $G(P,P_0)$ fulfilling 1, 2 and 3 is called a Green function for Dirichlet problem \eqref{3:6:1.1}, \eqref{3:6:1.2}. The existence of a Green function satisfying 1 , 2 and 3 is analyzed in [2](see [2] in references)
\begin{theorem}\label{th:3.6.2}
Assume that the Green function $G(P,P_0)$\,satisfying 1, 2 and 3 is found such that the normal derivative $\frac{\partial G}{\partial n}(s)\,of\,G$ at each $s\in S$ exists. Then the solution of Dirichlet problem \eqref{3:6:1.1}+\eqref{3:6:1.2} is represented by
$$u(P_0)-\mathop{\iint}\limits_{S}F_0(s)\frac{\partial G}{\partial n}(s)ds-\mathop{\iiint}\limits_{\Omega}G(P,P_0)f(P)dxdydz$$
\end{theorem}
\begin{proof}
Take $\delta>0$\,sufficiently small such that $B(P_0,\delta)\subseteq \Omega$\,and denote $\sigma=\partial B(P_0,\delta)$. Apply Green formula in the domain $\Omega'=\Omega\setminus B(P_0,\delta)$.
Let $u(P),\,P\in \Omega$,\,be the solution of the Dirichlet problem \eqref{3:6:1.1}+\eqref{3:6:1.2} and denote $\nu(P)=G(P,P_0),\,P\in \Omega'$. Both functions $u(P),\nu(P),P\in\Omega'$,\,are continuously differentiable and we get\\
\begin{equation}\label{3:6:1.5}
\mathop{\iiiint}\limits_{\Omega'}G(P,P_0)\Delta udP=\mathop{\iint}\limits_{S'}(u\frac{\partial G}{\partial n}-G\frac{\partial u}{\partial n})(s)ds
\end{equation}
where $S'=S\cup \sigma=\partial \Omega'$\\
Denote $n$\,the orthogonal vector at the surface $\sigma=\partial B(P_0,\delta),\,|n|=1$,\,oriented to the center $P_0$ and $n'$ is the orthogonal vector at $\sigma$ oriented in the opposite direction of $n$.Using $G|_{P\in S}=0$ we get
\begin{eqnarray}\label{3:6:1.6}
  \mathop{\iiint}\limits_{\Omega'}G(P,P_0)f(P)dx\,dy\,dz &=& \mathop{\iint}\limits_{S}F_0(s)\frac{\partial G}{\partial n}(s)ds-\frac{1}{4\pi}\mathop{\iint}\limits_{\sigma}(u\frac{\partial}{\partial n}\frac{1}{r}-\frac{1}{r}\frac{\partial u}{\partial n})d\sigma \nonumber\\
  &-& \mathop{\iint}\limits_{\sigma}(u\frac{\partial g}{\partial n}-g\frac{\partial u}{\partial n})d\sigma
\end{eqnarray}
Letting $\delta\rightarrow 0$ we get that  the last integral in \eqref{3:6:1.6} is vanishing (see $u,\,g$ are continuously differentiable with bounded erivatives). In  addition, the second integral in the right hand side of \eqref{3:6:1.6} will become.
\begin{equation}\label{3:6:1.7}
\mathop{lim}\limits_{\delta\rightarrow}\frac{1}{4\pi}\mathop{\iint}\limits_{\sigma}(u\frac{\partial}{\partial n}|\frac{1}{r}-\frac{1}{r}\frac{\partial u}{\partial n})d\sigma=\mathop{lim}\limits_{\delta\rightarrow}\frac{1}{4\pi\delta^2}\mathop{\iint}\limits_{\sigma}ud\sigma-\mathop{lim}\limits_{\delta\rightarrow}\frac{1}{4\pi^2}\mathop{\iint}\limits_{\sigma}\delta\frac{\partial u}{\partial n}d\sigma=u(p_{0})
\end{equation}
\begin{equation}\label{3:6:1.8}
u(p_{0})=\mathop{\iint}\limits_{S} F_{0}(s)\frac{\partial u}{\partial n}(s)ds-\mathop{\iiint}\limits_{\Omega}G(p,p_{0})f(p)dxdydz
\end{equation}
 and the proof is complete.
\end{proof}
\begin{theorem}\label{th:3.6.3}
 The function $u(p_{0}),p_{0}\in \Omega $ given in Theorem \ref{th:3.6.2} is the solution of the Dirichlet problem \eqref{3:6:1.1}+\eqref{3:6:1.2}.
\end{theorem}
\begin{proof}
It is enough to prove the existence of the Dirichlet problem solution \eqref{3:6:1.1}+\eqref{3:6:1.2}. Let $\varphi$  be the newtonian potential with the density function$f(p)$ on the domain $\Omega$,\\
\begin{equation}\label{3:6:1.9}
\varphi (p_{0})=-\frac{1}{4\pi}
\mathop{\iiint}\limits_{\Omega}\frac{1}{r}f(p)dxdydz
\end{equation}
It is  known that $\varphi $ satisfies the following Poisson equation \index{Poisson equation}$\delta \varphi =f$. Define $\nu=u-\varphi$ and it has to satisfy
\begin{equation}\label{3:6:1.10}
\left\{
\begin{array}{ll}
  \nu|S=u|S-\varphi=F_{0}(S)=\nu(s) \\
 \Delta\nu(P)=0,P\in \Omega(\nu(P_{0})=\mathop{\iint}\limits_{S}v_{0}(s)\frac{\partial G}{\partial n}ds)
  \end{array}
   \right.
\end{equation}
The existence of the harmonic function\index{Harmonic function } $\nu$\,verifying \eqref{3:6:1.10} determines the unknown $u$\,as a solution of the Dirichlet problem \eqref{3:6:1.1}+\eqref{3:6:1.2}.
\end{proof}
\section{Linear Parabolic Equations}\index{Parabolic equation}
The simplest linear parabolic equation is defined by the following heat equation\index{Heat equation}\\
 \begin{equation}\label{3:7.1}
\partial _{t} u(t,x)=\partial_{x}^{2} u (t,x),\,t>0,x\in \mathbb{R}
\end{equation}
\begin{equation}\label{3:7.2}
\mathop{\lim}\limits_{t\rightarrow 0}u(t,x)=\varphi(x),x\in\mathbb{R}
\end{equation}
 $\varphi\in C_{b}(R)$ is fixed. A solution is a continuous function $u(t,x);[0,\infty)\times\mathbb{R}\rightarrow \mathbb{R}$ which is second order continuously
  differentiable of $x\in \mathbb{R}$ for each $t> o$ and satisfying \eqref{3:7.1} + \eqref{3:7.2} for some fixed continuous and bounded function
  $\varphi\in \mathcal{C}_{b}(\mathbb{R})$.
\subsection[The Unique Solution of the C.P \eqref{3:7.1} and \eqref{3:7.2}]{The Unique Solution of the Cauchy Problem \eqref{3:7.1} and \eqref{3:7.2}}
It is expressed as follows.
\begin{equation}\label{3:7.3}
u(t,x)=\frac{1}{\sqrt{4\pi t}}\mathop{\int}\limits_{-\infty}^{\infty} \varphi(\xi)\exp-\frac{(x-\xi)^2}{4t}d\xi
\end{equation}
(Poisson formula for heat equation) Denote
\begin{equation}\label{3:7.4}
P(\sigma,x,y)dy=(\sqrt{4\pi\sigma})^{-1}\exp -\frac{(y-x^2)}{4\sigma},for\, \sigma>0,\,x,y\in \mathbb{R}
\end{equation}
By a direct computation we get the following equations
\begin{equation}\label{3:7.5}
\mathop{\int}\limits_{R} P(\sigma,x,y)dy=1 and \partial_{\sigma}P(\sigma,x,y)-\partial_{x}^2P(\sigma,x,y)=0\, for \,any\, \sigma>0,x,y\in \mathbb{R} \end{equation}
The first equation of \eqref{3:7.5} is obtained using a change of variable $\frac{y-x}{2\sqrt{\sigma}}=z$ and
\begin{equation}\label{3:7.6}
\mathop{\int}\limits_{R}P(\sigma,x,y)dy=\sqrt{4\sigma}\mathop{\int}\limits_{R}P(\sigma,x,x+2\sqrt{\sigma}z)dz=
\frac{1}{\sqrt{\pi}}\mathop{\int}\limits_{R}(\exp-|z|^2)dt=1
\end{equation}
where $\mathop{\int}\limits_{R}(\exp-t^2)dt=\sqrt{\pi}$ is used.
In addition, the function $P(\sigma,x,y)$ defined in \eqref{3:7.4} fulfils  the following.
\begin{equation}\label{3:7.7}
\mathop{\lim}\limits_{\sigma\rightarrow 0}\int_{R}\varphi(y)P(\sigma,x,y)dy=\varphi(x),\,for\,each\,x\in\mathbb{R}
\end{equation}
where $\varphi\in \mathcal{C}_b(R)$ is fixed. Using \eqref{3:7.5} we see easily that $\{u(t,x):t>0,x\in\mathbb{R}\}$ defined in \eqref{3:7.3} satisfies heat equation \index{Heat equation}\eqref{3:7.1}, provided we notice that
\begin{equation}\label{3:7.8}
u(t,x)=\int_{R}\varphi(\xi)P(t,x,\xi)d\xi,\, t>0,\,x\in\mathbb{R}
\end{equation}
and
\begin{equation}\label{3:7.9}
\left\{
      \begin{array}{ll}
       \partial_t\,u(t,x)= \int_{R}\varphi(\xi)\partial_tP(t,x,\xi)d\xi,\\
       \partial_{x}^{2}\,u(t,x)= \int_{R}\varphi(\xi)\partial_{x}^{2}P(t,x,\xi)d\xi,
      \end{array}
    \right.
\end{equation}
The property \eqref{3:7.7} used for \eqref{3:7.8} allows one to get \eqref{3:7.2}.
\begin{remark}
The unique solution of the heat equation \index{Heat equation} \eqref{3:7.1} + \eqref{3:7.2} can be expressed by the formula \eqref{3:7.3} even if the continuous $\varphi(x):\mathbb{R}\rightarrow\mathbb{R}$ satisfies a  polynomial growth condition
\begin{equation}\label{3:7.10}
|\varphi(x)|\leq c(1+|x|^N),\,\forall\,x\in\mathbb{R}
\end{equation}
There is no change in proving that the unique Cauchy problem solution of the heat equation for $x\in\mathbb{R}$.
\begin{equation}\label{3:7.11}
\partial_tu(t,x)=\Delta_x\,u(t,x),\,t>0,\,x\in\mathbb{R}
\end{equation}
\begin{equation}\label{3:7.12}
\mathop{lim}_{t\rightarrow 0}\,u(t,x)=\varphi(x),x\in\mathbb{R}^n,\,\Delta_x=\mathop{\sum}\limits_{i=1}^{n}\partial_{x_i}^{2}\
\end{equation}
is expressed by
\begin{equation}\label{3:7.13}
u(t,x)=(4\pi\,t)^{-\frac{n}{2}}\int_{\mathbb{R}^n}\varphi(\xi)\exp-\frac{|x-\xi|^2}{4t}d\xi
\end{equation}
where $\varphi\in\mathcal{C}_b(\mathbb{R}^n)$.
\end{remark}
\subsection{Exercises}
$(a_{1})$ Find a continuous and bounded function $u(t,x);[0,T]\times R\rightarrow R$ fulfilling the following
$$\left\{
   \begin{array}{ll}
     (1)\, \partial_{t} u(t,u)=a^2 \partial_{x}^2 u (t,x) , t\in(0,T],\, x \in \mathbb{R},\, a> 0\\
     (2)\,\mathop{\lim}\limits_{t\rightarrow 0} u(t,x)=cosx,\,x\in \mathbb{R}.
   \end{array}
 \right.
$$
 \textbf{Hint}.\, The equation (1) can be written with the constant $a^2=1$ provided  we use a change of variable $\xi=\frac{x}{a}$ and denote
$$u(t,x)=\nu(t,\frac{x}{a}),\,t\in[0,T],x\in\mathbb{R}$$
Here $\nu(t,y):[0,T]\times\mathbb{R}\rightarrow\mathbb{R}$ satisfies the equation
$$\partial_t\nu(t,y)=\partial_{y}^{2}\nu(t,y),\,(t,y)\in(0,T]\times\mathbb{R}$$
and initial condition $\nu(0,y)=cos(ay)$
$(a_2)$. Find a continuous and bounded function\\
$u(t,x,y):[0,T]\times\mathbb{R}\times\mathbb{R}\rightarrow\mathbb{R}$ satisfying the following linear parabolic equation
$$\left\{
  \begin{array}{ll}
   \textbf{(1)}\,\partial_{t} u(t,x,y)=a^2 \partial_{x}^2 u (t,x,y)+b^2 \partial_{y}^2 u (t,x,y),\,t\in(0,T]\\
   \textbf{(2)}\,\mathop{\lim}\limits_{t\rightarrow 0} u(t,x,y)=sinx+cosy,\,a>0,\,b>0,\,x,y\in\mathbb{R}
  \end{array}
\right.
$$\\
\textbf{Hint}.\,The equation (1) and (2) will be rewritten using the following changes $y_1=\frac{x}{a},\,y_2=\frac{y}{b}$\,and\\
$u(t,x,y)=\nu(t,\frac{x}{a},\frac{y}{b})$ where $\nu(t,y_1,y_2):[0,T]\times\mathbb{R}\times\mathbb{R}\rightarrow\mathbb{R}$ satisfy (see\eqref{3:7.11}$n=2$)
$$\left\{
   \begin{array}{ll}
     \partial_t\,\nu=\partial_{y_1}^{2}\nu+ \partial_{y_2}^{2}\nu\\
   \mathop{\lim}\limits_{t\rightarrow 0}\nu(t,y_1,y_2)=sin(ay_1)+cos(by_2)
   \end{array}
 \right.
$$
\subsection{Maximum Principle for Heat Equation \index{Heat equation}}
Any continuous solution $u(t,x),\,(t,x)\in[0,T]\times[A,B]$, satisfying heat equation \eqref{3:7.1} for $0<t\leq T$\ and $x\in(A,B)$ will achieve its extreme values \\$\mathop{\max}_{(t,x)\in[0,T]\times [A,B]}u(t,x)$ and $\mathop{\min}_{(t,x)\in[0,T]\times [A,B]}u(t,x)$ on the boundary $\hat{\partial D}$ of the domain $D=\{(t,x)\in[o,T]\times[A,B]\}$, where
$$\hat{\partial D}=(\{0\}\times[A,B])\bigcup([0,T]\times\{A\})\bigcup([0,T]\times\{B\})$$
\begin{proof}
Denote $M=\mathop{\max}_{(t,x)\in D}u(t,x) \,,m=\mathop{\max}_{(t,x)\in\hat{\partial D}}u(t,x)$ and assume $M>m$. Let $(t_0,x_0)\in D,\,(t_0,x_0) \,not\,in\, \hat{\partial D}$ be such that $u(t_0,x_0)=M$ and consider the following auxiliary function
\begin{equation}\label{3:7.15}
\nu(t,x)=u(t,x)+\frac{M-m}{2(B-A)^2}(x-x_0)^2
\end{equation}
We see easily that $\nu(t,x)$ satisfies
\begin{equation}\label{3:7.16}
\nu(t,x)\leq u(t,x)+\frac{M-m}{2(B-A)^2}(B-A)^2\leq m+\frac{M-m}{2}<M
\end{equation}
for any \\$(t,x)\in\hat{\partial D}$, and $\nu(t_0,x_0)=u(t_0,x_0)=M$, consider $\nu(t_1,x_1)=\mathop{\max}_{(t,x)\in D}\nu(t,x)\geq M$. As a consequence, $(t_1,x_1)$ not in $\hat{\partial D}$ and
\begin{equation}\label{3:7.17}
\left\{
       \begin{array}{ll}
      \partial_t\nu(t_1,x_1)=0,\, \partial_x\nu(t_1,x_1)=0,\, \partial_{x}^{2}\nu(t_1,x_1)\leq0 ,\,if\,(t_1,x_1)\in int D\\
       \partial_t\nu(t_1,x_1)\geq 0,\, \partial_x\nu(t_1,x_1)=0,\, \partial_{x}^{2}\nu(t_1,x_1)\leq 0,\,if\,(t_1,x_1)\in \{T\}\times(A,B)
       \end{array}
     \right.
\end{equation}
In both cases $(\partial_t\nu-\partial_{x}^{2}\nu)(t_1,x_1)\geq 0$ and $(t_1,x_1)$ not in $\hat{\partial D}$. On the other hand,using the heat equation \index{Heat equation}\eqref{3:7.1} satisfied by $u(t,x)$ when $(t,x)$ not in $\hat{\partial D}$ we get (see\eqref{3:7.15})
\begin{equation}\label{3:7.18}
(\partial_t\nu-\partial_{x}^{2}\nu)(t_1,x_1)=-\frac{M-m}{(B-A)^2}<0
\end{equation}
contradicting the above given inequality.  It proves that$M\leq m$. Replacing $u$ with $\{-u\}$ and using $$\mathop{\max}\limits_{(t,x)\in D}\{-u(t,x)\}=\mathop{\max}\limits_{(t,x)\in \hat{\partial D}}\{-u(t,x)\}$$
we get the second conclusion
$$\mathop{\min}\limits_{(t,x)\in D}\{-u(t,x)\}=\mathop{\min}\limits_{(t,x)\in \hat{\partial D}}\{-u(t,x)\}$$
\end{proof}
\begin{remark}
In the above given proof we may assume that the heat equation \eqref{3:7.1} $\partial_tu(t,x)=\partial_{x}^{2}u(t,x)$ is satisfied for any $(t,x)\in (0,T)\times(A,B)$ (omitting $t=T$) and the result is still valid noticing that $u(t,x)\leq m$ for $0\leq t\leq T-\varepsilon$ will imply $u(t,x)\leq m$\,for any $0\leq t\leq T$ (see $u$ is continuous). In addition, the conclusion of the maximum principle allows to extend it for $|u(t,x)|,\,i.e$
$|u(t,x)|\leq \max\{|u(t,x)|:(t,x)\in\hat{\partial D}\}$.
\end{remark}
\begin{remark}
The computation and arguments used for the scalar heat equation \index{Heat equation } \eqref{3:7.1} can be extended to the case $x\in\mathbb{R}^n$ replacing the equation \eqref{3:7.1} by \eqref{3:7.11} and considering a ball $B(P_*,\rho)\subseteq \mathbb{R}^n$ instead of $[A,B]\subseteq\mathbb{R}$. If it is the case, the corresponding maximum principle associated with heat equation (\eqref{3:7.11} says
$$\mathop{\max}\limits_{(t,x)\in D}u(t,x)=\mathop{\max}\limits_{(t,x)\in \hat{\partial D}}u(t,x)$$
and
$$\mathop{\min}\limits_{(t,x)\in D}u(t,x)=\mathop{\min}\limits_{(t,x)\in \hat{\partial D}}u(t,x)$$
where\\
$D=[0,T]\times B(P_*,\rho),\,\hat{\partial D}=(\{0\}\times B(P_*,\rho))\bigcup([0,T]\times \partial B),\,P_*\in\mathbb{R}^n,(fixed),\partial B=boundary\,of\,B(P_*,\rho)$.\\
Here the heat equation\index{Heat equation}\eqref{3:7.11} is assumed on the domain $(t,x)\in(0,T]\times int\,B(P_0,\rho)$ and the corresponding auxiliary function is given by(see\eqref{3:7.15})
$$\nu(t,x)=u(t,x)+\frac{M-m}{8\rho^2}|x-x_0|^2,$$where
$$u(t_0,x_0)=M=\mathop{\max}\limits_{(t,x)\in D}u(t,x)=m=\mathop{\max}\limits_{(t,x)\in \hat{\partial D}}u(t,x)$$ and $M=u(t_0,x_0)>m$ will lead us to a contradiction.
\end{remark}
\begin{remark}\label{rk;3.7.3}
Using maximum principle for heat equation \index{Heat equation} we get that the Cauchy problem solution for \eqref{3:7.1}+\eqref{3:7.2}(or \eqref{3:7.11}+\eqref{3:7.12})) is unique provided the initial condition $\varphi$\,and the solution $\{u(t,x):t\geq 0,x\in\mathbb{R}^n\}$ are restricted to the bounded continuous functions. In this respect,let $M>0$ be such that
$$|u(t,x)|\leq M,\,|\varphi(x)|\leq M\,for\, any\, t\geq 0,\,x\in\mathbb{R}(x\in\mathbb{R}^n)$$
Consider the following particular solution of \eqref{3:7.1}
$$\nu(t,x)=\frac{2M}{L^2}(x^2+2t)\,\,\,satisfying\,\,
\partial_t\nu=\partial_{x}^{2}\nu\,and\,\nu(0,x)=\frac{2M}{L^2}x^2\geq 0$$
$$\nu(\mathop{+}\limits_{-}L,t)=\frac{2M}{L^2}(L^2+2t)\geq 2M$$
If \eqref{3:7.1}+\eqref{3:7.2} has two bounded solutions then their difference $u(t,x)=u_1(t,x)-u_2(t,x)$ is a solution of heat equation \eqref{3:7.1}, satisfying $|u(t,x)|\leq 2M,\,(t,x)\in[0,\infty)\times \mathbb{R}$\,,and $u(0,x)=0$.On the other hand $h(t,x)=\nu(t,x)-u(t,x)$ satisfies \eqref{3:7.1} for any $(t,x)\in\, int\,D,\,D=[-L,L]\times[0,T]$\,and
$$h(t,x)=\nu(t,x)-u(t,x)\geq 0,\,\forall\,(t,x)\in\hat{\partial D}$$
Using a maximum principle we get
$$ u(t,x)\leq\frac{2M}{L^2}(x^2+2t)\,\forall\,(t,x)\in[0,T]\times[-L,L]$$
and similarly for $\overline{u}(t,x)=-u(t,x)$\,we obtain\\
$$ -u(t,x)\leq\frac{2M}{L^2}(x^2+2t)\,\forall\,(t,x)\in[0,T]\times[-L,L]$$
Combining the last two inequalities we get
$$|u(t,x)|\leq \frac{2M}{L^2}(x^2+2t)\,\forall\,(t,x)\in[0,T]\times[-L,L]$$
and for any arbitrary fixed $(t,x)(t>0)$ letting $L\uparrow\infty$ we obtain $u(t,x)=0$ which proves $u_1(t,x)=u_2(t,x)$ for each $t>0$.
\end{remark}
\noindent \textbf{Exercise}.\,Use the above given algorithm($n=1$)for the multidimensional heat equation \index{Heat equation}(11)($n\geq 1$) and get the conclusion:the Cauchy problem solution  of \eqref{3:7.11} + \eqref{3:7.12} is unique provided initial condition $\varphi$\,and the solutions$\{u(t,x):t\geq 0,x\in\mathbb{R}\}$\,are restricted to the bounded continuous functions.\\
\textbf{Hint}.\, Let $M>0$\,be such that $|u(t,x)|,|\varphi(x)|\leq M$\,for any $t\geq 0,x\in\mathbb{R}^n$\,and consider $V(t,x=\frac{2M}{L^2})(|x|^2+2t)$
satisfying \eqref{3:7.11} and \eqref{3:7.12} with $V(0,x)=\frac{2M}{L^2}|x|^2\geq 0(V(t,x)=\frac{2M}{L^2})(|x|^2+2t)\geq 2M\,if\,x\in \partial B(0,L)$. Proceed as in Remark \eqref{rk;3.7.3}.\\
\textbf{Problem $P_1$(Maximum Principle for Linear Elliptic Equation)}\\\index{Elliptic equation}
Consider the following linear elliptic equation \index{Elliptic equation}
\begin{eqnarray}\label{3:7:4.1}
  0 &=& \mathop{\sum}\limits_{i,j=1}^{n}a_{ij}\partial_{x_ix_j}^{2}\,u(x) \nonumber\\
  &=& \hbox{Trace}[A.\partial_{x}^{2}u(x)] \nonumber\\
  &=&\hbox{ Trace}[\partial_{x}^{2}u(x)A],\,x\in\Omega\hbox{ (bounded domain) }\subseteq\mathbb{R}^n
\end{eqnarray}
where the symmetric matrix \index{Matrix!symmetric}$A=(a_{i,j\in\{1,...,n\}})$ is strictly positive
definite($<x,A\,x>\geq \delta\parallel x\parallel^2,\,\forall\,x\in\mathbb{R}^n,\,for\,some\,\delta>0$). Under the above given conditions,
using an adequate transformation of coordinates and function,
we get a standard Laplace equation\index{Laplace equation} in $\mathbb{R}^n$\,for which the maximum principle is valid.\\
(R) Show that for a continuous and bounded function $u(x):\overline{\Omega}\rightarrow\mathbb{R}$ satisfying \eqref{3:7:4.1}
for any $x\in\Omega$ we get the following maximum principle:
$$\mathop{\max}\limits_{x\in\overline{\Omega}}u(x)=\mathop{\max}\limits_{x\in\Gamma}u(x)\,(\mathop{\min}\limits_{x\in\overline{\Omega}}u(x)=\mathop{\min}\limits_{x\in\Gamma}u(x)),\,where \,\overline{\Omega}=\Omega\sqcup\Gamma,\,\Gamma=\partial\Omega$$
\textbf{Hint.} Let $T:\mathbb{R}^n\rightarrow\mathbb{R}^n$\,be and orthogonal matrix \index{Matrix!orthogonal}$(T^*=T^{-1})$\,such that
\begin{equation}
A=T\,D\,T^{-1},\,where\,D=\diag(d_1,\dots,d_n),\,d_i>0
\end{equation}
Define $A^{1\setminus2}$(square root of the matrix $A$)$=T\,D^{\frac{1}{2}}T^{-1}$\,and make the following transformations
\begin{equation}
x=A^{1\setminus2}y,\,\nu(y)=u(A^{1\setminus2}y)
\end{equation}
Then notice that$\{\nu(y):y\in\overline{\Omega_1}\}$\,is a harmonic function\index{Harmonic function} satisfying\\
\begin{equation}\label{3:7:4.4}0=\Delta\nu(y)=Trace[\partial_{y}^{2}\nu(y)],\,\forall\,y\in\Omega_{1}
\end{equation}
where
$$\Omega_1=[A^{1\setminus2}]^{-1}\Omega \hbox{and} \Gamma_1=\partial\Omega_1=[A^{1\setminus2}]^{-1}\Gamma$$
As far as the maximum principle is valid for $\nu(y):\overline{\Omega_1}\rightarrow \mathbb{R}$ satisfying \eqref{3:7:4.4}we get that
$$u(x)=\nu([A^{1\setminus2}]^{-1}x),\,x\in\overline{\Omega}$$
satisfies the maximum principle too.\\
\textbf{Problen $P_2$(Maximum Princie for Linear Parabolic Equations)}\\
Consider the following linear parabolic equation\\
\begin{equation}\label{3:7:5.1}
\partial_tu(t,x)=\mathop{\sum}\limits_{i,j=1}^{n}a_{ij}\partial_{x_ix_j}^{2}\,u(t,x)=Trace[A.\partial_{x}^{2}u(t,x)]
=Trace[\partial_{x}^{2}u(t,x)A]
\end{equation}
$t\in(0,T]\,,x\in\Omega(bounded\,domain)\subseteq\mathbb{R}^n$ and let $u(t,x):[0,T]\times\overline{\Omega}\rightarrow\mathbb{R}$
be a continuous and bounded function satisfying the parabolic equation\index{Parabolic equation}\eqref{3:7:5.1}. If the matrix $A=(a_{ij})_{i,j\in\{1,...,n\}}$
is symmetric and strictly positive($<x,A\,x>\geq \delta\parallel x\parallel^2,\,\forall\,x\in\mathbb{R}^n,\,for\,some\,\delta>0$)then $\{u(x):x\in\overline{\Omega}\}$ satisfies the following maximum principle:
\begin{equation}
\mathop{\max}\limits_{(t,x)\in D}u(t,x)=\mathop{\max}\limits_{(t,x)\in\hat{\partial D}}u(t,x)\,(\mathop{\min}\limits_{(t,x)\in D}u(t,x)=\mathop{\min}\limits_{(t,x)\in\hat{\partial D}}u(t,x))
\end{equation}
where $D=[0,T]\times\overline{\Omega}\subseteq\mathbb{R}^{n+1}\,and\,\hat{\partial D}=(\{0\}\times \overline{\Omega})\sqcup([0,T]\times \partial \Omega)$.\\
\textbf{Hint}.\,The verification is based on the canonical form we may obtain in the right hand side of \eqref{3:7:5.1} provided the following transformations are performed
\begin{equation}
x=A^{1\setminus2}y,\,\nu(t,y)=u(t,A^{1\setminus2}y),\,y\in\Omega_1=[A^{1\setminus2}]^{-1}\Omega
\end{equation}
where the square root of a symmetric and positive matrix\index{Matrix!positive} $A^{1\setminus2}=T(\Gamma)^{\frac{1}{2}}T^{-1}$ is used. Here $\Gamma=\diag(\gamma_1,...,\gamma_n),\,\gamma_i>0$\,and $T:\mathbb{R}^n\rightarrow\mathbb{R}^n$\,is an orthogonal matrix ($T^*=T^{-1}$) such that $A=T\,\Gamma\,T^{-1}$. Notice that using \eqref{3:7:5.1} we get that $\{\nu(t,x):t\in[0,T],y\in\Omega_1\}$ satisfies the standard heat equation \index{Heat equation}
\begin{equation}
\partial_t\nu(t,y)=\Delta_y\nu(t,y)=Trace[\partial_{y}^{2}\nu(t,y)],\,t\in[0,t],\,y\in\Omega_1
\end{equation}
and the corresponding maximum principle
\begin{equation}\label{3:7:5.5}
\mathop{\max}\limits_{(t,y)\in D_1}\nu(t,x)=\mathop{\max}\limits_{(t,y)\in\hat{\partial D_1}}\nu(t,x)\,(\mathop{\min}\limits_{(t,y)\in D_1}\nu(t,x)=\mathop{\min}\limits_{(t,y)\in\hat{\partial D_1}}\nu(t,x))
\end{equation}
is valid,where $D_1=[0,T]\times\overline{D_1}\,and\,\hat{\partial D_1}=(\{0\}\times \overline{\Omega_1})\sqcup([0,T]\times \partial \Omega_1)$. Using $u(t,x)=\nu(t,[A^{1\setminus2}]^{-1}x),\,x\in\Omega$,\,and \eqref{3:7:5.5} we obtain the conclusion \eqref{3:7:5.5}.
\section{Weak Solutions(Generalized Solutions)}
\textbf{Separation of Variables(Fourier Method)\index{Fourier method}}\\
Boundary problems for parabolic and hyperbolic $PDE$\,can be solved using Fourier method.We shall confine ourselves to consider the following two types of $PDE$

$$(I)\,\,\,\,\,\,\,\frac{1}{a^2}\partial_tu(t,x,z)=\Delta u(t,x,y,z),\,t\in[0,T],\,(x,y,z)\in D\subseteq \mathbb{R}^3$$\\
$$(II)\,\,\,\,\,\,\,\partial_{t}^{2}u(t,x,z)=\Delta u(t,x,y,z),\,t\in[0,T],\,(x,y,z)\in D\subseteq \mathbb{R}^3$$\\
where the bounded  domain $D$\,has the boundary $S=\partial D$. The parabolic equation (I) is augmented with initial conditions(Cauchy conditions)
$$(I_a)\,\,\,\,\,\,\,\,\,\,\,\,\,u(0,x,y,z)=\varphi(x,y,z),\,(x,y,z)\in D$$
and the boundary conditions
$$(I_b) \,\,\,\,\,\,\,\,\,\,\,\,\,u(t,x,y,z)|_{(x,y,z)\in S}=0\,\,\,t\in[0,T]$$
The hyperbolic equation (II) is augmented with initial conditions
$$(II_{a})\,\,\,\,\,\,\,\,\,\,\,\,\, u(0,x,y,z)=\varphi_0(x,y,z),\,\partial_tu(0,x,y,z)=\varphi_1(x,y,z),\,(x,y,z)\in D$$
and the boundary conditions
$$(II_{b})\,\,\,\,\,\,\,\,\,\,\,\,\,\frac{\partial u}{\partial n}(t,x,y,z)=(\omega_1\partial_x u+\omega_2\partial_y u+\omega_3\partial_z u)|{(x,y,z)\in S}=0$$
where $n=(\omega_1,\omega_2,\omega_3)$\,is the unit orthogonal vector at $S$ oriented outside of $D$.
\subsection{Boundary Parabolic Problem}
To solve the mixed problem ($I,I_a,I_b$) we shall consider the particular solutions satisfying ($I,I_a,I_b$) in the form
\begin{equation}\label{3:8.1}
u(t,x,y,z)=T(t)U(x,y,z)
\end{equation}
It lead us directly to the following equations
\begin{equation}\label{3:8.2}
\frac{\Delta U(x,y,z)}{U(x,y,z)}=\frac{1}{a^2}\frac{T'(t)}{T(t)},\,t\in[0,T],\,(x,y,z)\in D
\end{equation}
and it implies that each term in \eqref{3:8.2} equals a constant $-\lambda$ and we obtain the following equations
\begin{equation}\label{3:8.3}
\Delta U+\lambda U=0;T'(t)+a^2\lambda T(t)=0
\end{equation}
(see $T(t)=C \exp-\lambda a^2 t,t\in[0,T]$). Using \eqref{3:8.3} and the boundary conditions ($I_b$) we get
\begin{equation}\label{3:8.4}
U(x,y,z)|_{(x,y,z)\in S}=0
\end{equation}
The values of the parameter $\lambda$ for which \eqref{3:8.3}+\eqref{3:8.4} has a solution are called eigenvalues\index{Eigen!values} associated with linear elliptic equation \index{Elliptic equation}$\Delta U+\lambda U=0$ and boundary condition \eqref{3:8.4}. The corresponding eigenvalues are found provided a Green function \index{Green!function}is used which allow us to rewrite the elliptic equation \index{Elliptic equation}as a Fredholm integral equation\index{Integral equation}. Recall the definition of a Green function.
\begin{definition}
Let $S=\partial D$ be defined by second order continuously differentiable function.A Green function for the Dirichlet problem $\Delta U=f(P),\,U|_{S}=F_0(S)$ is a symmetric function $G(P,P_0)$ satisfying the following conditions with respect to $P\in D$\\
$(\alpha)\,\Delta G(P,P_0)=0,\,\forall\,P\in D,P\neq P_0$\,,where $P_0\in D$ is fixed,\\
$(\beta)\, G(P,P_0)|_{S}=0,\,G(P,P_0)=G(P_0,P),\,and\,G(P,P_0)=\frac{1}{4\pi r}+g(P,P_0),\,r=|P-P_0|$\\
$(\gamma)\,g(P,P_0)$ is second order continuously differentiable and $\Delta g(P,P_0)=0,\,\forall\,P\in D$.
$(\delta)\,U(P_0)=\mathop{\iint}\limits_{S}F_0(s)\partial_nG(s)ds-\mathop{\iiint}\limits_{D}G(P,P_0)f(P)dxdydz$
where $f(P)=-\lambda U(P)$ and $F_0(s)=0,s\in S$
\end{definition}
The solution of the Dirichlet problem \eqref{3:8.3}+\eqref{3:8.4} can be expressed as in Theorem \ref{th:3.6.2}. Under these conditions,the integral representation formula ($\delta$) lead us to the following Fredholm integral equation
\begin{equation}\label{3:8.5}
U(P_0)=\lambda\mathop{\iiint}\limits_{D}G(P,P_0)U(P)dxdydz
\end{equation}
where $G(P,P_0)=\frac{1}{4\pi r}+g(P,P_0)$\, is an unbounded function (see $\frac{1}{r}$) verifying
\begin{equation}\label{3:8.6}
G(P,P_0)\leq \frac{1}{r^{\alpha}},\,0<\alpha<3,\,for\,some\,constant\,\,A>0
\end{equation}
Define
$$G^*(P,P_0)=\min(G(P,P_0),\frac{A}{\delta^{\alpha}}),\,where\,\delta>0\,\,is\,\,fixed$$
We get $G(P,P_0)-G^*(P,P_0)\geq 0\,\,\forall\,P\in D$ and
\begin{equation}\label{3:8.7}
0\geq G(P,P_0)-G^*(P,P_0)\leq A(\frac{1}{r^{\alpha}}-\frac{1}{\delta^{\alpha}})\,\,if\,\,r\leq\delta
\end{equation}
Using \eqref{3:8.7} and $\delta$\,sufficiently small we obtain
\begin{eqnarray}\label{3:8.8}
  \mathop{\iiint}\limits_{D}|G(P,P_0)-G^*(P,P_0)|dP &=& \mathop{\iiint}\limits_{D}[G(P,P_0) - G^*(P,P_0)]dP \nonumber\\
  &\leq&  A\mathop{\iiint}\limits_{\{r\leq \delta\}}\frac{1}{r^{\alpha}}dP\leq \frac{\varepsilon}{2}\hbox{ where }\varepsilon>0
\end{eqnarray}
is arbitrarily fixed.On the other hand, $G^*(P,P_0)$\,is a continuous and bounded function for $P\in D$ and approximate it by a degenerate kernel
\begin{equation}\label{3:8.9}
G^*(P,P_0)=\mathop{\sum}\limits_{i=1}^{N}\varphi_i(P)\psi_i(P_0)+G_2(P,P_0)
\end{equation}
where $$\mathop{\iiint}\limits_{D}|G_2(P,P_0)|dP\leq \frac{\varepsilon}{2}$$
From \eqref{3:8.8} and \eqref{3:8.9} we get that $G(P,P_0)$ in \eqref{3:8.5} can be rewritten as
\begin{equation}\label{3:8.10}
G(P,P_0)=\mathop{\sum}\limits_{i=1}^{N}\varphi_i(P)\psi_i(P_0)+G_1(P,P_0)
\end{equation}
where $G_1(P,P_0)=G_2(P,P_0)+[G(P,P_0)-G^*(P,P_0)]$ satisfies
\begin{equation}\label{3:8.11}
\mathop{\iiint}\limits_{D}|G_1(P,P_0)|dP\leq\,\frac{\varepsilon}{2}+\frac{\varepsilon}{2}=\varepsilon
\end{equation}
Using these remarks we replace the equation \eqref{3:8.5} by the following one
\begin{eqnarray}\label{3:8.12}
  U(P_0)-\lambda\mathop{\iiint}\limits_{D}G_1(P,P_0)U(P)dP &=& [(E-\lambda A_1)U](P_0) \\
  &=& \lambda\mathop{\sum}\limits_{i=1}^{N}\psi_i(P_0) \mathop{\iiint}\limits_{D}\varphi_i(P)U(P)dP \nonumber
\end{eqnarray}
where the operator $B_1\varphi=[E-\lambda A_1](\varphi)$ has an inverse
\begin{equation}\label{3:8.13}
B_{1}^{-1}=[E+\lambda A_1+\lambda^2 A_{1}^{2}+...+\lambda^kA_{1}^{k}+...],\hbox{ for any }|\lambda|<\frac{1}{\parallel A_1\parallel},\hbox{ where }\parallel A_1\parallel\leq \frac{\varepsilon}{2}
\end{equation}
is acting from $\mathcal{C}(D)$\,to $\mathcal{C}(D)$.Denote $\xi_i(P_0)=(B_{1}^{-1}\psi_i)(P_0)$ and rewrite \eqref{3:8.12} as the following equation
\begin{equation}\label{3:8.14}
U(P_0)=\lambda\mathop{\sum}\limits_{i=1}^{N}\xi_i(P_0) \mathop{\iiint}\limits_{D}\varphi_i(P)U(P)dP
\end{equation}
which has a nontrivial solution for any $\lambda\in\{\lambda_1,\lambda_2,...\}$ where the sequence $\{\lambda_j\}_{j\geq 1}$ of real numbers satisfies $|\lambda|\leq C$ only for a finite terms,for each constant $C>0$\,arbitrarily fixed. Let$\{U_j\}_{j\geq 1}$ be a sequence of solutions associated with equation \eqref{3:8.14}  and eigenvalues\index{Eigen!values} $\{\lambda_j\}_{j\geq 1}$, they are called eigen functions. Notice that (see $(\beta)$) $G(P,P_0)$ is a symmetric function which allows one to see that the eigenvalues are positive numbers, $\lambda_j>0$, and the corresponding eigenfunctions\index{Eigen!functions} $\{U_j\}_{j\geq 1}$ can be taken such that
\begin{equation}\label{3:8.15}
\mathop{\iiint}\limits_{D} U_i(P)U_j(P)dP=\left\{
                                                        \begin{array}{ll}
                                                         0\,\,\,i\neq j \\
                                                         1\,\,\,i=j
                                                        \end{array}
                                                      \right.
\end{equation}
$\{U_j\}_{j\geq 1}\hbox{ is a complete system in }\hat{C}(D)\subseteq C(D)\hbox{ i.e any }\varphi\in\hat{C}(D) \hbox{ can be represented }$
\begin{equation}\label{3:8.16}
\varphi(P)=\mathop{\sum}\limits_{j=1}^{\infty}a_j U_j(P)
\end{equation}
where the series is convergent in $L_2(D)\\(see\,\mathop{\sum}\limits_{j=1}^{\infty}a_{j}^{2} <\infty)$ and\\
$$\mathop{lim}\limits_{N\rightarrow\infty}(\mathop{\iiint}\limits_{D}|\mathop{\sum}\limits_{j=1}^{N}a_jU_j(P)-\varphi(P)|^{2}dP)^{1\setminus2}=0$$\\
The coefficients $\{a_j\}_{j\geq 1}$\,describing the continuous function $\varphi\in\hat{C}(D)$ are called Fourier coefficients\index{Fourier coefficients} and satisfy
\begin{equation}\label{3:8.17}
\mathop{\iiint}\limits_{D}\varphi(P)U_j(P)dP=a_j,\,j\geq 1,\,\varphi\in C(D)
\end{equation}
is fixed in $I_a$. Now we are in position to define a weak solution of mixed problem ($I,I_a,I_b$) and it will be given as the following series
\begin{equation}\label{3:8.18}
u(t,x,y,z)=\mathop{\sum}\limits_{i=1}^{\infty}(\exp-\lambda_ia^2t)a_iU_i(x,y,z)
\end{equation}
which is convergent in $L_2(D)$, uniformly with respect to $t\in[0,T]$. Using \eqref{3:8.17} and \eqref{3:8.16} the initial condition in ($I_a$) is satisfied in a weak sense, i.e
$$\mathop{lim}\limits_{N\rightarrow\infty}\parallel u_N(0,P)-\varphi(.)\parallel_2=0,u_N(t,x,y,z)=\mathop{\sum}\limits_{i=1}^{\infty}(\exp-\lambda_ia^2t)a_iU_i(x,y,z)$$
Here $\{u_N(t,x,y,z):t\in[0,T],\,(x,y,z)\in D\}_{N\geq 1}$ is defined as a sequence of solutions
\begin{equation}\label{3:8.19}
u_N(t,x,y,z)=\mathop{\sum}\limits_{i=1}^{N}(\exp-\lambda_ia^2t)a_iU_i(x,y,z)
\end{equation}
satisfying the parabolic equation\index{Parabolic equation} \eqref{3:8.1}, the boundary condition ($I_b$) and
\begin{equation}\label{3:8.20}
u_N(0,x,y,z)=\varphi_N(x,y,z),\,(x,y,z)\in D
\end{equation}
such that, $\mathop{lim}\limits_{N\rightarrow\infty}\parallel \varphi-\varphi_N\parallel_2=0$\,(($I_a$) is weakly satisfied).
\subsection{Boundary  Hyperbolic Problem}To solve the mixed problem ($II,II_a,II_b$) we shall proceed as in the parabolic case and look for a particular solution
\begin{equation}\label{3:8.21}
u(t,P)=T(t)U(P),\,u\in\mathcal{C}^2(D),\,T\in\mathcal{C}^2([0,T])
\end{equation}
satisfying (II) and the boundary condition($II_b$). The function \eqref{3:8.21} satisfies (II) if
\begin{equation}\label{3:8.22}
T(t)\,\Delta U(P)=U(P)T''(t)\,or\,\frac{T''(t)}{T(t)}=\frac{\Delta\,U(P)}{U(P)}=-\lambda^2(const)
\end{equation}
which imply the equations
\begin{equation}\label{3:8.23}
T''(t)+\lambda^2\,T(t)=0
\end{equation}
\begin{equation}\label{3:8.24}
\Delta U(P)+\lambda^2U(P)=0,\,\frac{\partial U}{\partial n}|_S=0
\end{equation}
For the Ne\'{u}mann problem solution in \eqref{3:8.24}, we use a Green function\index{Green!function} $G_1(P,P_0)$ satisfying
\begin{equation}\label{3:8.25}
\left\{
        \begin{array}{ll}
         \Delta\,G_1=\frac{1}{C}(where\,\,C=vol\,D)\,P\neq P_0 \\
         \partial_nG_1|_S=0
        \end{array}
      \right.
\end{equation}
In this case $\frac{1}{C}$\,stands for the solution of the adjoint equation\index{Adjoint!equation }
\begin{equation}\label{3:8.26}
\Delta\,\psi=0,\,\partial_n \psi|_S=0
\end{equation}
The Green function $G_1$\,has the structure
\begin{equation}\label{3:8.27}
G_1(P,P_0)=\frac{1}{4\pi r}+g(P),\,r=|P,P_0|,\,g(P)=\alpha|P|^2+g_1(P)
\end{equation}
where $g_1$\,verifies
\begin{equation}\label{3:8.28}
\left\{
        \begin{array}{ll}
          \Delta\,g_1(P)=0,\,\forall\,P\in D,and \\
         \partial_n\,g_1|_S=-[\alpha\partial_n|P|^2+\frac{1}{4\pi}\,\partial_n(\frac{1}{r})]|_S
        \end{array}
      \right.
\end{equation}
The constant $\alpha$ is found such that
\begin{eqnarray}\label{3:8.29}
  \mathop{\iiint}\limits_{D}(\Delta\,G_1)dxdydz &=& \frac{1}{4\pi}\mathop{\iiint}\limits_{D}\Delta(\frac{1}{r})dxdydz+6\alpha\,vol(D) \nonumber\\
  &=& -1+6\alpha\,vol(D)\Rightarrow \alpha=1\setminus3val(D)
\end{eqnarray}
In addition, using the Green function\index{Green!function} $G_1$ we construct a weak solution for the Ne\'{u}mann problem
\begin{equation}\label{3:8.30}
\Delta U(P)=f(P),\,\,\,(\partial_nU)|S=0
\end{equation}
assuming that\\
\begin{equation}\label{3:8.31}
\mathop{\iiint}\limits_{D}f(P)dxdydz=0
\end{equation}
The Green formula \eqref{3:8.4} used for $\nu=G_1$\,and $\{U(P:P\in D)\}$ satisfying \eqref{3:8.30}
and (31) will get the form
\begin{equation}\label{3:8.32}
\mathop{\iiint}\limits_{D}u(P)\Delta\,G_1(P,P_0))dxdydz=\mathop{\iiint}\limits_{D}G_1(P,P_0)f(P)dxdydz
\end{equation}
Looking for solution of \eqref{3:8.24} which verify
\begin{equation}\label{3:8.33}
\mathop{\iiint}\limits_{D}U(P)dxdydz=0(see\,\,f(P)=-\lambda^2\,U(P)\,in\,\,(30))
\end{equation}
from \eqref{3:8.25} and \eqref{3:8.32} we obtain an integral equation\index{Integral equation}
\begin{equation}\label{3:8.34}
U(P_0)=\lambda^2\mathop{\iiint}\limits_{D}G_1(P,P_0)U(P)dxdydz,\,P=(x,y,z)
\end{equation}
which has a symmetric kernel $G_1(P,P_0)$. As in the case of the parabolic mixed problem we get a sequence of eigenvalues\index{Eigen!values} and the corresponding eigenfunctions $\{U_j\}_{j\geq 1}$. In this case they are satisfying
\begin{equation}\label{3:8.35}
U_j\in\mathcal{C}^2(D),\,\partial_nU_j|_S=0\,and\,\,\mathop{\iiint}\limits_{D}U_j(P)dxdydz=0,\,j\geq 1
\end{equation}
Define the space $\hat{C_0}(D)\subseteq \hat{C}^1(D)$ consisting from all continuously differentiable functions $\varphi\in \mathcal{C}^1(D)$ verifying
\begin{equation}\label{3:8.36}
(\partial_n\varphi)(P\in S)=0,\mathop{\iiint}\limits_{D}\varphi(P)dxdydz=0\,and\,
\varphi(P)=\mathop{\sum}\limits_{j=1}^{\infty}a_j U_j(P),\,\mathop{\sum}\limits_{j=1}^{\infty}|a_j|^2<\infty
\end{equation}
The boundary problem ($(II),(II_b)$) \,has two independent solutions\\ (see $T_1(t)=cos\lambda_j t,T_2(t)=sin\lambda_j t$)
\begin{equation}\label{3:8.37}
U_j(P)cos\lambda_j t\,and\,U_j(P)sin\lambda_j t,\,\,for each\, j\geq 1
\end{equation}
and we are looking for a solution of the mixed problem ($(II),(II_a),(II_b)$) \,as a convergent series
\begin{equation}\label{3:8.38}
U(t,P)=\mathop{\sum}\limits_{j=1}^{\infty}[a_jU_j(P)cos\lambda_j t+b_jU_j(P)sin\lambda_j t]+b_0t
\end{equation}
in $L_2(D)$\,with respect to $P=(x,y,z)$ and uniformly with respect to $t\in[0,T]$. Here $\{a_j\}_{j\geq 1}$ must be determined as the Fourier coefficients\index{Fourier coefficient} associated with initial condition $\varphi_0\in\hat{C}_0(D)$
\begin{equation}\label{3:8.39}
\varphi_0(P)=u(0,P)=\mathop{\sum}\limits_{j=1}^{\infty}\alpha_j u_j(P)
\end{equation}
and $\{b_j\}_{j\geq 0}$ are found such that the second initial conditions $\partial_t u(0,P)=\varphi_1(P)$ (see$(II_a) and \varphi_1\in\hat{C}(D)$) are satisfied
\begin{equation}\label{3:8.40}
\left\{
        \begin{array}{ll}
        \varphi_1(P)=\mathop{\sum}\limits_{j=1}^{\infty}\beta_j U_j(P)+\beta_0,\,with\,\mathop{\sum}\limits_{j=1}^{\infty}(\beta_j)^2<\infty \\
          \partial_t u(0,P)=\mathop{\sum}\limits_{j=1}^{\infty}\lambda_jb_jU_j(P)+b_0=\varphi_1(P)
        \end{array}
      \right.
\end{equation}
Here $\hat{C}(D)=\mathcal{C}^1(D)$ is consisting from all continuously differentiable functions $\varphi_1(P)\in \mathcal{C}^1(D)$\,satisfying
\begin{equation}\label{3:8.41}
(\partial_n \varphi_1)(P\in S)=0\,\,and\,\,\varphi_1(P)=\mathop{\sum}\limits_{j=1}^{\infty}\beta_j U_j(P)+\beta_0\,,\,\,\mathop{\sum}\limits_{j=1}^{\infty}|\beta_j|^2<\infty
\end{equation}
We get
\begin{equation}\label{3:8.42}
b_0=\beta_0\,\,and\,\,\lambda_j\,b_j=\beta_j\,,\,j\geq 1
\end{equation}
In conclusion,the mixed hyperbolic problem  $(II),(II_a),(II_b)$ has a generalized (weak) solution
$$u(t,.):[0,T]\rightarrow L_2(D)$$
$$u(t,P)=\mathop{\sum}\limits_{j=1}^{\infty}[a_jcos\lambda_j t+b_jsin\lambda_j t]U_j(P)+b_0t$$
such that
$$U_N(t,P)=\mathop{\sum}\limits_{j=1}^{N}[a_jcos\lambda_j t+b_jsin\lambda_j t]U_j(P)+b_0t$$
satisfies ($II$) and ($II_a$) and ($II_b$)\,is fulfilled in a "weak sense" $$u_N(0,P)=\varphi_{0}^{N}(P), \partial_t u_N(0,P)=\varphi_{1}^{N}(P)$$ Here the "weak sense" means $$\mathop{lim}\limits_{N\rightarrow \infty}\varphi_{0}^{N}=\varphi_0\,\,in\,L_2(D)$$ and $$\mathop{lim}\limits_{N\rightarrow \infty}\varphi_{1}^{N}=\varphi_1\,\,in\,L_2(D)$$
\subsection{Fourier Method, Exercises}
\textbf{Exercise 1}\\
Solve the following mixed problem for a $PDE$ of a parabolic type using Fourier method
$$\left\{
    \begin{array}{ll}
    (I)\,\frac{1}{a^2}\partial_t u(t,x)=\partial_{x}^{2}u(t,x),\,\,\,t\in[0,T],\,x\in[0,1] \\
      (I_a)\,u(0,x)=\varphi(x),\,x\in[0,1],\,\,\,\varphi\in\mathcal{C_0}([0,1];\mathbb{R}) \\
      (I_b)\,u(t,0)=u(t,1)=0,\,\,\,t\in[0,T]
    \end{array}
  \right.
$$
where $\mathcal{C_0}([0,1];\mathbb{R})=\{\varphi\in\mathcal{C}([0,1];\mathbb{R}):\varphi(0)=\varphi(1)=0\}$.\\
\textbf{Hint}.\, We must notice from the very beginning that the space $\mathcal{C_0}([0,1];\mathbb{R})$ is too large for taking Cauchy condition($I_a$) and from the way of solving we are forced to accept only $\varphi\in $ with the following structure\\
\begin{equation}\label{3:8:3.1}
\varphi(x)=\mathop{\sum}\limits_{j=1}^{\infty}\alpha_j\hat{U}_j(x)\,\,,\,\,x\in[0,1]
\end{equation}
where
\begin{equation}\label{3:8:3.2}
\hat{U}_j(x=\frac{1}{\sqrt{2}}sin\,j\pi\,x,\,j\geq 1,\,\,x\in[0,1], \hbox{ (orthogonal in }L_2[0,T])
\end{equation}
$$\hbox{ satisfy }\mathop{\int}\limits_{0}^{1}\hat{U}_j(x)\hat{U}_k(x)dx=\left\{
                                                                              \begin{array}{ll}
                                                                                1\,\,\,j=k \\
                                                                               0\,\,\,j\neq k
                                                                              \end{array}
                                                                            \right.
$$
and the following series
\begin{equation}\label{3:8:3.3}
\mathop{\sum}\limits_{j=1}^{\infty}|\alpha_j|^2<\infty
\end{equation}
is a convergent one. The Fourier method\index{Fourier method} involves solutions as function of the following form
\begin{equation}\label{3:8:3.4}
u(t,x)=\mathop{\sum}\limits_{j=1}^{\infty}T_j(t)U_j(x)+a_0,\,t\in[0,T],\,x\in[0,1]
\end{equation}
where each term $U_j(t,x)=T_j(t)U_j(x)$\,satisfies the parabolic equation \index{Parabolic equation}(I) and the boundary conditions($I_b$). In this respect,
from the $PDE$ (I) we get the following
\begin{equation}\label{3:8:3.5}
\frac{\partial_{x}^{2}U_j(x)}{U_j(x)}=\frac{1}{a^2}\frac{\partial_tT_j(t)}{T_j(t)}=-\mu_j,\,\mu_j\neq 0,\,t\in[0,T],\,x\in[0,1],\,j\geq 1
\end{equation}
which are into a system\\
\begin{equation}\label{3:8:3.6}
\left\{
    \begin{array}{ll}
     \frac{dT_j(t)}{dt}+a^2\mu\,T_j(t)=0\,\,\,t\in[0,T]\\
     \frac{d^2U_j}{dx^2}(x)+\mu\,U_j(x)=0,\,\,\,x\in[0,1]
    \end{array}
  \right.
\end{equation}
On the other hand, to fulfil the boundary conditions ($I_b$)we need to impose
\begin{equation}\label{3:8:3.7}
U_j(0)=U_j(1)=0\,\,,\,\,\,j\geq 1
\end{equation}
and to get a solution $\{U_j(x):x\in[0,1]\}$ satisfying the second order differential equation\index{Differential!equation} in \eqref{3:8:3.6} and the boundary conditions \eqref{3:8:3.7} we need to make the choice
\begin{equation}\label{3:8:3.8}
\mu_j=(j\,\pi)^2>0,\,U_j(x)=c_j\,sin\,j\pi\,x,\,x\in[0,1],\,j\geq 1
\end{equation}
In addition, we get the general solution $T_j(t)$\,satisfying the first equation in \eqref{3:8:3.6}
\begin{equation}\label{3:8:3.9}
T_j(t)=a_j[\exp-(j\pi a)^2 t],\,t\in[0,T],\,j\geq 1,\,a_j\in\mathbb{R}
\end{equation}
Now the constants $c_j$\,in\,$U_j(x)$\,(see \eqref{3:8:3.8}) are taken such that
\begin{equation}\label{3:8:3.10}
\hat{U}_j(x)=\hat{c}_j\,sin\,j\pi\,x,\,\,x\in[0,1],\,j\geq 1
\end{equation}
is an orthonorma system in $L_2([0,1];\mathbb{R})$\,and it can be satisfied noticing that
\begin{eqnarray}\label{3:8:3.11}
  \mathop{\int}\limits_{0}^{1}(sin\,j\,\pi x)(sin\,k\,\pi x)dx &=& \mathop{\int}\limits_{0}^{1}(cos\,j\,\pi x)(cos\,k\,\pi x)dx \nonumber\\
  &=& \mathop{\int}\limits_{0}^{1}[cos(j+k)x]dx \nonumber\\
  &=& \mathop{\int}\limits_{0}^{1}(cos\,j\,\pi x)(cos\,k\,\pi x)dx \nonumber\\
  &=& \left\{
                                                                  \begin{array}{ll}
                                                                   1\setminus2\,\,\,\,j=k \\
                                                                  0\,\,\,\,\,j\neq k
                                                                  \end{array}
                                                                \right.
\end{eqnarray}
Here $$\mathop{\int}\limits_{0}^{1}(cos\,m\,\pi x)dx=0 \hbox{ for any }m\geq 1$$ and $$cos\alpha cos\beta=\frac{cos(\alpha+\beta)}{2}$$ are used. From \eqref{3:8:3.11} we get $$\hat{c}_j=\sqrt{2}\,\,\,j\geq 1$$ As a consequence
\begin{equation}\label{3:8:3.12}
\hat{U}_j(x)=\sqrt{2}sin\,j\,\pi\,x\,,\,x\in[0,T]\,,\,j\geq 1
\end{equation}
is an orthonormal system in $L_2([0,1];\mathbb{R})$ and the series in \eqref{3:8:3.4} becomes
\begin{equation}\label{3:8:3.13}
\hat{U}(t,x)=\mathop{\sum}\limits_{j=1}^{\infty}a_j\hat{U}_j(x)[\exp-(j\pi a)^2t]+a_0
\end{equation}
Now we are looking for the constants $a_j, j\geq 0$ such that the series in \eqref{3:8:3.13} is uniformly convergent on $(t,x)\in[0,T]\times[0,1]$ and in addition the initial condition($I_a$) must be satisfied
\begin{equation}\label{3:8:3.14}
\hat{U}(0,x)=\mathop{\sum}\limits_{j=1}^{\infty}a_j\hat{U}_j(x)+a_0=\varphi(x)=
\hat{U}(0,x)=\mathop{\sum}\limits_{j=1}^{\infty}\alpha_j\hat{U}_j(x)\
\end{equation}
where $\mathop{\sum}\limits_{j=1}^{\infty}|\alpha_j|^2<\infty$ is assumed (see\eqref{3:8:3.3}). As a consequence $a_0=0$ and $a_j=\alpha_j\,,\,j\geq 1$, are the corresponding Fourier coefficients \index{Fourier coefficient}associated with $\varphi$ satisfying \eqref{3:8:3.1} and \eqref{3:8:3.3}.\\
In conclusion, the uniformly convergent series given in \eqref{3:8:3.13} has the form
\begin{equation}\label{3:8:3.15}
\hat{u}(t,x)=\hat{U}(t,x)=\mathop{\sum}\limits_{j=1}^{\infty}\alpha_j\hat{U}_j(x)[\exp-(j\pi\,a)^2t],\,t\in[0,T],x\in[0,1]
\end{equation}
It is the weak solution of the problem $\{(I),(II_a),(II_b)\}$ in a sense that
\begin{equation}\label{3:8:3.16}
\hat{U}_N(t,x)=\mathop{\sum}\limits_{j=1}^{N}\alpha_j\hat{U}_j(x)[\exp-(j\pi\,a)^2t]
\end{equation}
satisfies the following properties
\begin{equation}\label{3:8:3.17}
\mathop{lim}\limits_{N\rightarrow \infty}\hat{U}_{N}(t,x)=\hat{U}(t,x)\,\,uniformly\,\,of\,\,t\in[0,1]
\end{equation}
\begin{equation}\label{3:8:3.18}
each\,\,\{\hat{U}_{N}(t,x):(t,x)\in[0,T]\times[0,1]\}\,\,fulfills \,(I)\,and\,(I_b)
\end{equation}
\begin{equation}\label{3:8:3.19}
\hat{U}(0,x)=\varphi(x)\,,\,x\in[0,1]\,,\,((I_a)\,is\,satisfied\,for\,\hat{u})
\end{equation}
provided $\varphi\in\mathcal{C}_0([0,1];\mathbb{R})$\,fulfil \eqref{3:8:3.1} and \eqref{3:8:3.3}.\\
\textbf{Exercise 2}\\
Solve the following mixed problem for a $PDE$ of hyperbolic type using Fourier method\index{Fourier method}\\
$$\left\{
    \begin{array}{ll}
    (I)\,\partial_{t}^{2} u(t,x)=\partial_{x}^{2}u(t,x),\,\,\,t\in[0,T],\,x\in[0,1] \\
      (I_a)\,u(0,x)=\varphi_0(x),\, \partial_t u(0,x)=\varphi_1(x),\,x\in[0,1],\,\,\,\varphi_0,\varphi_1\in\mathcal{C}([0,1];\mathbb{R}) \\
      (I_b)\,\partial_xu(t,0)=\partial_x u(t,1)=0,\,\,\,t\in[0,T]
    \end{array}
  \right.
$$\\
\textbf{Hint}.\,As in the previous exercise treating a mixed problem for parabolic equation we must notice that the space of continuous $\mathcal{C}([0,1];\mathbb{R})$\, is too large for the initial conditions $(II_a)$\,considered here. From the way of solving we are forced to accept only $\varphi_0,\varphi_1\in\mathcal{C}([0,1];\mathbb{R})$\,satisfying  the following conditions\\
\begin{equation}\label{3:8:4.1}
\varphi_0(x)=\mathop{\sum}\limits_{j=1}^{\infty}\alpha_j V_j(x)+\alpha_0,\,\varphi_1(x)=\mathop{\sum}\limits_{j=1}^{\infty}\beta_j V_j(x)+\beta_0\,,\,x\in[0,1]
\end{equation}
where
\begin{equation}\label{3:8:4.2}
\{V_j(x)=\sqrt{2}cos\,j\,\pi\,x\,,\,x\in[0,1]\}_{j\geq 1}
\end{equation}
is an orthonormal system in $L_2[0,1]$, and the corresponding Fourier coefficients\index{Fourier coefficient} \\$\{\alpha_j,\beta_j\}_{j\geq 1}$ define convergent series
\begin{equation}\label{3:8:4.3}
\mathop{\sum}\limits_{j=1}^{\infty}|\alpha_j|^2<\infty\,,\,\mathop{\sum}\limits_{j=1}^{\infty}|\beta_j|^2<\infty
\end{equation}
The Fourier method involves solutions $u(t,x):[0,T]\times[0,1]\rightarrow \mathbb{R}$\,possessing partial derivative $\partial_t u(0,x):[0,1]\rightarrow \mathbb{R}$\,and it must be of the following form
\begin{equation}\label{3:8:4.4}
u(t,x)=\mathop{\sum}\limits_{j=1}^{\infty}T_j(t)V_j(x)+a_0+b_0t\,,\,t\in[0,T]\,,\,x\in[0,1]
\end{equation}
Each term $u_j(t,x)=T_j(t)V_j(x)$\,must satisfy the hyperbolic  equation (II) and the boundary conditions$(II_b)$. It implies the following system of $ODE$
\begin{equation}\label{3:8:4.5}
\left\{
       \begin{array}{ll}
    \frac{d^2 T_j(t)}{dt^2}-\mu_j T_j(t)=0\,,\,t\in[0,T],\,j\geq 1 \\
          \frac{d^2 V_j(t)}{dx^2}-\mu_j V_j(t)=0\,,\,x\in[0,T],\,j\geq 1
       \end{array}
     \right.
\end{equation}
and, in addition, the boundary conditions
\begin{equation}\label{3:8:4.6}
\frac{dV_j}{dx}(0)=\frac{dV_j}{dx}(1)=0\,,\,j\geq 1
\end{equation}
are fulfilled. The conditions \eqref{3:8:4.6} implies\\
\begin{equation}\label{3:8:4.7}
\mu_j=-\lambda_{j}^{2}=-(j\,\pi)^2\,and\,\{V_j(x)=\sqrt{2}cos\,j\,\pi\,x:x\in[0,1]\}_{j\geq 1}
\end{equation}
is an orthonormal system in $L_2[0,1]$. On the other hand, using $\mu_j=-(j\pi)^2,\,j\geq 1$ (see\eqref{3:8:4.7}), from the first equation in \eqref{3:8:4.5} we get
\begin{equation}\label{3:8:4.8}
T_j(t)=a_j(cos\,j\,\pi t)+b_j(sin\,j\,\pi t),\ t\in[0,T],\,j\geq 1
\end{equation}
Using \eqref{3:8:4.7} and \eqref{3:8:4.8}, we are looking for $(a_j,b_j)_{j\geq 0}$ such that $\{u(t,x):(t,x)\in[0,t]\times[0,1]\}$ defined in \eqref{3:8:4.4} is a function satisfying initial condition
\begin{equation}\label{3:8:4.9}
u(0,x)=\mathop{\sum}\limits_{j=1}^{\infty}a_jV_j(x)+a_0=\varphi_0=\mathop{\sum}\limits_{j=1}^{\infty}\alpha_jV_j(x)+\alpha_0,\,x\in[0,1]
\end{equation}
In addition ,the function $\{\partial_tu(t,x):(t,x)\in[0,T]\times[0,1]\}$\,is a continuous one satisfying initial condition
\begin{equation}\label{3:8:4.10}
\partial_tu(0,x)=\mathop{\sum}\limits_{j=1}^{\infty}(j\,\pi)b_j V_j(x)+b_0=\varphi_1(x)=\mathop{\sum}\limits_{j=1}^{\infty}\beta_jV_j(x)+\beta_0\,,\,x\in[0,1]
\end{equation}
From the condition \eqref{3:8:4.9} and \eqref{3:8:4.10} and assuming \eqref{3:8:4.3} we get $(a_j,b_j)_{j\geq 1}$\,of the following form
\begin{equation}\label{3:8:4.11}
a_j=\alpha_j,\,\,(j\,\pi)b_j=\beta_j\,,\,j\geq 1\,,and\,\,a_0=\alpha_0\,\,b_0=\beta_0
\end{equation}
As a consequence, the series
\begin{equation}\label{3:8:4.12}
u(t,x)=\mathop{\sum}\limits_{j=1}^{\infty}[a_j(cos\,j\,\pi t)+b_j(sin\,j\,\pi t)]V_j(x)+a_0+b_0t
\end{equation}
with the coefficients $(a_j,b_j)_{j\geq 1}$ determined in \eqref{3:8:4.11} as a function for which each term $u(t,x)=a_j(cos\,j\,\pi x)+b_j(sun\,j\,\pi x)V_j(x)$ satisfies the hyperbolic equation (II) and the boundary conditions $(II_b)$. As a consequence, the uniformly convergent series \eqref{3:8:4.12} satisfies the mixed problem $(II),(II_a)\,\,and\,\,(II_b)$ in the weak sense, i.e
$$u_N(t,x)=\mathop{\sum}\limits_{j=1}^{N}u_j(t,x)+a_0+b_0t,(t,x)\in[0,T]\times[0,1]$$
verifies (II) and $(II_b)$ for each $N\geq 1$ and
\begin{equation}\label{3:8:4.13}
\mathop{lim}\limits_{N\rightarrow\infty}u_N(0,x)=\varphi_0(x),\,\mathop{lim}\limits_{N\rightarrow\infty}\partial_tu_N(0,x)=\varphi_1(x),\,x\in[0,1]
\end{equation}
stand for the initial conditions $(II_a)$.
\section{Some Nonlinear Elliptic and Parabolic $PDE$}
\subsection{Nonlinear Parabolic Equation}\index{Parabolic equation}
We consider the following nonlinear parabolic $PDE$
\begin{equation}\label{3:9.1}
\left\{
       \begin{array}{ll}
        (\partial_t-\Delta)(u)(t,x)=F(x,u(t,x),\partial_xu(t,x)),\,t\in(0,T],\,x\in\mathbb{R}^{n}\\
        \mathop{lim}\limits_{t\rightarrow 0}u(t,x)=0\,,\,\,x\in\mathbb{R}^n
       \end{array}
     \right.
\end{equation}
where $\partial_xu(t,x)=(\partial_1u,...,\partial_nu)(t,x),\,\partial_iu=\frac{\partial u}{\partial x_i},\,\partial_tu=\frac{\partial u}{\partial t},\,\Delta u=\mathop{\sum}\limits_{i=1}^{n}\partial_{i}^{2} u$. Here\\
$F(x,u,p):\mathbb{R}^{2n+1}\rightarrow\mathbb{R}$\,is a continuous function satisfying\\
\begin{equation}\label{3:9.2}
|F(x,u_0,p_0)|\leq C,\,|F(x,u_2,p_2)-F(x,u_1,p_1)|\leq L(|u_2-u_1|+|p_2-p_1|)
\end{equation}
for any $x\in\mathbb{R}^n,\,|u_i|,\,|p_i|\leq \delta,\,i=0,1,2$,\,where $L,C,\delta>0$ are some fixed constants. A standard solution for \eqref{3:9.1} means a continuous function $u(t,x):[0,a]\times\mathbb{R}^n\rightarrow\mathbb{R}$\,which is first order continuously derivable of $t\in(0,a)$ second order  continuously derivable  with respect to $x=(x_{1},...x_{n})\in \mathbb{R}^n$ such that \eqref{3:9.1} is satisfied for any $t\in(0,a),x\in \mathbb{R}^2$ a weak solution for \eqref{3:9.1} means a pair of continuous functions $(u (t,x) ,\partial_{x} u (t,x);[0,a]\times R^n\rightarrow R^{n+1}$ which are bounded such that the following system  of integral equations is satisfied
\begin{equation}\label{3:9.3}
\left\{
 \begin{array}{ll}
 u(t,x)=\mathop {\int}\limits_{0}^{t}[\mathop{\int}\limits_{\mathbb{R}^n}F(y,u(s,y),\partial_{y})u(s,y)P(t-s,x,y)dy]ds \\
 \partial_{x} u(t,x)=\mathop{\int}\limits_{0}^{t}[\mathop{\int}\limits_{\mathbb{R}^n}F(y,u(s,y),\partial_{y}u(s,y))\partial_{x}P(t-s,x,y)dy]ds \end{array}
  \right.
  \end{equation}
  for any $t\in[0,a],x\in R^{n} $, where $P(\sigma ,x,y),\,\sigma>0,x,y\in \mathbb{R}^n$ is the fundamental solution\index{Fundamental!solutions} of the parabolic equation \index{Parabolic equation} $(\partial_{\sigma}-\Delta_{x})P=0,\sigma>0,x,y\in \mathbb{R}^n$
\begin{equation}\label{3:9.4}
P(\sigma,x,y)=(4\pi \sigma)^{\frac{-n}{2}}\exp-\frac{|y-x|^2}{4\sigma},\sigma>0
\end{equation}
 A direct computation shows that $P(\sigma,x,y)$ satisfy the following properties
\begin{equation}\label{3:9.5}
\mathop {\int}\limits_{\mathbb{R}^n}P(\sigma,x,y)dy=1,\,\partial_{\sigma}P(\sigma,x,y)\Delta_xP(\sigma,x,y)
\end{equation}
 for any \,$\sigma>0,\,\,x,y\in\mathbb{R}^n$\,and\\
 $$\mathop {lim}\limits_{\sigma\downarrow0}P(\sigma,x,y)=0\,\,if\,\,x\neq y$$\\
 The unique solution for theintegral equation \eqref{3:9.3} is found using the standard approximations sequence defined recurrently by
 $$(u,P)_0(t,x)=(0,0)\in\mathbb{R}^{n+1}\,\,\,and$$
 \begin{equation}\label{3:9.6}
\left\{
        \begin{array}{ll}
          u_{k+1}(t,x)=\mathop{\int}\limits_{0}^{t}[\mathop{\int}\limits_{\mathbb{R}^n}F(y,u_k(s,y),p_k(s,y))P(t-s,x,y)dy]ds\\
p_{k+1}(t,x)=\partial_{x}u_{k+1}(t,x)=\mathop{\int}\limits_{0}^{t}[\mathop{\int}\limits_{\mathbb{R}^{n}}F(y,u_k(s,y),p_k(s,y))\partial_{x}P(t-s,x,y)dy]ds \end{array}
      \right.
\end{equation}
for any $k\geq 0,\,(t,x)\in[0,a]\times \mathbb{R}^n$\,where $a>0$\,is sufficiently small such that
\begin{equation}\label{3:9.7}
a\,C\leq \delta,\,\,2\sqrt{a}\,C\,C_1\leq\delta
\end{equation}
Here the constants $C,\delta>0$ are given in the hypothesis \eqref{3:9.2} and $C_1>0$ fixed satisfies
\begin{equation}\label{3:9.8}
(\pi)^{-n\setminus2}\mathop{\int}\limits_{\mathbb{R}^{n}}|z|(\exp-|z^2|)dz\leq C_1
\end{equation}
A constant $a>0$ verifying \eqref{3:9.7} allows one to get the boundedness of the sequence $\{(u_k,p_k)\}_{k\geq 1}$ as in the following lemma
\begin{lemma}\label{rk:3.9.1}
Let $F\in\mathcal{C}(\mathbb{R}^{2n+1},\mathbb{R})$ be given such that the hypothesis \eqref{3:9.2} is verified. Fix $a>0$ such that \eqref{3:9.7} are satisfied. Then the sequence $\{(u_k,p_k)\}_{k\geq 0}$ of continuous functions constructed in \eqref{3:9.6} has the following properties
\begin{equation}\label{3:9.9}
|u_k(t,x)|\leq\delta,\,|p_k(t,x)|\leq \delta,\,\,\forall\,(t,x)\in[0,a]\times\mathbb{R}^n\,,\,k\geq 0
\end{equation}
\begin{eqnarray}\label{3:9.10}
  \parallel(u'_{k+1},p_{k+1}(t)-(u_k,p_k)(t))\parallel &=& \mathop{\sup}\limits_{x\in\mathbb{R}^n}[|u_{k+1}(t,x)-u_{k}(t,x)| \\
  &-& |p_{k+1}(t,x)-p_{k}(t,x)|]\leq 2\delta C_{a}^{k}(\frac{t^k}{k!})^{1\setminus3} \nonumber
\end{eqnarray}
for any $t\in[0,a]\,,\,k\geq 0$ where $C_a=L(C_1\sqrt{a}+a^{2\setminus3})$.
\end{lemma}
\begin{remark}
Using the conclusion \eqref{3:9.10} of Lemma \eqref{rk:3.9.1}  and Lipschitz continuity of $F$ in \eqref{3:9.2} we get that the sequence $\{(u,p)_{k}\}_{k\geq 0}$ is uniformly convergent  of $(t,x)\in[0,\hat{a}]\times \mathbb{R}^{n},\mathop{lim}\limits_{k\rightarrow \infty} (u,p)_{k} (t,x)=(\hat{u},\hat{p})(t,x)$ there the continuous function$(\hat{u},\hat{p}),t\in[0,\hat{a}], x\in \mathbb{R}^n$  satisfies  the integral equation \eqref{3:9.3}, provided the  constant $\hat{a}> 0$ is fixed such that
\begin{equation}\label{3:9.11}
(\frac{\hat{a}}{2}^\frac{1}{3}).C_{\hat{a}}= \rho < 1
\end{equation}
 where $C_{\hat{a}}$ is given in \eqref{3:9.10}. In this respect, $\sum(t,x)=\mathop{\sum}\limits_{k=0}^{\infty}[(u,p)_{k+1}-(u,p)_{k}](t,x)$ is bounded by a numerical convergent series $|\sum(t,x)|\leq2\delta(1+\rho+,.....+\rho^k+...)=\frac{2^\delta}{1-\rho}$ which alow us to obtain the following
\end{remark}
\begin{lemma}
Let $F\in\mathcal{C}(\mathbb{R}^{2n+1},\mathbb{R})$ be given such that the  hypothesis \eqref{3:9.2} is satisfied. Then there exists a unique solution $(\hat{u}(t,x),\hat{p}tt,x), t\in[0,\hat{a}],\times\in\mathbb{R}^n,$ verifying the integral equations \index{Integral equations}\eqref{3:9.3} and
\begin{equation}\label{3:9.12}
|\hat{u}(t,x)|,|\hat{p}(t,x)|\leq \rho (forall) t \in [0,\hat{a}],\times \in\mathbb{ R}^n
\end{equation}
\begin{equation}\label{3:9.13}
\partial _{x}\hat{u}(t,x)=\hat{p}(t,x) \forall t\in [0,\hat{a}],x\in \mathbb{R}^n
\end{equation}
In addition, $((\hat{u}(t,x),\partial_{x}\hat{u}(t,x)))$ is the unique weak solution of the nonlinear equation \eqref{3:9.10}, i.e
 \begin{equation}\label{3:9.14}
\mathop{limt}\limits_{\varepsilon\downarrow 0}(\partial_{t}-\Delta)\hat{u_{\varepsilon}}(t,x)=F(x,\hat{u}(t,x),\partial_{x}\hat{u}(t,x))
\end{equation}
for each $0<t\leq\hat{a},x\in \mathbb{R}^n$ where
$$\hat{u_{\varepsilon}}(t,x)\mathop{\int}\limits_{0}^{t} [\mathop{\int}\limits_{\mathbb{R}^n}F(y,\hat{u},s,y),\partial_{y}\hat{u}(s,y)P(t-s,x,y)dy]ds $$
if $0<t=\hat{a}\,\, and\,\, x \in \mathbb{R}^n$ are fixed.
\end{lemma}
\subsection{Some Nonlinear Elliptic Equations}\index{Elliptic equation}
We consider the following nonlinear elliptic equation
\begin{equation}\label{3:9.15}
\Delta u(x)=f(x,u(x)),x\in R^2,n\geq3\,\, where \,f(x,u):\mathbb{R}^{n+1}\rightarrow \mathbb{R}
\end{equation}
is first order continuously differentiable satisfying
\begin{equation}\label{3:9.16}
f(x,u)=0 for\,\,x\in\, B(0,b),u\in R, \hbox{ where }B(0,b)\subseteq \mathbb{R}^n \hbox{ is fixed}
\end{equation}
 \begin{equation}\label{3:9.17}
\lambda=\{\max{|\frac{\partial t}{\partial u}|;x\in B(0,b),|u|\leq2K_{1}}\}
\end{equation}
 satisfies $\lambda K_{0}=\rho\in (0,\frac{1}{2})$ where $$K_{0}=\frac{2b^2}{n-2},K_{1}=C_{0}K_{0} and C_{0}=\{\max{|f(x,0)|;x\in B(0,b)}\}  $$
\begin{lemma}
Assume that ${f(x,u);(x,u)\in R^{n+1}}$ is given satisfying the  hypothesis \eqref{3:9.16} and \eqref{3:9.17}. Then there exits a unique bounded solution of the  nonlinear equation \eqref{3:9.15} $\hat{u}(x); R^{n}\rightarrow R.$ which satisfies the following integral equation
 \begin{equation}\label{3:9.18}
\hat{u}(x)=\mathop{\int}\limits_{\mathbb{R}^n}\hat{f}(y,\hat{u}(y) |y-x|^{2-n}dy=\mathop{\int}\limits_{B(o,b)}\hat{f}(y,\hat{u}(y))|y-x|^{2-n} dy
\end{equation}
where $\hat{f}(x,u)=\hat{C}f(x,u) \,and \,\hat{C}=\frac{1}{(2-n)|\Omega_{0}|},|\Omega|=meas S(0,1).$
\end{lemma}
\begin{remark}
The proof of this result usees the standard computation performed for the Poisson equation \index{Poisson equation} in the first section of this chapter.
More precisely, let$\Omega_{0}= S(0,1)\subseteq \mathbb{R}^n $ be the sphere centered at origin and with radius $\rho =1$.
Denote $ |\Omega_{0}|= meas \Omega_{0}$ and $\hat{f}(x,u)=\hat{C}f(x,u),(x,u)\in \mathbb{R}^n\times \mathbb{R}$ where $\hat{c}=\frac{1}{(2-n)|\Omega_{0}|}$.
 Associate the  integral equation \eqref{3:9.18} which can be rewritten as
\begin{equation}\label{3:9.19}
u(x)=\mathop{\int}\limits_{D}\hat{f}(x+z,u(x+z)).|z|^{2-n}dz ,\,for x\in B(o,b)
\end{equation}
where$ D=B(0,L),L=2b $. Define a sequence ${u_{k}(x); x\in B(0,b)}k\geq0 u_{0}(x)=0,x\in B(0,b)$
\begin{equation}\label{3:9.20}
u_{k+1}(x)=\mathop{\int}\limits_{D}\hat{f}(x+z,u_{k}(x+z))|z|^{2-n}dz,k\geq0\times ,\,x\in B(0,b)
\end{equation}
The sequence $\{u_{k}\}_{k\geq 0}$ is bounded and uniformly  convergent to $a$ continuous function as following estimates show
\begin{equation}\label{3:9.21}
|u_1(x)|\leq \hat{c}C_{0}\mathop{\int}\limits_{0}^{L} r dr \mathop{\int}\limits _{\Omega_{0}}dw = \frac{C_{0}}{n-2}\frac{L^2}{2}=\frac{C_{0}}{n-2}2b^2=K_1
\end{equation}
 for any $x\in B(0,b) $ and
\begin{equation}\label{3:9.22}
|u_{2}(x)-u,(x)|\leq \mathop{\int}\limits_{D}1+\hat{f}((x+z),(x+z))-\hat{f}(x+z,0)|z|^{2-n} dz\leq
\end{equation}
  $$\leq K_{1} \max |\frac{\partial\hat{f}}{\partial u}(y,u)|\mathop{\int}\limits_{S}|z|^{2-n}dz\leq K_1(\lambda.K_0)=K_1\rho\,for any\,x\in B(0,b)$$
An induction argument leads us to
\begin{equation}\label{3:9.23}
|u_{k+1}(x)-u_k(x)|\leq K_1\rho^{k},\,x\in B(0,b),\,k\geq 0
\end{equation}
where $\rho\in(0,1\setminus2)$ and $k_1>0$ are  defined in \eqref{3:9.17}. Rewrite
\begin{equation}\label{3:9.24}
u_{k+1}=u_1(x)+u_2(x)-u_1(x)+\dots+u_{k+1}(x)-u_k(x)
\end{equation}
and consider the following series of continuous functions
\begin{equation}\label{3:9.25}
\sum(x)=u_(x)+\nu_2(x)+\dots+\nu_{k+1}(x)+\dots
\end{equation}
$$x\in B(0,b),\,\nu_{j+1}(x)=u_{j+1}(x)-u_j(x)$$
 The series in \eqref{3:9.25} is dominated by a convergent numerical series
\begin{equation}\label{3:9.26}
|\sum(x)|\leq K_1(1+\rho+\rho^2+\dots+\rho^k+\dots)=K_1\frac{1}{1-\rho}\leq 2K_1
\end{equation}
and the  sequence defined in \eqref{3:9.20} is uniformly convergent to a continuous and bounded function
\begin{equation}\label{3:9.27}
\mathop{lim}\limits_{k\rightarrow \infty}u_k(x)=\hat{u}(x),\,|\hat{u}(x)|\leq 2K_1,\,x\in B(0,b)
\end{equation}
In addition$\{\hat{u}(x),\,x\in B(0,b)\}$ verifies
\begin{equation}\label{3:9.28}
\hat{u}(x)\mathop{\int}\limits_{\mathbb{R}^n}\hat{f}(y,\hat{u}(y))|y-x|^{2-n}dy=\mathop{\int}\limits_{B(0,b)}\hat{f}(y,\hat{u}(y))|y-x|^{2-n}dy
\end{equation}
The solution $\{\hat{u}(x),\,x\in B(0,b)\}$ is extended as a continuous function on $\mathbb{R}^n$ using the same as integral equation (see \eqref{3:9.28}) and the proof of \eqref{3:9.18} is complete.
\end{remark}
\section{Exercises}{Weak solutions for parabolic and hyperbolic boundary problems by Fourier's method\index{Fourier method}}
\textbf{$(P_1)$}.\,Using Fourier method, solve the following mixed problem for a scalar parabolic equation \index{Parabolic equation}
\begin{equation}\label{3:10.1.11p}
\left\{
    \begin{array}{ll}
      (a)\,\,\partial_t u(t,x)=\partial_{x}^{2}u(t,x)\,,\,t\in[0,T],\,x\in\,[A,B] \\
     (b)\,\,u(0,x)=\varphi_0(x)\,,\,x\in\,[A,B]\,,\,\,\varphi_0\in\mathcal{C}([A,B])\\
   (c)\,\,u(t,A)=u_A,\,\,u(t,B)=u_B\,,\,t\in[0,T],\,u_A,u_B\in\mathbb{R}\,\,given
    \end{array}
  \right.
\end{equation}
\textbf{Hint.} Make a function transformation
\begin{equation}
\nu(t,x)=u(t,x)-\{\frac{x-A}{B-A}u_B+\frac{B-x}{B-A}u_A\}
\end{equation}
which preserve the equation (a) but the boundary condition (c) becomes \\$\nu(t,A)=0, \nu(t,B)=0, t\in[0,T]$. The interval $[A,B]$ is shifted into $[0,1]$ by the following transformations
\begin{equation}
x=B\,y+(1-y)A,\,w(t,y)=\nu(t,B\,y+(1-y)A),\,y\in[0,1],\,t\in[0,T]
\end{equation}
The new function $\{w(t,y):y\in[0,T],y\in[0,1]\}$\,fulfils a standard mixed problem for a scalar parabolic equation
\begin{equation}\label{3:10.1}
\left\{
    \begin{array}{ll}
   (a)\,\,\partial_t\,w(t,y)=\frac{1}{(B-A)^2}\partial_{y}^{2}w(t,y),\,t\in[0,T],\,y\in[0,1] \\
      (b)\,\,w(0,y)=\varphi_1(y)\mathop{=}\limits^{def}\varphi_0(A+(B-A)y)-\{u_By+u_A(1-y)\} \\
     (c)\,\,w(t,0)=0\,,\,w(t,1)=0\,,\,t\in[0,T]
    \end{array}
  \right.
\end{equation}
\begin{equation}
w(t,y)=\mathop{\sum}\limits_{j=1}^{\infty}T_j(t)W_j(t)
\end{equation}
where
\begin{equation}\label{3:10.5}
\frac{T'_j(t)}{T_j(t)}=\frac{W''_j(y)}{W_j(y)}\frac{1}{(B-A)^2}=-\lambda_j,\,\lambda_j>0,\,j=1,2,...
\end{equation}
The boundary condition (\eqref{3:10.1.11p}, c) must be satisfied by the general solution
\begin{equation}
W_j(y)=C_1\,cos\,(B-A)\sqrt{\lambda_j}y+C_2\,sin\,(B-A)\sqrt{\lambda_j}y
\end{equation}
of \eqref{3:10.5} which implies $C_1=0$\,and\,$\lambda_j=(B-A)^{-2}(\pi j)^2$\\
\textbf{$(P_2)$}. Using Fourier method\index{Fourier method}, solve the following mixed problem for a scalar parabolic equation\index{Parabolic equation}\\
$$\left\{
    \begin{array}{ll}
      (a)\,\,\partial_{t}^{2} u(t,x)=\partial_{x}^{2}u(t,x)\,,\,t\in[0,T],\,x\in\,[A,B] \\
     (b)\,\,u(0,x)=\varphi_0(x)\,,\partial_tu(0,x)=\varphi_1(x)\,,\,x\in\,[A,B]\,,\,\,\varphi_0,\varphi_1\in\mathcal{C}([A,B])\\
   (c)\,\,u(t,A)=u_A,\,\,u(t,B)=u_B\,,\,t\in[0,T],\,u_A,u_B\in\mathbb{R}\,\,given
    \end{array}
  \right.
$$\\
\textbf{Hint}.\,Make a function transformation
\begin{equation}\label{3:10:1.1}
\nu(t,x)=u(t,x)-\{\frac{x-A}{B-A}u_B+\frac{B-x}{B-A}u_A\}
\end{equation}
which preserve the hyperbolic  equation (a) but the boundary condition (c) becomes
\begin{equation}\label{3:10:1.2}
\end{equation}
The interval $[A,B]$\,is shifted into $[0,1]$\,with preserving the boundary \eqref{3:10:1.2} if the following transformation are done
\begin{equation}\label{3:10:1.3}
x=B\,y+(1-y)A,\,w(t,y)=\nu(t,B\,y+(1-y)A),\,y\in[0,1],\,t\in[0,T]
\end{equation}
The new function $\{w(t,y):y\in[0,T],y\in[0,1]\}$\,fulfils a standard mixed problem for a scalar parabolic equation
\begin{equation}\label{3:10:1.3}
\left\{
    \begin{array}{ll}
   (a)\,\,\partial_{t}^{2}\,w(t,y)=\frac{1}{(B-A)^2}\partial_{y}^{2}w(t,y),\,t\in[0,T],\,y\in[0,1] \\
      (b)\left\{
           \begin{array}{ll}
             w(0,y)=\psi_0(y)\mathop{=}\limits^{def}\varphi_0(A+(B-A)y)-\{u_By+u_A(1-y)\},\,y\in[0,1] \\
            \partial_tw(0,y)=\psi_1(y)\mathop{=}\limits^{def}\varphi_1(A+(B-A)y),\,y\in[0,1]
           \end{array}
         \right.
 \\
     (c)\,\,w(t,0)=0\,,\,w(t,1)=0\,,\,t\in[0,T]
    \end{array}
  \right.
\end{equation}
Solution of \eqref{3:10:1.3} has the form
\begin{equation}\label{3:10:1.5}
w(t,y)=\mathop{\sum}\limits_{j=1}^{\infty}T_j(t)W_j(t)+C_0+C_1t
\end{equation}
and the algorithm of solving repeats the standard computation given $(P_1)$.
\section[Appendix I ]{Appendix I Multiple Riemann Integral and Gauss -Ostrogradsky Formula}\index{Gauss-Ostrogadsky}
\textbf{(A1)}  We shall recall the definition of the simple Riemann integral $\{f(x): x\in[a,b]\}$. Denote by
$\Pi$ a partition of the interval $[a,b]\subseteq R$:\\
$\Pi=\{a=x_{0}\leq x_{1}\leq ...\leq x_{i}\leq x_{i+1}\leq...\leq x_{N}=b\}$ and for $\Delta x_{i}=x_{i+1}-x_i, i\in \{0,1,...,N-1\}$, define the norm $d(\Pi)=\max_{i}\Delta x_{i}$ of the partition $\Pi$. For $\xi_{i}\in[x_{i},x_{i+1}], i\in \{0,1,...,N-1\}$ (marked points of $\Pi$) associate the integral sum
\begin{eqnarray*}
S_{\Pi}(f)=\sum_{i=0}^{N-1}f(\xi_{i})\Delta x_{i}
\end{eqnarray*}
\begin{definition}\label{df:3:12.1}
A number $I(f)$ is called Riemann integral of the function $\{f(x): x\in[a,b]\}$ on the interval $[a,b]$ if for any $\epsilon >0$ there is a $\delta >0$ such that
\begin{eqnarray*}
|S_{\Pi}(f)-I(f)|< \epsilon\ \  \forall\ \  \Pi\ \  \textrm{satisfying}\ \  d(\Pi)<\delta.
\end{eqnarray*}
\end{definition}
\begin{remark}
An equivalent definition can be expressed using sequences of partitions $\{\Pi_{k}\}_{k \geq 1}$ for which $$\mathop{lim}\limits_{k\rightarrow \infty}{S_{\Pi}}_{{}_{k}}(f)=I(f),$$ where the number $I(f)$ does not depend on the sequence $\{\Pi_{k}\}_{k\geq 1}$ and its marked points. If it is the case we call the number $I(f)$ as the integral of the function $f(x)$ on the interval $[a,b]$ and write
\begin{eqnarray*}
\int_{a}^{b}f(x)dx=\mathop{lim}\limits_{k\rightarrow \infty}{S_{\Pi}}_{{}_{k}}(f)
\end{eqnarray*}
Finally, there is a third equivalent definition of the integral which uses a "limit following a direction". Let $E$ be consisting of the all partitions $\Pi$ associated with their marked points, for a fixed $\delta >0$, denote $E_{\delta}\subseteq E$ the subset satisfying $d(\Pi)<\delta$, $\Pi\in E_{\delta}$. The subsets $E_{\delta}\subseteq E$, for different $\delta >0$, are directed using $d(\Pi)\rightarrow 0$. The integral of the function $\{f(x): x\in[a,b]\}$ is the limit $J(f)$ of the integral sums following this direction; the three integrals $J(f),\,\int_{a}^{b}f(x)dx$ and $I(f)$ exist and are equal iff one of them exists. It lead us to the conclusion that an integral of a continuous function $\{f(x): x\in[a,b]\}$ exists.
\end{remark}
\noindent\textbf{(A2)}
Following the above given steps we may and do define the general meaning of a Riemann integral when the interval $[a,b]\subseteq R$ is replaced by some metric space (X,d). Consider that a family of subsets $U\subseteq X$ of X are given verifying the following conditions:

\begin{enumerate}
  \item The set X and the empty set $\emptyset$ belong to $\mathcal{U}$
  \item f $A_{1}, A_{2}\in \mathcal{U}$ then their intersection   $A_{1}A_{2}$ belongs to $\mathcal{U}$.
  \item If $A_{1},A\in U$ \,and \,$A_{1}\subseteq A$ then there exist $A_{2}, ..., A_{p}\in \mathcal{U}$ such  that $A=A_1\cup ...\cup A_p$ and $A_{2}, ..., A_{p}\in \mathcal{U}$ are mutually disjoints. A system of subsets $u\subseteq X$ fulfilling (1), (2) and (3) is called demi-ring $\mathcal{U}$. In order to define a Riemann integral on the metric space X associated with a demi-ring $\mathcal{u}$ we need to assume another two conditions.
  \item For any $\delta>0$, there is a partition of the set $X=A_{1}\cup...\cup A_{p}$, with $A_{i}\in \mathcal{U}$, $A_{i}A_{j}=\phi$ if $i\neq j$ and $d(A_{i})=\sup\limits_{(x,y)\in A_{i}}\rho(x,y)<\delta$, $i\in \{1,...,p\}$ (This condition remind us the property that X is a precompact metric space). The last condition imposed on the demi-ring $\mathcal{U}$ gives the possibility to measure each individual of $U$.
  \item For each $A\in \mathcal{U}$, there is positive number $m(A)\geq0$ such that $m : \mathcal{U}\rightarrow[0,\infty)$ is additive, i.e
$m(A)=m(A_{1})+m(A_{2})+...+m(A_{p})$, if $A=A_{1}\cup...\cup A_{p}$ and $A_{1},...,A_{p}$ are mutually disjoints.
\end{enumerate}
\noindent The additive mapping $m : U\rightarrow[0,\infty)$ satisfying (5) is called measure on the cells composing $\mathcal{U}$. The metric space X associated with a demi-ring of cells $U$ and a finite additive measure $m : u\rightarrow[0,\infty)$ satisfying (1)-(5) will be called a measured  space $(X,\mathcal{U},m)$. Let $f(x) : X\rightarrow R$ be a real function defined on a measured space $(X,\mathcal{U},m)$ and for an arbitrary partition $\Pi=\{A_{1},...,A_{p}\}$ of $x=A_{1}\cup...\cup A_{p}$ ($A_{i}A_{j}=\phi_{1}\,\,i\neq j$) define an integral sum.\\
\begin{equation}\label{3:10:1.6}
S_{\Pi}(f)=\sum\limits_{i=1}^{\infty}f(\xi_{i})mA_{i},\, where \,\xi_{i}\in A_{i} \,\,is \,fixed
\end{equation}
The number
\begin{equation}\label{3:10:1.7}
I_{\mathcal{u}}(f)=\int\limits_{X}f(x)dx
\end{equation}
is called the integral of the function $f$ on the measured space $(X,\mathcal{U},m)$ if for any $\epsilon >0$, there is a $\delta>0$ such that
\begin{equation}\label{3:10:1.8}
|I_{\mathcal{u}}(f)-S_{\Pi}(f)|< \epsilon
\end{equation}
is verified for any partition $\Pi$ satisfying $d(\Pi)<\delta$, where $d(\Pi)=\max(d(A_{1}),...,d(A_{p}))$ and $d(A_{i})=\sup\limits_{(x,y)\in A_{i}}\rho(x,y)$. It is easily seen that this definition of an integral on a measured space $(X,\mathcal{U},m)$ coincides with the first definition of the Riemann integral given on a closed interval $[a,b]\subseteq R$.\\
The other two equivalent definitions using sequence of partitions $\{\pi_k\}_{k\geq 1}\,,\,d(\pi_k)\rightarrow 0$, and a "limit following a direction" will replicate the corresponding definitions in the one dimensional case $x\in[a,b]$.\\
A function $f(x):\{X,\mathcal{U},m\}\rightarrow \mathbb{R}$\,defined on a measured space and admitting Riemann integral is called integrable on $X\,,\,f\in J(X)$. The following elementary properties of the integral are direct consequence of the definition using sequence of partitions and their integral sums
\begin{equation}\label{3:10:1.9}
\int_{X}f(x)dx=c\,m(X),\,if\,\,f(x)=c(const),\,\,x\in X
\end{equation}
\begin{equation}\label{3:10:1.10}
\int_X cf(x)dx=c\int_Xf(x)dx,\,f\in J(X)\,\,and\,c=const
\end{equation}
\begin{equation}\label{3:10:1.11}
\int_X[f(x)+g(x)]dx=\int_xf(x)dx+\int_Xg(x)dx,\,if\,f,g\in J(X)
\end{equation}
\begin{equation}\label{3:10:1.12}
Any \,\,f\in J(X)\,\,is\,\, bounded \,\,on X,\,|f(x)|\leq c,\,x\in X
\end{equation}
if $f,g\in J(X)
$ and $f(x)\leq g(x),\,x\in X $ then
\begin{equation}\label{3:10:1.13}
\int_Xf(x)dx\leq \int_Xg(x)dx(\int_Xf(x)dx\leq \int_X|f(x)|dx\,if \,\,f|f|\in J(X))
\end{equation}
\begin{equation}\label{3:10:1.14}
C\,m(x)\leq \int_Xf(x)dx\leq C\,m(x),\,if\,f\in J(X)\,and\,c\leq f(x)\leq C\,\,\forall \,x\in X\,.
\end{equation}
\begin{theorem}
If a sequence $\{f_k(x):x\in X\}_{k\geq 1}\subseteq J(X)$ converges uniformly on $\{X,\mathcal{U},m\}$ to a function $f(x):X\rightarrow \mathbb{R}$ then $f\in J(X)$ and
$$\int_Xf(x)dx=\mathop{lim}\limits_{n\rightarrow\infty} \int_Xf_n(x)dx$$
\end{theorem}
\noindent \textbf{Hint} The proof replicates step by step the standard proof used for $X=[a,b]\subseteq\mathbb{R}$
\begin{example}
\textbf{($e_1$)}\, $X=[a,b]\subseteq \mathbb{R}$ and for a cell of $[a,b]$ can be taken any subinterval containing or not including its boundary points. The measure $m(A)=\beta-\alpha,\,(A=[\alpha,\beta],\,A=(\alpha,\beta))$ of a cell is the standard length of it and the corresponding measured space $(X,\mathcal{U},m)$ satisfy the necessary conditions (1)-(5). Notice that the first definition of the integral given in \eqref{3:10:1.8} coincides with that definition used in the definition \eqref{df:3:12.1}\\
\textbf{($e_2$)}\,\,\,Let $X$\,be a rectangle in $\mathbb{R}^2\,,\,X=\{x\in\mathbb{R}^2:a_1\leq x_1\leq b_1,a_2\leq x_2\leq b_2\}$\\
and as a cell of $X$\,we take any subset $A\subseteq X\,,\,A=\{x\in X:\alpha_1\prec x_1\prec \beta_1,\alpha_2\prec x_2\prec \beta_2\}$\,\,where the sign $"\prec"$\,means the standard $"\leq"\,\,or "<"$\,among real numbers.The measure $m(A)$\,associated with the cell $A$\,is given by its area $m(A)=(\beta_1-\alpha_1).(\beta_2-\alpha_2)$.The necessary condition (1)-(5) are satisfied by a direct inspection and the Riemann integral on $X$\,will be denoted by
$$\int_Xf(x)dx=\int_{a_1}^{b_1}\int_{a_2}^{b_2}f(x_1,x_2)dx_1\,dx_2\,(\,double\,\,Riemann\,\,integral)$$
Replacing $\mathbb{R}^2$\,by\,$\mathbb{R}^n\,,\,n\geq 2$\,and choosing $X=\mathop{\prod}\limits_{i=1}^{n}I_i$\,as a direct product of some intervals \,$I_i=\{x\in \mathbb{R}:a_i\prec x\prec b_i\}$\,we define a cell $A=\mathop{\prod}\limits_{i=1}^{n}\{\alpha_i\prec x\prec\beta_i\}$\,and its volume $m(A)=\mathop{\prod}\limits_{i=1}^{n}(\beta_i-\alpha_i)$\,as the associated measure.In this case the Riemann integral is denoted by\\
$$\int_Xf(x)dx=\int_{a_1}^{b_1}\dots\int_{a_n}^{b_n}f(x_1,\dots,x_n)dx_1\dots dx_n$$\\
and call it as the n-multiple integral of $f$.
\end{example}
\begin{theorem}\label{th:10.1.2}
Let $\{X,\mathcal{U},m\}$\,be a measured space and $\mu(x):X\rightarrow \mathbb{R}$\,is a uniformly continuous function. Then $f$\,is integrable, $f\in J(X)$.
\end{theorem}
\begin{proof}
By hypothesis the oscillation of the function $f$\,on $X$\
\begin{equation}\label{3:10:1.15}
 w_f(X,\delta)=\mathop{\sup}\limits_{\rho(x',x'')\leq\delta, x',x''\in X}|f(x')-f(x'')|
\end{equation}
satisfies $w_f(X,\delta)\leq\varepsilon$\,for some $\varepsilon>0$\,arbitrarily fixed provided $\delta>0$ is sufficiently small. In particular,this property is valid on any elementary subset $P\subseteq X\,,\,P=\mathop{\bigcup}\limits_{i=1}^{p}A_i\,,\,\{A_1,\dots,A_p\}$\,are mutually disjoint cells.Denote $w_f(P,\delta)$\,the corresponding oscillation of $f$\,restricted to $P$\,and notice\\
\begin{equation}\label{3:10:1.16}
|S_\pi(f,P)-S_\pi'(f,P)|\leq w_f(P,\delta)m(P)
\end{equation}
for any partition $\pi'$ of $P$ containing the given partition $\pi=\{A_1,\dots,A_p\}$ of the elementary subset $P\subset X(\pi'\supseteq \pi)$.\\
The property $\pi'\supseteq\pi$ is described by \lq\lq$\pi'$is following $\pi$\rq\rq($\pi'$ is more refine). Using \eqref{3:10:1.16} for $P=x$ noticing that $w_f(X,\delta)\leq \varepsilon$ for $\varepsilon>0$, arbitrarily fixed provided $\delta>0$\,is sufficiently small, we get that
$$\mathop{lim}\limits_{\delta(\pi)\rightarrow 0}S_\pi(f)=I(f)\,\,\,exists$$
as a consequence of the Cauchy criteria applied to integral sums.
\end{proof}
\noindent \textbf{Consequence} Any continuous function defined on a compact measured space $(X,\mathcal{U},m)$
is integrable(see $f(x):X\rightarrow \mathbb{R}$ is uniformly continuous).
\begin{theorem}\label{th:3:10!.2}
Let $(X,\mathcal{U},m)$ be a measured space and $Z\supseteq X$ is a negligible set.Assume that the bounded function $f(x)$ is uniformly continuous outside of any arbitrary neighborhood of $Z,\mathcal{U}_\sigma(Z)=\{x\in X:\rho(x,Z)<\delta\}$. Then $f$ is integrable on $X$.
\end{theorem}
\begin{proof}
By hypothesis $Z$ is a negligible set and for any $\varepsilon>0$ there is an elementary set $P=\mathop{\bigcup}\limits_{i=1}^{p}A_i$ such that $int P\supseteq Z$ and $m(P)< \varepsilon$. Let $M=\mathop{\sup}\limits_{x\in X}|f(x)|\hbox{ and for }\varepsilon>0$ define
\begin{equation}\label{3:10:1.17}
P=\mathop{\bigcup}\limits_{i=1}^{p}A_i\,\,such\,\,that\,\,int\,P\supseteq Z\,\,and\,\,\,m(P)\leq \frac{\varepsilon}{4 M}
\end{equation}
Denote $B=X-P$ and we get $d(Z,B)=2\rho>0$ where $d(Z,B)=inf\{d(z,b):z\in Z,b\in B\}$. By hypothesis,the function $f$ is uniformly continuous outside of a neighborhood $\mathcal{U}_\rho(Z)$ ,of the negligible set $Z$, i.e $f(x):Q\rightarrow \mathbb{R}$ is uniformly continuous, where $Q=X\setminus \mathcal{U}_\rho(Z)$. Using Theorem \ref{th:10.1.2} we get that $f$\,restricted to $(Q,\mathcal{U} Q,m)$ is integrable and any integral sum of $f$ on $X$ restricted to $P$ is bounded by $\varepsilon$(see $ int P\supseteq Z,\,\,m(P)\leq \frac{\varepsilon}{4M}$).\\
For an arbitrary partition $\pi=\{C_1,\dots,C_n\}$ of $X$ we divide it into two classes; the first class contains all cells of $\pi$ which are included in $P=\mathop{\bigcup}\limits_{i=1}^{p}A_i\,\supseteq Z_1$ and the second class is composed by the cells of $\pi$ which have common points with the set $B=X\setminus P$ and are entirely contained in $B$. In particular, take a partition $\pi$ with $d(\pi)<\sigma=min(\sigma,\rho)$ where $\sigma>0$ is sufficiently small such that $|f(x'')-f(x')|<\varepsilon\setminus{2m(X)}$ if $\rho(x'',x')<2\sigma,\,x',x''\in B$. Let $\pi'\supseteq\pi$ be a following partition($\pi'$ is more refined than $\pi$). A straight computation allows one to see that\\
\begin{equation}\label{3:10:1.18}
|S_\pi(f)-S_{\pi'}(f)|\leq |S_\pi(f,P)|+|S_{\pi'}(f,P)|+|S_\pi(f,Q)-S_{\pi'}(f,Q)|
\end{equation}
$Q=X\setminus \mathcal{U}_\rho(Z)$ the first two terms in \eqref{3:10:1.18} fulfil
\begin{equation}\label{3:10:1.19}
\left\{
        \begin{array}{ll}
         |S_\pi(f,P)|\leq M\mathop{\sum}\limits_{i=1}^{p}m(A_i)=M\,m(P)\leq M.\varepsilon\setminus{4M}=\varepsilon\setminus4  \\
                 |S_{\pi'}(f,P)|\leq M\mathop{\sum}\limits_{i=1}^{p}m(A_i)=M\,m(P) =\varepsilon\setminus4
        \end{array}
      \right.
\end{equation}
For the last term in \eqref{3:10:1.18} we use \eqref{3:10:1.16} in the proof of Theorem \ref{th:3:10!.2} and $Q=X\setminus \mathcal{U}_\rho(Z)\supseteq X\setminus P=B$. As far as $\sigma>d(\pi)\geq d(\pi')\hbox{ and }\sigma=\min(\sigma,\rho)$ we obtain
\begin{equation}\label{3:10:1.20}
\left\{
        \begin{array}{ll}
         |S_\pi(f,Q)-S_{\pi'}(f,Q)|\leq w_f(Q,\delta)m(X)\leq[ \varepsilon\setminus{2m(X)}]m(X)=\varepsilon\setminus2\\
           |S_\pi(f)-S_{\pi'}(f)|\leq\varepsilon
        \end{array}
      \right.
\end{equation}
and the Cauchy criteria used for integral sums $\{S_\pi(f):d(\pi)\rightarrow 0\}$\,lead us to the conclusion.The proof is complete.
\end{proof}
\begin{definition}\label{df:10:1.2}
 Let $\{X,\mathcal{U},m\}$ be a measured space and $G\subseteq X$ a subset. We say that $G$ is a jordanian set with respect to $(\mathcal{U},m)$ if its boundary $\partial G=\overline{G}\cap\{X\setminus G\}$, is a negligible set. A jordanian closed set A with the property $\overline{int\, A}=A$ is called a jordanian body.
 \end{definition}
\begin{remark}
The characteristic function of a jordanian set $G\subseteq X$
$$\chi_G(x)=\left\{
              \begin{array}{ll}
               1\,\,\,\,x\in G \\
               0\,\,\,\,x\in X\setminus G
              \end{array}
            \right.
$$
\end{remark}
\noindent is an integrable function if the measured space $\{X,\mathcal{U},m\}$ is compact; $\int_{X}\chi_G(x)dx=|G|$ is called volume of $G$.
\begin{definition}
A measured space $\{X,\mathcal{U},m\}$ for which all cells are jordanian sets is called normally measured;each cell $A\in \mathcal{U}$ has a volume and $m(A)=vol\,A=|A|$
\end{definition}
\noindent Let $f(x):X\rightarrow \mathbb{R}$ be a bounded function on a compact measured space $\{X,\mathcal{U},m\}$ and $G\subset X$ is a jordanian set with the boundary $\Gamma=\partial G$. By definition, the integral of the function $f$ on the set $G$ is given by
\begin{equation}\label{3:10:1.21}
\int_{G}f(x)dx=\int_{X}f(x)\chi_G(x)dx
\end{equation}
If $f$ is a continuous function on $G$ except a negligible $Z$ then $f(x)\chi_G(x)$ is  continuous on $X$ except the negligible set $Z\cup\partial G$ and \eqref{3:10:1.21} exists.\\
\textbf{($A_2$)\,\,\,\,\,\,Integration and derivation in $\mathbb{R}^n$; Gauss-Ostrogradsky formula}\\
Consider a domain $G\subseteq \mathbb{R}^n$ (a jordanian body)for which the boundary $\partial G$ is piecewise smooth surface. By \lq\lq piecewise smooth $\partial G$ \rq \rq we mean that $S=\partial G=\mathop{\bigsqcup}\limits_{p=1}^{q}S_p$ where $(int\,S_i)\cap(int\,S_j)=\phi$ for any $i\neq j\in\{1,\dots,q\}$ and for each $x\in int\,S_p$, there is a neighborhood $V\subseteq \mathbb{R}^n$ and a first order continuously differentiable mapping \\
$y=\varphi(u)=\mathop{\prod}\limits_{i=1}^{n-1}(-a_i,a_i)=D_{n-1}\rightarrow \,S_p\cap \,V,\,\varphi(0)=x$ such that \\
$rank\parallel\frac{\partial\varphi(u)}{\partial u}\parallel=n-1$. Denote $\varphi=(\varphi_1,...,\varphi_n),\,u=(u_1,...,u_{n-1})$ and define the vectorial product of the vectors
$[\frac{\partial\varphi}{\partial u_1}(u),...,\frac{\partial\varphi}{\partial u_{n-1}}(u)]\subseteq\mathbb{R}^n$ as a vector of $\mathbb{R}^n$ given by the following formula\\
\begin{equation}\label{3:12.1}
N=det\left(
            \begin{array}{ccccc}
              \overrightarrow{e_1} & . & . & . & \overrightarrow{e_n} \\
               \frac{\partial\varphi_1}{\partial u_1}(u)&  &  &  & \frac{\partial\varphi_n}{\partial u_1}(u) \\
              . &  &  &  & . \\
              . &  &  &  & . \\
              . &  &  &  & . \\
             \frac{\partial\varphi_1}{\partial u_{n-1}}(u) & . & . & . & \frac{\partial\varphi_n}{\partial u_{n-1}}(u) \\
            \end{array}
          \right)
=[\frac{\partial\varphi}{\partial u_1}(u),...,\frac{\partial\varphi}{\partial u_{n-1}}(u)]
\end{equation}
where $\{\overrightarrow{e_1},...,\overrightarrow{e_n}\}\subseteq \mathbb{R}^n$ is the canonical basis of $\mathbb{R}^n$ and the formal writing of \eqref{3:12.1} stands for a simple rule of computation, when the components of
{\small$[\frac{\partial\varphi}{\partial u_1}(u),...,\frac{\partial\varphi}{\partial u_{n-1}}(u)]\\\subseteq \mathbb{R}^n$} are involved. In addition, \eqref{3:12.1} allows one to se easily that $N$ is orthogonal to any vector $\frac{\partial\varphi}{\partial u_i}(u)\in\mathbb{R}^n,\,i\in\{1,...,n-1\}$ and as a consequence $N$ is orthogonal to the point $\varphi(u)\in S_p$. In this respect, the scalar product $<N,  \frac{\partial\varphi}{\partial u_i}(u)>$
\,coincides with the computation of the following expression
\begin{equation}\label{3:12.2}
<N,\frac{\partial\varphi}{\partial u_i}(u)>=det\left(
                                                      \begin{array}{c}
                                                          \frac{\partial\varphi}{\partial u_i}(u) \\
                                                        \frac{\partial\varphi}{\partial u_1}(u)\\
                                                        . \\
                                                        . \\
                                                        . \\
                                                       \frac{\partial\varphi}{\partial u_{n-1}}(u)\\
                                                      \end{array}
                                                    \right)
=0\,\,for\,\,each\,\,i\in\{1,..,n-1\}
\end{equation}
If it is the case then the computation of the oriented surface integral $\int_Sf(x)dS$\,on the surface $S=\{S_1,...,S_q\}$\,will be defined by the following formula \\
\begin{equation}\label{3:12.3}
\int_Sf(x)dS=\int_{D_{n-1}}f(\varphi(u))|N|du
\end{equation}
where $|N|$ stands for the length of the vector $N$ defined in \eqref{3:12.1}. On the other hand, the normalized vector
\begin{equation}\label{3:12.4}
m(u)=N/|N|
\end{equation}
can be oriented in two opposite directions with respect to the domain $G$ and for the Gauss-Ostrogradsky\index{Gauss-Ostrogadsky} formula we need to consider that $m(u)$ is oriented outside of the domain $G$. Rewrite $m(u)$ in (\eqref{3:12.4} as
\begin{equation}\label{3:12.5}
m(u)=\overrightarrow{e_1}cos\,\omega_1+...+\overrightarrow{e_n}cos\,\omega_n
\end{equation}
where $\omega_i$ is the angle of the unitary vector $m(u)$ and the axis $x_i,\,i\in\{1,...,n\}$. Let $P(x)=P(x_1,...,x_n)$ be a first order continuously differentiable function in the domain $G$. Assume that $G\subseteq\mathbb{R}^n$ is simple with respect to each axis $x_k,\,k\in\{1,...,n\}$. Then the following formula is valid
\begin{equation}\label{3:12.6}
\int_G\frac{\partial P(x)}{\partial x_k}dx=\oint_S(cos\,\omega_k)P(x)dS\,,\,for\,each\,\,k\in\{1,...,n\}
\end{equation}
where $S=\partial G$ and the surface integral $\oint_S$ is oriented outside of the domain $G$.
\begin{theorem}{(Gauss(1813)-Ostrogradsky(1828-1834)formula)}\\
Let $P_k(x):G\subseteq \mathbb{R}^n\rightarrow\mathbb{R}$, be a continuously differentiable function for each $k\in\{1,...,n\}$ and assume that the jordanian domain $G$ is simple with respect to each axis $x_k\,,\,k\in\{1,...,n\}$. Then the following formula is valid
\begin{equation}\label{3:12.7}
\int_G[\frac{\partial P_1}{\partial x_1}(x)+...+\frac{\partial P_n}{\partial x_n}(x)]dx=\oint_S[(cos\,\omega_1)P_1(x)+...+(cos\,\omega_n)P_n(x)]dS
\end{equation}
where the surface integral $\oint_S$ is oriented outside of the domain $G$ and $S=\partial G$
\end{theorem}
\begin{proof}
The meaning that $G$ is simple with respect to each axis $x_k\,, k\in\{1,...,n\}$ will be explained for $k=n$ and let $Q\subseteq\mathbb{R}^{n-1}$ be the projection of the domain $G$ on the hyperplane determined by coordinates $x_1,...,x_{n-1}=\hat{x}$. It is assumed that $Q\subseteq\mathbb{R}^{n-1}$ is jordanian domain(see jordanian body given in definition \ref{df:10:1.2} of ($A_1$))and the jordanian domain $G\subseteq\mathbb{R}^{n}$ can be described by the following inequalities
\begin{equation}\label{3:12.8}
\varphi(x_1,...,x_{n-1})\leq x_n\leq\psi(x_1,...,x_{n-1})\,,\,(x_1,...,x_{n-1})=\hat{x}\in Q
\end{equation}
where $\varphi(\hat{x})\leq\psi(\hat{x})\,,\,\hat{x}\in Q$ are continuously differentiable functions. The surface
$x_n=\psi(\hat{x}),\,(\hat{x},x_n)\in G$ is denoted by $S_u$(upper surface)and $x_n=\varphi(\hat{x}),\,(\hat{x},x_n)\in G$ will be denoted by $S_b$(lower surface). Notice that the unitary vector $m$ defined in \eqref{3:12.4} must be oriented outside of $G$ at each point $p\in S_u$. It implies$<m,e_n>\geq 0$, similarly $<m,e_n>\leq 0$ for each $p\in S_b$. A direct computation of the orthogonal vector $N$ at each point of the surfaces $S_u=\{x_n-\psi(\hat{x})=0\}$\,and\,$S_u=\{x_n-\varphi(\hat{x})=0\}$ will lead us to
\begin{equation}\label{3:12.9}
<m,e_n>=<N\setminus|N|,e_n>=\left\{
                                   \begin{array}{ll}
                                     1\setminus(1+|\partial_{\hat{x}}\psi(\hat{x})|^2)^{1\setminus2}\,,\,x\in S_u \\
                                     1\setminus(1+|\partial_{\hat{x}}\varphi(\hat{x})|^2)^{1\setminus2}\,,\,x\in S_b
                                   \end{array}
                                 \right.
\end{equation}
Using the standard decomposition method of a multiple integral into its iterated parts we get
\begin{eqnarray}\label{3:12.10}
  \int_G\frac{\partial P_n(x)}{\partial x_n} &=& \int_Q[\int_{x_n=\varphi(\hat{x})}^{x_n=\psi(\hat{x})}\frac{\partial P_n}{\partial x_n}(\hat{x},x_n)dx_n]d\hat{x} \nonumber\\
  &=& \int_QP_n(\hat{x},\psi(\hat{x}))d\hat{x}-\int_QP_n(\hat{x},\varphi(\hat{x}))d\hat{x} \nonumber\\
  &=& \int_QP_n(\hat{x},\psi(\hat{x}))<m,e_n>(1+|\partial_{\hat{x}}\psi(\hat{x})|^2)^{1\setminus2}d\hat{x} \nonumber\\
  && +\int_QP_n(\hat{x},\varphi(\hat{x}))<m,e_n>(1+|\partial_{\hat{x}}\varphi(\hat{x})|^2)^{1\setminus2}d\hat{x} \nonumber\\
  &=&\int_{S_u}P_n(x)<m,e_n>\,dS+\int_{S_b}P_n(x)<m,e_n>\,dS \nonumber\\
  &=&\oint_SP_n(x)(cos\,\omega_n)d\,D
\end{eqnarray}

Here we have used the definition of the unoriented integral given in \eqref{3:12.3} and the proof of \eqref{3:12.7} is complete.
\end{proof}
\begin{remark}
Assuming that $G=\mathop{\cup}\limits_{i=1}^{p}G_i$, with $(int\,G_i)\cap(int\,G_j)=\phi\,\,if\,\,i\neq j$ and each $G_i$ is a jordanian domain, simple with respect to any axis $x_k\,,\,k\in\{1,...,n\}$, then the Gauss-Ostrogradsky formula \eqref{3:12.7} is still valid.
\end{remark}
\section{Appendix II Variational Method Involving PDE}\index{Variational method}
\subsection{Introduction}
A variational method uses multiple integrals and their extremum values for deriving some $PDE$ as first order necessary conditions. Consider a functional
\begin{equation}\label{3:12:1.1}
J(z)=\int_{D_m}L(x,z(x),\partial_xz(x))dx\,,\,D_m=\mathop{\prod}\limits_{i=1}^{m}[a_i,b_i]
\end{equation}
where $L(x,z,u):V\times\mathbb{R}\times\mathbb{R}^n\rightarrow\mathbb{R}\,,\,V(open)\supseteq D_m$ is a first order continuously differentiable function. We are looking for a continuously differentiable function $\hat{z}(x):V\times\mathbb{R}\,\,\,\hat{z}\in\mathcal{C}^1(V)$ such that
\begin{equation}\label{3:12:1.2}
\min J(z)=J(\hat{z})\,\,\,z\in \mathcal{A}
\end{equation}
where $\mathcal{A}\subseteq \mathcal{C}^1(V)$ is the admissible set of functions satisfying the following boundary conditions
\begin{equation}\label{3:12:1.3}
\left\{
  \begin{array}{ll}
   \left\{
     \begin{array}{ll}
       z(a_1,x_2,...,x_m)=z_{0}^{1}(x_2,...,x_m)\,,\,\,\,(x_2,...,x_m)\in\mathop{\prod}_{i=1}^{m}[a_i,b_i] \\
      z(b_1,x_2,...,x_m)=z_{1}^{1}(x_2,...,x_m)
     \end{array}
   \right.
 \\
    .\\
    .\\
   .\\
  \left\{
    \begin{array}{ll}
  z(x_1,...,x_{m-1},a_m)=z_{0}^{m}(x_1,...,x_{m-1})\,,\,\,\,(x_1,...,x_{m-1})\in\mathop{\prod}_{i=1}^{m-1}[a_i,b_i]\\
     z(x_1,...,x_{m-1},b_m)=z_{1}^{m}(x_1,...,x_{m-1})
    \end{array}
  \right.

  \end{array}
\right.
\end{equation}
Here the functions $z_{i}^{j}\,,\,i\in\{0,1\},\,j\in\{1,...,m\}$, describing boundary conditions,are some given continuous functions. This problem belongs to the classical calculus of variations which has a long tradition with significant contributions of Euler\index{Euler}(1739)and Lagrange(1736). The admissible class $\mathcal{A}\subseteq \mathcal{C}^1(V)$ is too restrictive and the existence of an optimal solution $\hat{z}\in\Omega$ is under question even if we assume additional regularity conditions on the Lagrange function $L$. A more appropriate class of admissible function is defined as follows. Denote $\partial D_m=\Gamma_m$ the boundary of the domain $D_m=\mathop{\prod}_{i=1}^{m-1}[a_i,b_i]$ and let $L_2(D_m;\mathbb{R}^m)$ be Hilbert space of measurable functions $p(x):D_m\rightarrow \mathbb{R}^m$  admitting a finite norm
$\parallel p\parallel=(\int_{D_m}|p(x)|^2dx)^{1\setminus2}<\infty$. Define the admissible class $A\subseteq \mathcal{C}(V)$ as follows
\begin{equation}\label{3:12:1.4}
A=\{z\in\mathcal{C}(D_m):\partial_x z\in L_2(D_m;\mathbb{R}^m),z|_{\Gamma_m}=z_0\}
\end{equation}
where $z_0\in\mathcal{C}(\Gamma_m)$\,is fixed. In addition,the Euler-Lagrange equation(first order necessary conditions)for the problem
\begin{equation}\label{3:12:1.5}
\mathop{\min}\limits_{z\in\,A}J(z)=J(\hat{z})
\end{equation}
where $J$ is  given in \eqref{3:12:1.1}, and $A$ in \eqref{3:12:1.4} can be rewritten as a second order $PDE$ provided
\begin{equation}\label{3:12:1.6}
\hbox{ Lagrange function }L(x,z,u):V\times\mathbb{R}\times\mathbb{R}^m\rightarrow\mathbb{R}
\end{equation}
is second order continuously differentiable and $|\partial_u\,L(x,z,u)|\leq \mathcal{C}_N(1+|u|)\,\,\,\forall\,u\in\mathbb{R}^m\,,\,x\in D_m$\,and\,$|z|\leq N$\,,where $\mathcal{C}_N>0$\,is a constant for each $N>0$.
\subsection[E-L Equation for Distributed Parameters Functionals]{Euler-Lagrange\index{Euler!Lagrange} Equation for Distributed Parameters Functionals}
A functional of the type \eqref{3:12:1.1} is defined by a Lagrange function $L$ containing a multidimensional variable $x\in D_m$ (distributed parameter). Assume that $L$ in \eqref{3:12:1.1} fulfils conditions  \eqref{3:12:1.6} and let $\hat{z}\in A$ (defined in \eqref{3:12:1.4}) be the optimal element satisfying \eqref{3:12:1.5}(locally), i.e there is a ball $B(\hat{z},\rho)\subseteq \mathbb{C}(D_m)$ such that
\begin{equation}\label{3:12:1.7}
J(z)\geq J(\hat{z})\,\,\forall\,\,z\in B(\hat{z},\rho)\cap A\
\end{equation}
Denote
\begin{equation}\label{3:12:1.8}
W^{1\setminus2}=\{z\in\mathcal{C}(D_m):their\,\,exists\,\,\partial_x z\in L_2(D_m;\mathbb{R}^m)\}
\end{equation}
and define a linear subspace $Y\subseteq W^{1\setminus2}$
\begin{equation}\label{3:12:1.9}
Y=\{y\in W^{1\setminus2},y|_{\Gamma_m}=0\}
\end{equation}
An admissible variation of $\hat{z}$\,is given by
\begin{equation}\label{3:12:1.10}
z_\varepsilon(x)=\hat{z}(x)+\varepsilon\,\overline{y}(x)\,,\,x\in D_m
\end{equation}
where $\varepsilon\in[0,1]$\,and\,$\overline{y}\in Y$. By definition $z_\varepsilon\in A\,,\,\varepsilon\in[0,1]$,\,and using \eqref{3:12:1.7} we get the corresponding first order necessary condition of optimality using Frech\'{e}t differential
\begin{equation}\label{3:12:1.11}
0=d\,J(\hat{z};\overline{y})=\mathop{lim}\limits_{\varepsilon\downarrow 0}\frac{J(z_\varepsilon)-J(\hat{z})}{\varepsilon}\,,\,\,\forall \overline{y}\in Y
\end{equation}
where $d\,J(\hat{z};\overline{y})$ is the  Frech\'{e}t differential \index{Differential! }of $J$ at $\hat{z}$ computed for the argument $\overline{y}$. The first form of the E-L equation is deduced from \eqref{3:12:1.11} using the following subspace $\overline{Y}\subseteq Y$
\begin{equation}\label{3:12:1.12}
\overline{Y}=sp\{\overline{y}\in\mathcal{C}(D_m):\overline{y}(x)=\mathop{\prod}\limits_{i=1}^{m}y_i(x_i),y_i\in D_{0}^{1}([a_i,b_i])\}
\end{equation}
Here the linear space $D_{0}^{1}([\alpha,\beta])$\,is consisting of all continuous functions $\varphi$\,which are derivable satisfying $\varphi(\alpha)=\varphi(\beta)=0$\,and its derivative $\{\frac{d\varphi}{dt}\,,t\in[\alpha,\beta]\}$\,is a piecewise continuous function. A direct computation allows us to rewrite
\begin{equation}\label{3:12:1.13}
0=d\,J(\hat{z},\overline{y})(see\,\,(11))\,\,for\,\,any\,\,\overline{y}\in \overline{Y}
\end{equation}
as follows
\begin{equation}\label{3:12:1.14}
0=\int_{D_m}[\partial_z L(x;\hat{z}(x);\partial_x\hat{z}(x))\overline{y}(x)+
\mathop{\sum}\limits_{i=1}^{m}\partial_{u_i}L(x;\hat{z}(x);\partial_x\hat{z}(x))\partial_{x_i}\overline{y}(x)]dx
\end{equation}
for any $\overline{y}\in\overline{Y}$, where $u=(u_1,...,u_m)\,,\,x=(x_1,...,x_m)$\\
Denote
{\small
\begin{equation}\label{3:12:1.15}
\left\{
        \begin{array}{ll}
          \psi_1(x)=\int_{a_1}^{x_1}\partial_z L(t_1,x_2,...,x_m;\hat{z}(t_1,x_2,...,x_m);\partial_x\hat{z}(t_1,x_2,...,x_m))dt_1, \\
          . \\
         . \\
         . \\
           \psi_m(x)=\int_{a_m}^{x_m}\partial_z L(x_1,x_2,...,,x_{m-1}t_m;\hat{z}(tx_1,x_2,...,x_{m-1},t_m);\partial_x\hat{z}(x_1,x_2,...,x_{m-1},t_m))dt_m
        \end{array}
      \right.
\end{equation}}
and integrating by parts in \eqref{3:12:1.14} we get (see $\partial_z L=\partial_{x_i}\psi_i$)
\begin{equation}\label{3:12:1.16}
0=\int_{D_m}\{\mathop{\sum}\limits_{i=1}^{m}[-\frac{1}{m}\psi_i(x)+\partial_{u_i}L(x;\hat{z}(x);\partial_x\hat{z}(x))]\partial_{x_i}\overline{y}(x)\}dx
\end{equation}
for any $$\overline{y}\in W^{1\setminus2},\overline{y}(x)=\mathop{\prod}\limits_{i=1}^{m}y_i(x_i)\,,\,y_i\in D_{0}^{1}([a_i,b_i])\,\,i\in\{1,...,m\}$$ where $\psi_1(x),...,\psi_m(x)$ are  defined in \eqref{3:12:1.15}. The integral equation \eqref{3:12:1.16} stands for the first form of the Euler-Lagrange\index{Euler!Lagrange} equation associated with the variational problem  defined in \eqref{3:12:1.11}. To get a pointwise form of the E-L equation \eqref{3:12:1.16} we need to assume that
\begin{equation}\label{3:12:1.17}
\hbox{ The optimal element }\hat{z}\in \mathcal{C}^2(\theta\subseteq D_m)
\end{equation}
is second order continuously differentiable on some open subset $\theta\subseteq D_m$.
\begin{theorem}{(E-L)}
Let $L(x,z,u):D_m\times\mathbb{R}\times\mathbb{R}^m\rightarrow\mathbb{R}$ be a second order continuously differentiable function and consider that the local optimal element $\hat{z}$ fulfils \eqref{3:12:1.17}. Then the following pointwise $E-L$ equation
\begin{equation}\label{3:12:1.18}
\partial_z\,L(x,\hat{z}(x),\partial_x\hat{z}(x))=\mathop{\sum}\limits_{i=1}^{m}\partial_{x_i}[\partial_{u_i}\,L(x,\hat{z}(x),\partial_x\hat{z}(x))]
\end{equation}
for any $x\in \theta\subseteq D_m$, is valid and $\hat{z}|_{\Gamma_m}=z_0\in \mathcal{C}(\Gamma_m)$
\end{theorem}
\begin{proof}
By hypothesis, the $E-L$ equation \eqref{3:12:1.16} is verified where $\psi_i(x)$ and \\$\partial_{u_i}L(x,\hat{z}(x),\partial_x\hat{z}(x))$ are
 first order continuously differentiable functions of $x\in \theta\subseteq D_m$ for any $i\in\{1,...,m\}$. The integral in \eqref{3:12:1.16}
 can be rewritten as follows
\begin{equation}\label{3:12:1.19}
\mathop{\sum}\limits_{i=1}^{m}[-\frac{1}{m}\psi_i(x)+\partial_{u_i}L(x;\hat{z}(x);\partial_x\hat{z}(x))]\partial_{x_i}\overline{y}(x)=E_1(x)-E_2(x)
\end{equation}
$$=\mathop{\sum}\limits_{i=1}^{m}\partial_{x_i}\{[-\frac{1}{m}\psi_i(x)+\partial_{u_i}L(x;\hat{z}(x);\partial_x\hat{z}(x))]\overline{y}(x)\}-$$
$$-\mathop{\sum}\limits_{i=1}^{m}\{\partial_{x_i}[-\frac{1}{m}\psi_i(x)+\partial_{u_i}L(x;\hat{z}(x);\partial_x\hat{z}(x))]\}\overline{y}(x)$$
for any $x\in \theta\subseteq D_m$. For each $x_0\in \theta$\,arbitrarily fixed, let $D_{x_0}\subseteq \theta$\,be a cube centered at $x_0$
and using \eqref{3:12:1.19} for
\begin{equation}\label{3:12:1.20}
\overline{y}\in\overline{Y},\,\overline{y}(x)=\mathop{\prod}\limits_{i}^{m}y_i(x_i),\,\overline{y}\in\mathcal{C}(D_{x_0}),\,\overline{y}|_{\partial D_{x_0}}=0
\end{equation}
we rewrite \eqref{3:12:1.16} restricted to the cube $D_{x_0}\subseteq \theta$
\begin{equation}\label{3:12:1.21}
\int_{D_{x_0}}E_1(x)Dx=\int_{d_{x_0}}E_2(x)dx
\end{equation}
On the other hand,applying Gauss-Ostrogradsky\index{Gauss-Ostrogadsky} formula to the first integral in \eqref{3:12:1.21} we get
\begin{equation}\label{3:12:1.22}
\int_{D_{x_0}}E_1(x)dx=\oint_{\partial D_{x_0}}\overline{y}(x)\{\mathop{\sum}\limits_{i=1}^{m}[-\frac{1}{m}\psi_i(x)+\partial_{u_i}L(x;\hat{z}(x),\partial_x\hat{z}(x))]cos\,\omega_i\}dS
\end{equation}
and using \eqref{3:12:1.20} we obtain
\begin{equation}\label{3:12:1.23}
0=\int_{D_{x_0}}E_1(x)dx=\int_{d_{x_0}}E_2(x)dx
\end{equation}
for any $\overline{y}\in\overline{Y}$ satisfying \eqref{3:12:1.20} where
\begin{eqnarray}
  E_2(x) &=& \{\mathop{\sum}\limits_{i=1}^{m}[-\frac{1}{m}\psi_i(x)+\partial_{u_i}L(x;\hat{z}(x),\partial_x\hat{z}(x))]\}\overline{y}(x) \nonumber\\
  &=& \{-\partial_z L(x;\hat{z}(x),\partial_x\hat{z}(x)) \nonumber\\
  &+& \mathop{\sum}\limits_{i=1}^{m}\partial_{u_i}L(x;\hat{z}(x),\partial_x\hat{z}(x))\}\overline{y}(x)
\end{eqnarray}
Assimilating \eqref{3:12:1.23} as an equation for a linear functional on the space $\overline{Y}_{x_0}=\{\overline{y}\in\mathcal{C}(D_{x_0}):\overline{y}|_{\partial D_{x_0}}=0\}$
\begin{equation}\label{3:12:1.25}
0=\int_{D_{x_0}}E_1(x)dx=\int_{D_{x_0}}h(x)\overline{y}(x)dx,\,\forall\,\overline{y}\in\overline{Y}_{x_0}
\end{equation}
a standard argument used in the scalar case can be applied here and it shows that
\begin{equation}\label{3:12:1.26}
h(x)=0\,\,\,\forall\,\,x\in D_{x_0}
\end{equation}
where
$$h(x)=\mathop{\sum}\limits_{i=1}^{m}\partial_{x_i}[\partial_{u_i}L(x;\hat{z}(x),\partial_x\hat{z}(x))]-\partial_zL(x;\hat{z}(x),\partial_x\hat{z}(x))$$
In particular $h(x_0)=0$\,where $x_0\in\theta$\,is arbitrarily fixed and the proof is complete.
\end{proof}
\begin{remark}
The  equation \eqref{3:12:1.25} is contradicted if we assume that $h(x)>0$.It is acomplished by constructing an auxiliary function $\overline{y}_{0}(x)=\sigma^2-|x-x_0|^2$ on the ball $x\in B(x_0,\sigma)\subseteq D_{x_0}$ which satisfy\\
\begin{enumerate}
  \item $\overline{y}_0(x)=0,\,\,\,\forall\,\,x\in \Gamma_{x_0}=boundary\,\,of\,\,B(x_0,\sigma)$
  \item $\overline{y}_0(x)>0,\,\,\,\forall\,\,x\in int\,B(x_0,\sigma) $
  \item $h(x)>0,\,\,\,\forall\,\,x\in int\,B(x_0,\sigma)\,if\,\,\sigma>0$ is sufficiently small
\end{enumerate}
Define $\overline{y}_0(x)=0$\,for any $x\in D_{x_0}\setminus B(x_0,\sigma)$\,and\\
$$\int_{D_{x_0}}h(x)\overline{y}_0(x)dx>0\,\,\,\,\,contradicting $$.\\
\end{remark}
\subsection{Examples of PDE Involving E-L Equation}
\textbf{($E_1$) \,Elliptic Equations}\\
\begin{example}\label{ex:3.12.1}\index{Elliptic equation}
Consider Laplace equation\index{Laplace equation}
$$\Delta z(x)=\mathop{\sum}\limits_{i=1}^{m}\partial_{i}^{2}z(x)=0$$ on a bounded domain $x\in D_{m}^{0}=\mathop{\prod}\limits_{i=1}^{m}(a_i,b_i), D_m=\mathop{\sum}\limits_{i=1}^{m}[a_i,b_i]$, associated with a Drichlet boundary condition $z\setminus{\partial D_m}=z_0(x)$ where $z_0\in\mathcal{C}(\partial D_m)$.\\
Define Dirichlet integral
$$D(z)=\frac{1}{2}\int_{D_m}|\partial_x z(x)|^2dx\,,\,\partial_x z(x)=(\partial_1 z(z),...,\partial_m z(x)),$$\\
and notice that the corresponding Lagrange function is given by
$$L(x,z,u)=\frac{1}{2}|u|^2\,,\,\,u\in\mathbb{R}^m$$
If it is the case then compute $\partial_z L=0$ and $\mathop{\sum}\limits_{i=1}^{m}\partial_{x_i}[\partial_{u_i}L(x;\hat{z}(x),\partial_x\hat{z}(x))]=\Delta\hat{z}(x)$ which allows one to see that (E-L)  equation \eqref{3:12:1.18} coincides with the above given Dirichlet problem provided $\{\hat{z}(x):x\in\,int\,D_m\}$ is second order continuously differentiable
\end{example}
\begin{example}\label{ex:3.12.2} Consider Piosson rquation $\Delta z(x)=f(x0\,,\,\,x\in \mathop{D_{m}}\limits^{0}$ associated with a boundary condition $z|_{\partial D_m}=z_0$, where $z_0\in \mathcal{C}(\partial D_m)$.\\
This Dirichlet problem for Poisson equation \index{Poisson equation}can be deduced from (E-L) equation \eqref{3:12:1.18}(see theorem (E-L))and in this respect associated the following functional
$$J(z)=\int_{D_m}\{1\setminus2|\partial_x z(x)|^2+f(x)z(x)\}dx$$
where $L(x,z,u)=1\setminus2|u|^2+f(x)z\,,\,u\in\mathbb{R}^m,\,z\in\mathbb{R}$, is the corresponding Lagrange function.Notice that $\partial_z L=f(x)$ and
$$\mathop{\sum}\limits_{i=1}^{m}\partial_{x_i}[\partial_{u_i}L(x;\hat{z}(x),\partial_x\hat{z}(x))]=\Delta\hat{z}(x)$$
allows one to write (E-L)  equation \eqref{3:12:1.18} as
$$\Delta\hat{z}(x)=f(x),\,\,\,\,x\in \mathop{D_{m}}\limits^{0}\,\,\hat{z}|_{\partial D_m}=z_)(x),\,\,for\,\,z_0\in\mathcal{C}(\partial D_m)\,.$$
\end{example}
\begin{example}
 A semilinear Poisson equation \index{Poisson equation}$\Delta z(x)=f(x),\,\,\,\,x\in \mathop{D_{m}}\limits^{0}$, associated with a Dirichlet boundary condition $z|_{\partial D_m}=z_)(x),\,\,for\,\,z_0\in\mathcal{C}(\partial D_m)$ can be deduced from (E-L)equation \eqref{3:12:1.18} provided we associate the following functional
$$J(z)=\int_{D_m}\{1\setminus2|\partial_x z(x)|^2+g(z(x))\}dx$$
where $g(z):\mathbb{R}\rightarrow \mathbb{R}$\,is a primitive of $f(z)\,,\,\frac{dg|z|}{dz}=f(z)\,,\,z\in\mathbb{R}$\\
Similarly, the standard argument used in
 examples \ref{ex:3.12.1} and \ref{ex:3.12.2} lead us to the following non linear elliptic equation\index{Elliptic equation}
$$-\Delta\hat{z}(x)=\hat{z}(x)|\hat{z}(x)|^{p-1}+f(\hat{z}(x))\,,\,\,\,x\in \mathop{D_{m}}\limits^{0}$$
provided $L(x,z,u)=\{\frac{1}{2}|u|^2-\frac{1}{p+1}|z|^{p+1}-g(z)\}\,,\,\,p\geq 1$ where $g(z):\mathbb{R}\rightarrow \mathbb{R}$ is a primitive of $f(z)$. Assuming that $f(z)$ satisfy $f(0)=0,\,\mathop{lim}\limits_{z\rightarrow \infty}\frac{f(z)}{|z|^p}=0,\,for\,\,p=3$\\
we take $g(z)=\frac{1}{2}\lambda|z|^2$ and the corresponding (E-L) equation coincides with so called Yang-Milles equation which is significant in Physics.
\end{example}
\textbf{$(E_2)$\,\,\,\,\,Wave  equation}\\
(1) Consider $(x,t)\in\mathbb{R}^3\times\mathbb{R}=\mathbb{R}^4$ and a bounded interval $D_4\subseteq \mathbb{R}^4$. Associated the following functional
$$J(z)=\int_{D_4}1\setminus2\{\partial_{t}^{2} |\partial_x z(t,x)|^2\}dx dt$$
Notice that the corresponding (E-L) equation \eqref{3:12:1.18} can be written as a wave equation \index{Wave equation}
$$\Box\hat{z}(t,x)=\partial_{t}^{2}\hat{z}(t,x)-\Delta_x\hat{z}(t,x)=0,\,(t,x)\in \mathop{D_{4}}\limits^{0}$$
(\lq\lq$\Box$\rq\rq d'Alembert operator\index{D'Alembert!operator})\\
(2)\,With the same notations as above we get Klein-Gordan equation(mentioned in math.Physics equation)
$$\Box\hat{z}(t,x)+k^2\hat{z}(t,x)=0\,,\,\,(t,x)\in \mathop{D_{4}}\limits^{0}$$
which agrees with the following Lagrange Function
$$L(x,z,u)=\frac{1}{2}\{(u_0)^2-\mathop{\sum}\limits_{i=1}^{3}(u_i)^2-k^2\,z^2\},\,u=(u_0,u_1,u_2,u_3)$$
Adding $\frac{\lambda}{4}z^4+j \,z$ to the above $L$ we get another Klein-Gordon equation

$$\Box\hat{z}(t,x)+k^2\hat{z}(t,x)=\lambda(\hat{z}(t,x))^3+j,\,(t,x)\in \mathop{D_{4}}\limits^{0}$$\\
\textbf{$(E_3)$}\,\,$PDE$\, involving mimimal-area surface\\
Looking for a minimal-area surface $z=\hat{z}(x), x\in D_m$ satisfying \\$\hat{z}|_{\partial D_m}=z_0\in\mathcal{C}(\partial D_m)$ we associate the functional
$$J(z)=\int_{D_m}\sqrt{1+|\partial_x z(x)|^2}dx$$
The corresponding (E-L) equation \eqref{3:12:1.18} is $\div(T(z))(x)=0\,,\,\,x\in \mathop{D_{4}}\limits^{0}$\\
where $T(z)(x)=\partial_xz(x)\setminus{\sqrt{1+|\partial_x z(x)|^2}}$\\
\textbf{$(E_4)$} $ODE$ as (E-L)equation for m=1\\
A functional
$$J(y)=\int_{a}^{b}\sqrt{1+(y'(x))^2}dx$$
stands for the length of the curve $\{y(x):x\in[a,b]\}$ and the corresponding (E-L) equation \eqref{3:12:1.18} is
$$\left\{
    \begin{array}{ll}
     \frac{d}{dx}\,\frac{\hat{y}'(x)}{\sqrt{1+(\hat{y}'(x))}}=0\,,\,\,x\in (a,b) \\
     \hat{y}(a)=y_a,\,\hat{y}(b)=y_b
    \end{array}
  \right.
$$
The solutions are expressed by linear $\hat{y}(x)=\alpha\,x+\beta,\,\,\,x\in\mathbb{R}$, where $\alpha,\beta\in\mathbb{R}$ are determined such that the boundary conditions $\hat{y}(a)=y_a,\,\hat{y}(b)=y_b$ are satisfied.
\section[Appendix III  Harmonic Functions]{Appendix III Harmonic Functions\index{Harmonic functions }; \\Recovering a Harmonic Function from its Boundary Values}
\subsection{Harmonic Functions}
A vector field $H(x)=(H_1(x),\ldots,H_n(x)):
V\subseteq \mathbb{R}^n\rightarrow \mathbb{R}^n$ is called harmonic
in a domain $V\subseteq \mathbb{R}^n$ if $H\in \mathcal{C}^1(V;\mathbb{R}^n)$
and
\begin{equation}\label{3:13:1.1}
 \div H(x)=\sum_{i=1}^n\partial x_i H_i(x)=0, (\partial
x_i H_j-\partial x_j H_i)(x)=0, i,j\in\{i,\ldots,n\}
\end{equation}
for any $x\in V.$ In what follows we restrict ourselves to simple
convex domain $V$ and notice that the second condition in \eqref{3:13:1.1}
implies that $H(x)=grad \phi (x)=\partial_x\phi (x)$ of some scalar
function $\phi$ ($H$ has a potential) where $\phi$ is second order
continuously differentiable. Under these conditions, the first
constraint in \eqref{3:13:1.1} can be viewed as an equation for the scalar
function $\phi$
\begin{equation}\label{3:13:1.2}
0=\div H(x)=\div(\grad
\phi)(x)=\sum_{i=1}^n\partial_i^2\phi(x)=\Delta \phi(x)
\end{equation}
Any scalar function $\phi(x): V\subseteq \mathbb{R}^n\rightarrow
\mathbb{R}$ which is second order continuously differentiable and
satisfies the Laplace equation\index{Laplace equation} \eqref{3:13:1.2} for $x\in V$ will be
called a harmonic function\index{Harmonic function } on $V.$

\subsection{Green Formulas\index{Green!formula}}
Let $V\subseteq \mathbb{R}^n$ be a
bounded domain with a piecewise smooth boundary $S=\partial V$;
consider a continuously derivable scalar field $\psi(x):V\rightarrow
\mathbb{R}^n.$ By a direct computation we get
$(\partial=(\partial_1,\ldots,\partial_n))$
\begin{equation}\label{3:13:2.1}
\div\psi R=<\partial,\psi
R>=\psi<\partial,R>+<R,\partial\psi>=\psi\Delta\phi+<\partial\phi,\partial\psi>
\end{equation}
Applying Gauss-Ostrigradsky \index{Gauss-Ostrogadsky}formula and integrating both terms of
equality
we get\\
\begin{eqnarray*}
  \int_V\psi\Delta\phi dx + \int_V<\partial\phi,\partial\psi>dx &=& \int_V(\div\psi R)dx \\
  &=& \oint_V<m,\psi R>dS \\
  &=& \oint_S\psi<m,R>dS\\
  &=&\oint_S\psi(D_m\phi)dS
\end{eqnarray*}
which stands for
\begin{equation}\label{3:13:2.2}
\int_V<\partial\phi,\partial\psi>dx+\int_V\psi\Delta\phi
dx=\oint_S\psi(D_m\phi)dS
\end{equation}
Here $D_m\phi=<m,\partial\phi>$ is the derivative of the scalar
function $\phi$ in the normal direction represented by the unitary
orthogonal vector $m$ at the surface $S=\partial V$ oriented outside
of $V$. The expression in \ref{3:13:2.2} is the first Green formula
and by permutation and subtracting we get the second Green formula\index{Green!formula}
\begin{equation}\label{3:13:2.3}
\int_V(\psi\Delta\phi-\phi\Delta\psi)dx=\oint_S[\psi(D_m\phi)-\phi(D_m\psi)]dS
\end{equation}
\begin{theorem}
\begin{description}
  \item[(a)] If a harmonic function $h$ vanishes on the boundary
  $S=\partial V$ then $h(x)=0$ for any $x\in int V$
  \item[(b)]If $h_1$ and $h_2$ are two harmonic functions satisfying $h_1(x)=h_2(x),\,x\in S=\partial V,$ then $h_1(x)=h_2(x)$
  for all $x\in int\,V.$
  \item[(c)]If a harmonic vector field $H(x)$ satisfies $<m,H>=0$ on the boundary $S=\partial V$ then $H(x)=0$ for any $x\in int\, V.$
  \item[(d)] If $H_1$ and $H_2$ are two harmonic vector fields satisfying $<m,H_1>=<m,H_2>$ on the boundary $S=\partial V$ then $H_1(x)=H_2(x)$
   for any $x\in int\, V.$
\end{description}
\end{theorem}
\begin{proof}
$(a)$: Using the first Green formula \index{Green!formula} \ref{3:13:2.2} for $\phi=\psi=h$
and $\Delta h=0$ we get $\int_V |\partial h|^2dx=0$ and therefore
$\partial h=grad h(x)=0$ for any $x\in V$ which implies
$h(x)=const$, $x\in V.$ Using $h(x)=0$, $x\in S=\partial V$. We
conclude $h(x)=0$ for any $x\in V.$\\
$(b)$: For $h(x)=h_1(x)-h_2(x)$
use the conclusion of (a).\\
$(c)$: By hypothesis $H(x)=\partial h (x),\,x\in V,$ where the
potential function $h(x),\,x\in V,$ is harmonic, $\Delta h=0.$ Using
the first Green formula \eqref{3:13:2.2} for $\phi=\psi=h$ with $\Delta \phi=0$
and $D_m\phi=<m,H>(x)=0$ for $x\in S=\partial V.$. We get
\[
\int_V |\partial h|^2dx=\int_V |H(x)|^2dx=0\,\text{ and}
\,H(x)=0\text{ for all} x\in V.
\]
$(d)$: Follows from $(c)$ using the same argument as in $(b).$ The
proof is complete.
\end{proof}
\begin{theorem}
\begin{description}
  \item[(a)]If $h(x)$ is harmonic on the domain
  $V\subset \mathbb{R}^n$ then\\ $\oint _S(D_m h)(x)dS=0,$ where $D_m h=<m,\partial h>.$
  \item[(b)]Let $B(y,r)\subset V$ be a ball centered at $y\in int V,$
  and $h$ is a harmonic function\index{Harmonic function } on $V\subset \mathbb{R}^n$. Then
  the arithmetic mean value on $\sum=\partial B(y,r)$ equals the
  value $h(y)$ at the center $y$, $$h(y)=\frac{1}{|\sum|}\oint_\sum h(x)dS.$$
  \item[(c)]A harmonic function $h(x):\mathbb{R}^n\rightarrow
  \mathbb{R}$ which satisfies $\lim _{|x|\rightarrow \infty}h(x)=0$
  is vanishing everywhere, $h(x)=0,$ for all $x\in \mathbb{R}^n$
\end{description}
\end{theorem}
\begin{proof}
$(a)$: Use \eqref{3:13:2.3} for $\phi=h,\,\psi=1.$\\
$(b)$: Let $B(y,\rho)\subset B(y,r)$ be another ball, $\rho < r$,
and denote $\sum_r=\partial B(y,\rho)$ the corresponding boundaries.
Define $V_{r\rho}=W-Q$ and notice that the boundary $S=\partial
V_{r\rho}=\sum_r \coprod \sum_{\rho}$, where $W=B(y,r)$,
$Q=B(y,\rho)$. The normal derivative $D_m=<m,\partial>$ is oriented
outside of the domain $V_{r\rho}$ and it implies
$$D_m=\partial_r\text{ on }\sum_r$$
\[
D_m=-\partial_r\text{ on }\sum_{\rho}
\]
Take $V=V_{r\rho},\,\varphi=h,\,\psi=\frac{1}{|x-y|^{n-2}} $ in the
second Green formula \index{Green!formula}(\eqref{3:13:2.3}). It is known that
$\psi(x),\,x\in V$, is a harmonic function\index{Harmonic function } on $V$ (see $x\neq y$)
and as a consequence the left hand side in (\eqref{3:13:2.3}) vanishes.
The corresponding right hand side in (\eqref{3:13:2.3}) can be written
as a difference on $\sum_r$ and $\sum_{\rho}$ and the equation
(\eqref{3:13:2.3})lead us to
\[
\oint_{\sum_r}[\frac{1}{r^{n-2}}\frac{\partial h}{\partial r}-h
\frac{\partial}{\partial
r}(\frac{1}{r^{n-2}})]dS=\oint_{\sum_\rho}[\frac{1}{\rho^{n-2}}\frac{\partial
h}{\partial r}-h \frac{\partial}{\partial r}(\frac{1}{r^{n-2}})]dS.
\]
Performing the elementary derivatives we get
\[
\frac{1}{r^{n-2}}\oint_{\sum_r}\frac{\partial h}{\partial
r}dS+\frac{n-2}{r^{n-1}}\oint_{\sum_r}h\,
dS=\frac{1}{\rho^{n-2}}\oint_{\sum_{\rho}}\frac{\partial h}{\partial
r}dS+\frac{n-2}{\rho^{n-1}}\oint_{\sum_{\rho}}h\,dS.
\]
The first term in both sides of the last equality vanishes (see
$(a)$). Let $|S_1|$ be the area of the sphere $S(0,1)\subseteq
\mathbb{R}^n$ and dividing by $|S_1|$. We get
$\frac{1}{r^{n-1}|S_1|}\oint_{\sum_r}h\,
dS=\frac{1}{\rho^{n-1}|S_1|}\oint_{\sum_\rho}h\,dS$, where
$r^{n-1}|S_1|=|\sum_r|,\,\,\rho^{n-1}|S_1|=|\sum_\rho|$. Using the
continuity property of $h$ and letting $\rho\rightarrow 0$ from the
last equality we get the conclusion $(b)$.
$(c):$ Using the conclusion of $(b)$ for $r\rightarrow \infty$, we
get conclusion $(c)$.
\end{proof}
\subsection{Recovering a Harmonic Function Inside a Ball I}
\subsubsection{ By Using its Boundary Values}
The arguments for this conclusion are based on the second Green
formula\index{Green!formula} written on a domain $V=W-Q$, where $W=B(0,r)$,
$Q=B(y,\rho)\subseteq B(0,r).$ This time, we use the following
harmonic functions\index{Harmonic functions } on $V(\partial V=S=\sum_r\coprod \sum_\rho)$
\begin{equation}\label{3:13:3.1}
\varphi=h(x),\,\psi=\psi(x)=\frac{1}{|x-y|^{n-2}}-\frac{r^{n-2}}{|y|^{n-2}}\frac{1}{|x-y*|^{n-2}}
\end{equation}
where $y*=\frac{r^2}{|y|^2}\, y$, $\psi(x)=0$, $x\in \sum_r=\partial
B(0,r)$ and $\psi_0(x)=\frac{r^{n-2}}{|y|^{n-2}}
\frac{1}{|x-y*|^{n-2}}$ is harmonic on $W$. The second Green formula
becomes
\begin{equation}\label{3:13:3.2}
 -\oint_{\sum_r}h(x)\frac{\partial \psi}{\partial
r}dS=\oint_{\sum_{\rho}}[\psi \frac{\partial h}{\partial
\rho}(x)-h\frac{\partial \psi}{\partial \rho}(x)]dS
\end{equation}
where $\frac{\partial}{\partial \rho}$ stands for the derivative
following the direction of the radius $[y,x],\,x\in \sum_{\rho}.$ A
direct computation lead us to
\begin{equation}\label{3:13:3.3}
{\begin{cases}|\oint_{\sum_\rho}\psi \frac{\partial
h}{\partial \rho}(x)dS|\leq c_1\frac{1}{\rho^{n-2}}c_2
\rho^{n-1}\rightarrow 0\,
\text{for}\,\rho\rightarrow 0\\
|\oint_{\sum_\rho}h \frac{\partial \psi_0(x)}{\partial \rho}dS|\leq
c\rho^{n-1}\rightarrow 0\,\text{for}\, \rho\rightarrow 0\end{cases}}
\end{equation}
and
\begin{equation}\label{3:13:3.4}
\oint_{\sum_{\rho}}
h\frac{\partial}{\partial\rho}\frac{1}{\rho^
{n-2}}ds=(n-2)\oint_{\sum_{\rho}} h \frac{1}{\rho^{n-1}}ds=\frac{(n-2)\mid
S_1 \mid }{\mid \sum_\rho \mid}
\end{equation}
$$\oint_{\sum_{\rho}} h
\,ds\longrightarrow(n-2)
\mid S_1 \mid h(y),\, for
\rho\longrightarrow 0
$$
Letting $\rho\longrightarrow 0$ from \eqref{3:13:3.2} we get
\begin{equation}\label{3:13:3.5}
 h(y)=\frac{-1}{(n-2)\mid S_1 \mid}\oint_\sum \,
h(x)\frac{\partial \psi}{\partial r}(x)\, ds
\end{equation}
It remains to compute $\frac{\partial \psi (x)}{\partial r}$ on
$x\in \sum_r$ and it is easily seen that $see \, \psi(x)=const \,
x\in\sum_r$
\begin{equation}\label{3:13:3.6}
 \frac{\partial\psi}{\partial r}(x)=-\mid \, grad\,
\psi(x) \mid
\end{equation}
By definition, $y^*$ is taken such that
\begin{equation}\label{3:13:3.7}
\frac{\mid x-y \mid^2}{\mid x-y^* \mid
^2}=\frac{r^2-2<x,y>+\mid y \mid^2}{r^2-2<x,y>\frac{r^2}{\mid y
\mid^2}+\frac{r^4}{\mid y \mid^2}}=\frac{\mid y \mid^2}{r^2}=const
\end{equation}
for any $x\in \sum_r \, $ and the direct computation of $\mid grad
\psi(x) \mid$ shows that
\begin{equation}\label{3:13:3.8}
\mid grad \,\psi(x) \mid=\frac{(n-2)(r^2-\mid y
\mid^2)}{r\mid x-y \mid^n}\, , \mid x \mid=r
\end{equation}
provided \eqref{3:13:3.7}  is used. Using \eqref{3:13:3.6} and \eqref{3:13:3.8}  int \eqref{3:13:3.5}
we get the following Poisson formula
\begin{equation}\label{3:13:3.9}
  h(y)=\frac{r^2-\mid y \mid^2}{\mid S_1 \mid r}\oint_\sum
 \frac{h(x)}{\mid x-y \mid^n}\, ds=\oint_\sum P(x,y)\, h(x)\,ds ,
\end{equation}
where Poisson kernel
 \begin{equation}\label{3:13:3.10}
 P(x,y)=\frac{r^2-\mid y \mid^2}{\mid S_1 \mid r}\frac{1}{\mid x-y
 \mid^n}>0 \, \, for\, \,  any\, \,  \,  y \in \, B(0,r)
\end{equation}
\begin{theorem}
   Let $\lambda(x):\Sigma\subseteq\mathbb{R}^n\rightarrow \mathbb{R}$\,be a continuous function where $\Sigma=\partial W$\, is the boundary of a fixed ball $W\subset\mathbb{R}^n$. Then there exists a unique continuous function \,$h(y):W\longrightarrow R$\,which is harmonic for $y\in \,int \, W$\, and coincides with$\lambda(x)$\,on the   boundary $\Sigma$.
\end{theorem}
\begin{proof}
Let $ r>0 $\, be the radius of the ball$ W $\,and for any $ y \in \,int \,W  $\,define $ h(y)$\,as follows.
\begin{equation}\label{3:13:3.11}
h(y)=\oint_\sum P(x,y)\lambda(x)ds
\end{equation}
where $P(x,y)=\frac{r^2-\mid y\mid ^2}{\mid S_1\mid r}\frac{1}{\mid x-y\mid^n}$
is the Poisson kernel $ (see \eqref{3:13:3.10})$
We know that both  $\varphi (y)=\frac{1}{\mid x-y \mid^{n-2}},\, y\in \, int\, W$
 and its derivatives $\frac { \partial\varphi(y)}{\varphi y_i},\, \,   y \in \, int W,i\in {1,2\ldots,n} $
 are harmonic function\index{Harmonic functions }.
 As a consequence the following linear combination

 \begin{eqnarray}\label{3:13:3.12}
   \varphi(y)-\frac{2}{n-2}\sum_{i=1}^n x_i\frac{ \partial \varphi (y) }{\partial y_i} &=& \frac{1}{\mid x-y\mid^{n-2}}+2 \sum_{i=1}^n\frac{x_i(y_i-x_i)}{\mid x-y \mid^n} \nonumber\\
   &=& \frac{1}{{\mid x-y \mid}^n}(\sum_{i=1}^n (x_i-y_i)^2+2\sum_{i=1}^n x_iy_i-2\sum_{i=1}^n x_i^2) \nonumber\\
   &=& {\frac{1}{\mid x-y \mid^n}} \sum_{i=1}^n(y_i^2-x_i^2) \nonumber\\
   &=& \frac{\mid y^2 \mid -r^2}{\mid x-y\mid ^n}=-\mid S_i\mid P(x,y)
 \end{eqnarray}
is a harmonic function\index{Harmonic functions }. In conclusion, $h(y),\, y\in \,int \,W,$ defined in (\eqref{3:13:3.11}is a harmonic function. On the other hand, for a sequence $\{y_m\}_{m\geq 1}\subset \ int W ,\mathop{lim}\limits_{m\longrightarrow \infty} y_m=\chi_0 \in \Sigma$. We get
\begin{equation}\label{3:13:3.13}
 \oint_\Sigma P(x,y_m)dS=1   (see\eqref{3:13:3.9}, for h\equiv 1)
 \end{equation}
\begin{equation}\label{3:13:3.14}
\lim _{m\longrightarrow\infty} \int_\Sigma P(x,y_m)dS=0, where \Sigma'=\partial B(x_0,\delta)\cap\Sigma,\delta>0,
\end{equation}
and $\mid x-y_m\mid\geq c> 0,m \geq 1,$ are used.
\end{proof}
Using \eqref{3:13:3.13} and \eqref{3:13:3.14} we see easily that
\begin{equation}\label{3:13:3.15}
\lim _{m\rightarrow \infty} \oint _\sum P(x,y_m)\lambda(x)ds=\lambda(x_0)
\end{equation}
and $\{h(y): y  \in  W\}$  is a continuous function satisfying
\begin{equation}\label{3:13:3.16}
h(x) = \lambda(x), x\in\sum
\end{equation}
\begin{equation}\label{3:13:3.17}
 \triangle h(y) = 0 , (for all) y \in int W
\end{equation}
A continuous function satisfying \eqref{3:13:3.16} and \eqref{3:13:3.17} is unique provided the maximum principle for laplace equation\index{Laplace equation} is used.
\subsection{Recovering a Harmonic Function Inside a Ball II}
\subsubsection{By Using its Normal Derivative on the Boundary}
Let $h(y) : W \subseteq\mathbb{R}^n \rightarrow R $ be a harmonic function and $W=B(z,r)\subseteq \mathbb{R}^n$ is a ball. Then
\begin{equation}\label{3:14.1}
\phi(y) = \sum_{i=1}^n y_i \frac{\partial h(y)}{ \partial y_i } \,and\, \frac{\partial h(y)}{\partial y_i}, i \in {1,2,....,n}
\end{equation}
are harmonic functions on the ball W;
\begin{equation}\label{3:14.2}
\varphi(x) = \varphi \frac{\partial h(x)}{\partial\rho},\hbox{ for any x in }\Sigma = \partial W
\end{equation}
where
$\frac{\partial h(x)}{\partial\rho}$
is the normal derivative of h on $\Sigma.$
The property \eqref{3:14.1} is obtained by a direct computation the property \eqref{3:14.2} uses the following argument.
Without restricting generality, take
$z = 0, r = 1$, and for $y = \rho.x, |x| = 1, 0\leq\rho\leq1$, rewrite h(y) in \eqref{3:14.1} as follows
\begin{theorem}
Let $\lambda(x):\sum \subseteq \mathbb{R}^{n}\rightarrow \mathbb{R}$
be a continuous function satisfying
$\oint\limits_{\sum}\lambda(x)dS=0$, where $\sum=\partial W$ is the
boundary of a ball $W=B(z,r)\subseteq \mathbb{R}^{n}$ .Then there
exists a continuous function $h(y):W\rightarrow \mathbb{R}$ which is
harmonic on $y\in int W$ and admitting the normal derivative
$\frac{\partial h(\rho x)}{\partial \rho}\ (x\in \sum, 0<\rho
\leq1)$ which equals $\lambda(x)$ for $\rho=1$. Any other function
$\varphi$ with these properties verifies
\begin{equation}\label{3:14.3}
\varphi(y)-h(y)=const\ (\forall)\ y\in W.
\end{equation}
\end{theorem}
\begin{proof}
The existence of a function $h$ uses the Poisson kernel $P(x,y)$
given in $(\S 3.10)$,
\begin{equation}\label{3:14.4}
\varphi(y)=\oint\limits_{\sum}P(x,y)\lambda(x) dS.
\end{equation}
Using $(\S 3.16,\S 3.17)$ we know that the function $\{\varphi(y),
y\in W\}$ is continuous and
\begin{equation}\label{3:14.5}
\left\{
  \begin{array}{ll}
   \varphi(x)=\lambda(x), \indent x\in\sum \\
 \Delta \varphi(y)=0, \hbox{ y in int W (varphi is harmonic on int W)}
  \end{array}
\right.
\end{equation}
In  addition, using $\oint\limits_{\sum}\lambda(x)dS=0$ we get
\begin{equation}\label{3:14.6}
\varphi(0)=\oint\limits_{\sum}\lambda(x)dS=0
\end{equation}
Define
\begin{equation}\label{3:14.7}
 h(y)=h(\rho
x)=\int_{0}^{\rho}\frac{\varphi(\tau x)}{\tau}d \tau +h(0)
\end{equation}
where $\varphi(y)$ is continuously derivable at $y=0,\ \varphi(0)=0
$ and $h(0)=const$ (arbitrarily fixed). We shall show that
$h(y):W\rightarrow \mathbb{R}$ is a continuous function satisfying
\begin{equation}\label{3:14.8}
 \Delta h(y)=0,\ y\in int\,W \ (h \
\hbox{is harmonic in}\ int\,W)
\end{equation}
\begin{equation}\label{3:14.9}
\frac{\partial h(\rho x)}{\partial
\rho}\bigg|_{\rho =1}=\frac{\partial h(x)}{\partial
\rho}=\varphi(x)=\lambda (x),\ x\in \sum
\end{equation}
For simplicity, take $h(0)=0$, and using \eqref{3:14.7} for $\rho=1$ we compute
{\small
\begin{equation}\label{3:14.10} \oint\limits_{\sum}P(s,y)h(s) d S
=\oint\limits_{\sum}P(s,y)\Big\{\int_{0}^{1}\frac{\varphi(\tau
s)}{\tau}d \tau \Big\}d
S\int_{0}^{1}\bigg\{\oint\limits_{\sum}P(s,y)\frac{\varphi(\tau
s)}{\tau}d S\bigg\}d\tau
\end{equation}
}
By the definition $\frac{\varphi(\tau s)}{\tau}, \tau\in [0,1]$
fixed, is a harmonic function satisfying \eqref{3:14.5} and \eqref{3:14.10} is rewritten
as
\begin{equation}\label{3:14.11}
 \oint\limits_{\sum}P(s,y)h(s) d S
=\int_{0}^{1}\frac{\varphi(\tau s)}{\tau}d \tau
=\int_{0}^{\rho}\frac{\varphi(\xi x)}{\xi}d \xi =h(y).
\end{equation}
where
\begin{equation}
\bigg\{\frac{\varphi(\tau y)}{\tau}, \tau\in [0,1]\bigg\}\rightarrow
\bigg\{\frac{\varphi(\xi y)}{\xi}, \xi\in [0,\rho]\bigg\}
\end{equation}
and $\xi= \rho \tau$ are used. The equation \eqref{3:14.11} stands for \eqref{3:14.8} and
\eqref{3:14.9} is obtained from \eqref{3:14.7} by a direct derivation.
\end{proof}
\section*{Bibliographical Comments}
Mainly, it was written using the references \cite{5} and \cite{10}. The last section 3.9 follows the same presentation as in the reference \cite{12}. In the appendices are
are used the presentation contained in the references \cite{4} and \cite{9}.
\chapter[Stochastic Differential Equations]{Stochastic Differential Equations; \index{Stochastic!Differential Equations}\index{Differential!equation}Approximation and Stochastic Rule of Derivations}\index{Stochastic! Rule of Derivations}
\section[Properties of continuous semimartingales]{Properties of continuous semimartingales \index{semimartingale}and stochastic integrals}\index{Stochastic!integral}
Let a complete probability space $ \{ \Omega, F, P \}$ be given. Assume that a
family of sub $\sigma$-fields $F_t \subseteq F, t \in [0, T]$, is given such
that the following properties are fulfilled:

i) Each $F_t$ contains all null sets of $F$.

ii) $(F_t)$ is increasing, i.e. $F_t \subseteq F_s$ if $ t \geq s$.

iii) $(F_t)$ is right continuous, i.e. $\mathop{\bigcap}\limits_{\varepsilon >
0} F_{t + \varepsilon} = F_t$ for any $t<T$.
Let $X(t), t \in [0, T]$ be a measurable stochastic process \index{Stochastic! process}with values in $R$.
We will assume, unless otherwise mentioned, that it is $F_t$-adapted, i.e.,
$X(t)$ is $F_t$-measurable for any $t \in [0, T]$.
The process $X(t)$ is called continuous if $X(t; \omega)$ is a continuous
function of $t$ for almost all $\omega \in \Omega$. Let $L_c$ be the linear
space consisting of all continuous stochastic processes. We introduce the metric
$\rho$ by $|X-Y|= \rho (X,Y)= (E[\mathop{\sup}\limits_t | X(t) -Y(t)|^2 /( 1+
\mathop{\sup}\limits_t | X(t)-Y(t)|^2])^{1/2}$. It is equivalent to the topology
of the uniform convergence in probability. A sequence $\{X^n\}$ of $L_c$ is a
Cauchy sequence \index{Cauchy!sequence}iff for any $\varepsilon >0$, $P (\mathop{\sup}\limits_t |X^n
(t) - X^m (t)|>\varepsilon ) \to 0$ if $n, m \to \infty$. Obviously $L_c$ is a
complete metric space.
We introduce the norm $|| \cdot ||$ by $||X|| = (E [\mathop{\sup}\limits_t
|X (t)|^2])^{1/2}$ and denote by $L^2_c$ the set of all elements in $L_c$ with
finite norms. We may say that the topology of $L_c^2$ is the uniform convergence
in $L^2$. Since $\rho (X, 0) \leq ||X||$, the topology by $||\,||$ is stronger
than that by $\rho$ and it is easy to see that $L_c^2$ is a dense subset of $L_c$.\\
\begin{definition}
Let $X(t), t \in [0, T]$, be a continuous $F_t$-adapted process.
(i) It is called a martingale if $E|X(t)| < \infty$ for any $t$ and satisfies
$E[X(t)/F_s]=X(s)$ for any $t>s$.

(ii) It is called a local martingale if there exists an increasing sequence of
stopping times $(T_n)$ such that $T_n  \uparrow T$ and each stopped process
$X_{(t)}^{T_n} \equiv X (t \wedge T_n)$ is a martingale\\
(iii) It is called an increasing process if $X(t; \omega)$ is an increasing
function of $t$ almost surely (a.s.) with respect to $\omega \in \Omega$, i.e.
there is a $P$-null set $N \subseteq \Omega$ such that $X (t, \omega)$, $t \in
[0,T]$ is an increasing function for any $\omega \in \Omega \setminus N$.

(iv) It is called a process of bounded variation if it is written as the
difference of two increasing processes.

(v) It is called a semi-martingale if it is written as the sum of a local
martingale and a process of bounded variation.

We will quote two classical results of Doob concerning martingales \mbox{without}
giving proofs.(see A. Friedmann)
\end{definition}
\begin{theorem}
{\it Let $X(t), t \in [0,T]$ be a martingale.}

(i) {\bf Optional sampling theorem}. {\it Let $S$ and $\mathcal{U}$ be
stopping times with values in $[0,T]$. Then $X(S)$ is integrable and satisfies
$E[X(S) / F(\mathcal{U})]= X(S \wedge \mathcal{U})$, where $F(\mathcal{U})= \{A
\in F_T : A \cap \{ \mathcal {U} \leq t \}\} \in F(t)$ for any $t \in [0,T]\}$}

(ii) {\bf Inequality}. {\it Suppose $E|X(T)|^p < \infty$ with $p>1$. Then }
\[
E \mathop{\sup}\limits_t |X (t)|^p \leq q^p E|X(T)|^p
\]
{\it where}
\[
\frac{1}{p} + \frac{1}{q}=1
\]
\end{theorem}

\begin{remark}

\textit {Let $S$ be a stopping time. If $X(t)$ is a martingale the stopped
process $X^S (t) \equiv X (t \wedge S)$ is also a martingale. In fact, by Doob's
optional sampling theorem we have (for $ t \geq s$) $E[X^S (t)| F(s)]=X (t
\wedge S \wedge s) = X(S \wedge s) = X^S (s)$.
Similarly, if $X$ is a local martingale the stopped process $X^S$ is a local
martingale.}
{\it Let $X$ be a local martingale. Then there is an increasing sequence of
stopping times $S_k \uparrow T$ such that each stopped process $X^{S_k}$ is a
bounded martingale. In fact, define $S_k$ by $S_k= \inf \,\{ t>0: |X(t)|
\geq k\} (= T$} if {\it $\{ \cdots \} = \phi )$. Then $S_k \uparrow T$ and it
holds $\mathop{\sup}\limits_t |X^{S_k} (t) |\leq k$, so that each $X^{S_k}$
is a bounded martingale.}
\end{remark}
\begin{remark}
Let $M_c$ be the set of all square integrable martingale, $X(t)$ with $X(0)=0$.
Because of Doob's inequality the norm $||X||$ is finite for any $X \in M_c$.
Hence $M_c$ is a subset of $L_c^2$. We denote $M_c^{\mbox{loc}}$ the set of all
continuous local martingales $X(t)$ such that $X(0)=0$ ;\, it is a subset of
$L_c$.
\end{remark}

\begin{theorem}
{\it  $M_c$ is a closed subspace of $L^2_c\cdot M_c^{\rm{loc}}$
is a closed subspace of $L_c$. Further more, $M_c$ is dense in
$M_c^{\rm{loc}}$.}
\end{theorem}
\begin{remark}

{\it Denote by $H^2_t$ the set consisting of all random variables $h:\Omega
\to R$
which are $F_t$-measurable and $E|h|^2 < \infty$. Then for each $X \in L^2_c$
there holds $E[X(T) / F_t]= \widehat{h}$ where $\widehat{h} \in H_t^2$ is the
optional solution for}
\[
\mathop{\min}\limits_{h \in H^2_t} E|X(T)-h|^2=E|X(T)- \widehat{h}|^2
\]
\vspace{2mm}
{\bf Definition} Let $X(t)$ be a continuous stochastic process and $\Delta$ a
partition of the interval $[0,T]: \Delta = \{0=t_0 < \ldots < t_n=T$, and let
$|\Delta|= \max (t_{i+1}-t_i)$. Associated with the partition $\Delta$, we
define a continuous process ${\langle X \rangle}^{\Delta} (t)$ as
\[
{\langle X \rangle}^{\Delta} (t)=\sum_{i=0}^{k-1} (X(t_{i+1}-X(t_i))^2 +
(X(t)-X(t_k))^2
\]
where $k$ is the number such that $t_k \leq t < t_{k+1}$. We call it the
quadratic variation of $X(t)$ associated with the partition $\Delta$
Now let $\{ \Delta \}_m$ be a sequence of partitions such that $|\Delta_m | \to
0$. If the limit of $ {\langle X \rangle}^{\Delta_m} (t)$ exists in probability
and it is independent \index{Independent!}of the choice of sequence $\{ \Delta_m\}$ a.s., it is
called the quadratic variation of $X(t)$ and is denoted by $\langle X \rangle
(t)$. We will see that a natural class of processes where quadratic variations
are
well defined is that of continuous semimartingales.\index{Semimartingale}
\end{remark}

\begin{lemma}

{\it Let $X$ be a continuous process of bounded variation. Then the quadratic
variation exists and it equals zero a.s.}
\end{lemma}
\begin{proof}

Let $ |X| (t; \omega)$ be the total variation of the function $X(s, \omega), 0
\leq s \leq t$. Then there holds
\[
|X|(t) = \mathop{\sup}\limits_{\Delta} [|X(t_0)| + \sum_{j=0}^{k-1} |X (t_{j+1})-
X(t_j)|+|X(t)-X(t_k)|]
\]
and
$\langle X \rangle^{\Delta} (t)
 \leq\left ( \sum_{j=0}^{k-1} |X(t_{j+1})-
X(t_j)| + |X(t)-X(t_k)|\right ) \mathop{\max}\limits_{i} |X(t_{i+1})-X(t_i)|\\
 \leq |X|(t) \mathop{\max}\limits_{i} |X(t_{i+1}) - X (t_i)|
$

The right hand side converges to $0$ for $|\Delta| \to 0$ a.s.
\end{proof}
\begin{theorem}

{\it Let $M(t) \, t\in [0,T]$, be a bounded continuous martingale. Let $\{
\Delta_n \}$ be a sequence of partitions such that $|\Delta_n| \to 0$. Then
$\langle M \rangle^{\Delta_n} (t)$, $t\in [0,T]$, converges uniformly to a
continuous increasing process $\langle M \rangle (t)$ in $L^2$ sense, i.e.,}
\[
\lim_{n \to \infty} E [\mathop{\sup}\limits_{t}|\langle M \rangle^{\Delta_n} (t)
-\langle M \rangle (t)|^2]=0;
\]
in addition $M^2 (t)- \langle M \rangle (t), t \in [0,T]$ is a martingale.
\end{theorem}
The proof is based on the following two lemmas.

\begin{lemma}

{\it For any $t >s$, there holds,}
\[
E[\langle M \rangle^{\Delta} (t)/F(s)]-\langle M \rangle^{\Delta}
(s)=E[(M(t)-M(s))^2 |F(s)]
\]
{\it In particular, $M^2 (t)- \langle M \rangle^{\Delta} (t)$ is a continuous
martingale.}
\end{lemma}

{\bf Hint}.\,
Rewrite $M(t)$ and $M(s)$ as follows $(s<t)$
\[
M(t)=M(t)-M(t_k) + \sum_{i=0}^{k-1} [M(t_{i+1}-M(t_i)]+M(0), \mbox{where}\, t_k
\leq t < t_{k+1}
\]
\[
M(s)=M(s)-M(s_l)+ \sum_{i=0}^{l-1} [M(s_{i+1})-M(s_i)]+M(0), \mbox{where}\,s_l
\leq s <s_{l+1}
\]
and $\Delta \cap [0,s]=\{0=s_0<s_1< \ldots < s_l\}$. On the other hand\\
\begin{eqnarray}E[(M(t)-M(s))^2 / F(s)]&=&E[(M(t)-M(t_k))^2 + \sum_{i=l+1}^{k-1}
(M(t_{i+1})-M(t_i))^2 \nonumber\\
&+&(M(s_{l+1})-M(s))^2 /F(s)]
\end{eqnarray}
where  $t_{l+1}=s_{l+1}$  and
\begin{eqnarray}
&&E[(M(t_{i+1})-M(t_i))(M(t_{j+1})-M(t_j)] / F(s)]\nonumber\\
&=&E\{E[(M(t_{i+1})-M(t_i)](M(t_{j+1})-M(t_j))/ F(t_{j+1})]/F(s)\}=0
\,(\hbox{i} f i>j)
\end{eqnarray} are used.
Adding and subtracting $M^{\Delta}(s)=M(s)-M(s_l))^2 + \sum_{i=0}^{l-1}
(M(s_{i+1})-M(s_i))^2$ we get the first conclusion which can be written as a
martingale property.
\vspace{2mm}
\begin{lemma}

{\it It holds}
\[
\lim_{n,m \to \infty} E[|\langle M \rangle^{\Delta_n} (T) - \langle M
\rangle^{\Delta_m} (T)|^2]=0.
\]
\end{lemma}
\vspace{2mm}
\begin{theorem}

{\it $M(t)$ be a continuous local martingale. Then there is a continuous
increasing process $\langle M \rangle (t)$ such that $\langle M \rangle^{\Delta}
(t)$ converges uniformly to $\langle M \rangle (t)$ in probability.}
\end{theorem}
\vspace{2mm}
\begin{corollary}
$M^2 (t)- \langle M \rangle (t)$ is a local martingale if $M(t)$ is a
continuous local martingale.
\end{corollary}
\begin{theorem}
 Let $X(t)$ be a continuous semi-martingale. Then $\langle X
\rangle^{\Delta} (t)$ converges, uniformly to $\langle M \rangle (t)$ in
probability as $|\Delta| \to 0$, where $M(t)$ is the local martingale part of
$X(t)$
\end{theorem}
\vspace{2mm}
{\bf Stochastic integrals}\index{Stochastic !integral}
Let $M(t)$ be a continuous local martingale and let $f(t)$
be a continuous $F(t)$ adapted process. We will define the stochastic integral
of $f(t)$ by the differential\index{Differential !} $dM(t)$ using the properties of martingales,
especially those of quadratic variations.
Let $\Delta=(0=t_0 < \ldots < t_n=T)$ be a partition of $[0, T]$. For any $t \in
[0,T]$ choose $t_k$ of $\Delta$ such that $t_k \leq t <t_{k+1}$ and define
\begin{equation}\label{eq:4.1}
L^{\Delta} (t)=\sum_{i=0}^{k-1} f(t_i)(M (t_{i+1}) - M(t_i)) + f(t_k))
(M(t)-M(t_k))\end{equation}
It is easily seen that $L^{\Delta}(t)$ is a continuous local martingale. The
quadratic variation is computed directly as
\vspace{2mm}
\begin{eqnarray}\label{eq:4.2}
\langle L^\Delta \rangle (t) &=& \sum_{i=0}^{k-1} f^2 (t_i)( \langle M \rangle
(t_{i+1}) \nonumber\\&-& \langle M \rangle (t_i)) + f^2 (t_k) (\langle M \rangle (t) -
\langle M \rangle (t_k))
\int_0^t |f^{\Delta} (s)|^2 d \langle M \rangle (s)
\end{eqnarray}
where $f^{\Delta} (s)$ is a step process defined from $f(s)$ by $ f^{\Delta}
(s)=f(t_k)$ if $t_k \leq s < t_{k+1}$. Let $ \Delta'$ be another partition of
$[0, T]$.
We define $L^{\Delta'} (t)$ similarly, using the same $f(s)$ and $M(s)$.
Then there holds
\[
\langle L^{\Delta} - \L^{\Delta'} \rangle (t) = \int_0^t |f^{\Delta}_{(s)}-
f^{\Delta'}_{(s)}|^2 d \langle M \rangle (s).
\]
Now let $\{ \Delta_n \}$ be a sequence of partitions of $[0, T]$ such that
$|\Delta_n| \to 0$. Then $\langle L^{\Delta_n} - L^{\Delta_m} \rangle (T)$
converges to $0$ in probability as $n, m \to \infty$. Hence $\{L^{\Delta_n}\}$
is a Cauchy sequence\index{Cauchy!sequence} in $M_c^{loc}$. We denote the limit as $L(t)$.\\
\vspace{2mm}
\begin{definition}
The above $L(t)$ is called the It\^{o} integral of $f(t)$ by $dM(t)$ and is
denoted by $L(t)= \int_0^t f(s) dM(s)$.\\
\end{definition}
\begin{definition}
Let $X(t)$ be a continuous semimartingale\index{Semimartingale} decomposed to the sum of a continuous
local martingale $M(t)$ and a continuous process of bounded variation $A(t)$
\end{definition}
Let $f$ be an $F(t)$-adapted process such that $f \in L^2 ( \langle M \rangle)$
and\\
$\int_0^T |f(s) |d|A|(s) \langle \infty$. Then the Ito integral
of $f(t)$ by $dX(t)$ is defined as
\[
\int_0^t f(s) dX(s)= \int_0^t f(s) dM(s) + \int_0^t f(s) dA(s)
\]\\
\begin{remark}
If $f$ is a continuous semimartingale\index{Semimartingale} then $\int_0^t f(s) dX (s)$ exists.
We will define another stochastic integral\index{Stochastic! integral} by the differential \index{Differential !}''$\circ dX$
(s) $s\in [0,t]$
$\int_{0}^{t} f(s) \circ dX (s) = \lim_{|\Delta| \to 0} \Big [ \sum_{i=0}^{k-1}
\frac{1}{2} (f(t_{i+1}) + f(t_i)) (X(t_{i+1})-X (t_i))\\
+ \frac{1}{2} (f(t) + f(t_k))(X(t)-X(t_k)) \Big ]$
\end{remark}
\begin{definition}
If the above limit exists, it is called the Fisk-Stratonovich integral of $f(s)$
by $d X(s)$.
\end{definition}
\vspace{2mm}
\begin{remark}
If $f$ is a continuous semimartingale\index{Semimartingale}, the Fisk-Stratonovitch integral is
well
defined and satisfies
$
\int_0^t f(s) \circ dX (s) = \int_0^t f(s) dX (s) + \frac{1}{2} \langle f, X
\rangle (t)
$
\end{remark}
\begin{proof}
Is easily seen from the relation
\begin{eqnarray}&&\sum_{i=0}^{k-1} \frac{1}{2}  ( f(t_{i+1}) + f(t_i))(X(t_{i+1})-X(t_i))+
\frac{1}{2} (f(t) + f(t_k))(X(t)-X(t_k))\nonumber\\
&=&\sum_{i=0}^{k-1} f(t_i)(X(t_{i+1})-X(t_i))+ f(t_k)(X(t)-X(t_k))+\frac{1}{2}
\langle f, X \rangle^{\Delta} (t)
\end{eqnarray}

where the joint quadratic variation
$$\langle f, X \rangle^{\Delta} (t) = \sum_{i=0}^{k-1} ( f(
t_{i+1}-f(t_i))(X(t_{i+1}-X(t_i))
+ (f(t)-f(t_k))(X(t)-X(t_k))$$
and $k$ is the number such that $t_k \leq t < t_{k+1}$.
\end{proof}
\vspace{2mm}
\begin{theorem}
 Let $X$ and $Y$ be continuous semi-martingales. The joint quadratic
variation associated with the partition $\Delta$ is defined as before and is
written $\langle X, Y \rangle^{\Delta}$. Then $\langle X, Y \rangle^{\Delta}$
converges uniformly in probability to a continuous process of bounded
variation $\langle X, Y \rangle (t)$.
If $M$ and $N$ are local martingale parts of $X$ and $Y$, respectively, then
$\langle X, Y \rangle$ coincides with $\langle M,N \rangle$
\end{theorem}
\vspace{2mm}
\begin{remark}
Let $w (t)=( w^1 (t), \ldots , w^{m} (t))$ be an m-dimensional
$(F_{(t)})$-adapted continuous stochastic process. It is an $(F_{(t)})$-Wiener
process iff each $w^i (t)$ is a scalar Wiener process with the joint quadratic
variation fulfilling $\langle w_i, w_j \rangle (t)= \delta_{ij}t$ where
$\delta{ii}=1$ and $\delta_y=0$\, $i\neq j$
\end{remark}

\section[Approximations of the diffusion equations]{Approximations of the diffusion equations by smooth ordinary
differential equations\index{Differential !equation}\index{Ordinary differential equations}}
The $m$-dimensional Wiener process is approximated by a smooth process and it
allows one to use non anticipative smooth solutions of ordinary differential\index{Differential !equation}
equations as approximations for solutions of a diffusion equation.
It comes from Langevin's classical procedure of
defining a stochastic differential equation\index{Stochastic !differential equation}.
 By a standard $m$-dimensional Wiener process we
 mean a measurable function $ w (t, \omega) \in R^m$,
 $(t;\omega) \in [0, \infty) \times \Omega$ with continuous trajectories
 $w (t, \omega )$ for each $\omega \in \Omega$ such that $w (0, \omega)=0$,
 and
$i_1$ ) $E (w(_2))- w (t_1) / F_{t_1})=0$,
$i_2$) $E ([w(t_2)-w(t_1)] [w(t_2)-w(t_1)]^T / \mathcal{F}_{t_1})=I_m (t_2-t_1)$
for any $0 \leq t_1 < t_2$,
\noindent where $(\Omega, \mathcal{F}, P)$ is a given complete probability
space and $F_t \subseteq F$ is the $\sigma$-algebra generated by $(w (s), s
\leq t)$.
The quoted Langevin procedure replaces a standard $m$-dimensional Wiener
process by a $C^1$ non-anticipative process $v_{\varepsilon} (t)$ (see
$v_{\varepsilon} (t)$ is $F_t$ measurable) as follows
\begin{equation}\label{4.1}
v_\varepsilon (t) = w (t) - \int_0^t (\exp - \beta (t-s))dw (s), \,
\beta=\frac{1}{\varepsilon}, \varepsilon \downarrow 0
\end{equation}
where the integral in the right hand side is computed as
\begin{equation}
w(t)- \beta \int_0^t w(s) (\exp - \beta (t-s))ds \,\,\,\mbox{(the integration
by parts formula)}
\end{equation}
Actually, $v_{\varepsilon} (t), \, t \in [0,T]$, is the solution of the
following equations
\begin{equation}\label{4.2}
\frac{d v_{\varepsilon} (t)}{dt}=- \beta v_{\varepsilon} (t) +\beta w
(t)=\beta (w(t)-v_{\varepsilon}(t), \, v_{\varepsilon} (0)=0
\end{equation}
and by a direct computation we obtain
\begin{equation}\label{4.3}
E||v_\varepsilon (t) - w (t)||^2 \leq \varepsilon, \,\, t \in [0,T]
\end{equation}
Rewrite $v_\varepsilon (t)$ in \eqref{4.1} as
\[
v_\varepsilon (t)=w(t)-\eta_\varepsilon (t),
\]
where
\begin{equation}\label{4.4}
\eta_\varepsilon (t)=\int_0^t \exp - \beta (t-s) dw (s)
\end{equation}
fulfills $d \eta^i = -\beta \eta^i dt + dw^i (t), \,\, i=1, \ldots, m,$ and
\begin{equation}\label{4.5}
\frac{d v_\varepsilon}{dt} (t)=\beta \eta_\varepsilon (t), \,\,
\beta=\frac{1}{\varepsilon}, \, \varepsilon \downarrow 0.
\end{equation}
Now, we are given continuous functions
\[
f(t,x), g_j (t,x):[0,T] \times R^n \to R^n, \,\,j=1, \ldots, m,
\]
such that
\[
(\alpha)
\begin{cases}
i)\, f, g_j \,\, \mbox{are bounded,}\,\,\,\,\, j=1,\ldots ,m\\
ii)\,||h(t,x'')-h(t,x'||\leq L||x''-x'||\,\,(\forall)x',x'' \in R^n, t\in[0,T]
\end{cases}
\]
where $L>0$ is a constant,$ h\triangleq f, g_j$.
In addition, we assume that $(g_j \in C^{1,2}_b ([0,T]\times R^n)$\\
$\beta )\,\, \frac{d\partial g_j}{\partial t},\, \frac{d\partial g_j}{\partial
x}, \,\frac{d\partial^2 g_j}{\partial t \partial x},\, \frac{d\partial^2
g_j}{\partial x^2}, j=1, \ldots , m$,
are continuous and bounded functions. Let $x_0 (t)$ and $x_\varepsilon (t), t
\in [0,T]$, be the solution in
\begin{equation}\label{4.6}
dx=[f(t,x)+\frac{1}{2} \sum_{i=1}^m \frac{\partial g_i}{\partial x} (t,x) g_i
(t,x)] dt
\end{equation}
$$+\sum_{i=1}^m g_i (t,x)dw^i_{(t)},\,\, x(0)=x_0$$
and
\begin{equation}\label{4.7}
\frac{dx}{dt}=f(t,x)+\sum_{i=1}^m g_i (t,x) \frac{dv_\varepsilon^i (t)}{dt},
\,x(0)=x_0
\end{equation}
correspondingly, where $v_\varepsilon (t)$ is associated with $w (t)$ in \eqref{4.1}
and fulfills \eqref{4.2}-\eqref{4.5}. It is the Fisk-Stratonovich integral (see ''$\circ$'' below) which allow one
to rewrite the system \eqref{4.6} as
\begin{equation}
dx= f(t,x)dt+\sum_{i=1}^m g_i (t,x)\circ dw^i (t), x(0)=x_0,
\end{equation}
where
$$g_i (t,x)\circ dw^i(t) \stackrel{def}{=} g_i (t, x)dw^i (t) +
\frac{1}{2} \frac{\partial g_i}{\partial x} (t, x) g_i (t x)dt$$
\vspace{2mm}
\begin{theorem}\label{tt:4.2.1}
Assume that continuous functions $f (t,x), g_i (t,x), t \in [0,T], x \in
R^n$, are given such that $(\alpha)$ and $(\beta)$ are fulfilled. Then
$\lim_{\varepsilon \to 0} E||x_\varepsilon (t)-x_0 (t)||^2=0,\, t \in [0,T]$,
where $x_0(t)$  and $x_\varepsilon (t)$ are the solutions defined in \eqref{4.6}
and, respectively,\eqref{4.7}
\end{theorem}
\begin{proof}
Using \eqref{4.6} we rewrite the solution $x_\varepsilon (t)$ as
\begin{equation}\label{4.8}
x_\varepsilon (t)=x_0 + \int_0^t f(s), x_\varepsilon (s))ds
+\sum_{i=1}^m \beta \int_0^t \eta_\varepsilon^i (s) g_i (s, x_\varepsilon
(s))ds,\,\, t\in [0,T]
\end{equation}
and from $F_t$-measurability of $\eta_\varepsilon (t)$ we obtain that
$x_\varepsilon (t)$ is $\mathcal{F}_t$-measurable and non-anticipative with
respect to $\{\mathcal{F}_t\}$, \, $t\in [0,T]$. Therefore, the stochastic
integrals \index{Stochastic! integral}$\displaystyle{\int_0^t g_i (s, x_\varepsilon (s)) dw^i (s)}$ and
$\displaystyle{\int_0^t g_i (s, x_\varepsilon (s))d \eta_\varepsilon^i (s)}$
are well defined, and using \eqref{4.4} we obtain
\begin{equation}\label{4.9}
\int_0^t g_i (s, x_\varepsilon (s)) d \eta_\varepsilon^i (s)= -\beta \int_0^t
g_i
(s, x_\varepsilon (s))\eta_\varepsilon^i (s) ds\\
 + \int_0^t g_i (s, x_\varepsilon (s, x_\varepsilon (s))dw^i (s).
\end{equation}
{\underline{ \it Step 1}}

\vspace{2mm}
\noindent Using \eqref{4.8}in\eqref{4.8} there follows
\begin{eqnarray}\label{4.10}
x_\varepsilon (t)&=&x_0 + \int_0^t f(s, x_\varepsilon (s)) ds
 +\sum_{i=1}^m \int_0^t g_i (s, x_\varepsilon (s))dw^i (s)\nonumber\\
 &-&\sum_{i=1}^m \int_0^t g_i (s_s, x_\varepsilon (s)) d \eta_\varepsilon^i
(s), \, t \in [0,T]
\end{eqnarray}
In what follows it will be proved that (see step 2 and step 3)
\begin{equation}\label{4.11}
- \int_0^t g_i (s, x_\varepsilon (s)) d \eta_\varepsilon^i (s)=\frac{1}{2}
\int_0^t
\frac{\partial g_i}{\partial x} (s, x_\varepsilon (s)) g_i (s, x_\varepsilon
(s))ds +O_t (\varepsilon)
\end{equation}
where $E||O_t (\varepsilon)||^2 \leq c_1 \varepsilon, \,\, (\forall) \, t\in
[0,T]$, for some constant $c_1 >0$. Using \eqref{4.11} in \eqref{4.10} we rewrite \eqref{4.10} as
\begin{eqnarray}\label{4.12}
x_\varepsilon (t) &=& x_0 + \int_0^t \Big [ f(s, x_\varepsilon (s))\nonumber\\&+&
\frac{1}{2}\sum_{i=1}^{m}\frac{\partial g_i}{\partial x}
(s, x_\varepsilon (s)) g_i (s, x_\varepsilon (s)) \Big ] ds
 \sum_{i=1}^m \int_0^t g_i (s, x_\varepsilon (s)) dw^i (s)+O_t (\varepsilon),
\end{eqnarray}
where $E|O_t (\varepsilon)||^2 \leq c_1\varepsilon$.

The hypotheses $(\alpha)$ and $(\beta)$ allow one to check that
\[
\widetilde{f} (t,x)\triangleq f(t,x)+\frac{1}{2} \sum_{i=1}^m \frac{\partial
g_i}{\partial x} (t,x)g_i (t,x)
\]
and $g_i (t,x)$ fulfil $(\alpha)$ also but with a new Lipschitz constant
$\widetilde{L}$.

The proof will be complete noticing that
\begin{equation}\label{4.13}
E||x_\varepsilon (t) - x_0 (t)||^2 \leq c_2 \int_0^t E||x_\varepsilon
(s)-x_0(s)||^2 ds
+c_3 \varepsilon, \,\, t\in [0,T]
\end{equation}
for some constants $c_2,c_3>0$ and Gronwall's lemma applied to \eqref{4.13} implies
\begin{equation}\label{4.14}
E||x_\varepsilon (t)-x_0(t)||^2 \leq \varepsilon c_3 (\exp
T_{c_2})=\varepsilon c_4
\end{equation}
which proves the conclusion.

\vspace{2mm}

{\underline{ \it Step 2}}

\vspace{2mm}
\noindent To obtain \eqref{4.11} fulfilled we use an ordinary calculus as integration
by parts formula (see $x_\varepsilon (t)$ is a $C^1$ function in $t \in
[0,T]$). More precisely
\begin{eqnarray}\label{4.15}
- \int_0^t g_i (s, x_\varepsilon (s))d \eta_\varepsilon^i (s) &=&
-g_i (t, x_\varepsilon (t) \eta_\varepsilon^i (t) + \int_0^t
\eta_\varepsilon^i (s) \Big ( \frac{d}{ds} g_i (s, x_\varepsilon (s)) \Big )
ds\nonumber\\
 &=&\sum_{j=1}^m \int_0^t \frac{\partial g_i}{\partial x} (s, x_\varepsilon
(s)) g_j (s, x_\varepsilon (s)) \frac{dv_\varepsilon^j (s)}{ds}
\eta_\varepsilon^i (s) ds \nonumber\\
&+& O_t^i (\varepsilon)
\triangleq \sum_{j=1}^m T_{ij} (t) + O_t^i (\varepsilon),
\end{eqnarray}
Here
\[
T_{ij} (t) \triangleq \int_0^t \frac{\partial g_i}{\partial x} (s,
x_\varepsilon (s))g_j (s, x_\varepsilon (s)) \frac{dv_\varepsilon^j (s)}{ds}
\eta_\varepsilon^i (s) ds \,\,\,\mbox{and}
\]
\[
\begin{aligned}
O_t^i(\varepsilon)\triangleq \int_0^t \eta_\varepsilon^i (s) \Big [
\frac{\partial g_i}{\partial s} (s, x_\varepsilon (s)) + \frac{\partial
g_i}{\partial x} (s, x_\varepsilon (s))f (s, x_\varepsilon (s)) \Big ]ds\\
-g_i (t, x_\varepsilon (t)) \eta_\varepsilon^i (t)
\end{aligned}
\]
satisfies
\begin{equation}\label{4.16}
E|| O^1_t (\varepsilon) ||^2 \leq k \varepsilon, \, t \in [0,T],
\end{equation}
(see $E||\eta_\varepsilon^i (t)||^2 \leq \varepsilon$ in \eqref{4.4} and $f, g_i ,
\dfrac{\partial g_i}{\partial x}; \dfrac{\partial g_i}{\partial t}$ are
bounded). On the other hand, $\dfrac{dv_\varepsilon^j}{dt}= \beta
\eta_\varepsilon^j (t)$ (see \eqref{4.5}) and using $\beta \eta_\varepsilon^j
(t)dt=dw^j (t)-d\eta_\varepsilon^j (t)$ (see \eqref{4.4}) we obtain that $T_{ij}$ in
\eqref{4.15} fulfils
\begin{equation}\label{4.17}
T_{ij} (t)= - \int_0^t g_{ij} (s, x_\varepsilon (s)) \eta_\varepsilon^i (s)d
\eta_\varepsilon^j (s) + \theta_t^2 (\varepsilon),
\end{equation}

where $g_{ij} (t,x)\triangleq \dfrac{\partial g_i}{\partial x} (t,x) g_j
(t,x)$ and
\[
\theta_t^2 (\varepsilon) \triangleq \int_0^t g_{ij} (s, x_\varepsilon (s))
\eta_\varepsilon^i (s) dw^j (s)
\]
satisfies
\begin{equation}\label{4.18}
E||\theta^2_t (\varepsilon)||^2 \leq k_2 \varepsilon, \, t\in [0,T]
\end{equation}
(see $g_{ij}$ bounded and $E||\eta_\varepsilon^i (t) ||^2 \leq \varepsilon$).
\vspace{2mm}

{\underline{ \it Step 3}}
\vspace{2mm}

\noindent
(a) For $i=j$ we use formulas
\[
\eta_\varepsilon^i (t)= (\exp - \beta t) \mu_\varepsilon^i (t),
\mu_\varepsilon^i (t)= \int_0^t (\exp \beta s) dw^i (s)
\]
\[
\begin{aligned}
(\eta_\varepsilon^i (t))^2&=& (\exp- 2 \beta t)(\mu_\varepsilon^i (t))^2 =
(\exp - 2 \beta t)\Big [\int_0^t 2 (\exp 2 \beta s) \eta_\varepsilon^i (s)dw^i
(s)+ \int_0^t (\exp 2 \beta s)ds\Big ]\nonumber\\
& =&2 \int_0^t \eta_\varepsilon^i (s) d \eta_\varepsilon^i (s)+\int_0^t ds
\end{aligned}
\]
\[
T_{ii} (t)= - \int_0^t g_{ii} (s, x_\varepsilon (s)) \eta_\varepsilon^i (s) d
\eta_\varepsilon^i (s) + \theta_t^2 (\varepsilon)\,\, \mbox{(see \eqref{4.17})}
\]
We get
\begin{eqnarray}\label{4.19}
T_{ii}(t)&=&\frac{1}{2} \int_0^t g_{ii} (s, x_\varepsilon (s))ds
\hspace{2mm}-\frac{1}{2} \int_0^t g_{ii} (s, x_\varepsilon )d
(\eta_\varepsilon^i (t))^2 +\theta_t^2 (\varepsilon)\nonumber\\
&=& \frac{1}{2} \int_0^t g_{ii} (s, x_\varepsilon (s)) ds
\hspace{2mm}-\frac{1}{2} g_{ii} (t, x_\varepsilon (t)) (\eta_\varepsilon^i
(t))^2
\hspace{2mm}\nonumber\\&+&\frac{1}{2} \int_0^t \Big [ \frac{d}{ds} g_{ii} (s,
x_\varepsilon (s))\Big ] (\eta_\varepsilon^i (s))^2 ds +\theta^2_t
(\varepsilon)\nonumber\\
&=&\frac{1}{2} \int_0^t g_{ii} (s, x_\varepsilon (s))ds + O_t^2 (\varepsilon).
\end{eqnarray}
Here
\[
\begin{aligned}
O^2_t (\varepsilon) \triangleq  \,\theta_t^2 (\varepsilon)-\frac{1}{2} g_{ii}
(t, x_\varepsilon (t))(\eta_\varepsilon^i (t))^2\\
+\frac{1}{2} \int_0^t \Big [ \frac{\partial}{\partial s} g_{ii} (s,
x_\varepsilon (s)) +\frac{\partial g_{ii}}{\partial x} (s, x_\varepsilon (s))
f (s, x_\varepsilon (s))\\
+\beta \sum_{j=1}^m \frac{\partial g_{ii}}{\partial x} (s, x_\varepsilon (s))
g_j (s, x_\varepsilon (s))\eta_\varepsilon^j (s)\Big ] (\eta_\varepsilon^i
(s))^2 ds
\end{aligned}
\]

fulfils $E||O_t^2 (\varepsilon)||^2 \leq \widetilde{K}_2 \varepsilon$ taking
into account that
\begin{equation}\label{4.20}
E|| \theta_t^2 (\varepsilon)||^2 \leq k_2 \varepsilon \, \mbox{(see \eqref{4.17})},\,\,
E(\eta_\varepsilon^i (t))^2 = (\exp - 2 \beta t) E (\mu_\varepsilon^i (t))^2
\leq \varepsilon
\end{equation}
In addition, write $(\eta_\varepsilon^i (t))^2= (\exp - 2 \beta t)(
\mu_\varepsilon^i (t))^2$, where $(\mu_\varepsilon^i (t))^2= a_i (t)+ b(t)$ is
the corresponding decomposition using martingale $ a_i (t) \triangleq \int_0^t
2 \mu^i (s) (\exp \beta s) dw^i (s)$ and $b (t)\triangleq \int_0^t (\exp 2
\beta s)ds$.
Compute $(\eta_\varepsilon^i (t))^4= (\exp- 4 \beta t)[a_i (t))^2+ (b(t))^2+2
a_i (t)b(t)]$ and $(\eta_\varepsilon^i (t))^6= (\exp - 6 \beta t)[(a_i(t))^3+
(b(t))^3+3 (a_i (t))^2 b(t)+3a_i (t)(b(t))^2]$.
Noticing that $a_i (t), \, t \in [0,T]$, is a martingale with $a_i(0)=0$ we
can prove (exercise!) that any $(a_i (t))^m (m=2,3)$ satisfies
\[
E(a_i (t))^m=\dfrac{1}{2} \int_0^t m \,(m-1)\, E [a_i(s))^{m-2}\, (f_i (s))^2]ds
\]
where
$f_i (t)\triangleq 2 (\exp \beta t) \mu^i (t)$ and $a_i (t)\triangleq \int_0^t
f_i (s)dw^i (s)$.
We get
\begin{equation}\label{4.21}
E (\eta_\varepsilon^i (t))^4 \leq \,\mbox{(const)}\, \varepsilon^2, \,
E(\eta_\varepsilon^i (t))^6 \leq \mbox{(const)}\, \varepsilon^3\, t \in [0,T]
\end{equation}
and the conclusion $E||\theta_t^2 (\varepsilon)||^2 \leq \, \mbox{(const)}
\varepsilon $ used in \eqref{4.19}
follows directly from \eqref{4.20} and \eqref{4.21}.

(b)\, For any $i \neq j$, there holds
\[
T_{ij} (t)=- \int_0^t g_{ij} (s, x_\varepsilon (s)) \eta_\varepsilon^i (s) d
\eta_\varepsilon^j (s)+ \theta_t^2 (\varepsilon) \,\mbox{(see (17))}
\]
where
\begin{equation}\label{4.22}
d \eta_\varepsilon^j (s)= - \beta \eta_\varepsilon^j (s)ds+dw^j (s) \,
\mbox{(see \eqref{4.4})}
\end{equation}
Using \eqref{4.22} in \eqref{4.17} for $i \neq j$ we obtain that $ \eta_\varepsilon^i (t),
\eta_\varepsilon^j (t)$ are independent random variables\index{Independent!random variables} and
\begin{equation}\label{4.23}
T_{ij}(t)= - \frac{1}{2} \int_0^t g_{ij} (s, x_\varepsilon (s)) d
[\eta_\varepsilon^i (s) \eta_\varepsilon^j (s)]+\theta_t
(\varepsilon)+\widetilde{O}_t (\varepsilon),
\end{equation}
where
\[
\widetilde{O}_t (\varepsilon)= - \int_0^t g_{ij}(s, x_\varepsilon
(s))\eta_\varepsilon^i (s)dw^j (s)
\]
satisfies
\begin{equation}\label{4.24}
E||\widetilde{O}_t (\varepsilon)||^2 \leq \mbox{(const)} \varepsilon\,
( \mbox{see}\, g_{ij} \, \mbox{is bounded and}\, E(\eta_\varepsilon^i (s))^2
\leq \varepsilon)
\end{equation}
Integrating by parts the first term in \eqref{4.23} and using
\begin{equation}\label{4.25}
E(\eta_\varepsilon^i (t))^2 (\eta_\varepsilon^j (t))^2 \leq \mbox{(const)}
\varepsilon^2, \,\, E (\eta_\varepsilon^i (t))^2 (\eta_\varepsilon^j (t))^2
(\eta_\varepsilon^k (t))^2 \leq \mbox{(const)} \varepsilon^3
\end{equation}
we finally obtain
\begin{equation}\label{4.26}
T_{ij}(t)= \widetilde{O}_t^2 (\varepsilon) \mbox{with}\, E||\widetilde{O}_t^2
(\varepsilon)||^2 \leq \widetilde{C} \varepsilon, \, t \in [0,T]
\end{equation}
Using \eqref{4.19}, \eqref{4.26} in \eqref{4.15} we find \eqref{4.11} fulfilled.
The proof is complete
\end{proof}
\vspace{2mm}
\begin{remark}\label{rk:4:2.1}
\vspace{2mm}
{\it Under the conditions in theorem \eqref{tt:4.2.1} it might be useful to notice that the
computations remain unchanged if a stopping time $\tau : \Omega \to [0, T]$,\,
$\{\omega : \tau \geq t \} \in \mathcal{F}_t , \, (\forall ) t \in [0,T]$, is
used. Namely}
\[
\lim_{\varepsilon \downarrow 0} E||x_\varepsilon ( t \wedge \tau) - x_0 ( t
\wedge \tau)||^2=0 \,(\forall)\, t \in [0,T],
\]
{\it if the random varieble $\tau : \Omega \to [0, T]$ is adapted to $\{
\mathcal{F}_t \}$ i.e. $\{ \omega : \tau \geq t \} \in \mathcal{F}_t$ for $ t
\in [0,T]$.}
\vspace{3mm}

\begin{proof}
\vspace{2mm}

By definition
\[
x_\varepsilon (t \wedge \tau) = x_0 + \int_0^{ t \wedge \tau} f(s,
x_\varepsilon (s))ds +\sum_{j=1}^m \int_0^{t \wedge \tau} g_j (s,
x_\varepsilon (s) \frac{dv_\varepsilon^j}{ds} (s),
\]
\[
x_0 (t \wedge \tau)=x_0 + \int_0^{ t \wedge \tau} \widetilde{f} (t, x_0 (s))ds
+\sum_{j=1}^m \int_0^{t \wedge \tau} g_j (s, x_0 (s)dw^j (s), \, t \in [0,T],
\]
where
\[
\widetilde{f} (t,x)\triangleq f (t,x) +\frac{1}{2} \sum_{j=1}^m \frac{\partial
g_j}{\partial x} (t,x) g_j (t,x).
\]
Using the characteristic function
\[
\chi (t, \omega)=
\begin{cases}
1 \,\, \mbox{if} \,\, \tau (\omega) \geq t\\
0 \,\, \mbox{if} \,\, \tau (\omega) < t
\end{cases}
\]
which is a non-anticipative function, we rewrite $y_\varepsilon (t) \triangleq
x_\varepsilon (t \wedge \tau)$ and $y_0 (t) \triangleq \chi_0 (t \wedge \tau
)$ as
\[
y_\varepsilon (t) = x_0 + \int_0^t \chi (s) f (s, y_\varepsilon (s) ds +
\sum_{j=1}^m \int_0^t \chi (s) g_j (s, y_\varepsilon (s)
\frac{dv_\varepsilon^j (s)}{ds},
\leqno(*)
\]
\[
y_0 (t) = x_0 + \int_0^t \chi (s) \widetilde{f} (s, y_0 (s))ds + \sum_{j=1}^m
\int_0^t \chi (s) g_j (s, y_0 (s))dw^j (s).
\leqno(**)
\]
Now the computations in Theorem \ref{tt:4.2.1} repeated for (*) and (**) allow one to
obtain the conclusion.
\end{proof}
\end{remark}
\begin{remark}\label{rk:4:2.2}
\vspace{2mm}
{\it Using Remark 1 we may remove the boundedness assumption of $f, g_j$ in
the hypothesis $(\alpha)$ of Theorem \ref{tt:4.2.1}. That is to say, the solutions
$x_\varepsilon (t), x_0 (t),\, t \in [0,T]$, exist assuming only the
hypothesis $(\alpha, (ii))$ and $(\beta)$, and to obtain the conclusion we
multiply $f, g_j$ by a $C^\infty$ scalar function $0 \leq \alpha_N (x) \leq 1$
such that $\alpha_N (x)=1$ if $ x \in S_N (x_0)\,, \alpha_N (x)=0$ if $x \in
R^n \ S_{2N} (x_0)$ where $S_\rho (x_0) \subseteq R^n$ is the ball of radius
$\rho$ and centered at $x_0$. We obtain new bounded functions}
\[
f^N (t,x)= f(t,x) \alpha^N (x), \, g_j^N (t,x)=g_j (t,x) \alpha^N (x)
\]
{\it fulfilling $(\alpha)$ and $(\beta)$ of Theorem \ref{tt:4.2.1}, and therefore
$\lim_{\varepsilon \downarrow 0} E || x_\varepsilon^N (t) - x_0^N (t) ||^2 =
0$\,
$(\forall)\, t \in [0,T]$, where $ x_\varepsilon^N (t)$ and $x_0^N (t)$ are
the corresponding solutions. On the other hand, using a stopping time (exit
ball time)}
\[
\tau_N (\omega)= \inf \{ t \geq 0\, : \,x_0 (t, \omega) \not \in S_N (x_0)\}
\]
{\it we obtain $x_0 (t \wedge \tau_N)=x_0^N (t \wedge \tau_N),\, t \in [0,T]$
(see Friedmann) where $x_0 (t), t\in [0,T]$ is the solution of the equation
 \eqref{4.6} with $f, g_i$ fulfilling $(\alpha, ii)$ and $(\beta)$.

Finally we obtain:
\[
\lim_{\varepsilon \to 0} E||x_\varepsilon^N (t \wedge \tau_N) -x_0 (t \wedge
\tau_N) ||^2 =0
\leqno c)
\]
 for any $ t \in [0,T]$, and for arbitrarily fixed $N>0$. The conclusion (c)
re\-pre\-sents the approximation of the solution in \eqref{4.6} under the hypotheses
$(\alpha, ii)$ and $(\beta)$.}
\end{remark}
\begin{remark}
{\it The nonanticipative process $v_\varepsilon (t), \, t \in [0,T]$, used in
Theorem \ref{tt:4.2.1} is only of the class $C^1$ with respect to $t \in [0,T]$, but a
minor change in the appro\-xi\-ma\-ting equations as follows $(\beta =
\dfrac{1}{\varepsilon}, \, \varepsilon \downarrow 0)$}
\[
\frac{dv_\varepsilon}{dt} (t)=y_1, \, \varepsilon \frac{dy_1}{dt}=-y_1+y_2,
\ldots, \, \varepsilon \frac{dy_{k-1}}{dt}=-y_{k-1}+y_k
\]
\[
\varepsilon dy_k=-y_k dt +dw (t), \,t \in [0,T],\, y_j (0)=0, \, j=1, \ldots
k, v_\varepsilon (0)=0
\]
{\it will allow one to obtain a non-anticipative $v_\varepsilon (t), t \in
[0,T]$ of the class $C^k$, for an arbitrarily fixed $k$.}
\end{remark}
\section{ Stochastic rule of derivation}\index{Stochastic!Rule of Derivations}

\medskip

The results contained in Theorem \ref{tt:4.2.1} and in remarks following the
theorem give the possibility to obtain, in a straight manner, the standard
rule of stochastic derivation\index{Stochastic!derivation} associated with stochastic differential\index{Differential !equation}
equations ( SDE.)  $( 6_s )$  when the drift vector field  $f (  t, x  ) \in
R^n$ and diffusion vector fields $g_j ( t,x ) \in R^n $, $j \in \{ 1, \ldots ,
m \} $ are not bounded with respect to $ ( t,x ) \in  [ 0,T ] \times R^n$.
More precisely, we are given continuous functions

\[
f ( t,x ), g_j ( t,x) : [0,T] \times R^n \to R^n, j=1, \ldots ,m
\]

\noindent such that
\begin{itemize}
\item[(1)] $|| h (t, x'')-h (t, x') || \leq L || x'' - x' ||$, for any $x',
x'' \in R^n, t \in [0,T]$ where $L> 0$ is a constant and $h$ stands for $f$ or
$g_j$, $j \in \{1, \ldots , m\}$
\item[(2)] $g_j \in C_b^{1,2} ([0,T] \times R^n)$ i.e. $\partial_t g_j,
\partial_x g_j, \partial_{tx}^2 g_j$ and $\partial_{x^2}^2 g_j$, $j \in \{1,
\ldots , m\}$ are continuous and bounded functions

\noindent Consider the following system of SDE

\item[(3)]$ \displaystyle{dx= f(t,x) dt + \sum_{j=1}^m g_j (t,x) \circ  dw^j
(t), x (0)=x_0, \, t \in [0,T]}$ where Fisk-Stratonovich integral ''0'' is
related to It \^{o} stochatisc integral "." by
\[
g_j (t,x) \circ dw^j (t) \stackrel{def}{=} g_j (t,x) \bullet  dw^j (t) +
\frac{1}{2} [d_x g_j (t,x)] g_j (t,x)dt
\]
Assuming that $\{ f, g_j, j=1, \ldots , m \}$, fulfil the hypotheses (1), (2)
then
there is a unique solution $x(t):[0,T] \to R^n$ which is a continuous and
$F_t$-adapted process satisfying the corresponding integral equation

\item[(4)] $\displaystyle{x(t)=x_0 + \int_0^t f(s,x(s))ds + \sum_{j=1}^m
\int_0^t g_j (s, x(s)) \circ dw^j (s), t \in [0,T]}$\\
\mbox{(see A. Friedman)}.
Here $w(t)= (w^1 (t), \ldots , w^m (t) : [0,T] \to R^m$ is the standard
Wiener process over the filtered probability space
\mbox{$\{ \Omega, \{ \mathcal{F}_t\} \uparrow \subseteq \mathcal{F}, P \}$}
and $\{\mathcal{F}_t, t \in [0,T]\} \subseteq \mathcal{F}$ is the corresponding
increasing family of $\sigma$-algebras. Consider the exist ball time

\item[(5)] $\tau \stackrel{def}{=} \inf \{ t \in [0,T] : x(t) \not \in B (x_0,
\rho )\}$ associated with the unique solution of (4) and the ball centered at
$x_0 \in R^n$ with radius $\rho >0$.

\noindent By definition, each set $\{ \omega : \tau \geq t \}$ belongs to
$\mathcal{F}_t, t\in [0,T]$ and the characteristic function $\chi_\tau (t): [0,T]
\to \{0,1\}, \chi_\tau (t)=1$ for $\tau \geq t$, $\chi_\tau (t)=0$ for $\tau
<t$, is
an $\mathcal{F}_t$-adapted process
\end{itemize}

\begin{theorem}\label{th:4.t2} (stochastic rule of derivation)\index{Stochastic!Rule of Derivations}

Assume that the hypotheses (1) and (2) are
fulfilled and let $\{x(t):t \in [0,T]\}$ be the unique solution of (4).

Define a stopping time $\tau : \Omega \to [0,T]$ as in (5) and consider
a smooth scalar function $\varphi \in C^{1,3} ([0,T] \times R^n)$.
Then the following integral equation is valid

\begin{itemize}
\item[(6)] $\varphi (t \wedge \tau,
\, x (t \wedge \tau))= \varphi (0, x_0) +
\int_0^{t \wedge \tau} [ \partial_s \varphi (s, x(s))+ < \partial_x
\varphi (s, x (s)),\\
f(s, x (s))>]ds + \sum_{j=1}^m \int_0^{t \wedge \tau}
< \partial_x \varphi (s, x (s)) g_j (s, x(s))>
\circ \,dw^j (s), \, \\
t \in [0,T]$

\noindent where the Fisk-Stratonovich integral "0" is defined by

$h_j (s, x(s)) \circ dw^j
(s)= h_j (s, x (s)). dw^j (s)+ \frac{1}{2} [ \partial_x h_j
(s,x (s))] g_j (s, x(s))]ds$

using It\^{o} stochastic integral\index{Stochastic!integral} ".".
\end{itemize}
\end{theorem}
\begin{proof}
Denote $\alpha (t)= \varphi (t, x(t)) \in R, \, z(t)= \mbox{col} (\alpha (t),
x(t) \in
R^n$, $ t \in [0,T]$ where $\{ x(t), t \in [0,T]\}$ is the unique solution of
(4) and $ \varphi \in C^{1,3} ( [0,T] \times R^n)$ is fixed.
Define new vector fields $\widehat{f} (t,x) = \mbox{col} (h_0 (t,x), f(t,x)
\in R^{n+1}$,
\begin{itemize}
\item[(7)] $\widehat{g}_j (t,x) = \mbox{col} (h_j (t,x), g_j (t,x)) \in
R^{n+1}, \,
j \in \{ 1, \ldots, m \}$,

where

$h_0 (t,x) \stackrel{def}{=} \partial_t \varphi (t,x) + < \partial_x \varphi
(t,x), f (t,x)>$

and

$h_j (t,x) = < \partial_x \varphi (t,x), g_j (t,x)>, (t,x) \in [0,T] \times R^n,
\, j\in \{1 \ldots , m\}$,
$z= \mbox{col} (\alpha, x) \in R^{n+1}$.

Let $\gamma (x): R^n \to [0,1]$ be a smooth function $(\gamma \in C^\infty
(R^n))$ such that $\gamma(x)=1$ if $ x\in B (x_0, \rho)$, $ \gamma (x)=0$ if $x
\in R^n \setminus B (x_0, 2 \rho)$ and $0 \leq \gamma (x) \leq 1$ for any $x \in
B(x_0, 2 \rho) \setminus B(x_0, \rho)$, where $\rho >0$ is fixed arbitrarily.
Multiplying $\widehat{f}$ and $\widehat{g_j}$ by $\gamma \in C^\infty (R^n)$ we
get

\item[(8)]
$f^\rho (t,x) \stackrel{def}{=} \gamma (x)\, \widehat{f} (t,x), \, g_j^\rho
(t,x)= \gamma (x) \widehat{g_j} (t,x)$, $j=1, \ldots , m,$ as smooth functions
satisfying the hypothesis of Theorem \ref{tt:4.2.1} and denote $\{z^\rho (t), t \in
[0,T]\}$ the unique solution fulfilling the foloowing system of SDE. $(z^\rho
(t)= (\alpha^\rho (t), x^\rho (t))$

\item[(9)]
$\displaystyle{z^\rho (t)= z_0 + \int_0^t f^\rho (s, x^\rho (s))ds +
\sum_{j=1}^m \int_0^t
g_j^\rho (s,x^\rho (s)) \circ dw^j (s)}$, $t \in [0,T]$,
where
$z_0= \mbox{col} (\varphi (0,x_0), x_0) \in R^{n+1}$ and\\
\mbox{$g_j^\rho (s, x^\rho (s))\circ dw^j )(s)= g_j^\rho (s, x^\rho (s))
\bullet dw^j
(s) + \frac{1}{2} [ \partial_x g_j^\rho (s,x^\rho (s))]g_j^\rho (s, x^\rho (s))
ds$}.

In particular, for $ t \in [0, \tau]$ and using the characteristic function
$\chi_\tau (t), \, \\t \in [0,T]$, (see (5)) we rewrite (9) as

\item[(10)]
$z^\rho (t \wedge \tau) = z_0 + \displaystyle{\int_0^t} \chi_\tau (s)
\widehat{f} (s, x (s))ds +
\sum_{j=1}^m \int_0^t \chi_\partial (s) \widehat{g_j} (s,x (s)) \circ dw (s)$
for any $ t \in [0,T]$, where $z^\rho ( t \wedge \tau) = (\alpha^\rho (t
\wedge \tau), x (t \wedge \tau))$ and
$ \widehat{g_j} (\widehat{f})$ are defined in (7).

Using Remark \ref{rk:4:2.2}, we get

\item[(11)]
\[
\lim_{\varepsilon \to 0 } E|| z_\varepsilon^\rho (t) - z^\rho (t \wedge \tau)
||=0,\,
\mbox {for each}\,\, t \in [0,T]
\]
\end{itemize}
Here $z_\varepsilon^\rho \stackrel{def}{=} z_\varepsilon^\rho (t \wedge \tau), t
\in [0,T]$, verifies the following system of ODE.
\[
\begin{cases}
\dfrac{dz}{dt} = \chi_\tau (t) f^\rho (t,x) + \displaystyle{\sum_{j=1}^m}
\chi_\tau (t) g_j (t,x)
\dfrac{d v^j_\varepsilon (t)}{dt}, t \in [0,T]\\
z(0)=z_0= (\varphi (0, x_0), x_0)
\end{cases}
\leqno (12)
\]
where the vector fields $f^\rho, g_j^\rho, f \in \{1, \ldots m\}$ are defined in
(8) and fulfil the hypothesis of Theorem \ref{tt:4.2.1}.
By definition, $z_\varepsilon^\rho (t)= (\alpha_\varepsilon^\rho (t),
x_\varepsilon^\rho (t), t \in [0,T]$, and (12) can be rewritten as follows
\[
\begin{cases}
\dfrac{d \alpha_\varepsilon^\rho (t)}{dt} = \chi_\tau (t) \gamma
(x_\varepsilon^\rho (t)) \Big [ h_0 (t, x_\varepsilon^\rho (t)) +
\displaystyle{\sum_{j=1}^m}
h_j (t, x_\varepsilon^\rho (t)) \dfrac{d v_\varepsilon^j (t)}{dt} \Big ]\\
\dfrac{d x_\varepsilon^\rho (t)}{dt} = \chi_\tau (t) \gamma (x_\varepsilon^\rho
(t)) \Big [ f (t, x_\varepsilon^\rho (t)) + \displaystyle{\sum_{j=1}^m} g_j (t,
x_\varepsilon^\rho (t)) \dfrac{dv_\varepsilon^j (t)}{dt} \Big ]
\end{cases}
\leqno(13)
\]

\hspace{1,5 cm}$\alpha_\varepsilon^\rho (0)= \varphi (0, x_0), \,
x_\varepsilon^\rho (0)=x_0, \,
t \in [0,T],$
\vspace{3mm}

where the scalar functions $h_i, i \in \{0,1, \ldots , m \}$, are given in (7).

In a similar way, write $z^\rho (t)= (\alpha^\rho (t), x^\rho (t)), t\in [0,T]$
and using (10) we get that
$\alpha^\rho (t \wedge \tau)= \widehat{\alpha}^\rho (t)$
and $ x^\rho (t \wedge \tau)= \widehat{x}^\rho (t)$, $t \in [0,T]$
fulfil the following system of SDE.
\[
\begin{cases}
\widehat{\alpha}^\rho (t) = \varphi (0, x_0) + \displaystyle{\int_0^t}
\chi_\tau (s) h_0
(s, x (s))ds + \displaystyle{\sum_{j=1}^m, \int_0^t} \chi_\tau (s) h_j (s, x(s))
\circ dw^j (s)\\
x(t \wedge \tau) = \widehat{x}^\rho (t)= x_0+ \displaystyle{\int_0^t}
\chi_\tau (s)
f (s, x(s))ds + \displaystyle{\sum_{j=1}^m \int_0^t} \chi_\tau (s)g_j (s,x(s))
\circ dw^j (s)
\end{cases}
\leqno(14)
\]
Notice that $\alpha_\varepsilon^\rho (t)=\varphi (t \wedge \tau,
x_\varepsilon^\rho (t))$ and $ x_\varepsilon^\rho (t), t \in [0,T]$,
$\varepsilon > 0$, are bounded and convergent to $ \widehat{\alpha}^\rho (t), x
( t \wedge \tau)= \widehat{x}^\rho (t)$, correspondigly, for each $ t \in
[0,T]$\, (see (11)), when $\varepsilon \to 0$.
\vspace{2mm}

\noindent As a consequence
\[
\begin{aligned}
\varphi (t &\wedge \tau, x (t \wedge \tau))= \lim_{\varepsilon \to 0} \varphi ( t
\wedge \tau, x_\varepsilon^\rho (t))= \lim_{\varepsilon \to 0}
\alpha^\rho_\varepsilon (t)= \widehat{\alpha}^\rho (t)=\\
& = \varphi (0, x_0) + \int_0^{t \wedge \tau} h_0 (s,x (s))ds + \sum_{j=1}^m
\int_0^{t \wedge \tau} h_j (s,x (s)) \circ dw^j (s),\, t \in [0,T]
\end{aligned}
\leqno(15)
\]
and the proof of the conclusion (6) is complete.
\end{proof}
\noindent\textbf{Comment on stochastic rule of derivation}\\\index{Stochastic!Rule of Derivations}
The stochastic rule of derivation is based on the hypothesis (2) which involves
higher differentiability properties of the diffusion vector fields
$$
g_j \in C_b^{1,2} ([0,T] \times R^n; R^n), j \in \{ 1, \ldots , m\}
$$
On the other hand, using a stopping time $\tau : \Omega \to [0,T]$, the right
hand side of the equation (6) is a semimartingale\index{Semimartingale} without imposing any growth
condition on the test function $ \varphi \in C^{1,3} ([0,T] \times R^n)$.
The standard stochastic rule of derivation does not contain a stopping time
and it
can be accomplished assuming the following growth condition
\begin{itemize}
\item[(16)]
$ | \partial_t \varphi (t,x)|, | \partial_{x_i} \varphi (t,x)|,
|\partial^2_{x_i x_j} \varphi (t,x)| \leq k (1+||x||^p)$,
$i,j \in \{ 1, \ldots , n\}, \\ x \in R^n$, where $p \geq 1$ (natural) and $k >0$
are fixed.
Adding condition (16) to the hypotheses (1) and (2) of Theorem \ref{th:4.t2}, and using a
sequence of stopping times $ \tau_\rho : \Omega \to [0,T]$, $
\displaystyle{\lim_{\rho \to
\infty} \tau_\rho = T}$, we get the following stochastic rule of derivation\index{Stochastic!Rule of Derivations}
\end{itemize}
\[
\begin{aligned}
\varphi (t,x(t))&= \varphi (0, x_0)+ \int_0^t [\partial_s \varphi (s, x (s)) +
<\partial_x \varphi (s,x(s)), f(s, x(s))>]ds +\\
&+ \sum_{j=1}^m \int_0^t <\partial_x \varphi (s,x (s)), g_j (s,x(s))> \circ dw^j
(s), \, t \in [0,T]
\end{aligned}
\leqno(17)
\].

\section[Appendix]{Appendix}
\begin{center}
{\Large I.~Two problems for stochastic flows asociated with nonlinear
	parabolic equations}
\end{center}
\label{part:par}

(I. Molnar , C. Varsan, Functionals associated with gradient stochastic flows and nonlinear
	parabolic equations, preprint IMAR 12/2009)


\subsection{Introduction}
\label{sec:a}

Consider that
$\{\hat{x}_{\varphi}(t;\lambda):t\in[0,T]\}$
is the unique solution of SDE driven by complete vector fields
$f\in(\mathcal{C}_b\cap\mathcal{C}^1_b\cap\mathcal{C}^2)(\mathbb{R}^n;\mathbb{R}^n)$ and
$g\in(\mathcal{C}^1_b\cap\mathcal{C}^2)(\mathbb{R}^n;\mathbb{R}^n)$,
\begin{equation}
\label{a:1}
\left\{
\begin{aligned}
&d_t\hat{x}=\varphi(\lambda)f(\hat{x})dt+g(\hat{x})\circ dw(t),\,t\in[0,T],\,x\in\mathbb{R}^n,\\
&\hat{x}(0)=\lambda\in\mathbb{R}^n,
\end{aligned}
\right.
\end{equation}
where
$\varphi\in(\mathcal{C}^1_b\cap\mathcal{C}^2)(\mathbb{R}^n)$ and
$w(t)\in\mathbb{R}$ is a scalar Wiener process over a complete filtered
probability space
$\{\Omega,\mathcal{F}\supset\{\mathcal{F}_t\},P\}$.
We recall that Fisk-Stratonovich integral
``$\circ$'' in \eqref{a:1} is computed by
\[
\label{a:2}
g(x)\circ dw(t)=g(x)\cdot dw(t)+\frac{1}{2}\partial_x g(x)\cdot g(x)dt,
\]
using Ito stochastic integral
``$\cdot$''.

We are going to introduce some nonlinear SPDE or PDE of parabolic type
which describe the evolution of stochastic functionals
$u(t,x):=h(\psi(t,x))$, or $S(t,x):=Eh(\hat{x}_{\psi}(T;t,x))$, $t\in[0,T]$,
$x\in\mathbb{R}^n$, for a fixed
$h\in(\mathcal{C}^1_b\cap\mathcal{C}^2)(\mathbb{R}^n)$.
Here $\{\lambda=\psi(t,x):t\in[0,T],x\in\mathbb{R}^n\}$
is the unique solution satisfying integral equations
\begin{equation}
\label{a:3}
\hat{x}_\varphi(t;\lambda)=x\in\mathbb{R}^n,\,t\in[0,T].
\end{equation}
The evolution of
$\{S(t,x):t\in[0,T],x\in R^n\}$
will be defined by some  nonlinear backward parabolic equation considering that
$\{\hat{x}_{\psi}(s;t,x):s\in[t,T],x\in R^n\}$ is the unique solution of SDE
\[
\left\{
\begin{aligned}
&d_s\hat{x}=\varphi(\psi(t,x))f(\hat{x})ds+g(\hat{x})\circ dw(s),\,s\in[t,T],\\
&\hat{x}(t)=x\in\mathbb{R}^n.
\end{aligned}
\right.
\]

\subsection{Some problems and their solutions}
\label{sec:b}

\noindent
\textbf{Problem} (P1). Assume that
$g$ and $f$ commute using Lie bracket, i.e.
\begin{equation}
\label{b:1}
[g,f](x)=0,\,x\in\mathbb{R}^n,
\end{equation}
where $[g,f](x):=[\partial_xg(x)]f(x)-[\partial_xf(x)]g(x)$,
\begin{equation}
\label{b:2}
TVK=\rho\in[0,1),
\end{equation}
where $V=\sup\{|\partial_x\varphi(x)|:x\in\mathbb{R}\}^n$ and
$K=\sup\{|f(x)|;x\in\mathbb{R}^n\}$.

Under the hypotheses \eqref{b:1} and \eqref{b:2}, find the nonlinear
SPDE of parabolic type satisfied by
$\{u(t,x)=h(\psi(t,x)):\in[0,T],x\in\mathbb{R}^n\}$,
$h\in(\mathcal{C}^1_b\cap\mathcal{C}^2)(\mathbb{R}^n)$, where
$\{\lambda=\psi(t,x)\in\mathbb{R}^n:t\in[0,T],x\in\mathbb{R}^n\}$	
is the unique continuous and $\mathcal{F}_t$-adapted solution of the integral
equation \eqref{a:3}.

\noindent
\textbf{Problem} (P2). Using $\{\lambda=\psi(t,x)\}$ found in (P1), describe
the evolution of a functional $S(t,x):=Eh(\hat{x}_{\psi}(T;t,x))$ using
backward parabolic equations, where $\{\hat{x}_{\psi}(s;t,x):s\in[t,T]\}$ is
the unique solution of SDE
\begin{equation}
\label{b:3}
\left\{
\begin{aligned}
&d_s\hat{x}=\varphi(\psi(t,x))f(\hat{x})ds+g(\hat{x})\circ dw(s),\,s\in[t,T]\\
&\hat{x}(t)=x\in\mathbb{R}^n.
\end{aligned}
\right.
\end{equation}

\subsection{Solution for the Problem (P1)}
\label{ssec:b.1}

\begin{remark}
\label{rem-b:1}
Under the hypotheses \eqref{b:1} and \eqref{b:2} of (P1), the unique
solution of integral equations \eqref{a:3} will be found as a composition
\begin{equation}
\label{b:4}
\psi(t,x)=\hat{\psi}(t,\hat{z}(t,x)),
\end{equation}
where $\hat{z}(t,x):=G(-w(t))[x]$ and $\lambda=\hat{\psi}(t,z)$, $t\in[0,T]$,
$z\in\mathbb{R}^n$, is the unique deterministic solution satisfying integral
equations
\begin{equation}
\label{b:5}
\lambda=F(-\theta(t;\lambda))[z]=:\hat{V}(t,z;\lambda),\,
t\in[0,T],\,z\in\mathbb{R}^n.
\end{equation}
Here $F(\sigma)[z]$ and $G(\tau)[z]$, $\sigma,\tau\in\mathbb{R}$, are the
global flows generated by complete vector fields $f$ and $g$ correspondingly,
and $\theta(t;\lambda)=t\varphi(\lambda)$.
The unique solution of \eqref{b:5} is constructed in the following
\end{remark}

\begin{lemma}
\label{lem-b:1}
Assume that \eqref{b:2} is fulfilled.  Then there exists a unique smooth
deterministic mapping $\{\lambda=\hat{\psi}(t,z):t\in[0,T],x\in\mathbb{R}^n\}$
solving integral equations \eqref{b:5} such that
\begin{equation}
\label{b:6}
\left\{
\begin{aligned}
&F(\theta(t;\hat{\psi}(t,z)))[\hat{\psi}(t,z)]=z\in\mathbb{R}^n,\,t\in[0,T],\\
&|\hat{\psi}(t,z)-z|\leq R(T,z):=\frac{r(T,z)}{1-\rho},\,
t\in[0,T],\text{ where }r(T,z)=TK|\varphi(z)|,
\end{aligned}
\right.
\end{equation}
\begin{equation}
\label{b:7}
\left\{
\begin{aligned}
&\partial_t\hat{\psi}(t,z)+
\partial_z\hat{\psi}(t,z)f(z)\varphi(\hat{\psi}(t,z))=0,\,
t\in[0,T],\,x\in\mathbb{R}^n,\\
&\hat{\psi}(0,z)=z\in\mathbb{R}^n.
\end{aligned}
\right.
\end{equation}
\end{lemma}

\begin{proof}
The mapping $\hat{V}(t,z;\lambda)$ (see \eqref{b:5}) is a contractive
application with respect to $\lambda\in\mathbb{R}^n$, uniformly of
$(t,z)\in[0,T]\times\mathbb{R}^n$ which allows us to get the unique solution
of \eqref{b:5} using a standard procedure (Banach theorem).
By a direct computation, we get
\begin{equation}
\label{b:8}
|\partial_\lambda\hat{V}(t,z;\lambda)|=
|f(\hat{V}(t,z;\lambda))\partial_\lambda\theta(t;\lambda)|\leq
TVK=\rho\in[0,1),
\end{equation}
for any $t\in[0,T]$, $z\in\mathbb{R}^n$, $\lambda\in\mathbb{R}^n$, where
$\partial_\lambda\theta(t;\lambda)$ is a row vector.  The corresponding
convergent sequence $\{\lambda_k(t,z):t\in[0,T],z\in\mathbb{R}^n\}_{k\geq 0}$
is constructed fulfilling
\begin{equation}
\label{b:9}
\lambda_0(t,z)=z,\,\lambda_{k+1}(t,z)=\hat{V}(t,z;\lambda_k(t,z)),\,t\geq 0,
\end{equation}
\begin{equation}
\label{b:10}
\left\{
\begin{aligned}
&|\lambda_{k+1}(t,z)-\lambda_k(t,z)|\leq\rho^k|\lambda_1(t,z)-\lambda_0(t,z)|,\,
k\geq 0,\\
&|\lambda_1(t,z)-\lambda_0(t,z)|\leq|\hat{V}(t,z;z)-z|\leq TK|\varphi(z)|=:r(T,z).
\end{aligned}
\right.
\end{equation}
Using \eqref{b:10} we obtain that $\{\lambda_k(t,z)\}_{k\geq 0}$ is
convergent and
\begin{equation}
\label{b:11}
\hat{\psi}(t,z)=\lim_{k\to\infty}\lambda_k(t,z),\,
|\hat{\psi}(t,z)-z|\leq\frac{r(T,z)}{1-\rho}=:R(T,z),\,t\in[0,T].
\end{equation}
Passing $k\to\infty$ into \eqref{b:9} and using \eqref{b:11} we get the
first conclusion \eqref{b:6}. On the other hand, notice that
$\{\hat{V}(t,z;\lambda):t\in[0,T],z\in\mathbb{R}^n\}$ of \eqref{b:5} fulfils
\begin{equation}
\label{b:12}
\hat{V}(t,\hat{y}(t,\lambda);\lambda)=\lambda,\,t\in[0,T],
\text{ where }\hat{y}(t,\lambda)=F(\theta(t;\lambda))[\lambda].
\end{equation}
This shows that all the components of $\hat{V}(t,z;\lambda)\in\mathbb{R}^n$
are first integrals associated with the vector field
$f_{\lambda}(z)=\varphi(\lambda)f(z)$, $z\in\mathbb{R}^n$, for each
$\lambda\in\mathbb{R}^n$, i.e.
\begin{equation}
\label{b:13}
\partial_t\hat{V}(t,\hat{y}(t,\lambda);\lambda)+
[\partial_z\hat{V}(t,\hat{y}(t,\lambda);\lambda)]
f(\hat{y}(t,\lambda))\varphi(\lambda)=0,\,t\in[0,T]
\end{equation}
is valid for each $\lambda\in\mathbb{R}^n$.  In particular, for
$\lambda=\hat{\psi(t,z)}$ we get $\hat{y}(t,\hat{\psi}(t,z))=z$ and
\eqref{b:13} becomes (H-J)-equation
\begin{equation}
\label{b:14}
\partial_t\hat{V}(t,z;\hat{\psi}(t,z))+[\partial_z\hat{V}(t,z;\hat{\psi}(t,z)]
f(z)\varphi(\hat{\psi}(t,z))=0,\,t\in[0,T],\,z\in\mathbb{R}^n.
\end{equation}
Combining \eqref{b:5} and \eqref{b:14}, by direct computation, we convince
ourselves that $\lambda=\hat{\psi}(t,z)$ fulfils the following nonlinear
(H-J)-equation (see \eqref{b:7})
\begin{equation}
\label{b:15}
\left\{
\begin{aligned}
&\partial_t\hat{\psi}(t,z)+[\partial_z\hat{\psi}(t,z)]
f(z)\varphi(\hat{\psi}(t,z))=0,\,t\in[0,T],\,z\in\mathbb{R}^n,\\
&\hat{\psi}(0,z)=z\in\mathbb{R}^n,
\end{aligned}
\right.
\end{equation}
and the proof is complete.
\end{proof}

\begin{remark}
\label{rem-b:2}
Under the hypothesis \eqref{b:1}, the \index{stochastic flow}
$\{\hat{x}_{\varphi}(t;\lambda):t\in[0,T],\lambda\in\mathbb{R}^n\}$ generated
by SDE \eqref{a:1} can be represented as follows
\begin{equation}
\label{b:16}
\hat{x}_{\varphi}(t;\lambda)=G(w(t))\circ F(\theta(t;\lambda))[\lambda]=
H(t,w(t);\lambda),\,t\in[0,T],\,\lambda\in\mathbb{R}^n
\end{equation}
where $\theta(t;\lambda)=t\varphi(\lambda)$.
\end{remark}

\begin{lemma}
\label{lem-b:2}
Assume that \eqref{b:1} and \eqref{b:2} are satisfied and consider
$\{\lambda=\hat{\psi}(t,z):t\in[0,T],z\in\mathbb{R}^n\}$ found in
Lemma~\ref{lem-b:1}. Then the stochastic flow generated by SDE~\eqref{a:1}
fulfils
\begin{equation}
\label{b:17}
\{\hat{x}_{\varphi}(t;\lambda):t\in[0,T],\lambda\in\mathbb{R}^n\}
\text{ can be represented as in } \eqref{b:16},
\end{equation}
\begin{equation}
\label{b:18}
\begin{aligned}
&\psi(t,x)=\hat{\psi}(t,\hat{z}(t,x))
\text{ is the unique solution of integral equations } \eqref{a:3},\\
&\text{ where } \hat{z}(t,x)=G(-w(t))[x].
\end{aligned}
\end{equation}
\end{lemma}

\begin{proof}
Using the hypothesis \eqref{b:1}, we see easily that
\begin{equation}
\label{b:19}
y(\theta,\sigma)[\lambda]:=G(\sigma)\circ F(\theta)[\lambda],\,
\theta,\sigma\in\mathbb{R},\,\lambda\in\mathbb{R}^n
\end{equation}
is the unique solution of the gradient system
\begin{equation}
\label{b:20}
\left\{
\begin{aligned}
&\partial_{\theta}y(\theta,\sigma)[\lambda]=f(y(\theta,\sigma)[\lambda]),\,
\partial_{\sigma}y(\theta,\sigma)[\lambda]=g(y(\theta,\sigma)[\lambda]),\\
&y(0,0)[\lambda]=\lambda
\end{aligned}
\right.
\end{equation}
Applying the standard rule of stochastic derivation associated with the smooth
mapping $\varphi(\theta,\sigma):=y(\theta,\sigma)[\lambda]$ and the continuous
process $\theta=\theta(t;\lambda)=t\varphi(\lambda)$, $\sigma=w(t)$, we get
that $\hat{y}_{\varphi}(t;\lambda)=y(\theta(t;\lambda),w(t))$, $t\in[0,T]$,
fulfils SDE~\eqref{a:1}, i.e.
\begin{equation}
\label{b:21}
\left\{
\begin{aligned}
&d_t\hat{y}_{\varphi}(t;\lambda)=\varphi(\lambda)f(\hat{y}_{\varphi}(t;x))dt
+g(\hat{y}_{\varphi}(t;\lambda))\circ dw(t),\,t\in[0,T],\\
&\hat{y}_{\varphi}(0;\lambda)=\lambda.
\end{aligned}
\right.
\end{equation}
On the other hand, the unicity of the solution satisfying \eqref{a:1} lead us
to the conclusion that
$\hat{x}_{\varphi}(t;\lambda)=\hat{y}_{\varphi}(t;\lambda)$, $t\in[0,T]$, and
\eqref{b:17} is proved.  The conclusion \eqref{b:18} is a direct consequence
of \eqref{b:17} combined with
$\{\lambda=\hat{\psi}(t,z):t\in[0,T],z\in\mathbb{R}^n\}$ is the solution
defined in Lemma \ref{lem-b:1}. The proof is complete.
\end{proof}

\begin{lemma}
\label{lem-b:3}
Under the hypotheses in Lemma \ref{lem-b:2}, consider the continuous
and $\mathcal{F}_t$-adapted process $\hat{z}(t,x)=G(-w(t))[x]$, $t\in[0,T]$,
$x\in\mathbb{R}^n$.  Then the following SPDE of parabolic type is valid
\begin{equation}
\label{b:22}
\left\{
\begin{aligned}
&d_t\hat{z}(t,x)+\partial_x\hat{z}(t,x)g(x)\hat{\circ}dw(t)=0,\,
t\in[0,T],x\,x\in\mathbb{R}^n,\\
&\hat{z}(0,x)=x
\end{aligned}
\right.
\end{equation}
where the \index{Fisk-Stratonovich integral} ``$\hat{\circ}$'' is computed by
\[
h(t,x)\hat{\circ}dw(t)=h(t,x)\cdot dw(t)-\frac{1}{2}\partial_x h(t,x)g(x)dt,
\]
using Ito \index{stochastic integral} ``$\cdot$''.
\end{lemma}

\begin{proof}
The conclusion \eqref{b:22} is a direct consequence of applying standard
rule of stochastic derivation associated with $\sigma=w(t)$ and smooth
deterministic mapping $H(\sigma)[x]:=G(-\sigma)[x]$.  In this respect, using
$H(\sigma)\circ G(\sigma)[\lambda]=\lambda\in\mathbb{R}^n$ for any
$x=G(\sigma)[\lambda]$, we get
\begin{equation}
\label{b:23}
\left\{
\begin{aligned}
&\partial_{\sigma}\{H(\sigma)[x]\}=-\partial_x\{H(\sigma)[x]\}\cdot g(x),\,
\sigma\in\mathbb{R},\,x\in\mathbb{R}^n,\\
&\begin{split}
\partial^2_{\sigma}\{H(\sigma)[x]\}
&=\partial_{\sigma}\{\partial_{\sigma}\{H(\sigma)[x]\}\}=
\partial_{\sigma}\{-\partial_x\{H(\sigma)[x]\}\cdot g(x)\}\\
&=\partial_x\{\partial_x\{H(\sigma)[x]\}\cdot g(x)\}\cdot g(x),\,
\sigma\in\mathbb{R},\,x\in\mathbb{R}^n.
\end{split}
\end{aligned}
\right.
\end{equation}
The standard rule of stochastic derivation lead us to SDE
\begin{equation}
\label{b:24}
d_t\hat{z}(t,x)=\partial_{\sigma}\{H(\sigma)[x]\}_{\sigma=w(t)}\cdot dw(t)
+\frac{1}{2}\partial^2_{\sigma}\{H(\sigma)[x]\}_{\sigma=w(t)}dt,\,t\in[0,T],
\end{equation}
and rewritting the right hand side of \eqref{b:24} (see \eqref{b:23}) we
get SPDE of parabolic type given in \eqref{b:22}. The proof is complete.
\end{proof}

\begin{lemma}
\label{lem-b:4}
Assume the hypotheses \eqref{b:1} and \eqref{b:2} are fulfilled and
consider $\{\lambda=\psi(t,x):t\in[0,T],x\in\mathbb{R}^n\}$ defined in Lemma
\eqref{lem-b:2}. Then $u(t,x):=h(\psi(t,x))$, $t\in[0,T]$, $x\in\mathbb{R}^n$,
$h\in(\mathcal{C}^1_b\cap\mathcal{C}^2)(\mathbb{R}^n)$, satisfies the
following nonlinear SPDE of parabolic type
\begin{equation}
\label{b:25}
\left\{
\begin{aligned}
&d_tu(t,x)+\langle\partial_xu(t,x),f(x)\rangle\varphi(\psi(t,x))dt
+\langle\partial_xu(t,x),g(x)\rangle\hat{\circ}dw(t)=0\\
&u(0,x)=h(x),\,t\in[0,T],\,x\in\mathbb{R}^n,
\end{aligned}
\right.
\end{equation}
where the \index{Fisk-Stratonovich integral} ``$\hat{\circ}$''
is computed by
\[
h(t,x)\hat{\circ}dw(t)=h(t,x)\cdot dw(t)-\frac{1}{2}\partial_xh(t,x)g(x)dt.
\]
\end{lemma}

\begin{proof}
By definition (see Lemma \eqref{lem-b:2}),
$\psi(t,x)=\hat{\psi}(t,\hat{z}(t,x))$, $t\in[0,T]$, where
$\hat{z}(,x)=G(-w(t))[x]$ and
$\{\hat{\psi}(t,z)\in\mathbb{R}^n:t\in[0,T],z\in\mathbb{R}^n\}$ satisfies
nonlinear (H-J)-equations \eqref{b:7} of Lemma \ref{lem-b:1}. In addition
$\{\hat{z}(t,x)\in\mathbb{R}^n:t\in[0,T],x\in\mathbb{R}^n\}$ fulfils SPDE
\eqref{b:22} in Lemma~\ref{lem-b:3}, i.e.
\begin{equation}
\label{b:26}
d_t\hat{z}(t,x)+\partial_x\hat{z}(t,x)\hat{\circ}dw(t)=0,\,
t\in[0,T],\,x\in\mathbb{R}^n.
\end{equation}
Applying the standard rule of stochastic derivation associated with the smooth
mapping $\{\lambda=\hat{\psi}(t,z):t\in[0,T],z\in\mathbb{R}^n\}$ and
stochastic process $\hat{z}(t,x):=G(-w(t))[x]=:H(w(t))[x]$, $t\in[0,T]$, we
get the following nonlinear SPDE
\begin{equation}
\label{b:27}
\left\{
\begin{aligned}
&d_t\psi(t,x)+\partial_x\psi(t,x)f(x)\varphi(\psi(t,x))dt
+\partial_x\psi(t,x)g(x)\hat{\circ}dw(t)=0,\\
&\psi(0,x)=x,\,t\in[0,T].
\end{aligned}
\right.
\end{equation}
In addition, the functional $u(t,x)=h(\psi(t,x))$ can be rewritten
$u(t,x)=\hat{u}(t,\hat{z}(t,x))$, where $\hat{u}(t,z):=h(\hat{\psi}(t,z))$
is a smooth \index{deterministic functional} satisfying nonlinear (H-J)-equations
(see \eqref{b:7} of Lemma~\ref{lem-b:1})
\begin{equation}
\label{b:28}
\left\{
\begin{aligned}
&\partial_t\hat{u}(t,z)+\langle\partial_z\hat{u}(t,z),f(z)\rangle
\varphi(\hat{\psi}(t,z))=0,\,t\in[0,T],\,z\in\mathbb{R}^n,\\
&\hat{u}(0,z)=h(z).
\end{aligned}
\right.
\end{equation}
Using \eqref{b:26}and \eqref{b:28} we obtain SDPE fulfilled by $\{u(t,x)\}$,
\begin{equation}
\label{b:29}
\left\{
\begin{aligned}
&d_tu(t,x)+\langle\partial_z\hat{u}(t,\hat{z}(t,x)),f(\hat{z}(t,x))\rangle
\varphi(\psi(t,x))dt+\langle\partial_xu(t,x),g(x)\rangle\hat{\circ}dw(t)=0,\\
&u(0,x)=h(x),\,t\in[0,T],\,x\in\mathbb{R}^n.
\end{aligned}
\right.
\end{equation}
The hypothesis \eqref{b:1} allows us to write
\begin{equation}
\label{b:30}
\begin{split}
\langle\partial_z\hat{u}(t,\hat{z}(t,x)),f(\hat{z}(t,x))\rangle
&=\partial_z\hat{u}(t,\hat{z}(t,x))[\partial_x\hat{z}(t,x)]
[\partial_x\hat{z}(t,x)]^{-1}f(\hat{z}(t,x))\\
&=\langle\partial_xu(t,x),f(x)\rangle,\,t\in[0,T],\,x\in\mathbb{R}^n,
\end{split}
\end{equation}
and using \eqref{b:30} into \eqref{b:29} we get the conclusion \eqref{b:25},
\begin{equation}
\label{b:31}
\left\{
\begin{aligned}
&\partial_tu(t,x)+\langle\partial_xu(t,x),f(x)\rangle\varphi(\psi(t,x))dt
+\langle\partial_xu(t,x),g(x)\rangle\hat{\circ}dw(t)=0,\\
&u(0,x)=h(x),\,t\in[0,T],x\in\mathbb{R}^n,
\end{aligned}
\right.
\end{equation}
where the \index{Fisk-Stratonovich integral} ``$\hat{\circ}$'' is computed
by
\begin{equation}
\label{b:32}
h(t,x)\hat{\circ}dw(t)=-\frac{1}{2}\partial_xh(t,x)g(x)dt+h(t,x)\cdot dw(t),
\end{equation}
using Ito integral ``$\cdot$''. The proof is complete.
\end{proof}

\begin{remark}
\label{rem-b:3}
The complete solution of Problem (P1) is contained in Lemmas
\ref{lem-b:1}--\ref{lem-b:4}. We shall rewrite them as a theorem.
\end{remark}

\begin{theorem}
\label{thm-b:1}
Assume that the vector fields
$f\in(\mathcal{C}_b\cap\mathcal{C}^1_b\cap\mathcal{C}^2)
(\mathbb{R}^n;\mathbb{R}^n)$,
$g\in(\mathcal{C}^1_b\cap\mathcal{C}^2)(\mathbb{R}^n;\mathbb{R}^n)$, and
scalar function $\varphi\in(\mathcal{C}^1_b\cap\mathcal{C}^2)(\mathbb{R}^n)$
fulfil the hypotheses \eqref{b:1} and \eqref{b:2}. Consider the
continuous and $\mathcal{F}_t$-\index{adapted process}
$\{\lambda=\psi(t,x\in\mathbb{R}^n):t\in[0,T],x\in\mathbb{R}^n\}$ satisfying
integral equations \eqref{a:3}. Then $u(t,x):=h(\psi(t,x))$, $t\in[0,T]$,
$x\in\mathbb{R}^n$, fulfils nonlinear SPDE of parabolic type \eqref{b:25}
(see Lemma \ref{lem-b:4}), for each
$h\in(\mathcal{C}^1_b\cap\mathcal{C}^2)(\mathbb{R}^n)$.
\end{theorem}

\subsection{Solution for the Problem (P2)}
\label{ssec:b.2}

Using the same notations as in subsection \ref{ssec:b.1}, we consider the
unique solution $\{\hat{x}_{\psi}(s;t,x):s\in[t,T]\}$ satisfying
SDE~\eqref{b:3} for each $0\leq t<T$ and $x\in\mathbb{R}^n$.   As far as
SDE~\eqref{b:3} is a \index{non-markovian system}, the evolution of a functional
$S(t,x):=Eh(\hat{x}_{\psi}(T;t,x))$, $t\in[0,T]$, $x\in\mathbb{R}^n$,
$h\in(\mathcal{C}^1_b\cap\mathcal{C}^2)(\mathbb{R}^n)$, will be described
using the pathwise representation of the conditional mean values functional
\begin{equation}
\label{b:33}
v(t,x)=E\{h(\hat{x}_{\psi}(T;t,x))\mid\psi(t,x)\},\,
0\leq t<T,\,x\in\mathbb{R}^n.
\end{equation}
Assuming the hypotheses \eqref{b:1} and \eqref{b:2} we may and do
write the following integral representation
\begin{equation}
\label{b:34}
\hat{x}_{\psi}(T;t,x)=G(w(T)-w(t))\circ F[(T-t)\varphi(\psi(t,x))][x],\,
0\leq t<T,\,x\in\mathbb{R}^n,
\end{equation}
for a solution of SDE \eqref{b:3}, where $G(\sigma)[z]$ and $F(\tau)[z]$,
$\sigma,\tau\in\mathbb{R}$, $z\in\mathbb{R}^n$, are the global flows generated
by $g,f\in(\mathcal{C}^1_b\cap\mathcal{C}^2)(\mathbb{R}^n;\mathbb{R}^n)$.  The
right side hand of \eqref{b:34} is a continuous mapping of the two
independent random variables, $z_1=[w(T)-w(t)]\in\mathbb{R}$ and
$z_2=\psi(t,x)\in\mathbb{R}^n$ ($\mathcal{F}_t$-measurable) for each
$0\leq t<T$, $x\in\mathbb{R}^n$.   A direct consequence of this remark is to
use a \index{parameterized random variable}
\begin{equation}
\label{b:35}
y(t,x;\lambda)=G(w(T)-w(t))\circ F[(T-t)\varphi(\lambda)][x],\,0\leq t<T,
\end{equation}
and to compute the conditional mean values \eqref{b:33} by
\begin{equation}
\label{b:36}
v(t,x)=[Eh(y(t,x;\lambda))](\lambda=\psi(t,x)).
\end{equation}
Here the functional
\begin{equation}
\label{b:37}
u(t,x;\lambda):=Eh(y(t,x;\lambda)),\,t\in[0,T],\,x\in\mathbb{R}^n,
\end{equation}
satisfies a backward parabolic equation (Kolmogorov's equation) for each
$\lambda\in\mathbb{R}^n$ and rewrite \eqref{b:36} as follows,
\begin{equation}
\label{b:38}
v(t,x)=u(t,x;\psi(t,x)),\,0\leq t<T,\,x\in\mathbb{R}^n.
\end{equation}
In conclusion, the functional $\{S(t,x)\}$ can be written as
\begin{equation}
\label{b:39}
S(t,x)=E[E\{h(\hat{x}_{\psi}(T;t,x))\mid\psi(t,x)\}=Eu(t,x;\psi(t,x)),\,
0\leq t<T,\,x\in\mathbb{R}^n,
\end{equation}
where $\{u(t,x;\lambda):t\in[0,T],x\in\mathbb{R}^n\}$ satisfies the
corresponding backward parabolic equations with parameter
$\lambda\in\mathbb{R}^n$,
\begin{equation}
\label{b:40}
\left\{
\begin{aligned}
&\partial_tu(t,x;\lambda)+\langle\partial_xu(t,x;\lambda),f(x,\lambda)\rangle
+\frac{1}{2}\langle\partial^2_xu(t,x;\lambda)g(x),g(x)\rangle=0,\\
&u(T,x;\lambda)=h(x),\,f(x,\lambda):=\varphi(\lambda)f(x)+\frac{1}{2}[\partial_xg(x)]g(x).
\end{aligned}
\right.
\end{equation}
We conclude these remarks by a theorem.

\begin{theorem}
\label{thm-b:2}
Assume that the vector fields $f,g$ and the scalar function $\varphi$ of SDE
\eqref{b:3} fulfil the hypotheses \eqref{b:1} \eqref{b:2}, where the
continuous and $\mathcal{F}_t$-\index{adapted process}
$\{\psi(t,x)\in\mathbb{R}^n:t\in[0,T]\}$ is defined in Theorem \ref{thm-b:1}.
Then the evolution of the functional
\begin{equation}
\label{b:41}
S(t,x):=Eh(\hat{x}_{\psi}(T;t,x)),\,t\in[0,T],\,x\in\mathbb{R}^n,\,
h\in(\mathcal{C}^1_b\cap\mathcal{C}^2)(\mathbb{R}^n)
\end{equation}
can be described as in \eqref{b:39}, where
$\{u(t,x):t\in[0,T],x\in\mathbb{R}^n\}$ satisfies linear backward parabolic
equations \eqref{b:40} for each $\lambda\in\mathbb{R}^n$.
\end{theorem}

\begin{remark}
\label{rem-b:4}
Consider the case of several vector fields defining both the drift and
diffusion of SDE \eqref{a:1}, i.e.
\begin{equation}
\label{b:42}
\left\{
\begin{aligned}
&d_t\hat{x}=[\sum_{i=1}^{m}\varphi_i(\lambda)f_i(\hat{x})]dt
+\sum_{i=1}^{m}g_i(\hat{x})\circ dw_i(t),\,t\in[0,T],\\
&\hat{x}(0)=\lambda\in\mathbb{R}^n.
\end{aligned}
\right.
\end{equation}
We notice that the analysis presented in Theorems \ref{thm-b:1} and
\ref{thm-b:2} can be extended to this multiple vector fields case
(see next section).
\end{remark}

\subsection{Multiple vector fields case}
\label{sec:c}

We are given two finite sets of vector fields
$\{f_1,\dots,f_m\}\subset(\mathcal{C}_b\cap\mathcal{C}^1_b\cap\mathcal{C}^2)
(\mathbb{R}^n;\mathbb{R}^n)$ and
$\{g_1,\dots,g_m\}\subset(\mathcal{C}^1_b\cap\mathcal{C}^2)
(\mathbb{R}^n;\mathbb{R}^n)$ and consider the unique solution
$\{\hat{x}_{\varphi}(t,\lambda):t\in[0,T],\lambda\in\mathbb{R}^n\}$ of SDE
\begin{equation}
\label{c:1}
\left\{
\begin{aligned}
&d_t\hat{x}=[\sum_{i=1}^{m}\varphi_i(\lambda)f_i(\hat{x})]dt
+\sum_{i=1}^{m}g_i(\hat{x})\circ dw_i(t),\,t\in[0,T],\,\hat{x}\in\mathbb{R}^n,\\
&\hat{x}(0)=\lambda\in\mathbb{R}^n
\end{aligned}
\right.
\end{equation}
where
$\varphi=(\varphi_1,\dots,\varphi_m)\subset(\mathcal{C}^1_b\cap\mathcal{C}^2)$
are fixed and $w=(w_1(t),\dots,w_m(t))\in\mathbb{R}^m$ is a standard Wiener
process over a complete filtered probability space
$\{\Omega,\mathcal{F}\supset\{\mathcal{F}_t\},P\}$.
Each \index{Fisk-Stratonovich integral} ``$\circ$'' in \eqref{c:1} is computed by
\begin{equation}
\label{c:2}
g_i(x)\circ dw_i(t)=g_i(x)\cdot dw_i(t)+\frac{1}{2}[\partial_xg_i(x)]g_i(x)dt,
\end{equation}
using Ito integral ``$\cdot$''.

Assume that $\{\lambda=\psi(t,x)\in\mathbb{R}^n:t\in[0,T],x\in\mathbb{R}^n\}$
is the unique continuous and $\mathcal{F}_t$-adapted solution satisfying
integral equations
\begin{equation}\label{c:3}
\hat{x}_{\varphi}(t;\lambda)=x\in\mathbb{R}^n,\,t\in[0,T].
\end{equation}
For each $h\in(\mathcal{C}^1_b\cap\mathcal{C}^2)(\mathbb{R}^n)$, associate
stochastic functionals $\{u(t,x)=h(\psi(t,x):t\in[0,T],x\in\mathbb{R}^n)\}$
and \index{deterministic mappings}
$\{S(t,x)=Eh(\hat{x}_{\psi}(T;t,x)):t\in[0,T],x\in\mathbb{R}^n\}$, where
$\{\hat{x}_{\psi}(s;t,x):s\in[t,T],x\in\mathbb{R}^n\}$ satisfies the following SDE
\[
\left\{
\begin{aligned}
&d_s\hat{x}=[\sum_{i=1}^{m}\varphi_i(\psi(t,x))f_i(\hat{x})]ds
+\sum_{i=1}^{m}g_i(\hat{x})\circ dw_i(t),\,s\in[t,T],\\
&\hat{x}(t)=x.
\end{aligned}
\right.
\]

\noindent
\textbf{Problem} (P1). Assume that
\begin{equation}
\label{c:4a}
\left\{
\begin{aligned}
&M=\{f_1,\dots,f_m,g_1,\dots,g_m\} \text{ are muttualy commuting using Lie
bracket i.e.}\\
&[X_1,X_2](x)=0 \text{ for any pair } X_1,X_2\in M
\end{aligned}
\right.
\end{equation}
\begin{equation}
\label{c:4b}
TV_iK_i=\rho_i\in[0,\frac{1}{m}),
\end{equation}
where $V_i:=\sup\{|\partial_x\varphi_i(x)|:x\in\mathbb{R}^n\}$ and
$K_i=\{|f_i(x)|:x\in\mathbb{R}^n\}$, $i=1,\dots,m$.

Under the hypotheses \eqref{c:4a} and \eqref{c:4b}, find the nonlinear
SPDE of parabolic type satisfied by
$\{u(t,x)=h(\psi(t,x)),t\in[0,T],x\in\mathbb{R}^n\}$,
$h\in(\mathcal{C}^1_b\cap\mathcal{C}^2)(\mathbb{R}^n)$, where
$\{\lambda=\psi(t,x)\in\mathbb{R}^n:t\in[0,T],x\in\mathbb{R}^n\}$ is the
unique continuous and $\mathcal{F}_t$-adapted solution of integral equations
\eqref{c:3}.

\noindent
\textbf{Problem} (P2).
Using $\{\lambda=\psi(t,x)\in\mathbb{R}^n:t\in[0,T],x\in\mathbb{R}^n\}$ found
in (P1), describe the evolution of a functional
$S(t,x)=Eh(\hat{x}_{\psi}(T;t,x))$ using backward parabolic equations, where
$\{\hat{x}_{\psi}(s;t,x):s\in[t,T]\}$ is the unique solution of SDE
\begin{equation}
\label{c:5}
\left\{
\begin{aligned}
&d_s\hat{x}=[\sum_{i=1}^{m}\varphi_i(\psi(t,x))f_i(\hat{x})]ds
+\sum_{i=1}^{m}g_i(\hat{x})\circ dw_i(s),\,s\in[t,T],\\
&\hat{x}(t)=\hat{x}\in\mathbb{R}^n.
\end{aligned}
\right.
\end{equation}

\subsection{Solution for (P1)}
\label{ssec:c.1}

Under the hypotheses \eqref{c:4a} and \eqref{c:4b}, the unique solution
of SPDE \eqref{c:1} can be represented by
\begin{equation}
\label{c:6}
\hat{x}_{\varphi}(t;\lambda)=G(w(t))\circ F(\theta(t;\lambda))[\lambda]
=:H(t,w(t);\lambda)
\end{equation}
where
\[
\begin{aligned}
&G(\sigma)[z]=G_1(\sigma_1)\circ\dots\circ G_m(\sigma_m)[z],\,
\sigma=(\sigma_1,\dots,\sigma_m)\in\mathbb{R}^m,\\
&F(\sigma)[z]=F_1(\sigma_1)\circ\dots\circ F_m(\sigma_m)[z],\,
\theta(t;\lambda)=(t\varphi_1(\lambda),\dots,t\varphi_m(\lambda))\in\mathbb{R}^m
\text{ and }\\
&\{(F_i(\sigma_i)[z],G_i(\sigma_i)[z]):\sigma_i\in\mathbb{R},z\in\mathbb{R}^n\}
\end{aligned}
\]
are the global flows generated by $(f_i,g_i)$, $i\in\{1,\dots,m\}$.

The arguments for solving (P1) in the case of one pair $(f,g)$ of vector
fields (see subsection \eqref{ssec:b.1}) can be used also here and we get the
following similar results.  Under the representation \eqref{c:6}, the unique
continuous and $\mathcal{F}_t$-adapted solution
$\{\lambda=\psi(t,x):t\in[0,T],x\in\mathbb{R}^n\}$ solving equations
\begin{equation}
\label{c:7}
\hat{x}_{\varphi}(t;\lambda)=x\in\mathbb{R}^n,\,t\in[0,T]
\end{equation}
will be found as a composition
\begin{equation}
\label{c:8}
\psi(t,x)=\hat{\psi}(t,\hat{z}(t,x)),\,\hat{z}(t,x):=G(-w(t))[x].
\end{equation}
Here $\lambda=\hat{\psi}(t,z)$, $t\in[0,T]$, $z\in\mathbb{R}^n$ is the unique
solution satisfying \index{deterministic integral
eqations}
\begin{equation}
\label{c:9}
\lambda=F(-\theta(t;\lambda))[z]=:\hat{V}(t,z;\lambda),\,t\in[0,T],\,z\in\mathbb{R}^n.
\end{equation}

\begin{lemma}
\label{lem-c:1}
Asume that \eqref{c:4a} and \eqref{c:4b} is fulfilled. Then there exists a
unique smooth mapping $\{\lambda=\hat{\psi}(t,z):t\in[0,T],z\in\mathbb{R}^n\}$
solving \index{deterministic integral equations} \eqref{c:9} such that
\begin{equation}
\label{c:10}
\left\{
\begin{aligned}
&F(\theta(t;\hat{\psi}(t,z)))[\hat{\psi}(t,z)]=z\in\mathbb{R}^n,\,t\in[0,T],\\
&|\hat{\psi}(t,z)-z|\leq R(T,z):=\frac{r(T,z)}{1-\rho},\,t\in[0,T],\,z\in\mathbb{R}^n,
\end{aligned}
\right.
\end{equation}
where $\rho=\rho_1+\dots+\rho_m\in[0,1)$ and
$r(T,z)=T\sum_{i=1}^{m}K_i|\varphi_i(z)|$.

In addition, the following nonlinear (H-J)-equation is valid
\begin{equation}
\label{c:11}
\left\{
\begin{aligned}
&\partial_t\hat{\psi}(t,z)+\partial_z\hat{\psi}(t,z)
[\sum_{i=1}^{m}\varphi_i(\hat{\psi}(t,z))f_i(z)]=0,\,t\in[0,T],\,z\in\mathbb{R}^n,\\
&\hat{\psi}(0,z)=z.
\end{aligned}
\right.
\end{equation}
\end{lemma}
The proof is based on the arguments of Lemma \ref{lem-b:1} in subsection \ref{ssec:b.1}.

\begin{lemma}
\label{lem-c:2}
Assume that \eqref{c:4a} and \eqref{c:4b} are satisfied and consider
$\{\lambda=\hat{\psi}(t,z)\in\mathbb{R}^n:t\in[0,T],z\in\mathbb{R}^n\}$ found
in Lemma \eqref{lem-c:1}. Then the \index{stochastic flow} generated by SDE
\eqref{c:1} fulfils
\begin{equation}
\label{c:12}
\{\hat{x}_{\varphi}(t;\lambda):t\in[0,T],\lambda\in\mathbb{R}^n\}
\text{ can be represented as in \eqref{c:6}},
\end{equation}
\begin{equation}
\label{c:13}
\begin{aligned}
\psi(t,x)=\hat{\psi}(t,\hat{z}(t,x)), &\text{ is the unique solution of
\eqref{c:7}},\\
&\text{ where } \hat{z}(t,x)=G(-w(t))[x].
\end{aligned}
\end{equation}
\end{lemma}
The proof follows the arguments used in Lemma \ref{lem-b:2} of section \ref{ssec:b.1}.

\begin{lemma}
\label{lem-c:3}
Under the hypothesis \eqref{c:4a}, consider the continuous and
$\mathcal{F}_t$-\index{adapted process} $\hat{z}(t,x)=G(-w(t))[x]$, $t\in[0,T]$,
$x\in\mathbb{R}^n$.  Then the following SPDE of \index{parabolic type} is valid
\begin{equation}
\label{c:14}
\left\{
\begin{aligned}
&d_t\hat{z}(t,x)+\sum_{i=1}^{m}\partial_x\hat{z}(t,x)g_i(x)\hat{\circ}dw_i(t)=0,\,
t\in[0,T],\,x\in\mathbb{R}^n\\
&\hat{z}(0,x)=x,
\end{aligned}
\right.
\end{equation}
where the \index{Fisk-Stratonovich integral} ``$\hat{\circ}$'' is computed by
\[
h_i(t,x)\hat{\circ}dw_i(t)=h_i(t,x)\cdot dw_i(t)-\frac{1}{2}\partial_xh_i(t,x)g_i(x)dt
\]
using Ito \index{stochastic integral} ``$\cdot$''.
\end{lemma}

\begin{proof}
The conclusion \eqref{c:14} is a direct consequence of applying standard
rule of \index{stochastic derivation} associated with $\sigma=w(t)\in\mathbb{R}^m$ and
smooth deterministic mapping $H(\sigma)[x]=G(-\sigma)[x]$. In this respect,
using $H(\sigma)\circ G(\sigma)[\lambda]=\lambda\in\mathbb{R}^n$ for any
$x=G(\sigma)[\lambda]$, we get
\begin{equation}
\label{c:15}
\left\{
\begin{aligned}
&\partial_{\sigma_i}H(\sigma)[x]=-\partial_x\{H(\sigma)[x]\}g_i(x),\,
\sigma=(\sigma_1,\dots,\sigma_m)\in\mathbb{R}^m,\,x\in\mathbb{R}^n,\\
&\begin{split}
\partial^2_{\sigma_i}\{H(\sigma)[x]\}
&=\partial_{\sigma_i}\{\partial_{\sigma_i}\{H(\sigma)[x]\}\}
=\partial_{\sigma_i}\{-\partial_x\{H(\sigma)[x]\}g_i(x)\}\\
&=\partial_x\{\partial_x\{H(\sigma)[x]\}g_i(x)\}g_i(x),\,
\sigma\in\mathbb{R}^m,\,x\in\mathbb{R}^n
\end{split}
\end{aligned}
\right.
\end{equation}
for each $i\in\{1,\dots,m\}$.  Recall that the standard rule of \index{stachastic
derivation} lead us to
SDE
\begin{equation}
\label{c:16}
d_t\hat{z}(t,x)
=\sum_{i=1}^m\partial_{\sigma_i}\{H(\sigma)[x]\}_{(\sigma=w(t))}\cdot dw_i(t)
+\frac{1}{2}\sum_{i=1}^m\partial^2_{\sigma_i}\{H(\sigma)[x]\}_{(\sigma=w(t))}dt,
\end{equation}
for any $t\in[0,T]$, $x\in\mathbb{R}^n$.  Rewritting the right hand side of
\eqref{c:16} (see \eqref{c:15}) we get SPDE of parabolic type given in
\eqref{c:14}.
\end{proof}

\begin{lemma}
\label{lem-c:4}
Assume the hypotheses \eqref{c:4a} and \eqref{c:4b} are fulfilled and consider
$\{\lambda=\psi(t,x):t\in[0,T],x\in\mathbb{R}^n\}$ defined in
Lemma~\ref{lem-c:2}. Then $u(t,x):=h(\psi(t,x))$, $t\in[0,T]$,
$x\in\mathbb{R}^n$, $h\in(\mathcal{C}^1_b\cap\mathcal{C}^2)(\mathbb{R}^n)$,
satisfies the following nonlinear SPDE
\begin{equation}
\label{c:17}
\left\{
\begin{aligned}
&\begin{split}
d_tu(t,x)
+&\langle\partial_xu(t,x),\sum_{i=1}^{m}\varphi_i(\psi(t,x)f_i(x))\rangle dt\\
&+\sum_{i=1}^m\langle\partial_xu(t,x),g_i(x)\rangle\hat{\circ}dw_i(t)=0,\,
t\in[0,T]
\end{split}\\
&u(0,x)=h(x)
\end{aligned}
\right.
\end{equation}
where the nonstandard \index{Fisk-Stratonovich integral} ``$\hat{\circ}$'' is computed
by
\[
h_i(t,x)\hat{\circ}dw_i(t)
=h_i(t,x)\cdot dw_i(t)-\frac{1}{2}\partial_xh_i(t,x)g_i(x)dt.
\]
\end{lemma}
The proof uses the same arguments as in Lemma \ref{lem-b:4} of
section~\ref{ssec:b.1}.

\begin{theorem}
\label{thm-c:1}
Assume that the vector fields
$\{f_1,\dots,f_m\}\subset(\mathcal{C}_{b}\cap\mathcal{C}^{1}_{b}\cap\mathcal{C}^{2})
(\mathbb{R}^n;\mathbb{R}^n)$,
$\{g_1,\dots,g_m\}\subset(\mathcal{C}^{1}_{b}\cap\mathcal{C}^{2})
(\mathbb{R}^n;\mathbb{R}^n)$ and scalar functions
$\{\varphi_1,\dots,\varphi_m\}\subset(\mathcal{C}^{1}_{b}\cap\mathcal{C}^{2})
(\mathbb{R}^n)$ fulfil the hypotheses~\ref{c:4a} and \ref{c:4b}

Consider the continuous and $\mathcal{F}_t$-\index{adapted process}
$\{\lambda=\psi(t,x)\in\mathbb{R}^n:t\in[0,T],x\in\mathbb{R}^n\}$ satisfying
integral equations~\eqref{c:7} (see Lemma~\ref{lem-c:2}). Then
$\{u(t,x):=h(\psi(t,x)):t\in[0,T],x\in\mathbb{R}^n\}$ fulfils nonlinear SPDE
of parabolic type~\eqref{c:17} (see Lemma~\ref{lem-c:4}) for each
$h\in(\mathcal{C}^{1}_{b}\cap\mathcal{C}^{2})(\mathbb{R}^n)$.
\end{theorem}

\subsection{Solution for (P2)}
\label{ssec-c:2}

As far as SDE~\eqref{c:5} is a non-markovian system, the evolution of a
functional $S(t,x):=Eh(\hat{x}_{\psi}(T;t,x))$, $t\in[0,T]$, $x\in\mathbb{R}^n$,
for each $h\in(\mathcal{C}^{1}_{b}\cap\mathcal{C}^{2})(\mathbb{R}^n)$ will be
described using the pathwise representation of the conditioned mean values
functional
\begin{equation}
\label{c:18}
v(t,x):=E\{h(\hat{x}_{\psi}(T;t,x))\mid\psi(t,x)\},\,
0\leq t<T,\,x\in\mathbb{R}^n.
\end{equation}

Here $\hat{x}_{\psi}(T;t,x)$ can be expressed using the following integral
representation
\begin{equation}
\label{c:19}
\hat{x}_{\psi}(T;t,x)=G(w(T)-w(t))\circ F[(T-t)\varphi(\psi(t,x))](x),\,
0\leq t<T,
\end{equation}
where $G(\sigma)[z]$ and $F(\sigma)[z]$,
$\sigma=(\sigma_1,\dots,\sigma_m)\in\mathbb{R}^m$, $x\in\mathbb{R}^n$, are
defined in (P1) (see~\eqref{c:6}) for $\varphi:=(\varphi_1,\dots,\varphi_m)$.
The right hand side of~\eqref{c:19} is a continuous mapping of the two
independent random variables, $z_1=w(T)-w(t)\in\mathbb{R}^m$ and
$z_2=\psi(t,x)\in\mathbb{R}^n$ ($\mathcal{F}_t$-measurable) for each
$0\leq t<T$, $x\in\mathbb{R}^n$.

Using the parameterized random variable
\begin{equation}
\label{c:20}
y(t,x;\lambda)=G(w(T)-w(t))\circ F[(T-t)\varphi(\lambda)](x),\,
0\leq t<T
\end{equation}
we may and do compute the functional $v(t,x)$ in~\eqref{c:18} by
\begin{equation}
\label{c:21}
v(t,x)=[Eh(y(t,x;\lambda))](\lambda=\psi(t,x)),\,
0\leq t<T,\,x\in\mathbb{R}^n.
\end{equation}
Here, the functional
\begin{equation}
\label{c:22}
u(t,x;\lambda)=Eh(y(t,x;\lambda)),\,
t\in[0,T],\,x\in\mathbb{R}^n,
\end{equation}
satisfies a backward parabolic equation (Kolmogorov`s equation) for each
$\lambda\in\mathbb{R}^n$ and rewrite \eqref{c:21} as follows,
\begin{equation}
\label{c:23}
v(t,x)=u(t,x;\psi(t,x)),\,0\leq t<T,\,x\in\mathbb{R}^n.
\end{equation}
In conclusion, the functional $S(t,x)=Eh(\hat{x}_{\psi}(T;t,x))$ can be
reprezented by
\begin{equation}
\label{c:24}
S(t,x)=E[E\{h(\hat{x}_{\psi}(T;t,x))\mid\psi(t,x)\}]=Eu(t,x;\psi(t,x))
\end{equation}
for any $0\leq t<T$, $x\in\mathbb{R}^n$, where
$\{u(t,x;\lambda):t\in[0,T],x\in\mathbb{R}^n\}$ satisfies the corresponding
backward parabolic equations with parameter $\lambda\in\mathbb{R}^n$,
\begin{equation}
\label{c:25}
\left\{
\begin{aligned}
&\partial_tu(t,x;\lambda)+\langle\partial_xu(t,x;\lambda),f(x,\lambda)\rangle
+\frac{1}{2}\sum_{i=1}^{m}\langle\partial^2_xu(t,x;\lambda)g_i(x),g_i(x)\rangle=0,\\
&u(T,x;\lambda)=h(x),\,
f(x,\lambda)=\sum_{i=1}^{m}\varphi_i(\lambda)f_i(x)
+\frac{1}{2}\sum_{i=1}^{m}[\partial_xg_i(x)]g_i(x).
\end{aligned}
\right.
\end{equation}
We conclude these remarks by a theorem.

\begin{theorem}
\label{thm-c:2}
Assume that the vector fields
$\{f_1,\dots,f_m\}\subset(\mathcal{C}_b\cap\mathcal{C}^1_b\cap\mathcal{C}^2)
(\mathbb{R}^n;\mathbb{R}^n)$,
$\{g_1,\dots,g_m\}\subset(\mathcal{C}^1_b\cap\mathcal{C}^2)
(\mathbb{R}^n;\mathbb{R}^n)$, and scalar functions
$\varphi=(\varphi_1,\dots,\varphi_m)\subset(\mathcal{C}^1_b\cap\mathcal{C}^2)
(\mathbb{R}^n)$ of SDE \eqref{c:5} fulfil the hypotheses~\eqref{c:4a}
and~\eqref{c:4b}.
Then the evolution of the functional
\begin{equation}
\label{c:26}
S(t,x):=Eh(\hat{x}_{\psi}(T;t,x)),\,t\in[0,T],\,x\in\mathbb{R}^n,\,
h\in(\mathcal{C}^1_b\cap\mathcal{C}^2)(\mathbb{R}^n)
\end{equation}
can be described as in \eqref{c:24}, where
$\{u(t,x;\lambda):t\in[0,T],x\in\mathbb{R}^n\}$ satisfies linear backward
parabolic equations \eqref{c:25}, for each $\lambda\in\mathbb{R}^n$.
\end{theorem}

\textbf{Final remark}. One may wonder about the meaning of the martingale
representation associated with the non-markovian functionals
$h(\hat{x}_{\psi}(T;t,x))$,
$h\in(\mathcal{C}^1_b\cap\mathcal{C}^2)(\mathbb{R}^n)$.
In this respect, we may use the parameterized functional
$\{u(t,x;\lambda):t\in[0,T],x\in\mathbb{R}^n\}$ fulfilling backward parabolic
equations \eqref{c:25}. Write
\begin{equation}
\label{c:27}
h(\hat{x}_{\psi}(T;t,x))=u(T,\hat{x}_{\psi}(T;t,x);\hat{\lambda}=\psi(t,x))
\end{equation}
and apply the standard rule of stochastic derivation associated with smooth
mapping $\{u(s,x;\hat{\lambda}):s\in[0,T],x\in\mathbb{R}^n\}$ and stochastic
process $\{\hat{x}_{\psi}(s;t,x):s\in[t,T]\}$.  We get
\begin{equation}
\label{c:28}
\begin{split}
h(\hat{x}_{\psi}(T;t,x))=&u(t,x;\hat{\lambda})
+\int_t^T(\partial_s+L_{\hat{\lambda}})(u)(s,\hat{x}_{\psi}(s;t,x);\hat{\lambda})ds\\
&+\sum_{i=1}^{m}\int_t^T\langle\partial_xu(s,\hat{x}_{\psi}(s;t,x);\hat{\lambda}),
g_i(x)\rangle dw_i(s),
\end{split}
\end{equation}
where $L_{\hat{\lambda}}(u)(s,x;\hat{\lambda})
:=\langle\partial_xu(s,x;\hat{\lambda}),f(x,\hat{\lambda})\rangle
+\frac{1}{2}\sum_{i=1}^{m}\langle\partial^2_xu(s,x;\hat{\lambda})g_i(x),g_i(x)\rangle$
coincides with parabolic operator in PDE \eqref{c:25}.
Using \eqref{c:25} for $\hat{\lambda}=\psi(t,x)$, we obtain the following
martingale representation
\begin{equation}
\label{c:29}
h(\hat{x}_{\psi}(T;t,x))=u(t,x;\psi(t,x))
+\sum_{i=1}^{m}\int_t^T\langle\partial_xu(s,\hat{x}_{\psi}(s;t,x);\hat{\lambda}),
g_i(x)\rangle\cdot dw_i(s),
\end{equation}
which shows that the standard constant in the markovian case is replaced by a
$\mathcal{F}_t$-measurable random variable $u(t,x;\psi(t,x))$.  In addition,
the backward evolution of stochastic functional
$\{Q(t,x):=h(\hat{x}_{\psi}(T;t,x)):t\in[0,T],x\in\mathbb{R}^n\}$ given in
\eqref{c:29} depends essentially on the forward evolution process
$\{\psi(t,x)\}$ for each $t\in[0,T]$ and $x\in\mathbb{R}^n$.

\section*{Bibliographical Comments}
The writing of this part has much in common with the references \cite{1} and \cite{11}.


\newpage
\addcontentsline{toc}{chapter}{Index}
\pagestyle{plain}
\printindex



\begin{thebibliography}{9}
\bibitem[1]{1} A. Friedman, \emph{Stochastic Differential Equations and Applications}, Academic Press vol. 1, 1975.
\bibitem[2]{2} A. Halanay, \emph{Differential Equations }, Ed. Didactica and Pedagogica, 1972.
\bibitem[3]{3} P. Hartman, \emph{Ordinary Differential Equations}, The Johns Hopking Univerisrt, John Wiley Sons, 1964.
\bibitem[4]{4} M. Gianquinta, S. Hildebrandt, \emph{Calculas of Variations}, vol. 1, Springer, 1996.
\bibitem[5]{5} S. Godounov, \emph{E'quations de la Physique Mathe'matique}, Nauka, Moskow, Translated mir, 1973.
\bibitem[6]{6} P. J. Olver, \emph{Applications of Lie Groups to Differential Equations}, Springer, 1986 (Graduate texts in mathematics; 107).
\bibitem[7]{7} L. Pontriaguine, \emph{Equations Differentielles Ordinaires}, Editions MIR, Moskow, 1969.
\bibitem[8]{8} R. Racke, \emph{Lectures on Nonlinear Evolution Equations}, Vieweg, 1992.
\bibitem[9]{9} G. Silov, \emph{Multiple Variable Real Functions}, Analysis, MIR, Moskow, 1975.
\bibitem[10]{10} S. L. Sobolev, \emph{Mathematical Physics Equations}, Nauka, Moskow, 1966.
\bibitem[11]{11} C. Varsan, \emph{Applications of Lie Algebras to Hyperbolic and Stochastic Differential equations}, Kluwer Academic Publishers, 1999.
\bibitem[12]{12} C. Varsan, \emph{Basic of Mathematical Physics Equations and Element of Differential Equations}, Ex, PONTO, Constanta, 2000.
\bibitem[13]{13} J. J. Vrabie, \emph{Differential Equations}, Matrix-Rom, 1999.

\end{thebibliography}
\end{document}